\documentclass[reqno]{amsart}
\usepackage{amsmath,amsthm,multicol}

\usepackage[cal=esstix]{mathalpha}

\usepackage{mathrsfs}

\usepackage{bm}
\usepackage{amsfonts,amssymb,calc}
\usepackage{graphicx}
\usepackage{pdfpages}
\usepackage[percent]{overpic}

\usepackage[
  top=1.25in,
  bottom=1.25in,
  left=1.2in,
  right=1.2in
]{geometry}

\usepackage{hyperref}
\usepackage[capitalize]{cleveref}
\usepackage{import,color,transparent}

\crefname{theorem}{Theorem}{Theorems}
\crefname{fact}{Fact}{Facts}
\crefname{note}{Note}{Notes}
\crefname{lemma}{Lemma}{Lemmas}
\crefname{alg}{Algorithm}{Algorithms}
\crefname{remark}{Remark}{Remarks}
\crefname{example}{Example}{Examples}
\crefname{prop}{Proposition}{Propositions}
\crefname{conj}{Conjecture}{Conjectures}
\crefname{cor}{Corollary}{Corollaries}
\crefname{definition}{Definition}{Definitions}
\crefname{Relation}{Relation}{Relations}
\crefname{equation}{\!\!}{\!\!} 

\definecolor{labelkey}{rgb}{0.8, 0.0, 0.0}

\input xy
\usepackage[all]{xy}
\xyoption{line}
\xyoption{arrow}
\xyoption{color}
\SelectTips{cm}{}

\usepackage{tikz}
\usetikzlibrary{shapes,snakes}
\usetikzlibrary{decorations.markings}
\usetikzlibrary{decorations.pathreplacing}
\usetikzlibrary{cd}
\tikzstyle directed=[postaction={decorate,decoration={markings,
    mark=at position #1 with {\arrow{>}}}}]

\newcommand{\hackcenter}[1]{
 \xy (0,0)*{#1}; \endxy}

\tikzset{->-/.style={decoration={
  markings,
  mark=at position #1 with {\arrow{>}}},postaction={decorate}}}

\tikzset{middlearrow/.style={
        decoration={markings,
            mark= at position 0.5 with {\arrow{#1}} ,
        },
        postaction={decorate}
    }
}

\usepackage{graphicx}
\usepackage{color}

\usetikzlibrary{shapes.geometric}

\newcommand{\ord}{\textup{ord}}

\theoremstyle{plain}
\newtheorem{theorem}{Theorem}

\newtheorem{corollary}[theorem]{Corollary}
\newtheorem{proposition}[theorem]{Proposition}
\newtheorem{lemma}[theorem]{Lemma}

\theoremstyle{definition}
\newtheorem{example}[theorem]{Example}
\newtheorem{definition}[theorem]{Definition}
\newtheorem{conjecture}[theorem]{Conjecture}
\theoremstyle{definition}
\newtheorem{remark}[theorem]{Remark}

\numberwithin{equation}{subsection}
\numberwithin{theorem}{subsection}

\newcommand{\refequal}[1]{\xy {\ar@{=}^{#1}
(-1,0)*{};(1,0)*{}};
\endxy}

\newcommand{\Hom}{{\rm Hom}}

\renewcommand{\to}{\rightarrow}

\newcommand{\con}{\ensuremath{con}}
\renewcommand{\exp}{\ensuremath{exp}}

\def\Web{{\mathbf{Web}}}

\def\gl{\mathfrak{gl}}

\def\Webaff{\mathbf{Web}^{\textup{aff}}_{\calA}}
\def\WebAaIntro{\mathbf{Web}_{\calA}}
\def\AffWebAaIntro{\mathbf{Web}^{\textup{aff}}_{\calA}}
\def\WebAaI{\mathbf{Web}_{\calA}}
\def\AffWebAaI{\mathbf{Web}^{\textup{aff}}_{\calA}}

\def\ThinAffWebAaI{\mathbf{Web}^{\textup{aff},1}_{\calA}}
\def\KThinAffWebAaI{{}_{\KK}\mathbf{Web}^{\textup{aff},1}_{\calA}}
\def\KKThinAffWebAaI{\mathbf{Web}^{\textup{aff},1}_{\KK \otimes \calA}}
\def\WebAaIthin{\mathbf{Web}^{\textup{aff, thin}}_{\calA}}

\newcommand{\catEnd}{\bm{\mathcal{End}}}

\def\a{\mathfrak{a}}

\usepackage{bbm}

\def\Z{{\mathbbm Z}}

\def\tr{\mathrm{tr}}%
\usepackage{bbm}

\def \KK {\mathbbm{K}}

\def \k {\mathbbm{k}}
\def \Z {\mathbbm{Z}}

\def \Ob{\operatorname{Ob}}

\def \Id {{\rm Id}}

\newcommand{\nc}{\newcommand}
\nc\rnc{\renewcommand}
\nc\Kar{\operatorname{Kar}}
\nc\modQ {{\mathbb Q}}
\nc\modZ {{\mathbb Z}}
\nc\simeqto{\overset{\simeq}{\longrightarrow }}
\nc\K{\mathcal {K}}
\nc\CC{\mathbf{C}}
\newcommand{\modgl}{\mathrm{mod}\text{-}\gl_n(A)}

\nc\qh{\mathcal{H}}
\nc\hbm{\mathcal{B}}
\nc\bu{\mathbf{u}}
\nc\bv{\mathbf{v}}
\nc\bZ{\mathbf{Z}}
\nc\theirs{\mathrm{theirs}}
\nc\ours{\mathrm{ours}}
\nc\hE{\mathcal{\hat E}}
\nc\bK{\mathbf{K}}
\nc\bw{\mathbf{w}}
\nc\bh{\mathbf{h}}
\nc\tE{\tilde{E} }
\nc\ba{\mathbf{a}}
\nc\bb{\mathbf{b}}
\nc\bbx{\mathbf{x}}
\nc\bby{\mathbf{y}}

\def\WebLamThin{\mathbf{Web}^{\Lambda,\textup{thin}}_{\mathcal{A}}}
\def\WebAffThin{\mathbf{Web}^{\textup{aff}, \textup{thin}}_{\mathcal{A}}}
\def\WebLam{\mathbf{Web}^{\Lambda}_{\mathcal{A}}}
\def\WebLamThin{\mathbf{Web}^{\Lambda,\textup{thin}}_{\mathcal{A}}}

\nc{\bs}{\mathbf{s}}
\nc{\bc}{\mathbf{c}}
\nc{\bL}{\mathbf{L}}
\nc{\bM}{\mathbf{m}}
\nc{\br}{\mathbf{r}}
\nc{\bbf}{\mathbf{f}}
\nc{\bbg}{\mathbf{g}}
\nc{\bbh}{\mathbf{h}}
\nc{\bt}{\mathbf{t}}
\nc{\bx}{\mathbf{x}}
\nc{\by}{\mathbf{y}}
\nc{\bp}{\mathbf{p}}
\nc{\bq}{\mathbf{q}}

\allowdisplaybreaks

\usepackage{amsbsy,nccmath}
\usepackage{todonotes}
\DeclareMathAlphabet{\mathpzc}{OT1}{pzc}{m}{it}

\newcommand{\End}{\operatorname{End}}

\renewcommand{\SS}{\mathfrak{S}}
\newcommand{\bi}{\mathbf{i}}
\newcommand{\bj}{\mathbf{j}}
\newcommand{\bk}{\mathbf{k}}

\newcommand{\BasisB}{\mathtt{B}}
\newcommand{\BasisP}{\mathtt{P}}
\newcommand{\basisb}{\mathtt{b}}

\newcommand{\bnu}{\boldsymbol{\nu}}
\nc{\bom}{\boldsymbol{\omega}}
\nc{\bsi}{\boldsymbol{\sigma}}

\nc{\bmu}{\boldsymbol{\mu}}
\nc{\bz}{\mathbf{z}}

\newcommand{\Cliff}{\operatorname{Cl}}

\newcommand{\bd}{\mathbf{d}}

\newcommand{\0}{\bar{0}}
\renewcommand{\1}{\bar{1}}
\renewcommand{\WebAaI}{\mathbf{Web}_{\calA}}
\renewcommand{\AffWebAaI}{\mathbf{Web}^{\textup{aff}}_{\calA}}
\newcommand{\nakayamai}{i}
\newcommand{\nakayamaj}{j}
\newcommand{\nakayamak}{k}
\newcommand{\nakayamal}{\ell}
\newcommand{\nakayamaf}{{}_{1}\hspace{-0.4mm}f}
\newcommand{\nakayamag}{{}_{1}g}
\newcommand{\nakayamamu}{{}_{1}\mu}
\newcommand{\nakayamab}{{}_{1}\hspace{-0.2mm}b}
\newcommand{\nakayamac}{{}_{1}\hspace{-0.3mm}c}

\newcommand{\BasisC}{\mathtt{C}}

\newcommand{\fh}{\mathfrak{h}}
\newcommand{\fb}{\mathfrak{b}}
\newcommand{\fu}{\mathfrak{u}}

\newcommand{\gr}{\operatorname{gr}}
\newcommand{\calA}{\mathcal{A}}
\newcommand{\calM}{\mathcal{M}}
\newcommand{\KAffWebAI}{\mathbf{Web}^{\textup{aff}}_{\KK \otimes \calA}}
\newcommand{\id}{\Id}
\newcommand{\catC}{\bm{\mathcal{C}}}
\newcommand{\catD}{\bm{\mathcal{D}}}

\title[Superalgebra deformations of web categories: Affine and cyclotomic webs]{
Superalgebra deformations of web categories:\\ Affine and cyclotomic webs}

\begin{document}
\setcounter{tocdepth}{2}

\author{Nicholas Davidson}
\email{davidsonnj@cofc.edu}
\address{Department of Mathematics\\ College of Charleston \\ Charleston, SC, USA}

\author{Jonathan R. Kujawa}
\email{kujawaj@oregonstate.edu}
\address{Department of Mathematics\\ Oregon State University \\ Corvallis, OR, USA}

\author{Robert Muth}
\email{muthr@duq.edu}
\address{Department of Mathematics and Computer Science \\ Duquesne University \\ Pittsburgh, PA, USA}

\date{\today}

\begin{abstract}
Let \(\k\) be a characteristic zero domain.  
We define and study a diagrammatic monoidal \(\k\)-linear supercategory \(\AffWebAaI\) associated to any locally unital Frobenius \(\k\)-superalgebra \(A\).
This category can be viewed variously as an {\em affinization} of the finite web category \(\WebAaI\) previously defined by the authors and Zhu, as a {\em thickening} of the degenerate affine wreath product algebras defined by Savage, or as a {\em Frobenius deformation} of affine web categories defined by Song and Wang. 
We show that there is an asymptotically faithful family of functors from \(\AffWebAaI\) to the monoidal supercategory of endofunctors of \(\gl_n(A)\)-modules for every $n \geq 1$, and use this to establish a basis of `decorated double coset diagrams' for morphism spaces in \(\AffWebAaI\).   We also define and establish basis results for the cyclotomic quotient category \(\WebLam\) associated with a cyclotomic datum $\Lambda$.  \end{abstract}

\maketitle

\setcounter{tocdepth}{1}
\tableofcontents

\section{Introduction}

\subsection{Overview}  
  The goal of the present paper is to pull together several strands of modern representation theory.  Throughout, we work over an characteristic zero domain $\k$.

The first strand is web categories. These are $\k$-linear diagrammatic categories first introduced to describe certain categories of representations for Lie algebras \cite{CKM,Kuperberg}.  They have since played a prominent role in Lie theory, algebraic combinatorics, combinatorial representation theory, low-dimensional topology, cluster algebras, and other areas, with the references too numerous to list.   

The second strand is the degenerate affine wreath product algebra.  Wreath product algebras naturally generalize the symmetric groups and share many of their nice representation theoretic features.  In the past few decades it became clear the symmetric groups are part of a larger story that is revealed by viewing them as a cyclotomic quotient of the degenerate affine Hecke algebra.  In particular, this point of view leads to categorical representation theory, categorification, and the algebras introduced by Khovanov--Lauda and Rouquier.  These developments motivate recent work on degenerate affine wreath product algebras and related structures (e.g., see \cite{BSW,Khovanov-Heisenberg,LaiMinets,LaiNakanoXiang,RossoSavage,SavageAff,WanWang} and references therein).

The third strand is Schurification.  Given an associative $\k$-superalgebra $A = A_{\0}\oplus A_{\1}$ and a subalgebra $\a \subseteq A_{\0}$, one can construct a $\k$-superalgebra $T_{(A,\a)}(n,d)$ for any pair of natural numbers $n$ and $d$.  Let $V=A^{\oplus n}$.  Then $T_{(A,\a)}(n,d)$ is a $\k$-form inside the algebra of endomorphisms of the wreath product $A \wr \SS_{d}$ acting in the obvious way on $V^{\otimes d}$, where the choice of subalgebra $\a$ controls the $\k$-form one obtains.  Unlike the full endomorphism algebra, the $\k$-superalgebra $T_{(A,\a)}(n,d)$ often has the well-behaved and interesting representation theory one expects from something which deserves to be called the Schur algebra associated to $A$ (e.g., see \cite{KM2,KM,KMqh,KleschevWeinschelbaum}). The importance of this algebra is underscored by the fact that, as conjectured by Turner, the RoCK blocks of the symmetric group $\SS_{d}$ over a field $\mathbb{F}$ of positive characteristic are Morita equivalent to $T_{(A,\a )}(n,d)\otimes_{\Z}\mathbb{F}$ \cite{EK,EKRoCK} when $A$ is the type {\tt A} zigzag superalgebra defined over $\k = \Z$, see \cref{zigzagcat}. Thus, these algebras control a significant portion of the modular representation theory of the symmetric groups.  Similar statements are conjecturally true for the spin representations of the symmetric groups \cite[Conjecture 1]{KleLiv} and the classical Schur algebras \cite[Conjecture 7.58]{KMqh}.

\subsection{Finite webs}   Let $A = A_{\0} \oplus A_{\1}$ be a locally unital superalgebra over $\k$, with distinguished idempotents \(I \subseteq A\).  
Let $\a \subseteq  A_{\0}$ be locally unital even subalgebra, and assume there exists a $\k$-basis for $\a$ that can be extended to a $\k$-basis of $A$.  This is, of course, automatic if $\k$ is a field and is a mild assumption in general.  Write $\calA$ for the data $(A,\a)_I$.

In \cite{DKMZ} we introduced the web category $\WebAaIntro$ associated to this data.  This is a diagrammatic $\k$-linear monoidal category defined via generators and relations.  The objects are tuples of symbols of the form \(i^{(x)}\), where \(i \in I\) and \(x \in \Z_{\geq 0}\), with concatenation of words providing the monoidal product.  The generating morphisms of \(\WebAaIntro\) are given by the diagrams:
\begin{align}\label{introgen1}
\hackcenter{
\begin{tikzpicture}[scale=.8]
  \draw[ultra thick,blue] (0,0)--(0,0.2) .. controls ++(0,0.35) and ++(0,-0.35) .. (-0.4,0.9)--(-0.4,1);
  \draw[ultra thick,blue] (0,0)--(0,0.2) .. controls ++(0,0.35) and ++(0,-0.35) .. (0.4,0.9)--(0.4,1);
      \node[above] at (-0.4,1) {$ \scriptstyle i^{\scriptstyle (x)}$};
      \node[above] at (0.4,1) {$ \scriptstyle i^{\scriptstyle (y)}$};
      \node[below] at (0,0) {$ \scriptstyle i^{\scriptstyle (x+y)} $};
\end{tikzpicture}},
\qquad
\qquad
\hackcenter{
\begin{tikzpicture}[scale=.8]
  \draw[ultra thick,blue ] (-0.4,0)--(-0.4,0.1) .. controls ++(0,0.35) and ++(0,-0.35) .. (0,0.8)--(0,1);
\draw[ultra thick, blue] (0.4,0)--(0.4,0.1) .. controls ++(0,0.35) and ++(0,-0.35) .. (0,0.8)--(0,1);
      \node[below] at (-0.4,0) {$ \scriptstyle i^{ \scriptstyle (x)}$};
      \node[below] at (0.4,0) {$ \scriptstyle i^{ \scriptstyle (y)}$};
      \node[above] at (0,1) {$ \scriptstyle i^{ \scriptstyle (x+y)}$};
\end{tikzpicture}},
\qquad
\qquad
\hackcenter{
\begin{tikzpicture}[scale=.8]
  \draw[ultra thick,red] (0.4,0)--(0.4,0.1) .. controls ++(0,0.35) and ++(0,-0.35) .. (-0.4,0.9)--(-0.4,1);
  \draw[ultra thick,blue] (-0.4,0)--(-0.4,0.1) .. controls ++(0,0.35) and ++(0,-0.35) .. (0.4,0.9)--(0.4,1);
      \node[above] at (-0.4,1) {$ \scriptstyle j^{ \scriptstyle (y)}$};
      \node[above] at (0.4,1) {$ \scriptstyle i^{ \scriptstyle (x)}$};
       \node[below] at (-0.4,0) {$ \scriptstyle i^{ \scriptstyle (x)}$};
      \node[below] at (0.4,0) {$ \scriptstyle j^{ \scriptstyle (y)}$};
\end{tikzpicture}},
\qquad
\qquad
\hackcenter{
\begin{tikzpicture}[scale=.8]
  \draw[ultra thick, blue] (0,0)--(0,0.5);
   \draw[ultra thick, red] (0,0.5)--(0,1);
   \draw[thick, fill=yellow]  (0,0.5) circle (7pt);
    \node at (0,0.5) {$ \scriptstyle f$};
     \node[below] at (0,0) {$ \scriptstyle i^{ \scriptstyle (z)}$};
      \node[above] at (0,1) {$ \scriptstyle j^{ \scriptstyle (z)}$};
\end{tikzpicture}},
\end{align}
for \(i,j \in I\), \(x,y \in \Z_{\geq 0}\), \(z \in \Z_{>0}\), and \(f \in jAi\) if $z=1$ and $f \in  j \a i$ if $z \geq 2$.  We call these the split, merge, crossing and coupon morphisms, respectively.  Our convention is to read diagrams from bottom to top.  Composition is given by vertical concatenation and the monoidal product is given by horizontal concatenation.  Morphisms in $\WebAaIntro$ are then $\k$-linear combinations of diagrams built by repeated concatenation of generating diagrams, subject to a fairly simple set of local relations. 

To understand a $\k$-linear monoidal category given by generators and relations, one would like an explicit $\k$-basis for each morphism space. Doing so establishes that the morphism spaces have the expected size and that they do not contain torsion.  Having an explicit basis also facilitates computations and proofs.  A standard way of obtaining such bases in this setting is to define a representation of the category that is sufficiently faithful.

We did this for $\WebAaIntro$ in \cite{DKMZ}.  In that paper, we showed that for every $n \geq 1$, there exists a  functor
\[
G_{n}: \WebAaIntro \to \modgl, 
\] where $\modgl$ is the category of right supermodules for the Lie superalgebra $\gl_{n}(A)$.   On objects, $G_n(i^{(x)}) = S^x({}_iV)$, where ${}_iV = (iA)^{\oplus n}$ is the natural right (\(i\)-colored) supermodule for $\gl_n(A)$, and $S^x({}_iV)$ is its $x$th symmetric power.  On generating morphisms, the merge and split webs  are sent to homomorphisms induced by the usual product and coproducts on the symmetric superalgebra of ${}_iV$, while the coupon labeled by $f\in jAi$ is sent to the homomorphism given by left multiplication by $f$.  In \cref{Gthm} below we recall this functor more precisely.  

We also showed that the family of functors $\{G_n \}_{n \geq 1}$ is \emph{asympotically faithful} in the sense that, for any nonzero morphism $\phi$ in $\WebAaIntro$, there exists an $N$ so that $G_n(\phi) \neq 0$ for $n \geq N$.  This allowed us to determine a $\k$-basis for the morphism spaces of $\WebAaIntro$.  We also proved that $T_{(A,\a)}(n,d)$ appears naturally in $\WebAaIntro$ as an explicit sum of morphism spaces.  In particular, this establishes a diagrammatic description for $T_{(A,\a)}(n,d)$.  In addition, we obtained a Howe duality for the pair $\left( \gl_{m}(A), \gl_{n}(A)\right)$ under the assumption that $\k$ is a field and $A$ is semisimple, generalizing Howe dualities that appear in the literature.

\subsection{Affine webs}
Assume that $\calA = (A,\a)_I$ is as above.  In addition, assume that $A$ is a Frobenius superalgebra (we in fact require the slightly more strict assumption that \((A,\a)_I\) constitutes a `great pair'; see \cref{S:PlurimoreFrobenius} for further details and examples). 
In the present paper we introduce and study a (degenerate) affine web category $\AffWebAaIntro$ attached to the data $\calA$, along with their cyclotomic quotients.   This amounts to adding additional generators to $\WebAaIntro$,
\begin{align}\label{introgen2}
\hackcenter{
\begin{tikzpicture}[scale=.8]
  \draw[ultra thick, blue] (0,0)--(0,0.5);
   \draw[ultra thick, blue] (0,0.5)--(0,1);
   \draw[thick, fill=black]  (0,0.5) circle (5pt);
     \node[below] at (0,0) {$ \scriptstyle i^{(z)}$};
      \node[above] at (0,1) {$ \scriptstyle i^{(z)}$};
\end{tikzpicture}},
\end{align} for every $i \in I, z \in \Z_{\geq 1}$.  We call these generators `affine dots'. There are, of course, additional local relations to describe how these generators interact with the generating morphisms from $\WebAaIntro$.  See \cref{defwebaa} for details. Morphisms in \(\AffWebAaI\) are \(\k\)-linear combinations of diagrams which locally look like \cref{introgen1,introgen2}.
For example, if \(i,j \in I\), \(\alpha \in iAj\), \(\beta \in i \a i\) and \(\gamma \in j \a i\), then  
a representative morphism in \(\AffWebAaI(i^{(2)} j^{(4)} i^{(2)}, j^{(5)}i^{(3)})\) is shown below.
\begin{align*}
\\
\hackcenter{}
\hackcenter{
\begin{overpic}[height=27mm]{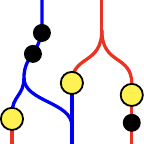}
  \put(8,-1){\makebox(0,0)[t]{$\scriptstyle i^{(\hspace{-0.2mm}2\hspace{-0.2mm})}$}}    
    \put(50,-1){\makebox(0,0)[t]{$\scriptstyle j^{(\hspace{-0.2mm}4\hspace{-0.2mm})}$}}    
      \put(92,-1){\makebox(0,0)[t]{$\scriptstyle i^{(\hspace{-0.2mm}2\hspace{-0.2mm})}$}}    
      \put(29,101){\makebox(0,0)[b]{$\scriptstyle j^{(\hspace{-0.2mm}5\hspace{-0.2mm})}$}}  
      \put(71,101){\makebox(0,0)[b]{$\scriptstyle i^{(\hspace{-0.2mm}3\hspace{-0.2mm})}$}}  
       \put(8.2,17){\makebox(0,0)[]{$\scriptstyle \gamma$}}  
       \put(50,41.8){\makebox(0,0)[]{$\scriptstyle \alpha$}}  
        \put(91.5,34){\makebox(0,0)[]{$\scriptstyle \beta$}}  
       \put(4,34){\makebox(0,0)[]{$\scriptstyle j^{(\hspace{-0.2mm}2\hspace{-0.2mm})}$}}  
        \put(32,20){\makebox(0,0)[]{$\scriptstyle j^{(\hspace{-0.2mm}3\hspace{-0.2mm})}$}}   
           \put(61,30){\makebox(0,0)[]{$\scriptstyle j^{(\hspace{-0.2mm}1\hspace{-0.2mm})}$}}   
           \put(55,63){\makebox(0,0)[]{$\scriptstyle i^{(\hspace{-0.2mm}1\hspace{-0.2mm})}$}}   
\end{overpic}
}
\;
+
\;\;
\hackcenter{
\begin{overpic}[height=27mm]{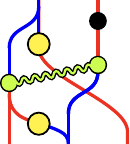}
  \put(8,-1){\makebox(0,0)[t]{$\scriptstyle i^{(\hspace{-0.2mm}2\hspace{-0.2mm})}$}}    
    \put(50,-1){\makebox(0,0)[t]{$\scriptstyle j^{(\hspace{-0.2mm}4\hspace{-0.2mm})}$}}    
      \put(92,-1){\makebox(0,0)[t]{$\scriptstyle i^{(\hspace{-0.2mm}2\hspace{-0.2mm})}$}}    
      \put(29,101){\makebox(0,0)[b]{$\scriptstyle j^{(\hspace{-0.2mm}5\hspace{-0.2mm})}$}}  
      \put(71,101){\makebox(0,0)[b]{$\scriptstyle i^{(\hspace{-0.2mm}3\hspace{-0.2mm})}$}}  
       \put(27,13.5){\makebox(0,0)[]{$\scriptstyle \alpha$}}  
        \put(27,69){\makebox(0,0)[]{$\scriptstyle \gamma$}}  
       \put(19,28){\makebox(0,0)[]{$\scriptstyle i^{(\hspace{-0.2mm}1\hspace{-0.2mm})}$}}  
           \put(59,22){\makebox(0,0)[]{$\scriptstyle j^{(\hspace{-0.2mm}3\hspace{-0.2mm})}$}}   
                  \put(8,84){\makebox(0,0)[]{$\scriptstyle j^{(\hspace{-0.2mm}3\hspace{-0.2mm})}$}}   
\end{overpic}
}
\\
\end{align*}
The squiggly green decoration is called a `teleporter' and is shorthand for a certain collection of coupon morphisms (see \cref{teleportsec}). Here, nomen est omen: these teleporters satisfy remarkable properties (\cref{telerel}) inherited from the Frobenius structure of \(A\). Our teleporters can be viewed as thickened analogues of those in \cite{BSW}.

\subsection{Main results}  Let $\catEnd (\modgl)$ be the category of endofunctors of the category of $\gl_{n}(A)$-modules.  For $n \geq 1$, in \cref{Hthm} we establish a monoidal functor, 
\[
H_{n}: \AffWebAaIntro \to \catEnd (\modgl).
\] On generating objects, $H_{n}(x) = - \otimes S^{x}({}_iV)$.  On generating morphisms the split, merge, crossing and coupon go to the obvious natural transformations that act on the symmetric powers, while the affine dots go to a map defined using a `thick' Casimir tensor.  We call this the \emph{(symmetric) defining representation} of $\AffWebAaIntro$.   Moreover, we prove in Theorem~\ref{defrepind77} that this family of functors is asymptotically faithful, allowing us to determine a $\k$-basis of `decorated double coset diagrams' for the morphism spaces of $\AffWebAaIntro$ in \cref{affwebbasis}.

We also consider cyclotomic quotients of $\AffWebAaIntro$.  Given a {\em regular} cyclotomic datum \(\Lambda\) of level \(\ell\), the cyclotomic quotient category \(\WebLam\) is defined in \cref{cycquosecmain}, after which we establish that the morphism spaces are free $\k$-modules with a basis consisting of those diagrams with fewer than \(\ell\) affine dots on a given strand.  The regularity condition on our cyclotomic data is inconvenient to verify in specific examples, but it is required to establish our basis results when working integrally (over $\k$).  After proving our basis results, we finish \cref{cycquosecmain} by establishing the regularity condition holds in several important cases, including the case where \(\k\) is a field of characteristic zero, or \(\a = A_{\bar 0}\), or \(A\) is a zigzag algebra.  Since these cover the examples of greatest interest to the authors, we leave the general case of when the regularity holds as an open question.

\subsection{Related works} 
As noted in this introduction, \(\AffWebAaI\) can be be considered as an \emph{affinization} of the finite web category \(\WebAaI\), and up to Morita equivalence, {\em affinization} of the Schurification \(T_{(A,\a)}(n,d)\). In \cite{SavageAff}, Savage constructed {\em (degenerate) affine wreath product algebras} associated to Frobenius algebras \(A\), and their cyclotomic quotients. Our categories \(\AffWebAaI\) and \(\WebLam\) can alternately be viewed as a {\em thickening} of these objects.  In particular, Savage's constructions live inside these categories as thin-strand subcategories, see \cref{degenaffcon,degencycconn}.

In independent work, Song--W.\ Wang introduce the affine web category and affine Schur category in \cite{SongWang-Schur,SongWang-webs} and study their relation to cyclotomic Schur algebras and finite W-algebras.  Their starting point is the same as ours in the special case $A=\k$, but we have different goals in mind.  For example, they only give the (skew-symmetric) defining representation under the assumption that $\k$ is a field; conversely, we do not consider connections to finite $W$-algebras.  Shen--Song--W.\ Wang extend some of these results to the quantum setting in \cite{ShenSongWang-quantumwebs}.  Likewise, generalizing work from \cite{BrKu,BCK}, Song--X.\ Wang independently introduce type $Q$ affine webs in \cite{SongWang-Qwebs} which corresponds to the current paper when  $A=\Cliff_{1}$,  the Clifford superalgebra on one generator. On a related front, in the forthcoming work \cite{BrundanIvanov}, Brundan-Ivanov study diagrammatics for the degenerate affine Schur category modified to incorporate symmetric polynomials on thick strands, and study the diagrammatic image of the Yangian associated to \(\mathfrak{gl}_n(\k)\).
Other related works appear in \cite{BSW,LaiMinets,MaksimauMinets}, see below.

\subsection{Further directions}
This paper establishes the foundations of affine $A$-webs and their cyclotomic qoutients, and sets the stage for a number of areas of research.

As discussed above, finite webs allow for a diagrammatic description of algebras which are Morita equivalent to the RoCK blocks of the symmetric groups and (conjecturally) related structures in positive characteristic.  This paper is part of a program to prove these conjectures as well as suitably modified versions at affine and higher cyclotomic levels.  It is worth mentioning that the generators-and-relations presentation and $\k$-bases for $\AffWebAaIntro$ and its cyclotomic quotients established here will be an essential ingredient in that work. In particular, we propose the following, which generalizes the Morita equivalence of \cite[Corollary 6.17]{KMaffzig} to the arbitrary characteristic setting.

\begin{conjecture}\label{KLRconj}
Let \({\tt X} \in \{{\tt A}_n, {\tt D}_n, {\tt E}_{6,7,8}\}\), and take \(\mathcal{A}\) to be the data of the zigzag algebra \(\bar{Z}({\tt X})\) as described in \cref{zigzagcat}. Then the affine web category \(\AffWebAaI\) is equivalent to the category \(\bigoplus_{d =0}^\infty C_{d\delta} \textup{-mod}\) of finitely-generated {\em imaginary KLR} modules in affine type \({\tt X}^{(1)}\) (see \cite[\S2.6]{KMaffzig}).
\end{conjecture}

Work related to \cref{KLRconj} has appeared elsewhere in the literature. The case \({\tt A}_1\) was solved in \cite{MaksimauMinets}, and we expect our \(\AffWebAaI\) category should be connected to the {\em curve Schur algebras} \(\mathcal{S}^{\mathbb{P}^1}_n\) defined therein when one takes \(A = \k[x]/(x^2)\), \(\a = \k\). In \cite{LaiMinets}, the authors have defined {\em coil Schur algebras} which also conjecturally describe imaginary KLR strata---we expect that (for certain parameter choices),  our \(\AffWebAaI\) provides a diagrammatic, generators-and-relations presentation for their algebras up to Morita equivalence (see \cite[Conjecture 8.3, Remark 6.11]{LaiMinets}). It would be interesting to fully delineate the connections to all the objects above.

An integral codeterminant basis is presented for finite webs for $A=\k$ in \cite{BEPO}.  See \cite{Elias-LightLadders} for a similar basis.  These bases, in particular, show these categories have the structure of an object-adapted cellular category as defined in \cite{EliasLauda}.  In light of the results of \cite{KMqh} --- where it is shown that if $A$ is based quasi-hereditary, then $T_{(A,\a)}(n,d)$ is again based quasi-hereditary via a generalized codeterminant basis --- it is natural to speculate that the bases of \cite{BEPO,Elias-LightLadders} should generalize to $\AffWebAaIntro$ and its cyclotomic quotients when $\calA$ is based quasi-hereditary.

When $A$ is Frobenius, one can also define oriented webs.  This would be a thick version of the Heisenberg category studied in \cite{BSW}. As the authors discuss in Remark~10.8 of \emph{loc.\  cit.}, such a thick calculus could be helpful in computing Grothendieck rings when the base ring is an integral domain.  Another interesting problem would be to define quantized affine Frobenius webs using the quantum wreath products from \cite{LaiNakanoXiang}.

\subsection{Acknowledgements}  The main results of this paper were obtained by the authors while supported by the American Institute of Mathematics SQuaRE program.  We are pleased to thank them for providing a productive environment for our collaboration.  The second author was partially supported by the Simons Foundation (Grants No.\ 963912 and 525043).  The third author was partially supported by an AMS-Simons PUI Research Enhancement Grant, and by ICERM while in residence for the program Categorification and Computation in Algebraic Combinatorics. 

\section{Frobenius superalgebras}\label{S:PlurimoreFrobenius}

\subsection{Frobenius great pairs}

Let $\k$ be an characteristic zero domain viewed as an associative superalgebra concentrated in parity $\0 \in \Z_{2}=\Z /2\Z$.  Let $\KK$ be the field of fractions for $\k$, also viewed as a superalgebra concentrated in parity $\0$.  Unless otherwise noted, unlabelled tensor products of $\k$-supermodules are over $\k$: $M \otimes N = M \otimes_{\k} N$.

The prefix ``super'' indicates a $\Z_{2}$-grading and that signs are introduced according to the ``sign rule'' for graded objects.   We often leave the prefix ``super'' implicit.  We assume the reader is familiar with the basics of (locally unital) associative $\k$-superalgebras, supermodules, supermodule homomorphisms, (monoidal) supercategories, supernatural transformations, etc., and only remark as needed to set our notations and conventions.

\subsubsection{Frobenius superalgebras}

  Let \(A = A_{\0} \oplus A_{\1}\) be a locally unital \(\k\)-superalgebra with distinguished idempotent set \(I \subseteq A_{\0 }\). In particular, $A = \oplus_{i,j \in I} jAi$.  By definition, a basis for a locally unital $\k$-superalgebra $A$ will be a homogeneous \(\k\)-basis \({}_j\BasisB_{i}\) for $jAi$, for all \(i,j \in I\). We say $A$ is \emph{locally finite} if \({}_j\BasisB_{i}\) is a finite set for all \(i,j \in I\).  Unless stated otherwise, all locally unital superalgebras will be assumed to be locally finite.  Write $\BasisB=\sqcup_{i,j \in I} ({}_j\BasisB_{i})$. Write $\BasisB_{\0}$, $\BasisB_{\1}$ for the even and odd elements of $\BasisB$, respectively.  For $k \in I$ write ${}_{k}\BasisB$ for $\sqcup_{i} \left( {}_{k}\BasisB_{i}\right)$ and write $\BasisB_{k}$ for $\sqcup_{j} \left( {}_{j}\BasisB_{k}\right)$.

By definition, an automorphism of a locally unital superalgebra \(A\) is an even superalgebra isomorphism $\eta: A \to A$ that fixes the distinguished idempotents of $A$.   For such an automorphism \(\eta\) and a right \(A\)-module \(M\), we define the right \emph{\(\eta\)-twisted} module \(M^\eta\) to be \(M\) as a \(\k\)-module, with \(A\)-action given by \(m \cdot f = m\eta(f)\) for all $f \in A$ and $m \in M$. Likewise, for a left $A$-module $M$ define the left \emph{$\eta$-twisted} module ${}^{\eta}M$ via the \(A\)-action given by \(f \cdot m = \eta(f)m\).

Let $A$ be a locally unital $\k$-superalgebra and let \(A^* = \bigoplus_{i,j \in I} (jAi)^* = \bigoplus_{i,j \in I} \Hom_\k(jAi, \k)\) denote the (local) dual.  Consider \(A^*\) as an \(A\)-bimodule via
\begin{equation}\label{E:dual-action}
(f \cdot \theta \cdot g)(h) := (-1)^{\bar f \bar \theta + \bar f \bar g + \bar f \bar h}\theta(ghf),
\end{equation}
for homogneous \(\theta \in A^*\) and homogeneous \(f,g,h \in A\).  Using this formula it is an easy calculation to show $j(A^{*})i = (i A j)^{*}$ for all $i,j \in I$.

Here and below, when given a homogeneous element $m$ of a $\k$-supermodule $M= M_{\0}\oplus M_{\1 }$, we write $\overline{m} \in \Z_{2}$ for the parity of $m$.  Note, too, that we often give a formula only on homogeneous elements and leave it understood that the general case is given by linearity.

\begin{definition}\label{D:Frobenius-Superalgebra}
If a locally unital superalgebra \(A\) admits an even $\k$-linear map \(\tr:A \to \k\) and an automorphism \(\psi: A \to A\) of locally unital superalgebras such that the map 
\begin{equation}\label{E:trace-condition3}
\phi_{\tr} : A \to {}^{\psi}\left(A^{*} \right), \quad f \mapsto \left( \phi_{tr}(f) : g \mapsto \tr(fg) \right).
\end{equation} is a parity preserving \(A\)-bimodule isomorphism, 
then we say the data $(A, \tr , \psi )$ makes $A$ a \emph{Frobenius superalgebra}.
\end{definition}

For a Frobenius superalgebra, we call $\tr$ the \emph{trace map} and $\psi$ the \emph{Nakayama automorphism}.  For brevity in upcoming formulas, write ${}_{t}a$ for $\psi^{-t}(a)$ for all $t \in \Z$ and $a \in A$.

A straightforward check verifies that $(A, \tr , \psi )$ makes $A$ a Frobenius superalgebra if and only if the  map $\varphi_{\tr} := \phi_{\tr} \circ \psi^{-1}$  given by
\[\varphi_\tr : A \to \left( A^*\right)^{\psi^{-1}}, \qquad f \mapsto (\varphi_\tr(f): g \mapsto (-1)^{\bar f \bar g}\tr(gf)).
\] is an $A$-bimodule isomorphism.  Our definition of Frobenius superalgebra is an extension of the definition used in \cite{SavageAff} to the locally unital setting.  Many of the basic results about these algebras can be found there.   For example, the following result is standard.

\begin{lemma} \label{L:nakayama} Say $A$ is a locally unital $\k$-superalgebra such that $(A, \tr , \psi)$ is a Frobenius superalgebra. Then,
\[
\tr \left(fg \right) = (-1)^{\bar{f}\bar{g}}\tr \left( g\psi(f)\right) =(-1)^{\bar{f}\bar{g}}\tr \left( \psi^{-1}(g)f\right),
\] for homogeneous $f,g \in A$.
\end{lemma}

\subsubsection{Results on Frobenius superalgebras}

Assume $(A, \tr , \psi)$ is a Frobenius $\k$-superalgebra.  Let \(\BasisB\) be a basis for $A$. Then $A^*$ has a dual basis $\{b^* \mid b \in \BasisB\}$ defined by $b^*(c) = \delta_{b,c}$ for all $b,c \in \BasisB$.  Here $\delta_{b,c}$ is the Kronecker delta function; more generally, if $P$ is a proposition, then we write $\delta_{P}$ for the function that is $1$ if $P$ is true, and $0$ if $P$ is false.

If 
\[
b^\vee := \phi_{\tr}^{-1} (b^*)
\]
for all $b \in \BasisB$, then the set $\BasisB^{\vee} = \{b^\vee \mid b \in \BasisB\}$ forms a homogeneous basis for $A$ satisfying $\tr(b^\vee c) = \delta_{b,c}$ for all \(b,c \in \BasisB\). We call $\BasisB^{\vee}$ the \emph{(left) dual basis} for $A$.

Note that the operation $c \mapsto  c^{\vee}$ is dependent on the choice of a basis.  We sometimes leave this choice implicit. For example, when we write $(b^{\vee})^{\vee}$ in the next lemma and later, this should be understood as the dual of $b^{\vee}$ with respect to the basis $\BasisB^{\vee}$.  That is, it is the element of $A$ determined by $\tr \left(\left( (b^{\vee})^{\vee} \right)\left( c^{\vee}\right)\right)=\delta_{b^{\vee},c^{\vee}}=\delta_{b,c}$ for all $c \in \BasisB$.

The following properties are straightforward to check.

\begin{lemma}\label{FrobFacts}  Let $(A, \tr , \psi)$ be a Frobenius superalgebra with basis $\BasisB$.  The following statements are true.
\begin{enumerate}
\item The trace function satisfies
\begin{align*}
\tr(jAi) = \begin{cases}
\k & \textup{if } j = i;\\
0 & \textup{otherwise}.
\end{cases}
\end{align*}
\item The map \(\varphi_\tr\) restricts to a linear isomorphism \(\varphi_{\tr} : jAi \xrightarrow{\sim} (i Aj)^*\) for all \(i,j \in I\).
\item The set \(({}_j \BasisB_i)^\vee = \{b^\vee \mid b \in {}_j \BasisB_i\}\) is a \(\k\)-basis for \(iAj\).
\item If $I \subseteq \BasisB$, then for all \(b \in \BasisB\), \(\tr(b^\vee) = \delta_{b \in I}\).
\item For all \(b \in \BasisB\), \((b^\vee)^\vee = (-1)^{\bar b}\psi^{-1}(b)\).
\item For all $b \in \BasisB$, $\psi (b^{\vee}) = \psi (b)^{\vee}$, where $\psi(b)^{\vee}$ is the dual of $\psi(b)$ defined using the basis $\left\{\psi(b) \mid b \in  \BasisB \right\}$.
\end{enumerate}
\end{lemma}

\subsubsection{Good pairs and Frobenius great pairs}

\begin{definition}
   Let \(I \subseteq \a  \subseteq A_{\bar 0}\) be an even subalgebra of \(A\) that contains $I$, the set of distinguished idempotents. We say that \(\calA = (A,\a )_I\) is a {\em good pair} provided there exists a basis $\BasisB$ for $A$ such that \({}_j\basisb_{i}:={}_j\BasisB_i \cap \a \) is a \(\k\)-basis for \(j\a i\), and \(i \in {}_i \basisb_i \) for all $i \in I$.  
\end{definition}
We call a basis of $A$ a \emph{good basis} if it satisfies the above definition.  If $\calA = (A,\a )_{I}$ is a good pair we will usually choose to use a good basis for $A$.  However, we will occasionally need to make use of other bases.  For example, if $\BasisB$ is a good basis, then $\BasisB^{\vee}$ may fail to be a good basis.

\begin{definition}\label{FrobDef}
We say that the data \(\calA = ((A,\a )_I, \tr , \psi) \) is a \emph{Frobenius good pair} provided that $(A,\a )_{I}$ is a good pair, $(A,\tr, \psi)$ is a locally unital Frobenius superalgebra, and the Nakayama automorphism restricts to an automorphism of $\a $. 
\end{definition}

Let \(\calA = ((A,\a )_I, \tr , \psi) \) be a Frobenius good pair.  Let $\left( A/(A_{\1}+\a )\right)^* \xrightarrow{\pi^*} A^{*}$ be the dual of the canonical projection map $\pi: A \to A/(A_{\1}+\a )$.
\begin{lemma}\label{L:Plurimore-Frobenius-Conditions}  Let \(\calA = ((A,\a )_I, \tr , \psi) \) be a Frobenius good pair.  Then the following are equivalent:
\begin{enumerate}
\item There exists a good basis $\BasisB$ with the property that $b^{\vee} \in \a $ for all $b \in \BasisB_{\0}\setminus \basisb_{\0}$.
\item The $\k$-linear map
\begin{equation}
\left(A/(A_{\1}+\a ) \right)^* \xrightarrow{\pi^{*}} A^* \xrightarrow{\phi^{-1}_{\tr}} A
\end{equation} has image contained in $\a $.
\item Every good basis $\BasisB$ has the property that $b^{\vee} \in \a $ for all $b \in \BasisB_{\0}\setminus \basisb_{\0}$.
\end{enumerate}
\end{lemma}

\begin{proof} 

We first prove $(1)$ implies $(2)$.  Say $\BasisB$ is a good basis that satisfies the property in $(1)$.  Since $\left\{ \pi (c) \mid c\in  \BasisB_{\0} \setminus \basisb_{\0}  \right\}$ is a basis for $A/(A_{\1}+\a )$, a basis for $\left(A/(A_{\1}+\a ) \right)^{*}$ is given by $\left\{\pi (b)^{*} \mid b \in \BasisB_{\0} \setminus \basisb_{\0}   \right\}$, where $\pi(b)^{*}$ is defined by declaring $\pi(b)^{*}(\pi(c)) = \delta_{b,c}$ for all $ c\in  \BasisB_{\0} \setminus \basisb_{\0}$.  
The linear functional $\pi^{*}(\pi(b)^{*}) \in A^*$ evaluates on $c \in \BasisB$ by
$$
\pi^{*}(\pi(b)^{*})(c) = \pi(b)^*(\pi(c)) = \delta_{b,c}.
$$  In other words, $\pi^{*}(\pi (b)^{*})=b^{*}$ as linear functionals on $A$.  By the definition of $b^{\vee}$, $ \phi_{\tr}^{-1}\left(  \pi^{*}(\pi (b)^{*}) \right)=\phi_{\tr}^{-1}\left( b^{*}\right) =b^{\vee}$, which is an element of $\a$ by assumption.  Thus, we have found a basis of $\left(A/(A_{\1}+\a ) \right)^*$ whose image under $\phi^{-1}_{\tr} \circ \pi^*$ is contained in $\a$, establishing the desired containment in (2).

We next prove that $(2)$ implies $(3)$.   Let $b \in \BasisB_{\0 } \setminus \basisb_{\0}$. Using the notation of the previous paragraph and the definition of $b^{\vee}$, $\phi_{\tr}(b^{\vee}) = b^{*}=\pi^{*}(\pi(b)^{*})$ and so $b^{\vee} = \varphi_{\tr}^{-1}(\pi^{*}(\pi(b)^{*}))$ is in $\a$, as desired.

It is immediate that $(3)$ implies $(1)$.
\end{proof}

\begin{definition}\label{D:plurimore-Frobenius}  We say that the data $\calA = ((A,\a)_I, \tr , \psi)$ is a Frobenius {\em great} pair (or simply that $\calA$ is \emph{great}) if it is a Frobenius good pair and if it satisfies the equivalent conditions of \cref{L:Plurimore-Frobenius-Conditions}. 
\end{definition}

Unless otherwise stated, going forward we will assume that $\calA = ((A,\a)_I, \tr , \psi)$ is a Frobenius great pair. To better understand the role of these assumptions, see \cref{SS:Examples} for examples and \cref{SS:Assumptions} for further discussion.

\subsection{Examples}\label{SS:Examples}

\subsubsection{ \texorpdfstring{$A= \k$}{A=k} and other Frobenius superalgebras}\label{SS:FirstExamplesofFrobeniusAlgebras}  Let $A=\a =\k$, let $I = \{1 \}$, let $\tr : A \to \k$ be the identity map, and let $\psi: A \to A$ be the identity map.  Then $\calA = ((\k, \k )_{\{1 \}}, \tr, \psi)$ is trivially a Frobenius great pair.

More generally, let $(A, \tr , \psi)$ be a locally unital Frobenius superalgebra with distinguished idempotents $I$. If $A$ admits a homogeneous basis $\BasisB$ for $A$ that contains $I$, then $(A,A_{\0})_{I}$ is a good pair.   Furthermore, $ ((A, A_{\0})_{I}, \tr , \psi)$ is trivially a Frobenius great pair.  In particular, since a locally unital Frobenius algebra $(A,\tr , \psi)$ with distinguished idempotents $I$ can always be viewed as a superalgebra concentrated in parity $\0$, it follows that $((A,A)_{I}, \tr , \psi)$ is always a Frobenius great pair if it admits a $\k$-basis that contains $I$.

If $A=\k G$ is the group ring of a finite group $G$ with identity element $e_{G}\in G$, then it is well-known that taking $\tr (g) = \delta_{e_{G},g}$ and $\psi (g) =g $ for all $g \in G$ makes $A$ into a Frobenius superalgebra concentrated in parity $\0$.  If $H$ is a proper subgroup of $G$, then $((\k G, \k H)_{\{e_{G} \}},\tr , \psi)$ is a Frobenius good pair but is never great.

If \(A = \k[x]/(x^m)\) is a truncated polynomial algebra for some \(m \in \Z_{>0}\), then \(A = A_{\bar 0}\) is symmetric Frobenius with trace map given by \(\tr(x^t) = \delta_{t, m-1}\). If \(\a = A\) then \(\mathcal{A} = ((A,\a)_I, \tr, \psi))\) is a Frobenius great pair by the above discussion. One other interesting option for a Frobenius great pair is the case \(\a = \k\langle x^2 \rangle\) when \(m\) is even.

The next example is a locally unital version of Laurent polynomials.  Let $\mathcal{C}$ be the $\k$-linear category with set of objects equal to the integers and with $\Hom_{\mathcal{C}}(a,b) = \k T_{b,a}$, where, by definition, $T_{c,b} \circ T_{b,a} = T_{c,a}$ for all $a,b,c \in \Z$.  Let $A = \oplus_{a,b \in \Z} \Hom_{\mathcal{C}}(a,b) = \oplus_{a,b \in \Z} \k T_{b,a}$ be the path algebra.  This can be made into a superalgebra by declaring $\overline{T}_{b,a} = \overline{b-a} \in \Z_{2}$.  Let $\a$ be the $\k$-span of $\left\{T_{b,a} \mid \overline{b-a} = \bar{0}, \text{ and }  b \geq a \right\}$ and let $I = \left\{T_{a,a} \mid a \in \Z \right\}$.  Define $\tr : A \to \k$ by $\tr (T_{b,a}) = \delta_{a,b}$, and $\psi: A \to A$ be the identity map.  Then both $\left( (A,A_{\0})_{I}, \tr, \psi\right)$ and $\left( (A,\a )_{I}, \tr, \psi\right)$ are Frobenius great pairs.

\subsubsection{Clifford and Grassmann superalgebras}  

Let $A=\Cliff_{t}=\k[c_{1}, \dotsc , c_{t}]/(\{ c_{i}c_{j}=-c_{j}c_{i}, c_{i}^{2}=1\mid 1 \leq i \neq j \leq  t \})$ be the Clifford superalgebra on $t$ generators with $\Z_{2}$-grading given by declaring $\bar{c_{i}}=\1$ for $i=1, \dotsc , t$.  Then $A$ has $\k$-basis given by $\BasisB = \left\{c_{1}^{r_{1}}\dotsb c_{t}^{r_{t}} \mid r_{k} \in \left\{0,1 \right\}  \right\}$. Set $I=\{1_{A} \}$.  Let $\tr : A \to \k$ be given by 
\[
\tr (c_{1}^{r_{1}}\dotsb c_{t}^{r_{t}}) = \begin{cases}  1  & \text{ if } r_{1}=\dotsb = r_{t}=0;   \\
                0 & \text{ otherwise,} 

\end{cases} 
\] and let $\psi :A \to A$ be given by $\psi (x) = (-1)^{\bar{x}}x$.  Then, both $((A,\k 1_{A})_{\{1_{A} \}},\tr , \psi)$ and $((A,A_{\0})_{\{1_{A} \}},\tr , \psi)$ are good Frobenius pairs, but only $((A,A_{\0})_{\{1_{A} \}},\tr , \psi)$ is a Frobenius great pair when $t > 1$. 

Now let $A = \Lambda_{t} =  \k [x_{1}, \dotsc , x_{t}]/(\{ x_{i}x_{j}=-x_{j}x_{i}, x_{i}^{2}=0 \mid 1 \leq i \neq j \leq  t \} )$ be the Grassmann superalgebra on $t$ generators with $\Z_{2}$-grading given by declaring $\bar{x_{i}}=\1$ for $i=1, \dotsc , t$.  Then $A$ has $\k$-basis given by $\BasisB = \left\{x_{1}^{r_{1}}\dotsb x_{t}^{r_{t}} \mid r_{k} \in \left\{0,1 \right\}  \right\}$. Set $I=\{1_{A} \}$.  Let $\tr : A \to \k$ be given by
\[
\tr (x_{1}^{r_{1}}\dotsb x_{t}^{r_{t}})=\begin{cases}  1  & \text{ if } r_{1}=\dotsb = r_{t}=1;   \\
                0 & \text{ otherwise,} 

\end{cases} 
\] and let $\psi :A \to A$ be given by $\psi (x) = (-1)^{\bar{x}}x$. In order for $\tr$ to be an even map, let us assume $t=2p$ for some $p > 0$.  Let $\a$ be the $\k$-span of $1_{A}$ and all even basis elements that satisfy $\sum_{i=1}^{t} r_{i} \geq p$.  Then $\a$ is easily seen to be a subalgebra of $A_{\0}$.  A direct check verifies that $((A,\k 1_{A})_{\{1_{A} \}},\tr , \psi)$,  $((A,A_{\0})_{\{1_{A} \}},\tr , \psi)$, and $((A,\a )_{\{1_{A} \}},\tr , \psi)$ are Frobenius good pairs, but only $((A,A_{\0})_{\{1_{A} \}},\tr , \psi)$ and $((A,\a )_{\{1_{A} \}},\tr , \psi)$ are Frobenius great pairs

\subsubsection{Zigzag superalgebras}\label{zigzagcat}
Let \(\Gamma = (I,E)\) be a simple connected graph with no loops. For each pair of connected vertices \(i\textup{\textemdash} j\), fix \(\varepsilon_{ij}, \varepsilon_{ji} \in \{\pm 1\}\). Then the associated locally unital {\em zigzag superalgebra} \(\overline{Z}(\Gamma) = \overline{Z} = \bigoplus_{i,j \in I} j\overline{Z}i\) has distinguished set of orthogonal idempotents corresponding to the vertices \(I\) of \(\Gamma\), with the following \(\k\)-basis for each subspace:
\begin{align*}
j\overline{Z}i =
\begin{cases}
\k\{i, c_i\} & \textup{if } i =j;\\
\k\{a_{ji} \} & \textup{if } i\textup{\textemdash} j;\\
0 & \textup{otherwise},
\end{cases}
\end{align*} 
Multiplication is defined via the orthogonality of the idempotents \(I\), coupled with the additional rules:
\begin{align*}
a_{kj}a_{ji} =\delta_{ik} \varepsilon_{ji} c_i,
\qquad
c_ia_{ij} = a_{ij}c_j = 0,
\qquad
c_i^2 = 0.
\end{align*}
The parity of elements is defined by \(\bar i = \bar{c}_i = \bar 0\), \(\bar{a}_{ij} = \bar 1\). 
Taking \(\tr(i) = \tr(a_{ji}) = 0\), \(\tr(c_i) = 1\) for all \(i,j \in I\), and \(\psi(a_{ji}) = -\varepsilon_{ij} \varepsilon_{ji} a_{ji}\), it is straightforward to verify that \(((\bar{Z}, \k I)_I, \tr, \psi)\) is a Frobenius great pair.

\subsubsection{Trivial extension superalgebras}\label{trivextsec}
The trivial extension algebra is a well-known technique for turning a finite-dimensional associative algebra into a Frobenius algebra. Suitably modified, this machine also applies here.

Let $C$ be a locally finite, locally unital $\k$-superalgebra and let $C^{*}$ be the dual viewed as a bimodule via \cref{E:dual-action}.  Set 
\[
E(C) = C^{*} \oplus C
\] as a $\k$-module.  The $\Z_{2}$-grading is given by declaring $E(C)_{r} = (C^{*}_{r}, C_{r})$ for $r \in \Z_{2}$.  In particular, this implies $\overline{(\alpha,a)}= \overline{\alpha}=\overline{a}$.   Define the product by
\[
(\alpha, a) \cdot (\beta, b) = (\alpha \cdot b + a \cdot \beta, ab).
\]  This makes $E(C)$ into a locally unital $\k$-superalgebra, where $C$ is identified as a subsuperalgebra of $E(C)$ in the obvious way.  In particular, $I \subseteq C \subseteq E(C)$ provides a set of distinguished idempotents.  Moreover, $jE(C)i = j\cdot(C^{*})\cdot i \oplus jCi = (iCj)^{*}\oplus jCi$ for all $i,j \in I$, and $E(C)=\bigoplus_{i,j \in I} jE(C)i$.

If $\BasisC$ is a homogenous basis for $C$, $c^{*} \in C^{*}$ is defined by $c^{*}(d) = \delta_{c,d}$ for all $d \in \BasisC$, and if we set $\BasisC^{*}= \left\{c^{*} \mid c \in \BasisC  \right\}$, then the basis $\BasisB = \left\{(c^{*},0), (0,c) \mid c \in \BasisC  \right\}$ makes $(E(C), C_{\0})_{I}$ a good pair.  Define a trace map by
\[
\tr : E(C) \to \k, \quad \tr (\alpha,a)= \sum_{i \in I} \alpha(i),
\] and let the Nakayama automorphism $\psi: E(C) \to E(C)$ be the identity map.  We claim the data $\left((E(C), C_{\0})_{I}, \tr , \psi \right)$ is a Frobenius great pair.

Consider the $\k$-linear map $\phi_{\tr}: E(C) \to E(C)^{*}$ defined by the trace map as in \cref{E:trace-condition3}. We claim this map is a $k$-linear isomorphism. Note that for homogeneous $(\alpha, a), (\beta, b) \in E(c)$ one has
\begin{equation*}
\left[  \phi_{\tr}(\alpha,a)\right](\beta, b)  = \tr ((\alpha, a)(\beta, b)) 
        = \tr (\alpha \cdot b + a \cdot \beta, ab)
       = \alpha (b) + (-1)^{\bar{a}\bar{\beta}} \beta(a).
\end{equation*}    Let $c,d \in \BasisC$.  The previous calculation verifies that
\begin{align*}
\phi_{\tr}(c^{*},0)(0,d) &= c^{*}(d) = \delta_{c,d}, \\
\phi_{\tr}(c^{*},0)(d^{*},0) & =0, \\
\phi_{\tr}(0,c)(0,d) &= 0, \\
\phi_{\tr}(0,c)(d^{*},0) &= (-1)^{\bar{c}\bar{d}} d^{*}(c) = (-1)^{\bar{c}} \delta_{c,d}.
\end{align*}  That is, $\phi_{\tr}(c^{*},0) = (0,c)^{*}$ and $\phi_{\tr}(0,c) = (-1)^{\bar{c}}(c^{*},0)^{*}$.  Thus, $\phi_{\tr}$ defines a homogeneous bijection between $\k$-bases of $E(C)$ and $E(C)^{*}$ and, hence, defines an isomorphism of $\k$-modules.  Moreover, this calculation shows that $(c^{*}, 0)^{\vee} = \phi_{\tr}^{-1}\left( (c^{*},0)^{*}\right) = (0,c) \in C_{\0}$ for $c \in \BasisC_{\0}$ and so the data $\left((E(C), C_{\0})_{\{1 \}}, \tr , \psi \right)$ satisfies the great condition.

We next claim that $\phi_{\tr }$ is a morphism of $(E(C), E(C))$-bimodules.  Let $(\alpha, a)$, $(\beta, b)$, and $(\gamma, g)$ be homogenous elements of $E(C)$.  Then
\begin{align*}
\left((\alpha, a) \cdot \phi_{\tr}(\beta, b) \right) (\gamma, g)
&=  (-1)^{\overline{(\alpha, a)}(\overline{(\beta,b)}+\overline{(\gamma, g)})}\phi_{\tr}(\beta,b)((\gamma, g)(\alpha,a)) \\
&= (-1)^{\overline{a}(\overline{b}+\overline{g})}\phi_{\tr}(\beta,b)((\gamma \cdot a + g\cdot \alpha,ga)) \\
&= (-1)^{\overline{a}(\overline{b}+\overline{g})}\tr\left((\beta,b)(\gamma \cdot a + g\cdot \alpha,ga)\right) \\
&= (-1)^{\overline{a}(\overline{b}+\overline{g})}\tr\left(\beta\cdot ga + b \cdot\gamma \cdot a + bg\cdot \alpha,bga)\right) \\
&=(-1)^{\overline{a}(\overline{b}+\overline{g})}\left(\beta(ga) + (-1)^{\bar{b}(\bar{g}+\bar{a})}\gamma(ab)
+ (-1)^{(\bar{b}+\bar{g})\bar{\alpha}} \alpha(bg) \right),
\end{align*}
whereas,
\begin{align*}
\left(\phi_{\tr}((\alpha, a)(\beta, b) \right) (\gamma, g) &= \left(\phi_{\tr} (\alpha \cdot b + a \cdot \beta, ab) \right)(\gamma, g) \\
&= \tr \left( (\alpha \cdot b + a \cdot \beta, ab)(\gamma, g) \right) \\
&= \tr \left( \alpha \cdot bg + ab \cdot \gamma + a \cdot \beta \cdot g, abg \right) \\
& = \alpha(bg) + (-1)^{(\bar{a}+\bar{b})\bar{g}}\gamma(ab) + (-1)^{\bar{a}(\bar{b}+\bar{g})} \beta(ga).
\end{align*}  Thus $\phi_{\tr}$ is a homomorphism of left $E(C)$-modules.

On the other hand,
\begin{equation*}
\left( \phi_{\tr}(\beta, b) \cdot (\alpha, a) \right) (\gamma, g) 
= \phi_{\tr}(\beta, b)  \left( (\alpha, a)  (\gamma, g) \right) 
= \tr \left((\beta, b)  (\alpha, a)  (\gamma, g) \right)
= \tr \left((\beta, b) (\alpha, a)\right)(\gamma, g).
\end{equation*} Thus $\phi_{\tr}$ is a homomorphism of right $E(C)$-modules.  This verifies that $\left((E(C), C_{\0})_{I}, \tr , \psi \right)$ is a Frobenius great pair.

\subsection{A discussion of assumptions}\label{SS:Assumptions} Our assumptions on $\k$ and $\calA = ((A,a)_{I}, \tr , \psi)$ were selected to include the examples and applications of interest to the authors.  The reader may wish to make different assumptions.  In this section we briefly discuss how toggling on/off various assumptions affects the results of the paper.

\subsubsection{Ground ring}
The reader may prefer to work with a ground field of characteristic zero. If so, then the affine web category for any choice of $\a \subseteq A_{\0}$  is isomorphic to the one obtained by selecting $\a =A_{\0}$, in which case the Frobenius great pair requirements are trivially satisfied.  Likewise, cyclotomic datum are automatically regular (see \cref{cycquosecmain}).  Thus the results of this paper hold unconditionally for any locally unital Frobenius superalgebra when working over a field of characteristic zero. Moreover, a number of the definitions and proofs can be simplified.

\subsubsection{Locally finite, locally unital Frobenius superalgebras}
Readers who are not interested in $\Z_{2}$-gradings may assume that $A$ is an associative $\k$-algebra.  One can view $A$ as a superalgebra concentrated in parity $0$ by declaring $A = A_{\0} \oplus A_{\1} = A \oplus 0$.  The results of this paper apply with the simplification that all signs due to $\Z_{2}$-gradings disappear.

Associative superalgebras frequently come with a distinguished set of idempotents. For example, if the superalgebra is defined as the quotient of the path superalgebra of some quiver (e.g., see the zigzag superalgebras and \cite[Section IV.16]{SY1}).  Another important example is if $\mathcal{C}$ is a small category enriched over the category of free $k$-supermodules with a fixed set of isomorphism classes of objects, $X$.  The associated path algebra, 
\[
\bigoplus_{x,y \in X} \Hom_{\mathcal{C}}(x,y),
\] is naturally a locally unital $\k$-superalgebra with $I=\left\{\Id_{x} \mid x \in X  \right\}$ the set of distinguished idempotents.

The Frobenius property \cref{E:trace-condition3} forces $A$ to satisfy a suitable finiteness property.  When $A$ is unital, a Frobenius structure forces $A$ to have finite rank.  If $A$ is locally unital, then local finiteness is the appropriate choice (see (2) of \cref{FrobFacts}). By allowing locally unital superalgebras, our framework allows for examples that have infinite total rank like the one in \cref{SS:FirstExamplesofFrobeniusAlgebras}. Another important advantage of including a set of distinguished idempotents is that it allows for greater refinement in the definitions and results.

The reader who prefers their associative superalgebras to be unital may choose $I=\left\{1_{A} \right\}$.  The results of this paper simplify in that strands in $\AffWebAaI$ have a single color (and, hence, are determined by their thickness), the crossing generators in \cref{WebAaGens} are omitted, and the relations \cref{Cox,SplitIntertwineRel,MergeIntertwineRel} are omitted.

For simplicity's sake this paper assumes the Nakayama automorphism fixes the distinguished idempotents.   The results of this paper go through with appropriate modifications as long as the Nakayama automorphism fixes the distinguished idempotents as a set.  In that case, for example,  the affine dot would define a morphism in $\Hom_{\AffWebAaI}(i, \psi^{-1} (i))$ for each $i \in I$.  Fearing most readers would find the associated bookkeeping a journey into madness (and, indeed, the authors spent some time on this particular journey), we forgo that extra level of generality.

Likewise, while we choose not to do so, it would be interesting to allow the trace map to be odd.  This odd trace map occurs, for example, when studying Frobenius structures on odd rank Grassman superalgebras, or by modifying the Frobenius structure on odd rank Clifford algebras that we described in the previous section.  We expect the results of this paper to generalize with suitable modifications (e.g., odd Casimir elements in $A\otimes  A$ and odd affine dots in $\AffWebAaI$) along with additional signs that would need to be tracked.

\subsubsection{Frobenius great pairs}  The Frobenius structure on $A$ is essential for defining the Casimir elements, the affine dot (e.g., see \cref{DotCrossRel}), and the action of the affine dot on $\gl_{n}(A)$-supermodules (e.g., see \cref{L:hatH}).  The even subalgebra $\a \subseteq A_{\0}$ and the notion of a good pair $(A,\a)_{I}$ is needed if we wish to define $\AffWebAaI$ integrally (i.e., over the integral domain $\k$).  Different choices of $\a $ can be viewed as making different choices of $\k$-form inside the  $\KK$-linear category $\KAffWebAI$ defined using the superalgebra $\KK \otimes_{\k} A$.  It is known that the choice of good pair $(A,\a)_{I}$  dramatically affects the representation theory of the superalgebras $T_{(A, \a)_{I}}(n,d)$ obtained from the Schurification of $(A,\a)_{I}$ \cite{KM}.  We expect the choice of good pair to be influential here, as well.

Likewise, the fact we wish to work integrally motivates the assumption that $\calA$ is a Frobenius great pair.  Specifically, the great condition  ensures the coefficients which appear in relation \cref{DotCrossRel} lie in $\k$.

\subsection{Casimir elements in a Frobenius superalgebra}\label{SS:Casimir-elements-in-A}  Let $(A,\tr , \psi)$ be a locally unital Frobenius superalgebra.  Let \(M\) be a left \(A\)-module, \(N\) be a right \(A\)-module, and let \(\eta\) be an automorphism of \(A\). Then there is a natural isomorphism of \((A,A)\)-bimodules:
\begin{align}\label{natisodual}
D_{\eta}: N^* \otimes (M^*)^\eta \to ((M \otimes N)^*)^\eta, \qquad \beta \otimes \alpha \mapsto ((x \otimes y) \mapsto (-1)^{\bar \beta \bar \alpha + \bar \beta \bar x} \alpha(x) \beta(y)
\end{align}

For \(i \in I\), define \(m_{i}: A \nakayamai  \otimes \nakayamai A \to A\) as the restriction of the multiplication map on \(A\). Then we define an \((A,A)\)-bimodule homomorphism
\begin{align*}
\Delta_i: A \to  Ai \otimes \nakayamai A
\end{align*}
as the composition
\begin{equation*}
A \xrightarrow{\varphi_\tr} (A^*)^{\psi^{-1}} \xrightarrow{m_{i}^*}
(A\nakayamai  \otimes \nakayamai A)^*)^{\psi^{-1}}
\xrightarrow{D_{\psi^{-1}}^{-1}}
(\nakayamai A)^* \otimes ((A\nakayamai )^*)^{\psi^{-1}}
\xrightarrow{\varphi_\tr^{-1} \otimes \varphi_\tr^{-1}} Ai \otimes \nakayamai A, 
\end{equation*}
where \(D_{\psi^{-1}}\) is given by \cref{natisodual}.

\begin{lemma}\label{DelC}
Let \(i,j  \in I\). Then, \(\Delta_i(j) = \sum_{b \in {}_j\BasisB_i} b \otimes b^\vee\).
\end{lemma}
\begin{proof}
Let \(u,v \in \BasisB_i\). Set \(x = (u^\vee)^\vee\), \(y = v^\vee\). By Lemma~\ref{FrobFacts}(3), such elements \(x,y\) comprise bases for the spaces \(A \nakayamai  \) and \( \nakayamai A\) respectively. We have
\begin{align*}
 (m_{j,i}^* \circ \varphi_{\tr}(j))(x \otimes y) &= \varphi_\tr(j)(xy) = \tr(xyj) = \delta_{y \in Aj}\tr(xy) =\delta_{v \in {}_j\BasisB_i} \tr((u^\vee)^\vee v^\vee) =\delta_{v \in {}_j\BasisB_i} \delta_{u,v}.
\end{align*}
On the other hand, 
\begin{align*}
\left( (D_{\psi^{-1}} \circ \varphi_\tr \otimes \varphi_\tr\left(  \sum_{b \in {}_j \BasisB_i} b \otimes b^\vee\right)\right)(x \otimes y)
&=
D_{\psi^{-1}} \left(   \sum_{b \in {}_j \BasisB_i}  \varphi_\tr(b) \otimes \varphi_\tr(b^\vee) \right)( x \otimes y) \\
&=
 \sum_{b \in {}_j \BasisB_i} (-1)^{\bar b + \bar b \bar x}\tr(xb^\vee) \tr(yb)\\
 &=
  \sum_{b \in {}_j \BasisB_i} (-1)^{\bar b + \bar b \bar u} \tr((u^\vee)^\vee b^\vee) \tr(v^\vee b) \\
 & =   \sum_{b \in {}_j \BasisB_i} (-1)^{\bar b + \bar b \bar u} \delta_{u,b} \delta_{v,b}
 = \delta_{v \in {}_j\BasisB_i}\delta_{u,v}.
\end{align*}
Thus \( m_{j,i}^* \circ \varphi_{\tr}(j) =D \circ \varphi_\tr \otimes \varphi_\tr\left(  \sum_{b \in {}_j \BasisB_i} b \otimes b^\vee\right) \).  The result then follows from the definition of \(\Delta_i\).
\end{proof}

It follows that elements of the form \(\sum_{b \in {}_j\BasisB_i} b \otimes b^\vee\) are independent of the choice of basis \(\BasisB\).  In particular, for any $i,j \in I$, we can replace the basis ${}_j \BasisB_i$ for $jAi$ with the basis $({}_i \BasisB_j)^{\vee} = \{b^{\vee} \mid b \in {}_i \BasisB_j\}$, which implies
\begin{align}\label{indchoice}
\sum_{b \in {}_j\BasisB_i} b \otimes b^\vee = \sum_{b \in {}_{i} \BasisB_{j}} b^\vee \otimes (b^\vee)^\vee.
\end{align}  
We call these the \emph{(local) Casimir elements} of $A \otimes A$.

Let \(\tau: A \otimes A \to A \otimes A\) be the signed flip map \(x \otimes y \mapsto (-1)^{\bar x \bar y} y \otimes x\).
\begin{lemma}\label{taulem}
For all \(i,j \in I\), we have \(\tau \left(  \sum_{b \in {}_j\BasisB_i} b \otimes b^\vee  \right) = \sum_{b \in {}_i\BasisB_{j}} \psi^{-1}(b) \otimes b^\vee\).
\end{lemma}
\begin{proof}
The result follows from a direct calculation:
\begin{align*}
\tau \left(\sum_{b \in {}_j\BasisB_i} b \otimes b^\vee   \right)
=
\tau \left( \sum_{b \in {}_{i} \BasisB_{j}} b^\vee \otimes (b^\vee)^\vee   \right) =  \sum_{b \in {}_{i} \BasisB_{j}}  (-1)^{\bar b} (b^\vee)^\vee \otimes b^\vee\;\;
=
\sum_{b \in {}_i\BasisB_{j}} \psi^{-1}(b) \otimes b^\vee,
\end{align*} where the first and last equalities follow from \cref{indchoice} and \cref{FrobFacts}(5), respectively.
\end{proof}

Given a locally unital $\k$-superalgebra $A$, the $\k$-supermodule $A \otimes A$ is again a locally unital $\k$-superalgebra with product given by $(f \otimes g)(x \otimes y) = (-1)^{\bar{g}\bar{x}}fx \otimes gy$.

\begin{lemma}\label{translem}
Let \(f \in kA j\), \(g \in \ell A\nakayamai \). Then the product on $A \otimes A$ satisfies 
\begin{align}\label{transport}
(f \otimes g) \left( \sum_{b \in {}_j\BasisB_i} b \otimes b^\vee\right) =(-1)^{\bar f \bar g} \left( \sum_{b \in {}_k \BasisB_{\ell}} b \otimes b^\vee \right)(\psi(g) \otimes f).
\end{align}
and
\begin{align}\label{transportrev}
(g \otimes f) \left(  \sum_{b \in {}_i\BasisB_{j}} \psi^{-1}(b) \otimes b^\vee\right) =(-1)^{\bar f \bar g} \left(  \sum_{b \in {}_{\ell}\BasisB_{k}} \psi^{-1}(b) \otimes b^\vee \right)(f \otimes \psi(g)).
\end{align}
\end{lemma}
\begin{proof}
First, recalling that \(\Delta_i\) is an \(A\)-bimodule map,
\begin{equation}\label{leftpast}
(f \otimes \nakayamai ) \left( \sum_{b \in {}_j\BasisB_i} b \otimes b^\vee\right)  = f \cdot \Delta_i(j) = \Delta_i(f)
= \Delta_i(k) \cdot f = \left( \sum_{b \in {}_k\BasisB_i} b \otimes b^\vee\right) ( i \otimes f),
\end{equation}
where the first and last equalities are from \cref{DelC}. 
This identity along with use of the flip map yields: 
\begin{align*}
(j \otimes g)\left( \sum_{b \in {}_j\BasisB_i} b \otimes b^\vee\right)
&=
\tau\left( (g \otimes j) \tau\left( \sum_{b \in {}_j\BasisB_i} b \otimes b^\vee   \right)  \right)\\
&= \tau\left(  \sum_{b \in {}_i \BasisB_{j}} g\psi^{-1}(b) \otimes b^\vee \right) \\
&= \tau  (\psi^{-1} \otimes 1) \left(  \sum_{b \in {}_i \BasisB_{j}} \psi(g)b \otimes b^\vee  \right) \\
& = \tau \left( (\psi^{-1} \otimes 1) \left(  \sum_{b \in {}_{\ell} \BasisB_{j}} b \otimes b^\vee\psi(g) \right)\right)\\
&= \tau \left(  \sum_{b \in {}_{\ell} \BasisB_{j}} \psi^{-1}(b) \otimes b^\vee\psi(g) \right) \\
& =\tau \left(  \sum_{b \in {}_{\ell} \BasisB_{j}} \psi^{-1}(b) \otimes b^\vee \right) \tau( j \otimes \psi(g))\\
&= \left( \sum_{b \in {}_j \BasisB_{\ell}} b \otimes b^\vee \right) (\psi(g) \otimes j).
\end{align*}
Since $(f \otimes g) = (f \otimes \ell)(j \otimes g)$, this establishes  \cref{transport}.  The second identity follows from applying \(\tau\) to both sides of \cref{transport} and applying \cref{taulem}.
\end{proof}

\section{The finite web category \texorpdfstring{$\WebAaI$}{WebAaI}}

We assume the reader is familar with constructing monoidal supercategories by generators and relations, diagrammatic supercategories, and the like.  A recommended reference is \cite{BE}.  This section summarizes needed definitions and results from \cite{DKMZ}.  Further details can be found in therein.

\subsection{Combinatorics}
\subsubsection{Counting and sets}
We use the generalized binomial coefficient: 
\begin{align*}
\binom{n}{k}:= \frac{n(n-1) \cdots (n-k+1)}{k!} \in \Z,
\end{align*}
for any \(n \in \Z\), \(k \in \Z_{\geq 0}\).  Given integers $a,b$ with $a \leq b$, let $[a,b] = \{a, a+1, \dotsc , b \}$.  Given sets $X$ and $Y$, we write $X^{Y}$ for the set of tuples indexed by $Y$ consisting of elements from $X$. 

\subsubsection{Compositions}
For \(\bx = (x_1, \ldots, x_t) \in \Z_{\geq 0}^t\),  set
\begin{align*}
 |\bx| = t,\qquad\lVert \bx \rVert = \sum_{k=1}^t x_k,\qquad  \bx! =  \prod_{k=1}^t (x_k!).
 \end{align*}
Set
\begin{align*}
\Omega(n,d) = \{ \bx = (x_1, \ldots, x_n) \in \Z_{\geq 0}^n \mid \lVert \bx \rVert = d\} \qquad \text{and} \qquad 
\Omega =  \bigsqcup_{n,d \in \Z_{\geq 0}}  \Omega(n,d).
\end{align*}
We also set
\begin{align*}
\mathfrak{S}_{\bx} := \mathfrak{S}_{x_1} \times \cdots \times \mathfrak{S}_{x_t} \leq \mathfrak{S}_{\lVert \bx \rVert}.
\end{align*}
If \(\lVert \bx \rVert = \lVert \by \rVert\) and \(\mathfrak{S}_{\by} \leq \mathfrak{S}_{\bx}\), then we write \(\mathscr{D}_{\bx}^{\by}\) for the set of minimal length right coset representatives for \(  \mathfrak{S}_{\by} \backslash \mathfrak{S}_{\bx}\).

\subsubsection{Colored compositions}
For \(n,d \in \Z_{\geq 0}\) and a set \(I\), an {\em \(n\)-part \(I\)-colored composition of \(d\)} is the data of a composition \(\bx = (x_1, \ldots, x_n) \in \Omega(n,d)\), and a tuple \(\bi = (i_1, \ldots, i_n) \in I^n\), which we write as \(\bi^{(\bx)} = (i_1^{(x_1)}, \ldots, i_n^{(x_n)})\). For two such \(I\)-colored compositions, we declare \(\bi^{(\bx)} = \bj^{(\bx)}\) provided \(i_t = j_t\) whenever \(x_t \neq 0\); i.e., the coloring of the zero components of a composition is irrelevant for us. Set
\begin{align*}
\Omega_I(n,d) = \{ \bi^{(\bx)} = (i_1^{(x_1)}, \ldots, i_n^{(x_n)}) \mid \bi \in I^n, \, \bx \in \Omega(n,d)\} \qquad \text{and} \qquad 
\Omega_I =  \bigsqcup_{n,d \in \Z_{\geq 0}}  \Omega_I(n,d).
\end{align*}
We consider \(\Omega_I\) as a monoid under concatenation, with identity the empty composition \(() \in \Omega_I(0,0)\). We define the monoid of {\em reduced} \(I\)-colored compositions \(\widehat{\Omega}_I\) by imposing the following additional relation on \(\Omega_I\): \(() = (i^{(0)})\) for all \(i \in I\). In other words, we have \(\bi^{(\bx)} = \bj^{(\by)}\) in \(\widehat{\Omega}_I\) provided the colored compositions are identical after deleting all the zero components. We write
\begin{align*}
\widehat{\Omega}_I(n,d) = \{ \bi^{(\bx)} = (i_1^{(x_1)}, \ldots, i_n^{(x_n)}) \mid \bi \in I^n, \, \bx \in \Omega(n,d)\} \subseteq \widehat{\Omega}_I.
\end{align*}
For \(\bi^{(\bx)}= i_1^{(x_1)} \cdots i_t^{(x_t)}  \in \widehat{\Omega}\), we will write \(\bi^\bx := ((i_1^{(1)})^{ x_1},  \ldots, (i_t^{(1)})^{ x_t}) \in \widehat{\Omega}\). We will also associate \(\bi \in I^n\) with \(\bi^{(1^n)}=(i_1^{(1)}, \ldots, i_n^{(1)}) \in \widehat{\Omega}\).

\begin{example}
If \(I = \{\alpha, \beta\}\), then \(\alpha^{(2)} \beta^{(0)} \beta^{(3)} = \alpha^{(2)} \alpha^{(0)} \beta^{(3)}\) in both \( \Omega_I\) and \(\widehat{\Omega}_I\), while \(\alpha^{(2)} \beta^{(0)} \beta^{(3)} =  \alpha^{(2)} \beta^{(3)}\) in \(\widehat{\Omega}_I\) but not in \(\Omega_I\).
\end{example}

\subsection{The \texorpdfstring{$\WebAaI$}{WebAaI} category}\label{defweb}  Let \(\calA = (A, \a)_I\) be a good pair.  It will be useful to set \(A^{(1)} =A\) and \(A^{(z)} = \a \) for \(z > 1\).
\begin{definition}\label{defwebaa}
 Let \(\WebAaI\) be the strict monoidal \(\k\)-linear supercategory defined as follows.  
We have \(\Ob(\WebAaI)=\widehat{ \Omega}_I\). The monoidal structure on objects in \(\WebAaI\) is inherited from the monoid \(\widehat{ \Omega}_I\); in particular the objects of \(\WebAaI\) are monoidally generated by objects \(i^{(x)}\) for \(i \in I\), \(x \in \Z_{\geq 0}\), with \(i^{(0)}\) equating to the unit object \(\mathbbm{1} \in \WebAaI\) for all \(i \in I\). 

The generating morphisms of \(\WebAaI\) are given by the diagrams:
\begin{align}\label{WebAaGens}
\hackcenter{
\begin{tikzpicture}[scale=.8]
  \draw[ultra thick,blue] (0,0)--(0,0.2) .. controls ++(0,0.35) and ++(0,-0.35) .. (-0.4,0.9)--(-0.4,1);
  \draw[ultra thick,blue] (0,0)--(0,0.2) .. controls ++(0,0.35) and ++(0,-0.35) .. (0.4,0.9)--(0.4,1);
      \node[above] at (-0.4,1) {$ \scriptstyle i^{\scriptstyle (x)}$};
      \node[above] at (0.4,1) {$ \scriptstyle i^{\scriptstyle (y)}$};
      \node[below] at (0,0) {$ \scriptstyle i^{\scriptstyle (x+y)} $};
\end{tikzpicture}},
\qquad
\qquad
\hackcenter{
\begin{tikzpicture}[scale=.8]
  \draw[ultra thick,blue ] (-0.4,0)--(-0.4,0.1) .. controls ++(0,0.35) and ++(0,-0.35) .. (0,0.8)--(0,1);
\draw[ultra thick, blue] (0.4,0)--(0.4,0.1) .. controls ++(0,0.35) and ++(0,-0.35) .. (0,0.8)--(0,1);
      \node[below] at (-0.4,0) {$ \scriptstyle i^{ \scriptstyle (x)}$};
      \node[below] at (0.4,0) {$ \scriptstyle i^{ \scriptstyle (y)}$};
      \node[above] at (0,1) {$ \scriptstyle i^{ \scriptstyle (x+y)}$};
\end{tikzpicture}},
\qquad
\qquad
\hackcenter{
\begin{tikzpicture}[scale=.8]
  \draw[ultra thick,red] (0.4,0)--(0.4,0.1) .. controls ++(0,0.35) and ++(0,-0.35) .. (-0.4,0.9)--(-0.4,1);
  \draw[ultra thick,blue] (-0.4,0)--(-0.4,0.1) .. controls ++(0,0.35) and ++(0,-0.35) .. (0.4,0.9)--(0.4,1);
      \node[above] at (-0.4,1) {$ \scriptstyle j^{ \scriptstyle (y)}$};
      \node[above] at (0.4,1) {$ \scriptstyle i^{ \scriptstyle (x)}$};
       \node[below] at (-0.4,0) {$ \scriptstyle i^{ \scriptstyle (x)}$};
      \node[below] at (0.4,0) {$ \scriptstyle j^{ \scriptstyle (y)}$};
\end{tikzpicture}},
\qquad
\qquad
\hackcenter{
\begin{tikzpicture}[scale=.8]
  \draw[ultra thick, blue] (0,0)--(0,0.5);
   \draw[ultra thick, red] (0,0.5)--(0,1);
   \draw[thick, fill=yellow]  (0,0.5) circle (7pt);
    \node at (0,0.5) {$ \scriptstyle f$};
     \node[below] at (0,0) {$ \scriptstyle i^{ \scriptstyle (z)}$};
      \node[above] at (0,1) {$ \scriptstyle j^{ \scriptstyle (z)}$};
\end{tikzpicture}},
\end{align}
for \(i,j \in I\), \(x,y \in \Z_{\geq 0}\), \(z \in \Z_{>0}\), \(f \in jA^{(z)}i\), where diagrams are to be read from bottom to top. We call these morphisms {\em split}, {\em merge}, {\em crossing}, and {\em coupon}, respectively. Splits, merges and crossings have parity \(\bar 0\), and the parity of the \(f\) coupon is \(\bar{f}\). We say a strand labeled by \(i^{(x)}\) has {\em color} \(i\) and {\em thickness} \(x\). We refer to strands of thickness 1 as {\em thin} strands, otherwise we call them {\em thick}.  Note that, since \(A^{(1)} =A\) and \(A^{(z)} = \a \) for \(z > 1\), thick strands may only be decorated by coupons in the even subalgebra \(\a \), whereas thin strands may be decorated with arbitrary coupons in \(A\).

Going forward, we will use the following conventions:
\begin{itemize}
\item Strands of thickness $0$ (and any coupons thereon) are to be deleted;
\item Diagrams containing a strand of negative thickness are to be read as zero.
\item When coupons on separate strands are depicted at the same height, the coupon on the left should be considered to be slightly above the coupon on the right.
\end{itemize}

The relations imposed on the morphisms in \(\WebAaI\) are given in \cref{AssocRel} - \cref{AaIntertwine}:

{\em Web-associativity.} For all \(i \in I\), \(x,y,z \in \Z_{\geq 0}\):
\begin{align}\label{AssocRel}
\hackcenter{
}
\end{align}

{\em Odd knothole annihilation.} For all \(i,j \in I\), \(f \in jA_{\bar 1}i\):
\begin{align}\label{OddKnotholeRel}
\hackcenter{}
\hackcenter{
%
}
\end{align}
\end{definition}

\subsection{Some implied relations in \(\WebAaI\)} \label{someimplied1}
The following two lemmas follow from \cite[Lemma 4.2.3]{DKMZ} and \cite[Corollary 4.2.7]{DKMZ}, respectively.

\begin{lemma}\label{CrossAbsorb}
For all \(i \in I\), \(a,b \in \Z_{\geq 0}\) we have the following equalities in \(\WebAaI\):
\begin{align*}
\hackcenter{}
\hackcenter{
%
}.
\end{align*}
\end{lemma}

\begin{lemma}\label{CoxArb}
The Coxeter, split- and merge-intertwining relations (\ref{Cox})--(\ref{MergeIntertwineRel}) hold in \(\WebAaI\) for any \(i,j,k \in I\) (not just distinct \(i,j,k\)).
\end{lemma}

\subsection{Explosion and contraction}\label{ExplSec}
Define morphisms 
\begin{align}\label{E:multisplitmerge}
\hackcenter{}
y_i^{(x_1, \ldots, x_n)}:=
\hackcenter{
%
},
\end{align}
in \(\WebAaI(i^{(x_1 + \cdots + x_n)}, i^{(x_1)} \cdots  i^{(x_n)})\)
and
 \(\WebAaI(i^{(x_1)} \cdots i^{(x_n)}, i^{(x_1 + \cdots + x_n)})\), respectively.  The diagrams should be interpreted as \(n-1\) vertically composed splits (or merges). By \cref{AssocRel} the resulting morphism is independent of the split (or merge) order. 

For \(\bi^{(\bx)} = i_1^{(x_1)} \cdots i_t^{(x_t)} \in \textup{Ob}(\WebAaI)\) recall that \(\bi^\bx = (i_1^{(1)})^{ x_1}  \cdots (i_t^{(1)})^{ x_t} \in \textup{Ob}(\WebAaI)\), and 
\begin{align*}
Y_{\bi^{(\bx)}}&:= y_{i_1}^{(1, \ldots, 1)} \otimes \cdots \otimes y_{i_t}^{(1, \ldots, 1)} \in \WebAaI(\bi^{(\bx)}, \bi^\bx)\\
Z_{\bi^{(\bx)}}&:= z_{i_1}^{(1, \ldots, 1)} \otimes \cdots \otimes z_{i_t}^{(1, \ldots, 1)} \in \WebAaI(\bi^\bx,\bi^{(\bx)})
\end{align*}
Then for \(\bi^{(\bx)}, \bj^{(\by)} \in \textup{Ob}(\WebAaI)\), we have linear maps: 
\begin{align*}
\exp_{\bi^{(\bx)}, \bj^{(\by)}}: \WebAaI(\bi^{(\bx)}, \bj^{(\by)}) \to \WebAaI(\bi^\bx, \bj^\by),
\qquad
f \mapsto Y_{\bj^{(\by)}}\circ f\circ Z_{\bi^{(\bx)}}\\
\con_{\bi^{(\bx)}, \bj^{(\by)}}: \WebAaI(\bi^{\bx}, \bj^{\by}) \to \WebAaI(\bi^{(\bx)}, \bj^{(\by)}),
\qquad
f \mapsto Z_{\bj^{(\by)}}\circ f\circ Y_{\bi^{(\bx)}}
\end{align*}
We refer to these maps as \emph{explosion} and \emph{contraction}, respectively. 
The next lemma follows by repeated application of \cref{KnotholeRel}.

\begin{lemma}\label{L:ExplCon} For all \(f \in \WebAaI(\bi^{(\bx)}, \bj^{(\by)})\), we have
\begin{align*}
\con_{\bi^{(\bx)}, \bj^{(\by)}} \circ \exp_{\bi^{(\bx)}, \bj^{(\by)}} (f)= \bx! \,\by! \, f.
\end{align*}
\end{lemma}

For any \(x \in \Z_{\geq 0}\) and \(i_1,\ldots, i_x \in I\), Lemma~\ref{CoxArb} implies there is a natural way to associate a permutation \(\sigma \in \mathfrak{S}_x\) to a morphism in \(\Hom_{\AffWebAaI}\left( i_1^{(1)} \cdots i_x^{(1)}, i_{\sigma^{-1}(1)}^{(1)} \cdots i_{\sigma^{-1}(x)}^{(1)}\right)\) which associates simple transpositions with the appropriately colored crossing. 

\begin{lemma}\label{bigwishtosym}
For all \(i \in I\), and \(\bx = (x_1, \ldots, x_m) \in \Z_{\geq 0}^m\), \(\by = (y_1, \ldots, y_n) \in \Z_{\geq 0}^n\) with \(\lVert \bx \rVert = \lVert \by \rVert\), we have
\begin{align}\label{GreenProdRule}
\hackcenter{}
\hackcenter{
\begin{overpic}[height=35mm]{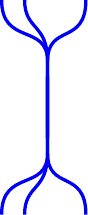}
    \put(-2,-1){\makebox(0,0)[t]{$\scriptstyle i^{(\hspace{-0.2mm}x_1\hspace{-0.4mm})}$}}
    \put(14,-1){\makebox(0,0)[t]{$\scriptstyle  i^{(\hspace{-0.2mm}x_2\hspace{-0.3mm})}$}}
     \put(27,-1){\makebox(0,0)[t]{$\scriptstyle \cdots $}}
     \put(40,-1){\makebox(0,0)[t]{$\scriptstyle  i^{(\hspace{-0.2mm}x_m\hspace{-0.2mm})}$}}
      \put(10,50){\makebox(0,0)[]{$\scriptstyle  i^{(\lVert \bx \rVert)}$}}
    \put(-2,101){\makebox(0,0)[b]{$\scriptstyle i^{(\hspace{-0.2mm}y_1\hspace{-0.4mm})}$}}
    \put(14,101){\makebox(0,0)[b]{$\scriptstyle  i^{(\hspace{-0.2mm}y_2\hspace{-0.3mm})}$}}
     \put(27,101){\makebox(0,0)[b]{$\scriptstyle \cdots $}}
     \put(40,101){\makebox(0,0)[b]{$\scriptstyle  i^{(\hspace{-0.2mm}y_n\hspace{-0.2mm})}$}}
     %
\end{overpic}
}
=
\hackcenter{}
\;
\sum_{\alpha}
\;\;
\hackcenter{
\begin{overpic}[height=35mm]{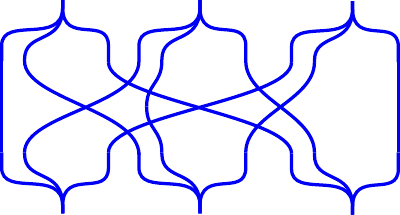}
    \put(17,-0){\makebox(0,0)[t]{$\scriptstyle i^{(\hspace{-0.2mm}x_1\hspace{-0.4mm})}$}}
    \put(51,-0){\makebox(0,0)[t]{$\scriptstyle  i^{(\hspace{-0.2mm}x_2\hspace{-0.3mm})}$}}
     \put(68,-0){\makebox(0,0)[t]{$\scriptstyle \cdots $}}
     \put(88,-0){\makebox(0,0)[t]{$\scriptstyle  i^{(\hspace{-0.2mm}x_m\hspace{-0.2mm})}$}}
      \put(17,54){\makebox(0,0)[b]{$\scriptstyle i^{(\hspace{-0.2mm}y_1\hspace{-0.4mm})}$}}
    \put(51,54){\makebox(0,0)[b]{$\scriptstyle  i^{(\hspace{-0.2mm}y_2\hspace{-0.3mm})}$}}
     \put(68,54){\makebox(0,0)[b]{$\scriptstyle \cdots $}}
     \put(88,54){\makebox(0,0)[b]{$\scriptstyle  i^{(\hspace{-0.2mm}y_n\hspace{-0.2mm})}$}}
     \put(6,6){\makebox(0,0)[t]{$\scriptstyle {i}^{(\hspace{-0.2mm}\alpha_{1\hspace{-0.2mm},\hspace{-0.2mm}1\hspace{-0.2mm}})}$}}
     \put(13,18){\makebox(0,0)[t]{$\scriptstyle {i}^{(\hspace{-0.2mm}\alpha_{1\hspace{-0.2mm},\hspace{-0.2mm}2}\hspace{-0.2mm})}$}}
      \put(20,13){\makebox(0,0)[t]{$\scriptstyle {}^{\cdots}$}}
      \put(24,6){\makebox(0,0)[t]{$\scriptstyle {i}^{(\hspace{-0.2mm}\alpha_{1\hspace{-0.2mm},\hspace{-0.2mm}n}\hspace{-0.2mm})}$}}
      \put(40,6){\makebox(0,0)[t]{$\scriptstyle {i}^{(\hspace{-0.2mm}\alpha_{2\hspace{-0.2mm},\hspace{-0.2mm}1}\hspace{-0.2mm})}$}}
     \put(47.5,18){\makebox(0,0)[t]{$\scriptstyle {i}^{(\hspace{-0.2mm}\alpha_{2\hspace{-0.2mm},\hspace{-0.2mm}2}\hspace{-0.2mm})}$}}
      \put(53,13){\makebox(0,0)[t]{$\scriptstyle {}^{\cdots}$}}
      \put(58,6){\makebox(0,0)[t]{$\scriptstyle {i}^{(\hspace{-0.2mm}\alpha_{2\hspace{-0.2mm},\hspace{-0.2mm}n}\hspace{-0.2mm})}$}}
      \put(78,6){\makebox(0,0)[t]{$\scriptstyle {i}^{(\hspace{-0.2mm}\alpha_{m\hspace{-0.2mm},\hspace{-0.2mm}1}\hspace{-0.2mm})}$}}
     \put(86,18){\makebox(0,0)[t]{$\scriptstyle {i}^{(\hspace{-0.2mm}\alpha_{m\hspace{-0.2mm},\hspace{-0.2mm}2}\hspace{-0.2mm})}$}}
      \put(90,13){\makebox(0,0)[t]{$\scriptstyle {}^{\cdots}$}}
      \put(96,6){\makebox(0,0)[t]{$\scriptstyle {i}^{(\hspace{-0.2mm}\alpha_{m\hspace{-0.2mm},\hspace{-0.2mm}n}\hspace{-0.2mm})}$}}
      \put(8,48){\makebox(0,0)[b]{$\scriptstyle {i}^{(\hspace{-0.2mm}\alpha_{1\hspace{-0.2mm},\hspace{-0.2mm}1}\hspace{-0.2mm})}$}}
       \put(13,36){\makebox(0,0)[b]{$\scriptstyle {i}^{(\hspace{-0.2mm}\alpha_{2\hspace{-0.2mm},\hspace{-0.2mm}1}\hspace{-0.2mm})}$}}
        \put(20,39){\makebox(0,0)[b]{$\scriptstyle {}^{\cdots}$}}
         \put(24.5,48){\makebox(0,0)[b]{$\scriptstyle {i}^{(\hspace{-0.2mm}\alpha_{m\hspace{-0.2mm},\hspace{-0.2mm}1})}$}}
        \put(42,48){\makebox(0,0)[b]{$\scriptstyle {i}^{(\hspace{-0.2mm}\alpha_{1\hspace{-0.2mm},\hspace{-0.2mm}2}\hspace{-0.2mm})}$}}
     \put(47.5,36){\makebox(0,0)[b]{$\scriptstyle {i}^{(\hspace{-0.2mm}\alpha_{2\hspace{-0.2mm},\hspace{-0.2mm}2}\hspace{-0.2mm})}$}}
      \put(53,39){\makebox(0,0)[b]{$\scriptstyle {}^{\cdots}$}}
      \put(59,48){\makebox(0,0)[b]{$\scriptstyle {i}^{(\hspace{-0.2mm}\alpha_{m\hspace{-0.2mm},\hspace{-0.2mm}2}\hspace{-0.2mm})}$}}
        \put(81,48){\makebox(0,0)[b]{$\scriptstyle {i}^{(\hspace{-0.2mm}\alpha_{1\hspace{-0.2mm},\hspace{-0.2mm}n}\hspace{-0.2mm})}$}}
     \put(86,36){\makebox(0,0)[b]{$\scriptstyle {i}^{(\hspace{-0.2mm}\alpha_{2\hspace{-0.2mm},\hspace{-0.2mm}n}\hspace{-0.2mm})}$}}
      \put(90,39){\makebox(0,0)[b]{$\scriptstyle {}^{\cdots}$}}
      \put(97,48){\makebox(0,0)[b]{$\scriptstyle {i}^{(\hspace{-0.2mm}\alpha_{m\hspace{-0.2mm},\hspace{-0.2mm}n}\hspace{-0.2mm})}$}}
\end{overpic}
},
\\
\nonumber
\end{align}
where the sum is over all \(m \times n\) matrices \(\alpha = (\alpha_{r,s})_{r \in [1,m],s \in [1,n]}\) such that \(\sum_{r=1}^m \alpha_{r,t} = y_t\) and \(\sum_{s=1}^n \alpha_{t,s} = x_t\). In particular, we have
\begin{align*}
{}
\hackcenter{
\begin{tikzpicture}[scale=0.8]
  \draw[ultra thick, blue] (0,0)--(0,0.2) .. controls ++(0,0.35) and ++(0,-0.35)  .. (0.7,0.8)--(0.7,1.2)
  .. controls ++(0,0.35) and ++(0,-0.35)  .. (0,1.8)--(0,2); 
    \draw[ultra thick, blue] (1.4,0)--(1.4,0.2) .. controls ++(0,0.35) and ++(0,-0.35)  .. (0.7,0.8)--(0.7,1.2)
  .. controls ++(0,0.35) and ++(0,-0.35)  .. (1.4,1.8)--(1.4,2); 
     \node[below] at (0,0) {$\scriptstyle i^{\scriptstyle (1)}$};
     \node[below] at (1.4,0) {$\scriptstyle i^{\scriptstyle (1)}$};
      \node[above] at (0,2) {$\scriptstyle i^{\scriptstyle (1)}$};
      \node[above] at (1.4,2) {$\scriptstyle i^{\scriptstyle (1)}$};
                  \node[left] at (0.6,1) {$\scriptstyle i^{\scriptstyle (x)}$};
        \node[above] at (0.7,2) {$\scriptstyle \cdots$};
         \node[below] at (0.7,-0.2) {$\scriptstyle \cdots$};
\end{tikzpicture}}
=
\sum_{\sigma \in \mathfrak{S}_x}
\hackcenter{
\begin{tikzpicture}[scale=0.8]
  \draw[ultra thick, blue] (0,0)--(0,2);
   \draw[ultra thick, blue] (1.4,0)--(1.4,2);
     \node[below] at (0,0) {$\scriptstyle i^{\scriptstyle (1)}$};
     \node[below] at (1.4,0) {$\scriptstyle i^{\scriptstyle (1)}$};
      \node[above] at (0,2) {$\scriptstyle i^{\scriptstyle (1)}$};
      \node[above] at (1.4,2) {$\scriptstyle i^{\scriptstyle (1)}$};
                  \node[left] at (0.6,1) {$\scriptstyle i^{\scriptstyle (x)}$};
        \node[above] at (0.7,2) {$\scriptstyle \cdots$};
         \node[below] at (0.7,-0.2) {$\scriptstyle \cdots$};
         \draw[draw=black, rounded corners, thick, fill=orange!50!yellow] (-0.2,0.5) rectangle ++(1.8,1);
          \node[] at (0.7,1) {$ \sigma$};
\end{tikzpicture}}.
\end{align*}
for all \(x \in \Z_{>0}\).
\end{lemma}
\begin{proof}
One may see this via a straightforward induction proof using relation (\ref{MSrel}) as the base case. More directly, we have by \cite[Theorem 7.4.2]{DKMZ} that \(\bigoplus_{\bx, \by,  \lVert \bx \rVert = \lVert \by \rVert = d} \WebAaI(i^{(\bx)}, i^{(\by)})\) is isomorphic to the classical Schur algebra \(S_{\k}(d,d)\), and from this point of view, (\ref{GreenProdRule}) merely expresses Green's product rule (see \cite[(2.3b)]{Green}) applied to the product
\(\xi_{1^{y_1}2^{y_2} \cdots n^{y_n}, 1^d} \cdot \xi_{1^d, 1^{x_1} 2^{x_2} \cdots m^{x_m}}\).
\end{proof}

\section{Teleporter morphisms}\label{teleportsec}
\label{SSS:Additional-shorthand-coupons}
Assume throughout this section that \(\calA = ((A, \a)_I, \tr , \psi)\) is a Frobenius great pair. 
Let \(x \in \Z_{\geq 0}\), \(i,j \in I\), \(f \in jAi\). We first define some shorthand coupons in \(\WebAaI(i^{(x)}, j^{(x)})\).
Set:
\begin{align}\label{greendotdef}
\hackcenter{
\begin{tikzpicture}[scale=0.8]
  \draw[ultra thick, blue] (0,0)--(0,0.9);
  \draw[ultra thick, red] (0,0.9)--(0,1.8);
 \draw[draw=black, rounded corners, thick, fill=lime] (-0.5,0.6) rectangle ++(1,0.6);
    \node at (0,0.9) {$\scriptstyle [f]^{\scriptstyle \diamond}$};
     \node[below] at (0,0) {$\scriptstyle i^{\scriptstyle(x)}$};
      \node[above] at (0,1.8) {$\scriptstyle j^{\scriptstyle(x)}$};
\end{tikzpicture}}
:=
\hackcenter{
\begin{tikzpicture}[scale=0.8]
  \draw[ultra thick, blue] (0,0)--(0,0.1) .. controls ++(0,0.35) and ++(0,-0.35) .. (-0.6,0.6)--(-0.6,0.9);
    \draw[ultra thick, red]  (-0.6,0.9)--(-0.6,1.2) 
  .. controls ++(0,0.35) and ++(0,-0.35) .. (0,1.7)--(0,1.8);
  \draw[ultra thick, blue] (0,0)--(0,0.1) .. controls ++(0,0.35) and ++(0,-0.35) .. (0.6,0.6)--(0.6,0.9);
    \draw[ultra thick, red]  (0.6,0.9)--(0.6,1.2) 
  .. controls ++(0,0.35) and ++(0,-0.35) .. (0,1.7)--(0,1.8);
  .. controls ++(0,0.35) and ++(0,-0.35) .. (0,1.7)--(0,1.8);
   \draw[thick, fill=yellow] (-0.6,0.9) circle (7pt);
     \draw[thick, fill=yellow] (0.6,0.9) circle (7pt);
       \node[above] at (0,1.8) {$\scriptstyle j^{\scriptstyle(x)}$};
         \node[below] at (0,0) {$\scriptstyle i^{\scriptstyle(x)}$};
           \node at (-0.6,0.9) {$\scriptstyle f$};
             \node at (0.6,0.9) {$\scriptstyle f$};
             \node[right] at (0.5,0.3) {$\scriptstyle i^{\scriptstyle(1)}$};
               \node[left] at (-0.5,0.3) {$\scriptstyle i^{\scriptstyle(1)}$};
                \node at (0,0.9) {$\scriptstyle \cdots $};
\end{tikzpicture}}
\qquad
\text{ and }
\qquad
\hackcenter{
\begin{tikzpicture}[scale=0.8]
  \draw[ultra thick, blue] (0,0)--(0,0.9);
  \draw[ultra thick, red] (0,0.9)--(0,1.8);
 \draw[draw=black, rounded corners, thick, fill=lime] (-0.5,0.6) rectangle ++(1,0.6);
    \node at (0,0.9) {$\scriptstyle [f] $};
     \node[below] at (0,0) {$\scriptstyle i^{\scriptstyle(x)}$};
      \node[above] at (0,1.8) {$\scriptstyle j^{\scriptstyle(x)}$};
\end{tikzpicture}}
\;
&:=
\begin{cases}
\;\;\;\;
\hackcenter{
\begin{tikzpicture}[scale=0.8]
  \draw[ultra thick, blue] (0,0)--(0,0.9);
  \draw[ultra thick, red] (0,0.9)--(0,1.8);
     \draw[thick, fill=yellow] (0,0.9) circle (8pt);
    \node at (0,0.9) {$\scriptstyle f$};
     \node[below] at (0,0) {$\scriptstyle i^{\scriptstyle(x)}$};
      \node[above] at (0,1.8) {$\scriptstyle j^{\scriptstyle(x)}$};
\end{tikzpicture}}
&
\textup{if }f \in j\a i, \textup{ or } x=1 \textup{ and } f \in jA_{\bar 1}i;
\\
\;\;\;\;
\hackcenter{
\begin{tikzpicture}[scale=0.8]
  \draw[ultra thick, blue] (0,0)--(0,0.9);
  \draw[ultra thick, red] (0,0.9)--(0,1.8);
 \draw[draw=black, rounded corners, thick, fill=lime] (-0.5,0.6) rectangle ++(1,0.6);
    \node at (0,0.9) {$\scriptstyle [f]^{\scriptstyle \diamond}$};
     \node[below] at (0,0) {$\scriptstyle i^{\scriptstyle(x)}$};
      \node[above] at (0,1.8) {$\scriptstyle j^{\scriptstyle(x)}$};
\end{tikzpicture}}
&
\textup{if }
f \in jA_{\bar 0}i \backslash j\a i.
\end{cases}
\end{align}

We will say a finitely supported function \(\mu: {}_j\BasisB_i \to \Z_{\geq 0}\) is a {\em restricted \( {}_j\BasisB_i\)-composition} provided that \(\mu(b) \leq 1\) whenever \(\overline{b} = \bar 1\). We define the following associated values:
\begin{align*}
|\mu| = \sum_{b \in  {}_j\BasisB_i} \mu(b), 
\qquad
\{\mu\}^! = \prod_{b \in {}_j\BasisB_i} \mu(b)!,
\qquad
\{\mu\}^!_{\a} = \prod_{\substack{b \in {}_j\BasisB_i \\ b \in \a} } \mu(b)!,
\qquad \text{and}\qquad
\{\mu^\vee\}^!_{\a} = \prod_{\substack{b \in {}_j\BasisB_i \\ b^\vee \in \a } } \mu(b)!.
\end{align*}
Then for such \(\mu\), we enumerate \(\textup{supp}(\mu) = \{b^{(1)}_\mu >  \cdots > b^{(d_\mu)}_\mu\}\), where \(d_\mu = |\textup{supp}(\mu)|\), and define the following morphisms:
\begin{align}\label{resmu}
\hackcenter{
\begin{tikzpicture}[scale=0.8]
  \draw[ultra thick, blue] (0,0)--(0,0.9);
  \draw[ultra thick, red] (0,0.9)--(0,1.8);
  \draw[draw=black, rounded corners, thick, fill=lime] (-0.3,0.6) rectangle ++(0.6,0.6);
    \node at (0,0.9) {$\scriptstyle \mu$};
     \node[below] at (0,0) {$\scriptstyle i^{\scriptstyle(|\mu|)}$};
      \node[above] at (0,1.8) {$\scriptstyle j^{\scriptstyle(|\mu|)}$};
\end{tikzpicture}}
:=
\hackcenter{
\begin{tikzpicture}[scale=0.8]
  \draw[ultra thick, blue] (0,0)--(0,0.05) .. controls ++(0,0.35) and ++(0,-0.35) .. (-1.2,0.45)--(-1.2,1.2);
    \draw[ultra thick, red]  (-1.2,1.2)--(-1.2,1.35) 
  .. controls ++(0,0.35) and ++(0,-0.35) .. (0,1.75)--(0,1.8);
  \draw[ultra thick, blue] (0,0)--(0,0.05) .. controls ++(0,0.35) and ++(0,-0.35) .. (1.2,0.45)--(1.2,1.2);
    \draw[ultra thick, red]  (1.2,1.2)--(1.2,1.35) 
  .. controls ++(0,0.35) and ++(0,-0.35) .. (0,1.75)--(0,1.8);
  .. controls ++(0,0.35) and ++(0,-0.35) .. (0,1.75)--(0,1.8);
    \draw[draw=black, rounded corners, thick, fill=lime] (-1.9,0.6) rectangle ++(1.2,0.6);
        \draw[draw=black, rounded corners, thick, fill=lime] (0.7,0.6) rectangle ++(1.2,0.6);
  \node at (-1.3,0.9) {$\scriptstyle [b_{\mu}^{(1)}]$};
    \node at (1.3,0.9) {$\scriptstyle [b_{\mu}^{(d_\mu)}]$};
       \node[above] at (0,1.8) {$\scriptstyle j^{\scriptstyle(|\mu|)}$};
         \node[below] at (0,0) {$\scriptstyle i^{\scriptstyle(|\mu|)}$};
             \node[left] at (-0.6,0.05) {$\scriptstyle i^{\scriptstyle(\mu(b_{\mu}^{(1)}))}$};
               \node[right] at (0.6,0.05) {$\scriptstyle i^{\scriptstyle(\mu(b_{\mu}^{(d_\mu)}))}$};
                \node at (0.04,0.9) {$\scriptstyle \cdots $};
\end{tikzpicture}}
\qquad
\textup{and}
\qquad
\hackcenter{
\begin{tikzpicture}[scale=0.8]
  \draw[ultra thick, blue] (0,0)--(0,0.9);
  \draw[ultra thick, red] (0,0.9)--(0,1.8);
  \draw[draw=black, rounded corners, thick, fill=lime] (-0.3,0.6) rectangle ++(0.6,0.6);
    \node at (0,0.9) {$\scriptstyle \nakayamamu$}; 
     \node[below] at (0,0) {$\scriptstyle \nakayamai^{\scriptstyle(|\mu|)}$};
      \node[above] at (0,1.8) {$\scriptstyle \nakayamaj^{\scriptstyle(|\mu|)}$};
\end{tikzpicture}}
:=
\hackcenter{
\begin{tikzpicture}[scale=0.8]
  \draw[ultra thick, blue] (0,0)--(0,0.05) .. controls ++(0,0.35) and ++(0,-0.35) .. (-1.2,0.45)--(-1.2,1.2);
    \draw[ultra thick, red]  (-1.2,1.2)--(-1.2,1.35) 
  .. controls ++(0,0.35) and ++(0,-0.35) .. (0,1.75)--(0,1.8);
  \draw[ultra thick, blue] (0,0)--(0,0.05) .. controls ++(0,0.35) and ++(0,-0.35) .. (1.2,0.45)--(1.2,1.2);
    \draw[ultra thick, red]  (1.2,1.2)--(1.2,1.35) 
  .. controls ++(0,0.35) and ++(0,-0.35) .. (0,1.75)--(0,1.8);
  .. controls ++(0,0.35) and ++(0,-0.35) .. (0,1.75)--(0,1.8);
    \draw[draw=black, rounded corners, thick, fill=lime] (-1.9,0.6) rectangle ++(1.2,0.6);
        \draw[draw=black, rounded corners, thick, fill=lime] (0.6,0.6) rectangle ++(1.4,0.6);
  \node at (-1.3,0.9) {$\scriptstyle [\nakayamab_{\mu}^{(1)}]$};
    \node at (1.3,0.9) {$\scriptstyle [\nakayamab_{\mu}^{(d_\mu)}]$};
       \node[above] at (0,1.8) {$\scriptstyle \nakayamaj^{\scriptstyle(|\mu|)}$};
         \node[below] at (0,0) {$\scriptstyle \nakayamai^{\scriptstyle(|\mu|)}$};
             \node[left] at (-0.6,0.05) {$\scriptstyle \nakayamai^{\scriptstyle(\mu(b_{\mu}^{(1)}))}$};
               \node[right] at (0.6,0.05) {$\scriptstyle \nakayamai^{\scriptstyle(\mu(b_{\mu}^{(d_\mu)}))}$};
                \node at (0.04,0.9) {$\scriptstyle \cdots $};
\end{tikzpicture}}
\end{align}
recalling that we write ${}_{t}a$ for $\psi^{-t}(a)$ for all $t \in \Z$ and $a \in A$.
We further define a `checked' version of this morphism:
\begin{align}\label{resmu}
{}
\hackcenter{
\begin{tikzpicture}[scale=0.8]
  \draw[ultra thick, red] (0,0)--(0,0.9);
  \draw[ultra thick, blue] (0,0.9)--(0,1.8);
  \draw[draw=black, rounded corners, thick, fill=lime] (-0.3,0.6) rectangle ++(0.6,0.6);
    \node at (0,0.9) {$\scriptstyle \mu^\vee$};
     \node[below] at (0,0) {$\scriptstyle j^{\scriptstyle(|\mu|)}$};
      \node[above] at (0,1.8) {$\scriptstyle \nakayamai^{\scriptstyle(|\mu|)}$};
\end{tikzpicture}}
:=
\hackcenter{
\begin{tikzpicture}[scale=0.8]
  \draw[ultra thick, red] (0,0)--(0,0.05) .. controls ++(0,0.35) and ++(0,-0.35) .. (-1.2,0.45)--(-1.2,1.2);
    \draw[ultra thick, blue]  (-1.2,1.2)--(-1.2,1.35) 
  .. controls ++(0,0.35) and ++(0,-0.35) .. (0,1.75)--(0,1.8);
  \draw[ultra thick, red] (0,0)--(0,0.05) .. controls ++(0,0.35) and ++(0,-0.35) .. (1.2,0.45)--(1.2,1.2);
    \draw[ultra thick, blue]  (1.2,1.2)--(1.2,1.35) 
  .. controls ++(0,0.35) and ++(0,-0.35) .. (0,1.75)--(0,1.8);
  .. controls ++(0,0.35) and ++(0,-0.35) .. (0,1.75)--(0,1.8);
    \draw[draw=black, rounded corners, thick, fill=lime] (-2.1,0.6) rectangle ++(1.6,0.6);
        \draw[draw=black, rounded corners, thick, fill=lime] (0.5,0.6) rectangle ++(1.6,0.6);
  \node at (-1.3,0.9) {$\scriptstyle [(b_{\mu}^{(d_\mu)})^\vee]$};
    \node at (1.3,0.9) {$\scriptstyle [(b_{\mu}^{(1)})^\vee]$};
       \node[above] at (0,1.8) {$\scriptstyle \nakayamai^{\scriptstyle(|\mu|)}$};
         \node[below] at (0,0) {$\scriptstyle j^{\scriptstyle(|\mu|)}$};
             \node[left] at (-0.6,0.05) {$\scriptstyle j^{\scriptstyle(\mu(b_{\mu}^{(d_\mu)}))}$};
               \node[right] at (0.6,0.05) {$\scriptstyle j^{\scriptstyle(\mu(b_{\mu}^{(1)}))}$};
                \node at (0.04,0.9) {$\scriptstyle \cdots $};
\end{tikzpicture}}
\end{align}
For \(t \in \Z_{\geq 0}\), we write \({}_j\mathcal{B}^t_i\) for the set of restricted \({}_j \BasisB_i\)-compositions \(\mu\) such that \(|\mu| = t\). Then we set \({}_j \mathcal{B}_i = \bigcup_{t \in \Z_{\geq 0}} {}_j \mathcal{B}_i^t\). For any \(m \in \Z_{\geq 0}\), \(i,j, k_1, \ldots, k_m \in I\), \(x, y_1, \ldots, y_m \in \Z_{\geq 0}\), we utilize the following shorthand notation:
\begin{align}\label{thicktransporterdef}
{}
\hackcenter{
\begin{tikzpicture}[scale=0.8]
  \draw[ultra thick, blue] (0,0)--(0,0.9);
  \draw[ultra thick, red] (0,0.9)--(0,1.8);
     \node[below] at (0,0) {$\scriptstyle i^{\scriptstyle(x)}$};
      \node[above] at (0,1.8) {$\scriptstyle j^{\scriptstyle(x)}$};
      \draw[ultra thick, red] (0+3,0)--(0+3,0.9);
        \draw[ultra thick, blue] (0+3,0.9)--(0+3,1.8);
     \node[below] at (0+3,0) {$\scriptstyle j^{\scriptstyle(x)}$};
      \node[above] at (0+3,1.8) {$\scriptstyle \nakayamai^{\scriptstyle(x)}$};
        \draw[ultra thick, gray] (1,0)--(1,1.8);
           \node[below] at (1,0) {$\scriptstyle k_1^{\scriptstyle(y_1)}$};
              \node[above] at (1,1.8) {$\scriptstyle k_1^{\scriptstyle(y_1)}$};
         \draw[ultra thick, gray] (2,0)--(2,1.8);
              \node[below] at (2,0) {$\scriptstyle k_m^{\scriptstyle(y_m)}$};
              \node[above] at (2,1.8) {$\scriptstyle k_m^{\scriptstyle(y_m)}$};
              \node[] at (1.5,1.3) {$\scriptstyle \cdots$};
              \node[] at (1.5,0.5) {$\scriptstyle \cdots$};
                        \draw[thick, black, fill=black]  (0,0.9) circle (6pt);
        \draw[thick,black, fill=black]  (3,0.9) circle (6pt);
                   \draw[decorate, decoration={snake,segment length=5pt, amplitude=1pt}, line width=1mm, black] (0,0.9)--(3,0.9);
      \draw[thick, lime, fill=lime]  (0,0.9) circle (5pt);
        \draw[thick, lime, fill=lime]  (3,0.9) circle (5pt);
      \draw[decorate, decoration={snake,segment length=5pt, amplitude=1pt}, line width=0.5mm, lime] (0,0.9)--(3,0.9);
\end{tikzpicture}}
:=
\sum_{\mu \in {}_j\mathcal{B}^x_i}
\frac{\{\mu\}^!_{\a}\{\mu^\vee\}^!_{\a}}{\{\mu\}^!}
\hackcenter{
\begin{tikzpicture}[scale=0.8]
  \draw[ultra thick, blue] (0,0)--(0,0.9);
  \draw[ultra thick, red] (0,0.9)--(0,1.8);
  \draw[draw=black, rounded corners, thick, fill=lime] (-0.3,0.6) rectangle ++(0.6,0.6);
    \node at (0,0.9) {$\scriptstyle \mu$};
     \node[below] at (0,0) {$\scriptstyle i^{\scriptstyle(x)}$};
      \node[above] at (0,1.8) {$\scriptstyle j^{\scriptstyle(x)}$};
      \draw[ultra thick, red] (0+3,0)--(0+3,0.9);
        \draw[ultra thick, blue] (0+3,0.9)--(0+3,1.8);
  \draw[draw=black, rounded corners, thick, fill=lime] (-0.3+3,0.6) rectangle ++(0.6,0.6);
    \node at (0+3,0.9) {$\scriptstyle \mu^\vee$};
     \node[below] at (0+3,0) {$\scriptstyle j^{\scriptstyle(x)}$};
      \node[above] at (0+3,1.8) {$\scriptstyle \nakayamai^{\scriptstyle(x)}$};
        \draw[ultra thick, gray] (1,0)--(1,1.8);
           \node[below] at (1,0) {$\scriptstyle k_1^{\scriptstyle(y_1)}$};
              \node[above] at (1,1.8) {$\scriptstyle k_1^{\scriptstyle(y_1)}$};
         \draw[ultra thick, gray] (2,0)--(2,1.8);
              \node[below] at (2,0) {$\scriptstyle k_m^{\scriptstyle(y_m)}$};
              \node[above] at (2,1.8) {$\scriptstyle k_m^{\scriptstyle(y_m)}$};
              \node[] at (1.5,0.9) {$\scriptstyle \cdots$};
\end{tikzpicture}}
\end{align}
and
\begin{align}\label{twistedthicktransporterdef}
{}
\hackcenter{
\begin{tikzpicture}[scale=0.8]
  \draw[ultra thick, blue] (0,0)--(0,0.9);
  \draw[ultra thick, red] (0,0.9)--(0,1.8);
     \node[below] at (0,0) {$\scriptstyle i^{\scriptstyle(x)}$};
      \node[above] at (0,1.8) {$\scriptstyle j^{\scriptstyle(x)}$};
      \draw[ultra thick, red] (0+3,0)--(0+3,0.9);
        \draw[ultra thick, blue] (0+3,0.9)--(0+3,1.8);
     \node[below] at (0+3,0) {$\scriptstyle j^{\scriptstyle(x)}$};
      \node[above] at (0+3,1.8) {$\scriptstyle \nakayamai^{\scriptstyle(x)}$};
        \draw[ultra thick, gray] (1,0)--(1,1.8);
           \node[below] at (1,0) {$\scriptstyle k_1^{\scriptstyle(y_1)}$};
              \node[above] at (1,1.8) {$\scriptstyle k_1^{\scriptstyle(y_1)}$};
         \draw[ultra thick, gray] (2,0)--(2,1.8);
              \node[below] at (2,0) {$\scriptstyle k_m^{\scriptstyle(y_m)}$};
              \node[above] at (2,1.8) {$\scriptstyle k_m^{\scriptstyle(y_m)}$};
              \node[] at (1.5,1.3) {$\scriptstyle \cdots$};
              \node[] at (1.5,0.5) {$\scriptstyle \cdots$};
                        \draw[thick, black, fill=black]  (0,0.9) circle (6pt);
        \draw[thick,black, fill=black]  (3,0.9) circle (6pt);
                   \draw[decorate, decoration={snake,segment length=5pt, amplitude=1pt}, line width=1mm, black] (0,0.9)--(3,0.9);
      \draw[thick, lime, fill=lime]  (0,0.9) circle (5pt);
        \draw[thick, lime, fill=lime]  (3,0.9) circle (5pt);
      \draw[decorate, decoration={snake,segment length=5pt, amplitude=1pt}, line width=0.5mm, lime] (0,0.9)--(3,0.9);
      \node[] at (0,0.9) {$\scriptstyle 1$};
\end{tikzpicture}}
:=
\sum_{\mu \in {}_j\mathcal{B}^x_i}
\frac{\{\mu\}^!_{\a}\{\mu^\vee\}^!_{\a}}{\{\mu\}^!}
\hackcenter{
\begin{tikzpicture}[scale=0.8]
  \draw[ultra thick, blue] (0,0)--(0,0.9);
  \draw[ultra thick, red] (0,0.9)--(0,1.8);
  \draw[draw=black, rounded corners, thick, fill=lime] (-0.3,0.6) rectangle ++(0.6,0.6);
    \node at (0,0.9) {$\scriptstyle \nakayamamu$};
     \node[below] at (0,0) {$\scriptstyle i^{\scriptstyle(x)}$};
      \node[above] at (0,1.8) {$\scriptstyle j^{\scriptstyle(x)}$};
      \draw[ultra thick, red] (0+3,0)--(0+3,0.9);
        \draw[ultra thick, blue] (0+3,0.9)--(0+3,1.8);
  \draw[draw=black, rounded corners, thick, fill=lime] (-0.3+3,0.6) rectangle ++(0.6,0.6);
    \node at (0+3,0.9) {$\scriptstyle \mu^\vee$};
     \node[below] at (0+3,0) {$\scriptstyle j^{\scriptstyle(x)}$};
      \node[above] at (0+3,1.8) {$\scriptstyle \nakayamai^{\scriptstyle(x)}$};
        \draw[ultra thick, gray] (1,0)--(1,1.8);
           \node[below] at (1,0) {$\scriptstyle k_1^{\scriptstyle(y_1)}$};
              \node[above] at (1,1.8) {$\scriptstyle k_1^{\scriptstyle(y_1)}$};
         \draw[ultra thick, gray] (2,0)--(2,1.8);
              \node[below] at (2,0) {$\scriptstyle k_m^{\scriptstyle(y_m)}$};
              \node[above] at (2,1.8) {$\scriptstyle k_m^{\scriptstyle(y_m)}$};
              \node[] at (1.5,0.9) {$\scriptstyle \cdots$};
\end{tikzpicture}}
\end{align}
The assumption that $\calA$ is a Frobenius great pair implies that either $b \in \a$ or $b^{\vee} \in \a$ for each $b \in {}_{j}\BasisB_{i}$.  This, in turn, implies  $\{\mu\}^!$  evenly divides the product $\{\mu\}^!_{\a}\{\mu^\vee\}^!_{\a}$ and, hence,  the coefficients in the above sums are integers that can be viewed as elements of $\k$. 
We call elements of the form \cref{thicktransporterdef} (\emph{\(x\)-thick}) \emph{teleporters}, as they can `teleport' coupons from one corner of the barbell shape to the other, see \cref{teleporterrules}. These may be viewed as thickened versions of the teleporters in \cite{BSW}.

\subsection{Teleporter relations}\label{telerel}
Although the definition of the thick teleporter decorations above is somewhat dense and combinatorial, we show below that, morally speaking, the \(x\)-thick teleporter is  obtained by aggregating \(x\) many thin teleporters.

\begin{lemma}\label{blowupteleporter} 
For all \(x \in \Z_{\geq 0}\) and \(i,j \in I\), in  $\WebAaI$ we have:
\begin{align*}
\hackcenter{}
x!
\hackcenter{
\begin{tikzpicture}[scale=0.8]
  \draw[ultra thick,blue] (-0.9,0)--(-0.9,0.9);
   \draw[ultra thick,red] (0.9,0)--(0.9,0.9);
  \draw[ultra thick,blue] (0.9,0.9)--(0.9,1.8);
    \draw[ultra thick,red] (-0.9,0.9)--(-0.9,1.8);
      \node[above] at (-0.9,1.8) {$ \scriptstyle j^{ \scriptstyle (x)}$};
      \node[above] at (0.9,1.8) {$ \scriptstyle \nakayamai^{ \scriptstyle (x)}$};
       \node[below] at (-0.9,0) {$ \scriptstyle i^{ \scriptstyle (x)}$};
      \node[below] at (0.9,0) {$ \scriptstyle j^{ \scriptstyle (x)}$};
        \draw[thick, black, fill=black]  (-0.9,0.9) circle (6pt);
        \draw[thick,black, fill=black]  (0.9,0.9) circle (6pt);
                   \draw[decorate, decoration={snake,segment length=5pt, amplitude=1pt}, line width=1mm, black] (-0.9,0.9)--(0.9,0.9);
      \draw[thick, lime, fill=lime]  (-0.9,0.9) circle (5pt);
        \draw[thick, lime, fill=lime]  (0.9,0.9) circle (5pt);
      \draw[decorate, decoration={snake,segment length=5pt, amplitude=1pt}, line width=0.5mm, lime] (-0.9,0.9)--(0.9,0.9);
\end{tikzpicture}}
=
\hackcenter{
\begin{tikzpicture}[scale=0.8]
  \draw[ultra thick, blue] (0,0)--(0,0.1) .. controls ++(0,0.35) and ++(0,-0.35) .. (-0.6,0.6)--(-0.6,0.9);
    \draw[ultra thick, red]  (-0.6,0.9)--(-0.6,1.2) 
  .. controls ++(0,0.35) and ++(0,-0.35) .. (0,1.7)--(0,1.8);
  \draw[ultra thick, blue] (0,0)--(0,0.1) .. controls ++(0,0.35) and ++(0,-0.35) .. (0.6,0.6)--(0.6,0.9);
    \draw[ultra thick, red]  (0.6,0.9)--(0.6,1.2) 
  .. controls ++(0,0.35) and ++(0,-0.35) .. (0,1.7)--(0,1.8);
  .. controls ++(0,0.35) and ++(0,-0.35) .. (0,1.7)--(0,1.8);
       \node[above] at (0,1.8) {$\scriptstyle j^{\scriptstyle(x)}$};
         \node[below] at (0,0) {$\scriptstyle i^{\scriptstyle(x)}$};
             \node[right] at (0.1,0.1) {$\scriptstyle i^{\scriptstyle(1)}$};
               \node[left] at (-0.1,0.1) {$\scriptstyle i^{\scriptstyle(1)}$};
                \node at (0,0.85) {$\scriptstyle \ddots $};
  \draw[ultra thick, red] (0+1.8,0)--(0+1.8,0.1) .. controls ++(0,0.35) and ++(0,-0.35) .. (-0.6+1.8,0.6)--(-0.6+1.8,0.9);
    \draw[ultra thick, blue]  (-0.6+1.8,0.9)--(-0.6+1.8,1.2) 
  .. controls ++(0,0.35) and ++(0,-0.35) .. (0+1.8,1.7)--(0+1.8,1.8);
  \draw[ultra thick, red] (0+1.8,0)--(0+1.8,0.1) .. controls ++(0,0.35) and ++(0,-0.35) .. (0.6+1.8,0.6)--(0.6+1.8,0.9);
    \draw[ultra thick, blue]  (0.6+1.8,0.9)--(0.6+1.8,1.2) 
  .. controls ++(0,0.35) and ++(0,-0.35) .. (0+1.8,1.7)--(0+1.8,1.8);
  .. controls ++(0,0.35) and ++(0,-0.35) .. (0+1.8,1.7)--(0+1.8,1.8);
       \node[above] at (0+1.8,1.8) {$\scriptstyle \nakayamai^{\scriptstyle(x)}$};
         \node[below] at (0+1.8,0) {$\scriptstyle j^{\scriptstyle(x)}$};
           \node[right] at (0.1+1.8,0.1) {$\scriptstyle j^{\scriptstyle(1)}$};
               \node[left] at (-0.1+1.8,0.1) {$\scriptstyle j^{\scriptstyle(1)}$};
                \node at (0+1.8,1.2) {$\scriptstyle \ddots $};             
                        \draw[thick, black, fill=black]  (-0.6,1.1) circle (5pt);
        \draw[thick,black, fill=black]  (1.2,1.1) circle (5pt);
                   \draw[decorate, decoration={snake,segment length=5pt, amplitude=1pt}, line width=1mm, black] (-0.6,1.1)--(1.2,1.1);
      \draw[thick, lime, fill=lime]  (-0.6,1.1) circle (4pt);
        \draw[thick, lime, fill=lime]  (1.2,1.1) circle (4pt);
      \draw[decorate, decoration={snake,segment length=5pt, amplitude=1pt}, line width=0.5mm, lime] (-0.6,1.1)--(1.2,1.1);
                       \draw[thick, black, fill=black]  (-0.6+1.2,0.7) circle (5pt);
        \draw[thick,black, fill=black]  (1.2+1.2,0.7) circle (5pt);
                   \draw[decorate, decoration={snake,segment length=5pt, amplitude=1pt}, line width=1mm, black] (-0.6+1.2,0.7)--(1.2+1.2,0.7);
      \draw[thick, lime, fill=lime]  (-0.6+1.2,0.7) circle (4pt);
        \draw[thick, lime, fill=lime]  (1.2+1.2,0.7) circle (4pt);
      \draw[decorate, decoration={snake,segment length=5pt, amplitude=1pt}, line width=0.5mm, lime] (-0.6+1.2,0.7)--(1.2+1.2,0.7);
\end{tikzpicture}}
\qquad
\textup{and}
\qquad
x!
\hackcenter{
\begin{tikzpicture}[scale=0.8]
  \draw[ultra thick,blue] (-0.9,0)--(-0.9,0.9);
   \draw[ultra thick,red] (0.9,0)--(0.9,0.9);
  \draw[ultra thick,blue] (0.9,0.9)--(0.9,1.8);
    \draw[ultra thick,red] (-0.9,0.9)--(-0.9,1.8);
      \node[above] at (-0.9,1.8) {$ \scriptstyle \nakayamaj^{ \scriptstyle (x)}$};
      \node[above] at (0.9,1.8) {$ \scriptstyle i^{ \scriptstyle (x)}$};
       \node[below] at (-0.9,0) {$ \scriptstyle i^{ \scriptstyle (x)}$};
      \node[below] at (0.9,0) {$ \scriptstyle j^{ \scriptstyle (x)}$};
        \draw[thick, black, fill=black]  (-0.9,0.9) circle (6pt);
        \draw[thick,black, fill=black]  (0.9,0.9) circle (6pt);
                   \draw[decorate, decoration={snake,segment length=5pt, amplitude=1pt}, line width=1mm, black] (-0.9,0.9)--(0.9,0.9);
      \draw[thick, lime, fill=lime]  (-0.9,0.9) circle (5pt);
        \draw[thick, lime, fill=lime]  (0.9,0.9) circle (5pt);
      \draw[decorate, decoration={snake,segment length=5pt, amplitude=1pt}, line width=0.5mm, lime] (-0.9,0.9)--(0.9,0.9);
         \node[] at (-0.9,0.9) {$\scriptstyle 1$};
\end{tikzpicture}}
=
\hackcenter{
\begin{tikzpicture}[scale=0.8]
  \draw[ultra thick, blue] (0,0)--(0,0.1) .. controls ++(0,0.35) and ++(0,-0.35) .. (-0.6,0.6)--(-0.6,0.9);
    \draw[ultra thick, red]  (-0.6,0.9)--(-0.6,1.2) 
  .. controls ++(0,0.35) and ++(0,-0.35) .. (0,1.7)--(0,1.8);
  \draw[ultra thick, blue] (0,0)--(0,0.1) .. controls ++(0,0.35) and ++(0,-0.35) .. (0.6,0.6)--(0.6,0.9);
    \draw[ultra thick, red]  (0.6,0.9)--(0.6,1.2) 
  .. controls ++(0,0.35) and ++(0,-0.35) .. (0,1.7)--(0,1.8);
  .. controls ++(0,0.35) and ++(0,-0.35) .. (0,1.7)--(0,1.8);
       \node[above] at (0,1.8) {$\scriptstyle \nakayamaj^{\scriptstyle(x)}$};
         \node[below] at (0,0) {$\scriptstyle i^{\scriptstyle(x)}$};
             \node[right] at (0.1,0.1) {$\scriptstyle i^{\scriptstyle(1)}$};
               \node[left] at (-0.1,0.1) {$\scriptstyle i^{\scriptstyle(1)}$};
                \node at (0,0.85) {$\scriptstyle \ddots $};
  \draw[ultra thick, red] (0+1.8,0)--(0+1.8,0.1) .. controls ++(0,0.35) and ++(0,-0.35) .. (-0.6+1.8,0.6)--(-0.6+1.8,0.9);
    \draw[ultra thick, blue]  (-0.6+1.8,0.9)--(-0.6+1.8,1.2) 
  .. controls ++(0,0.35) and ++(0,-0.35) .. (0+1.8,1.7)--(0+1.8,1.8);
  \draw[ultra thick, red] (0+1.8,0)--(0+1.8,0.1) .. controls ++(0,0.35) and ++(0,-0.35) .. (0.6+1.8,0.6)--(0.6+1.8,0.9);
    \draw[ultra thick, blue]  (0.6+1.8,0.9)--(0.6+1.8,1.2) 
  .. controls ++(0,0.35) and ++(0,-0.35) .. (0+1.8,1.7)--(0+1.8,1.8);
  .. controls ++(0,0.35) and ++(0,-0.35) .. (0+1.8,1.7)--(0+1.8,1.8);
       \node[above] at (0+1.8,1.8) {$\scriptstyle i^{\scriptstyle(x)}$};
         \node[below] at (0+1.8,0) {$\scriptstyle j^{\scriptstyle(x)}$};
           \node[right] at (0.1+1.8,0.1) {$\scriptstyle j^{\scriptstyle(1)}$};
               \node[left] at (-0.1+1.8,0.1) {$\scriptstyle j^{\scriptstyle(1)}$};
                \node at (0+1.8,1.2) {$\scriptstyle \ddots $};             
                        \draw[thick, black, fill=black]  (-0.6,1.1) circle (5pt);
        \draw[thick,black, fill=black]  (1.2,1.1) circle (5pt);
                   \draw[decorate, decoration={snake,segment length=5pt, amplitude=1pt}, line width=1mm, black] (-0.6,1.1)--(1.2,1.1);
      \draw[thick, lime, fill=lime]  (-0.6,1.1) circle (4pt);
        \draw[thick, lime, fill=lime]  (1.2,1.1) circle (4pt);
      \draw[decorate, decoration={snake,segment length=5pt, amplitude=1pt}, line width=0.5mm, lime] (-0.6,1.1)--(1.2,1.1);
                       \draw[thick, black, fill=black]  (-0.6+1.2,0.7) circle (5pt);
        \draw[thick,black, fill=black]  (1.2+1.2,0.7) circle (5pt);
                   \draw[decorate, decoration={snake,segment length=5pt, amplitude=1pt}, line width=1mm, black] (-0.6+1.2,0.7)--(1.2+1.2,0.7);
      \draw[thick, lime, fill=lime]  (-0.6+1.2,0.7) circle (4pt);
        \draw[thick, lime, fill=lime]  (1.2+1.2,0.7) circle (4pt);
      \draw[decorate, decoration={snake,segment length=5pt, amplitude=1pt}, line width=0.5mm, lime] (-0.6+1.2,0.7)--(1.2+1.2,0.7);
        \node[] at (-0.6,1.1) {$\scriptstyle 1$};
          \node[] at (0.6,0.7) {$\scriptstyle 1$};
\end{tikzpicture}}.
\end{align*}
\end{lemma}
\begin{proof}
We prove the first statement, the second being similar. We argue by induction on \(x\), with the base case \(x=1\) trivial. Assuming the claim holds for \(x\), we have:

\begin{align*}
\hackcenter{
\begin{tikzpicture}[scale=0.8]
  \draw[ultra thick, blue] (0,0)--(0,0.1) .. controls ++(0,0.35) and ++(0,-0.35) .. (-0.6,0.6)--(-0.6,0.9);
    \draw[ultra thick, red]  (-0.6,0.9)--(-0.6,1.2) 
  .. controls ++(0,0.35) and ++(0,-0.35) .. (0,1.7)--(0,1.8);
  \draw[ultra thick, blue] (0,0)--(0,0.1) .. controls ++(0,0.35) and ++(0,-0.35) .. (0.6,0.6)--(0.6,0.9);
    \draw[ultra thick, red]  (0.6,0.9)--(0.6,1.2) 
  .. controls ++(0,0.35) and ++(0,-0.35) .. (0,1.7)--(0,1.8);
  .. controls ++(0,0.35) and ++(0,-0.35) .. (0,1.7)--(0,1.8);
       \node[above] at (0,1.8) {$\scriptstyle j^{\scriptstyle(x+1)}$};
         \node[below] at (0,0) {$\scriptstyle i^{\scriptstyle(x+1)}$};
             \node[right] at (0.1,0.1) {$\scriptstyle i^{\scriptstyle(1)}$};
               \node[left] at (-0.1,0.1) {$\scriptstyle i^{\scriptstyle(1)}$};
                \node at (0,0.85) {$\scriptstyle \ddots $};
  \draw[ultra thick, red] (0+1.8,0)--(0+1.8,0.1) .. controls ++(0,0.35) and ++(0,-0.35) .. (-0.6+1.8,0.6)--(-0.6+1.8,0.9);
    \draw[ultra thick, blue]  (-0.6+1.8,0.9)--(-0.6+1.8,1.2) 
  .. controls ++(0,0.35) and ++(0,-0.35) .. (0+1.8,1.7)--(0+1.8,1.8);
  \draw[ultra thick, red] (0+1.8,0)--(0+1.8,0.1) .. controls ++(0,0.35) and ++(0,-0.35) .. (0.6+1.8,0.6)--(0.6+1.8,0.9);
    \draw[ultra thick, blue]  (0.6+1.8,0.9)--(0.6+1.8,1.2) 
  .. controls ++(0,0.35) and ++(0,-0.35) .. (0+1.8,1.7)--(0+1.8,1.8);
  .. controls ++(0,0.35) and ++(0,-0.35) .. (0+1.8,1.7)--(0+1.8,1.8);
       \node[above] at (0+1.8,1.8) {$\scriptstyle \nakayamai^{\scriptstyle(x+1)}$};
         \node[below] at (0+1.8,0) {$\scriptstyle j^{\scriptstyle(x+1)}$};
           \node[right] at (0.1+1.8,0.1) {$\scriptstyle j^{\scriptstyle(1)}$};
               \node[left] at (-0.1+1.8,0.1) {$\scriptstyle j^{\scriptstyle(1)}$};
                \node at (0+1.8,1.2) {$\scriptstyle \ddots $};             
                        \draw[thick, black, fill=black]  (-0.6,1.1) circle (5pt);
        \draw[thick,black, fill=black]  (1.2,1.1) circle (5pt);
                   \draw[decorate, decoration={snake,segment length=5pt, amplitude=1pt}, line width=1mm, black] (-0.6,1.1)--(1.2,1.1);
      \draw[thick, lime, fill=lime]  (-0.6,1.1) circle (4pt);
        \draw[thick, lime, fill=lime]  (1.2,1.1) circle (4pt);
      \draw[decorate, decoration={snake,segment length=5pt, amplitude=1pt}, line width=0.5mm, lime] (-0.6,1.1)--(1.2,1.1);
                       \draw[thick, black, fill=black]  (-0.6+1.2,0.7) circle (5pt);
        \draw[thick,black, fill=black]  (1.2+1.2,0.7) circle (5pt);
                   \draw[decorate, decoration={snake,segment length=5pt, amplitude=1pt}, line width=1mm, black] (-0.6+1.2,0.7)--(1.2+1.2,0.7);
      \draw[thick, lime, fill=lime]  (-0.6+1.2,0.7) circle (4pt);
        \draw[thick, lime, fill=lime]  (1.2+1.2,0.7) circle (4pt);
      \draw[decorate, decoration={snake,segment length=5pt, amplitude=1pt}, line width=0.5mm, lime] (-0.6+1.2,0.7)--(1.2+1.2,0.7);
\end{tikzpicture}}
&=
x!
\hackcenter{
\begin{tikzpicture}[scale=0.8]
  \draw[ultra thick, blue] (0,0)--(0,0.1) .. controls ++(0,0.35) and ++(0,-0.35) .. (-0.6,0.6)--(-0.6,0.9);
    \draw[ultra thick, red]  (-0.6,0.9)--(-0.6,1.2) 
  .. controls ++(0,0.35) and ++(0,-0.35) .. (0,1.7)--(0,1.8);
  \draw[ultra thick, blue] (0,0)--(0,0.1) .. controls ++(0,0.35) and ++(0,-0.35) .. (0.6,0.6)--(0.6,0.9);
    \draw[ultra thick, red]  (0.6,0.9)--(0.6,1.2) 
  .. controls ++(0,0.35) and ++(0,-0.35) .. (0,1.7)--(0,1.8);
  .. controls ++(0,0.35) and ++(0,-0.35) .. (0,1.7)--(0,1.8);
       \node[above] at (0,1.8) {$\scriptstyle j^{\scriptstyle(x+1)}$};
         \node[below] at (0,0) {$\scriptstyle i^{\scriptstyle(x+1)}$};
             \node[right] at (0.1,0.1) {$\scriptstyle i^{\scriptstyle(1)}$};
               \node[left] at (-0.1,0.1) {$\scriptstyle i^{\scriptstyle(x)}$};
                \node at (0,0.85) {$\scriptstyle \ddots $};
  \draw[ultra thick, red] (0+1.8,0)--(0+1.8,0.1) .. controls ++(0,0.35) and ++(0,-0.35) .. (-0.6+1.8,0.6)--(-0.6+1.8,0.9);
    \draw[ultra thick, blue]  (-0.6+1.8,0.9)--(-0.6+1.8,1.2) 
  .. controls ++(0,0.35) and ++(0,-0.35) .. (0+1.8,1.7)--(0+1.8,1.8);
  \draw[ultra thick, red] (0+1.8,0)--(0+1.8,0.1) .. controls ++(0,0.35) and ++(0,-0.35) .. (0.6+1.8,0.6)--(0.6+1.8,0.9);
    \draw[ultra thick, blue]  (0.6+1.8,0.9)--(0.6+1.8,1.2) 
  .. controls ++(0,0.35) and ++(0,-0.35) .. (0+1.8,1.7)--(0+1.8,1.8);
  .. controls ++(0,0.35) and ++(0,-0.35) .. (0+1.8,1.7)--(0+1.8,1.8);
       \node[above] at (0+1.8,1.8) {$\scriptstyle \nakayamai^{\scriptstyle(x+1)}$};
         \node[below] at (0+1.8,0) {$\scriptstyle j^{\scriptstyle(x+1)}$};
           \node[right] at (0.1+1.8,0.1) {$\scriptstyle j^{\scriptstyle(1)}$};
               \node[left] at (-0.1+1.8,0.1) {$\scriptstyle j^{\scriptstyle(x)}$};
                \node at (0+1.8,1.2) {$\scriptstyle \ddots $};             
                        \draw[thick, black, fill=black]  (-0.6,1.1) circle (5pt);
        \draw[thick,black, fill=black]  (1.2,1.1) circle (5pt);
                   \draw[decorate, decoration={snake,segment length=5pt, amplitude=1pt}, line width=1mm, black] (-0.6,1.1)--(1.2,1.1);
      \draw[thick, lime, fill=lime]  (-0.6,1.1) circle (4pt);
        \draw[thick, lime, fill=lime]  (1.2,1.1) circle (4pt);
      \draw[decorate, decoration={snake,segment length=5pt, amplitude=1pt}, line width=0.5mm, lime] (-0.6,1.1)--(1.2,1.1);
                       \draw[thick, black, fill=black]  (-0.6+1.2,0.7) circle (5pt);
        \draw[thick,black, fill=black]  (1.2+1.2,0.7) circle (5pt);
                   \draw[decorate, decoration={snake,segment length=5pt, amplitude=1pt}, line width=1mm, black] (-0.6+1.2,0.7)--(1.2+1.2,0.7);
      \draw[thick, lime, fill=lime]  (-0.6+1.2,0.7) circle (4pt);
        \draw[thick, lime, fill=lime]  (1.2+1.2,0.7) circle (4pt);
      \draw[decorate, decoration={snake,segment length=5pt, amplitude=1pt}, line width=0.5mm, lime] (-0.6+1.2,0.7)--(1.2+1.2,0.7);
\end{tikzpicture}}\\
& =
x!
\sum_{\substack{\mu \in {}_j \mathcal{B}^x_i \\ b \in {}_j \BasisB_i}}
\frac{\{\mu\}^!_{\a} \{\mu^\vee\}^!_{\a}}{\{\mu\}^!}
\hackcenter{
\begin{tikzpicture}[scale=0.8]
  \draw[ultra thick, blue] (0,0)--(0,0.1) .. controls ++(0,0.35) and ++(0,-0.35) .. (-0.6,0.6)--(-0.6,0.9);
    \draw[ultra thick, red]  (-0.6,0.9)--(-0.6,1.2) 
  .. controls ++(0,0.35) and ++(0,-0.35) .. (0,1.7)--(0,1.8);
  \draw[ultra thick, blue] (0,0)--(0,0.1) .. controls ++(0,0.35) and ++(0,-0.35) .. (0.6,0.6)--(0.6,0.9);
    \draw[ultra thick, red]  (0.6,0.9)--(0.6,1.2) 
  .. controls ++(0,0.35) and ++(0,-0.35) .. (0,1.7)--(0,1.8);
  .. controls ++(0,0.35) and ++(0,-0.35) .. (0,1.7)--(0,1.8);
       \node[above] at (0,1.8) {$\scriptstyle j^{\scriptstyle(x+1)}$};
         \node[below] at (0,0) {$\scriptstyle i^{\scriptstyle(x+1)}$};
             \node[right] at (0.1,0.1) {$\scriptstyle i^{\scriptstyle(1)}$};
               \node[left] at (-0.1,0.1) {$\scriptstyle i^{\scriptstyle(x)}$};
                \node at (0,0.85) {$\scriptstyle \ddots $};
  \draw[ultra thick, red] (0+1.8,0)--(0+1.8,0.1) .. controls ++(0,0.35) and ++(0,-0.35) .. (-0.6+1.8,0.6)--(-0.6+1.8,0.9);
    \draw[ultra thick, blue]  (-0.6+1.8,0.9)--(-0.6+1.8,1.2) 
  .. controls ++(0,0.35) and ++(0,-0.35) .. (0+1.8,1.7)--(0+1.8,1.8);
  \draw[ultra thick, red] (0+1.8,0)--(0+1.8,0.1) .. controls ++(0,0.35) and ++(0,-0.35) .. (0.6+1.8,0.6)--(0.6+1.8,0.9);
    \draw[ultra thick, blue]  (0.6+1.8,0.9)--(0.6+1.8,1.2) 
  .. controls ++(0,0.35) and ++(0,-0.35) .. (0+1.8,1.7)--(0+1.8,1.8);
  .. controls ++(0,0.35) and ++(0,-0.35) .. (0+1.8,1.7)--(0+1.8,1.8);
       \node[above] at (0+1.8,1.8) {$\scriptstyle \nakayamai^{\scriptstyle(x+1)}$};
         \node[below] at (0+1.8,0) {$\scriptstyle j^{\scriptstyle(x+1)}$};
           \node[right] at (0.1+1.8,0.1) {$\scriptstyle j^{\scriptstyle(1)}$};
               \node[left] at (-0.1+1.8,0.1) {$\scriptstyle j^{\scriptstyle(x)}$};
                \node at (0+1.8,1.2) {$\scriptstyle \ddots $};             
                  \draw[draw=black, rounded corners, thick, fill=lime] (-0.6-0.3,1.1-0.3) rectangle ++(0+0.3+0.3,+0.3+0.3);
                   \node at (-0.6,1.1) {$\scriptstyle \mu$};
                    \draw[draw=black, rounded corners, thick, fill=lime] (1.2-0.3,1.1-0.3) rectangle ++(0+0.3+0.3,+0.3+0.3);
                      \node at (1.2,1.1) {$\scriptstyle \mu^\vee$};
                       \draw[thick, fill=yellow]  (0.6,0.7) circle (8pt);
                        \node at (0.6,0.7) {$\scriptstyle b$};
                         \draw[thick, fill=yellow]  (2.4,0.7) circle (8pt);
                        \node at (2.4,0.7) {$\scriptstyle b^\vee$};
\end{tikzpicture}} \\
& =
x!
\sum_{\substack{\mu \in {}_j \mathcal{B}^x_i \\ b \in {}_j \BasisB_i}}
\frac{\{\mu\}^!_{\a} \{\mu^\vee\}^!_{\a}}{\{\mu\}^!} \xi(\mu,b)
\hackcenter{
\begin{tikzpicture}[scale=0.8]
  \draw[ultra thick, blue] (0,0)--(0,0.9);
    \draw[ultra thick, red]  (0,0.9)--(0,1.8); 
      \draw[ultra thick, red] (0+1.8,0)--(0+1.8,0.9);
    \draw[ultra thick, blue]  (0+1.8,0.9)--(0+1.8,1.8); 
       \node[above] at (0,1.8) {$\scriptstyle j^{\scriptstyle(x+1)}$};
         \node[below] at (0,0) {$\scriptstyle i^{\scriptstyle(x+1)}$};
       \node[above] at (0+1.8,1.8) {$\scriptstyle \nakayamai^{\scriptstyle(x+1)}$};
         \node[below] at (0+1.8,0) {$\scriptstyle j^{\scriptstyle(x+1)}$};
                  \draw[draw=black, rounded corners, thick, fill=lime] (0-0.6,0.9-0.3) rectangle ++(0+0.6+0.6,+0.3+0.3);
                   \node at (0,0.9) {$\scriptstyle (\mu \cup b)$};
                    \draw[draw=black, rounded corners, thick, fill=lime] (1.8-0.6,0.9-0.3) rectangle ++(0+0.6+0.6,+0.3+0.3);
                      \node at (1.8,0.9) {$\scriptstyle (\mu \cup b)^\vee$};
\end{tikzpicture}},
\end{align*}

where 
\begin{equation}\label{xicases}
\xi(\mu,b) =
\begin{cases}
(\mu(b)+1)^2 &\textup{if } b \in \a,\; b^\vee \in \a;\\
\mu(b) + 1 & \textup{if }b \in \a, \; b^\vee \notin \a \textup{ or } b \in \a, \; b^\vee \in \a;\\
1 &\textup{if }b \in A_{\bar{1}}, \mu(b)=0;\\
0 &\textup{if }b \in A_{\bar{1}}, \mu(b)=1.
\end{cases}
\end{equation}
Now, for \(b \in \textup{supp}(\mu)\), we write \(\mu \backslash b\) for the restricted \({}_j \BasisB_i\)-composition given by setting \((\mu \backslash b)(c) = \mu(c) - \delta_{c,b}\). Then after reindexing we write the above sum:
\begin{equation}\label{teleporter456}
x!
\hspace{-3mm}
\sum_{\substack{\mu \in {}_{j} \mathcal{B}^{x+1}_{i} \\ b \in \textup{supp}(\mu)}}
\frac{\{\mu \backslash b\}^{!}_{\a}\{(\mu \backslash b)^{\vee} \}^{!}_{\a}}{\{(\mu \backslash b)\}^{!}}
\xi(\mu \backslash b, b)
\hackcenter{
\begin{tikzpicture}[scale=0.8]
  \draw[ultra thick, blue] (0,0)--(0,0.9);
    \draw[ultra thick, red]  (0,0.9)--(0,1.8); 
      \draw[ultra thick, red] (0+1.8,0)--(0+1.8,0.9);
    \draw[ultra thick, blue]  (0+1.8,0.9)--(0+1.8,1.8); 
       \node[above] at (0,1.8) {$\scriptstyle j^{\scriptstyle(x+1)}$};
         \node[below] at (0,0) {$\scriptstyle i^{\scriptstyle(x+1)}$};
       \node[above] at (0+1.8,1.8) {$\scriptstyle \nakayamai^{\scriptstyle(x+1)}$};
         \node[below] at (0+1.8,0) {$\scriptstyle j^{\scriptstyle(x+1)}$};
                  \draw[draw=black, rounded corners, thick, fill=lime] (0-0.3,0.9-0.3) rectangle ++(0+0.3+0.3,+0.3+0.3);
                   \node at (0,0.9) {$\scriptstyle \mu $};
                    \draw[draw=black, rounded corners, thick, fill=lime] (1.8-0.3,0.9-0.3) rectangle ++(0+0.3+0.3,+0.3+0.3);
                      \node at (1.8,0.9) {$\scriptstyle \mu^{\vee} $};
\end{tikzpicture}}
\end{equation}
Finally, note that by consideration of \cref{xicases}, we have
\begin{equation*}
\frac{\{\mu \backslash b\}^{!}_{\a}\{(\mu\backslash b)^{\vee} \}^{!}_{\a}}{\{(\mu \backslash b)\}^{!}}
\xi(\mu \backslash b, b) = \frac{\{\mu\}^{!}_{\a}\{\mu^{\vee} \}^{!}_{\a}}{\{\mu\}^{!}} \mu(b)
\end{equation*}
for all \(\mu \in {}_j \mathcal{B}_{i}^{x+1}\) and \(b \in \textup{supp}(\mu)\). Moreover \(\sum_{b \in \textup{supp}(\mu)} \mu(b) = x+1\) for all \(\mu \in {}_{j} \mathcal{B}_{i}^{x+1}\). Thus it follows that \cref{teleporter456} is equal to 
\begin{equation*}
(x+1)!\sum_{\mu \in {}_j \mathcal{B}^{x+1}_{i} }
\frac{\{\mu \}^{!}_{\a}\{\mu^{\vee} \}^{!}_{\a}}{\{\mu\}^{!}}
\hackcenter{
\begin{tikzpicture}[scale=0.8]
  \draw[ultra thick, blue] (0,0)--(0,0.9);
    \draw[ultra thick, red]  (0,0.9)--(0,1.8); 
      \draw[ultra thick, red] (0+1.8,0)--(0+1.8,0.9);
    \draw[ultra thick, blue]  (0+1.8,0.9)--(0+1.8,1.8); 
       \node[above] at (0,1.8) {$\scriptstyle j^{\scriptstyle(x+1)}$};
         \node[below] at (0,0) {$\scriptstyle i^{\scriptstyle(x+1)}$};
       \node[above] at (0+1.8,1.8) {$\scriptstyle \nakayamai^{\scriptstyle(x+1)}$};
         \node[below] at (0+1.8,0) {$\scriptstyle j^{\scriptstyle(x+1)}$};
                  \draw[draw=black, rounded corners, thick, fill=lime] (0-0.3,0.9-0.3) rectangle ++(0+0.3+0.3,+0.3+0.3);
                   \node at (0,0.9) {$\scriptstyle \mu$};
                    \draw[draw=black, rounded corners, thick, fill=lime] (1.8-0.3,0.9-0.3) rectangle ++(0+0.3+0.3,+0.3+0.3);
                      \node at (1.8,0.9) {$\scriptstyle \mu^\vee$};
\end{tikzpicture}}
=
\hackcenter{
\begin{tikzpicture}[scale=0.8]
  \draw[ultra thick,blue] (-0.9,0)--(-0.9,0.9);
   \draw[ultra thick,red] (0.9,0)--(0.9,0.9);
  \draw[ultra thick,blue] (0.9,0.9)--(0.9,1.8);
    \draw[ultra thick,red] (-0.9,0.9)--(-0.9,1.8);
      \node[above] at (-0.9,1.8) {$ \scriptstyle j^{ \scriptstyle (x)}$};
      \node[above] at (0.9,1.8) {$ \scriptstyle \nakayamai^{ \scriptstyle (x)}$};
       \node[below] at (-0.9,0) {$ \scriptstyle i^{ \scriptstyle (x)}$};
      \node[below] at (0.9,0) {$ \scriptstyle j^{ \scriptstyle (x)}$};
        \draw[thick, black, fill=black]  (-0.9,0.9) circle (6pt);
        \draw[thick,black, fill=black]  (0.9,0.9) circle (6pt);
                   \draw[decorate, decoration={snake,segment length=5pt, amplitude=1pt}, line width=1mm, black] (-0.9,0.9)--(0.9,0.9);
      \draw[thick, lime, fill=lime]  (-0.9,0.9) circle (5pt);
        \draw[thick, lime, fill=lime]  (0.9,0.9) circle (5pt);
      \draw[decorate, decoration={snake,segment length=5pt, amplitude=1pt}, line width=0.5mm, lime] (-0.9,0.9)--(0.9,0.9);
\end{tikzpicture}},
\end{equation*}
as desired.
\end{proof}

The following is immediate from an application of \cref{blowupteleporter}:

\begin{corollary}\label{miniblowupteleporter}
For all \(x,y \in \Z_{\geq 0}\) and \(i,j \in I\), in $\WebAaI$ we have:
\begin{equation*}
\hackcenter{}
(x+y)!
\hspace{-3mm}
\hackcenter{
\begin{tikzpicture}[scale=0.8]
  \draw[ultra thick,blue] (-0.9,0)--(-0.9,0.9);
   \draw[ultra thick,red] (0.9,0)--(0.9,0.9);
  \draw[ultra thick,blue] (0.9,0.9)--(0.9,1.8);
    \draw[ultra thick,red] (-0.9,0.9)--(-0.9,1.8);
      \node[above] at (-0.9,1.8) {$ \scriptstyle j^{ \scriptstyle (x+y)}$};
      \node[above] at (0.9,1.8) {$ \scriptstyle \nakayamai^{ \scriptstyle (x+y)}$};
       \node[below] at (-0.9,0) {$ \scriptstyle i^{ \scriptstyle (x+y)}$};
      \node[below] at (0.9,0) {$ \scriptstyle j^{ \scriptstyle (x+y)}$};
        \draw[thick, black, fill=black]  (-0.9,0.9) circle (6pt);
        \draw[thick,black, fill=black]  (0.9,0.9) circle (6pt);
                   \draw[decorate, decoration={snake,segment length=5pt, amplitude=1pt}, line width=1mm, black] (-0.9,0.9)--(0.9,0.9);
      \draw[thick, lime, fill=lime]  (-0.9,0.9) circle (5pt);
        \draw[thick, lime, fill=lime]  (0.9,0.9) circle (5pt);
      \draw[decorate, decoration={snake,segment length=5pt, amplitude=1pt}, line width=0.5mm, lime] (-0.9,0.9)--(0.9,0.9);
\end{tikzpicture}}
\hspace{-3mm}
=
x! y!
\hackcenter{
\begin{tikzpicture}[scale=0.8]
  \draw[ultra thick, blue] (0,0)--(0,0.1) .. controls ++(0,0.35) and ++(0,-0.35) .. (-0.6,0.6)--(-0.6,0.9);
    \draw[ultra thick, red]  (-0.6,0.9)--(-0.6,1.2) 
  .. controls ++(0,0.35) and ++(0,-0.35) .. (0,1.7)--(0,1.8);
  \draw[ultra thick, blue] (0,0)--(0,0.1) .. controls ++(0,0.35) and ++(0,-0.35) .. (0.6,0.6)--(0.6,0.9);
    \draw[ultra thick, red]  (0.6,0.9)--(0.6,1.2) 
  .. controls ++(0,0.35) and ++(0,-0.35) .. (0,1.7)--(0,1.8);
  .. controls ++(0,0.35) and ++(0,-0.35) .. (0,1.7)--(0,1.8);
       \node[above] at (0,1.8) {$\scriptstyle j^{\scriptstyle(x+y)}$};
         \node[below] at (0,0) {$\scriptstyle i^{\scriptstyle(x+y)}$};
             \node[right] at (0.1,0.1) {$\scriptstyle i^{\scriptstyle(x)}$};
               \node[left] at (-0.1,0.1) {$\scriptstyle i^{\scriptstyle(y)}$};
  \draw[ultra thick, red] (0+1.8,0)--(0+1.8,0.1) .. controls ++(0,0.35) and ++(0,-0.35) .. (-0.6+1.8,0.6)--(-0.6+1.8,0.9);
    \draw[ultra thick, blue]  (-0.6+1.8,0.9)--(-0.6+1.8,1.2) 
  .. controls ++(0,0.35) and ++(0,-0.35) .. (0+1.8,1.7)--(0+1.8,1.8);
  \draw[ultra thick, red] (0+1.8,0)--(0+1.8,0.1) .. controls ++(0,0.35) and ++(0,-0.35) .. (0.6+1.8,0.6)--(0.6+1.8,0.9);
    \draw[ultra thick, blue]  (0.6+1.8,0.9)--(0.6+1.8,1.2) 
  .. controls ++(0,0.35) and ++(0,-0.35) .. (0+1.8,1.7)--(0+1.8,1.8);
  .. controls ++(0,0.35) and ++(0,-0.35) .. (0+1.8,1.7)--(0+1.8,1.8);
       \node[above] at (0+1.8,1.8) {$\scriptstyle \nakayamai^{\scriptstyle(x+y)}$};
         \node[below] at (0+1.8,0) {$\scriptstyle j^{\scriptstyle(x+y)}$};
           \node[right] at (0.1+1.8,0.1) {$\scriptstyle j^{\scriptstyle(x)}$};
               \node[left] at (-0.1+1.8,0.1) {$\scriptstyle j^{\scriptstyle(y)}$};          
                        \draw[thick, black, fill=black]  (-0.6,1.1) circle (5pt);
        \draw[thick,black, fill=black]  (1.2,1.1) circle (5pt);
                   \draw[decorate, decoration={snake,segment length=5pt, amplitude=1pt}, line width=1mm, black] (-0.6,1.1)--(1.2,1.1);
      \draw[thick, lime, fill=lime]  (-0.6,1.1) circle (4pt);
        \draw[thick, lime, fill=lime]  (1.2,1.1) circle (4pt);
      \draw[decorate, decoration={snake,segment length=5pt, amplitude=1pt}, line width=0.5mm, lime] (-0.6,1.1)--(1.2,1.1);
                       \draw[thick, black, fill=black]  (-0.6+1.2,0.7) circle (5pt);
        \draw[thick,black, fill=black]  (1.2+1.2,0.7) circle (5pt);
                   \draw[decorate, decoration={snake,segment length=5pt, amplitude=1pt}, line width=1mm, black] (-0.6+1.2,0.7)--(1.2+1.2,0.7);
      \draw[thick, lime, fill=lime]  (-0.6+1.2,0.7) circle (4pt);
        \draw[thick, lime, fill=lime]  (1.2+1.2,0.7) circle (4pt);
      \draw[decorate, decoration={snake,segment length=5pt, amplitude=1pt}, line width=0.5mm, lime] (-0.6+1.2,0.7)--(1.2+1.2,0.7);
\end{tikzpicture}}
\quad
\textup{and}
\quad
(x+y)!
\hspace{-3mm}
\hackcenter{
\begin{tikzpicture}[scale=0.8]
  \draw[ultra thick,blue] (-0.9,0)--(-0.9,0.9);
   \draw[ultra thick,red] (0.9,0)--(0.9,0.9);
  \draw[ultra thick,blue] (0.9,0.9)--(0.9,1.8);
    \draw[ultra thick,red] (-0.9,0.9)--(-0.9,1.8);
      \node[above] at (-0.9,1.8) {$ \scriptstyle \nakayamaj ^{ \scriptstyle (x+y)}$};
      \node[above] at (0.9,1.8) {$ \scriptstyle i^{ \scriptstyle (x+y)}$};
       \node[below] at (-0.9,0) {$ \scriptstyle i^{ \scriptstyle (x+y)}$};
      \node[below] at (0.9,0) {$ \scriptstyle j^{ \scriptstyle (x+y)}$};
        \draw[thick, black, fill=black]  (-0.9,0.9) circle (6pt);
        \draw[thick,black, fill=black]  (0.9,0.9) circle (6pt);
                   \draw[decorate, decoration={snake,segment length=5pt, amplitude=1pt}, line width=1mm, black] (-0.9,0.9)--(0.9,0.9);
      \draw[thick, lime, fill=lime]  (-0.9,0.9) circle (5pt);
        \draw[thick, lime, fill=lime]  (0.9,0.9) circle (5pt);
      \draw[decorate, decoration={snake,segment length=5pt, amplitude=1pt}, line width=0.5mm, lime] (-0.9,0.9)--(0.9,0.9);
         \node[] at (-0.9,0.9) {$\scriptstyle 1$};
\end{tikzpicture}}
\hspace{-3mm}
=
x! y!
\hackcenter{
\begin{tikzpicture}[scale=0.8]
  \draw[ultra thick, blue] (0,0)--(0,0.1) .. controls ++(0,0.35) and ++(0,-0.35) .. (-0.6,0.6)--(-0.6,0.9);
    \draw[ultra thick, red]  (-0.6,0.9)--(-0.6,1.2) 
  .. controls ++(0,0.35) and ++(0,-0.35) .. (0,1.7)--(0,1.8);
  \draw[ultra thick, blue] (0,0)--(0,0.1) .. controls ++(0,0.35) and ++(0,-0.35) .. (0.6,0.6)--(0.6,0.9);
    \draw[ultra thick, red]  (0.6,0.9)--(0.6,1.2) 
  .. controls ++(0,0.35) and ++(0,-0.35) .. (0,1.7)--(0,1.8);
  .. controls ++(0,0.35) and ++(0,-0.35) .. (0,1.7)--(0,1.8);
       \node[above] at (0,1.8) {$\scriptstyle \nakayamaj ^{\scriptstyle(x+y)}$};
         \node[below] at (0,0) {$\scriptstyle i^{\scriptstyle(x+y)}$};
             \node[right] at (0.1,0.1) {$\scriptstyle i^{\scriptstyle(x)}$};
               \node[left] at (-0.1,0.1) {$\scriptstyle i^{\scriptstyle(y)}$};
  \draw[ultra thick, red] (0+1.8,0)--(0+1.8,0.1) .. controls ++(0,0.35) and ++(0,-0.35) .. (-0.6+1.8,0.6)--(-0.6+1.8,0.9);
    \draw[ultra thick, blue]  (-0.6+1.8,0.9)--(-0.6+1.8,1.2) 
  .. controls ++(0,0.35) and ++(0,-0.35) .. (0+1.8,1.7)--(0+1.8,1.8);
  \draw[ultra thick, red] (0+1.8,0)--(0+1.8,0.1) .. controls ++(0,0.35) and ++(0,-0.35) .. (0.6+1.8,0.6)--(0.6+1.8,0.9);
    \draw[ultra thick, blue]  (0.6+1.8,0.9)--(0.6+1.8,1.2) 
  .. controls ++(0,0.35) and ++(0,-0.35) .. (0+1.8,1.7)--(0+1.8,1.8);
  .. controls ++(0,0.35) and ++(0,-0.35) .. (0+1.8,1.7)--(0+1.8,1.8);
       \node[above] at (0+1.8,1.8) {$\scriptstyle i^{\scriptstyle(x+y)}$};
         \node[below] at (0+1.8,0) {$\scriptstyle j^{\scriptstyle(x+y)}$};
           \node[right] at (0.1+1.8,0.1) {$\scriptstyle j^{\scriptstyle(x)}$};
               \node[left] at (-0.1+1.8,0.1) {$\scriptstyle j^{\scriptstyle(y)}$};      
                        \draw[thick, black, fill=black]  (-0.6,1.1) circle (5pt);
        \draw[thick,black, fill=black]  (1.2,1.1) circle (5pt);
                   \draw[decorate, decoration={snake,segment length=5pt, amplitude=1pt}, line width=1mm, black] (-0.6,1.1)--(1.2,1.1);
      \draw[thick, lime, fill=lime]  (-0.6,1.1) circle (4pt);
        \draw[thick, lime, fill=lime]  (1.2,1.1) circle (4pt);
      \draw[decorate, decoration={snake,segment length=5pt, amplitude=1pt}, line width=0.5mm, lime] (-0.6,1.1)--(1.2,1.1);
                       \draw[thick, black, fill=black]  (-0.6+1.2,0.7) circle (5pt);
        \draw[thick,black, fill=black]  (1.2+1.2,0.7) circle (5pt);
                   \draw[decorate, decoration={snake,segment length=5pt, amplitude=1pt}, line width=1mm, black] (-0.6+1.2,0.7)--(1.2+1.2,0.7);
      \draw[thick, lime, fill=lime]  (-0.6+1.2,0.7) circle (4pt);
        \draw[thick, lime, fill=lime]  (1.2+1.2,0.7) circle (4pt);
      \draw[decorate, decoration={snake,segment length=5pt, amplitude=1pt}, line width=0.5mm, lime] (-0.6+1.2,0.7)--(1.2+1.2,0.7);
        \node[] at (-0.6,1.1) {$\scriptstyle 1$};
          \node[] at (0.6,0.7) {$\scriptstyle 1$};
\end{tikzpicture}}
\end{equation*}
\end{corollary}

Finally, we show that the thick teleporter decorations are worthy of the name; they `teleport' a thick coupon from one corner of the teleporter to the other.

\begin{lemma}\label{teleporterrules}
For all \(x \in \Z_{\geq 0}\), \(i,j \in I\), \(f \in {}_iA^{(x)}_k\), \(g \in {}_jA^{(x)}_\ell\), 
\begin{equation*}
\hackcenter{
\begin{tikzpicture}[scale=0.8]
  \draw[ultra thick,blue] (-0.6,0)--(-0.6,1.2);
   \draw[ultra thick,red] (0.6,0)--(0.6,1.2);
  \draw[ultra thick,blue] (0.6,1.2)--(0.6,2.4);
    \draw[ultra thick,red] (-0.6,1.2)--(-0.6,2.4);
     \draw[ultra thick,purple] (-0.6,0)--(-0.6,0.5);
       \draw[ultra thick,cyan] (0.6,0)--(0.6,0.5);
      \node[above] at (-0.6,2.4) {$ \scriptstyle j^{ \scriptstyle (x)}$};
      \node[above] at (0.6,2.4) {$ \scriptstyle \nakayamai^{ \scriptstyle (x)}$};
       \node[below] at (-0.6,0) {$ \scriptstyle k^{ \scriptstyle (x)}$};
      \node[below] at (0.6,0) {$ \scriptstyle \ell^{ \scriptstyle (x)}$};
        \node[left] at (-0.7,0.85) {$ \scriptstyle i^{ \scriptstyle (x)}$};
         \node[right] at (0.7,0.9) {$ \scriptstyle j^{ \scriptstyle (x)}$};
        \draw[thick, black, fill=black]  (-0.6,1.2) circle (6pt);
        \draw[thick,black, fill=black]  (0.6,1.2) circle (6pt);
                   \draw[decorate, decoration={snake,segment length=5pt, amplitude=1pt}, line width=1mm, black] (-0.6,1.2)--(0.6,1.2);
      \draw[thick, lime, fill=lime]  (-0.6,1.2) circle (5pt);
        \draw[thick, lime, fill=lime]  (0.6,1.2) circle (5pt);
      \draw[decorate, decoration={snake,segment length=5pt, amplitude=1pt}, line width=0.5mm, lime] (-0.6,1.2)--(0.6,1.2);
           \draw[thick, fill=yellow]  (-0.6,0.5) circle (7pt);
    \node at (-0.6,0.5) {$ \scriptstyle f$};
              \draw[thick, fill=yellow]  (0.6,0.5) circle (7pt);
    \node at (0.6,0.5) {$ \scriptstyle g$};
      \end{tikzpicture}}
      =
      (-1)^{\bar f \bar g}
      \hackcenter{
\begin{tikzpicture}[scale=0.8]
  \draw[ultra thick,purple] (-0.6,0)--(-0.6,1.2);
   \draw[ultra thick,cyan] (0.6,0)--(0.6,1.2);
  \draw[ultra thick,purple] (0.6,1.2)--(0.6,2.4);
    \draw[ultra thick,cyan] (-0.6,1.2)--(-0.6,2.4);
     \draw[ultra thick,blue] (-0.6,1.9)--(-0.6,2.4);
       \draw[ultra thick,red] (0.6,1.9)--(0.6,2.4);
      \node[above] at (-0.6,2.4) {$ \scriptstyle j^{ \scriptstyle (x)}$};
      \node[above] at (0.6,2.4) {$ \scriptstyle \nakayamai^{ \scriptstyle (x)}$};
       \node[below] at (-0.6,0) {$ \scriptstyle k^{ \scriptstyle (x)}$};
      \node[below] at (0.6,0) {$ \scriptstyle \ell^{ \scriptstyle (x)}$};
        \node[left] at (-0.7,1.5) {$ \scriptstyle \ell^{ \scriptstyle (x)}$};
         \node[right] at (0.7,1.6) {$ \scriptstyle \nakayamak^{ \scriptstyle (x)}$};   
        \draw[thick, black, fill=black]  (-0.6,1.2) circle (6pt);
        \draw[thick,black, fill=black]  (0.6,1.2) circle (6pt);
                   \draw[decorate, decoration={snake,segment length=5pt, amplitude=1pt}, line width=1mm, black] (-0.6,1.2)--(0.6,1.2);
      \draw[thick, lime, fill=lime]  (-0.6,1.2) circle (5pt);
        \draw[thick, lime, fill=lime]  (0.6,1.2) circle (5pt);
      \draw[decorate, decoration={snake,segment length=5pt, amplitude=1pt}, line width=0.5mm, lime] (-0.6,1.2)--(0.6,1.2);
           \draw[thick, fill=yellow]  (-0.6,1.9) circle (7pt);
    \node at (-0.6,1.9) {$ \scriptstyle g$};
              \draw[thick, fill=yellow]  (0.6,1.9) circle (7pt);
    \node at (0.6,1.9) {$ \scriptstyle \nakayamaf$};
      \end{tikzpicture}}
      \qquad
      \textup{and}
      \qquad
      \hackcenter{
\begin{tikzpicture}[scale=0.8]
  \draw[ultra thick,blue] (-0.6,0)--(-0.6,1.2);
   \draw[ultra thick,red] (0.6,0)--(0.6,1.2);
  \draw[ultra thick,blue] (0.6,1.2)--(0.6,2.4);
    \draw[ultra thick,red] (-0.6,1.2)--(-0.6,2.4);
     \draw[ultra thick,purple] (-0.6,0)--(-0.6,0.5);
       \draw[ultra thick,cyan] (0.6,0)--(0.6,0.5);
      \node[above] at (-0.6,2.4) {$ \scriptstyle \nakayamaj^{ \scriptstyle (x)}$};
      \node[above] at (0.6,2.4) {$ \scriptstyle i^{ \scriptstyle (x)}$};
       \node[below] at (-0.6,0) {$ \scriptstyle k^{ \scriptstyle (x)}$};
      \node[below] at (0.6,0) {$ \scriptstyle \ell^{ \scriptstyle (x)}$};
        \node[left] at (-0.7,0.85) {$ \scriptstyle i^{ \scriptstyle (x)}$};
         \node[right] at (0.7,0.9) {$ \scriptstyle j^{ \scriptstyle (x)}$};
        \draw[thick, black, fill=black]  (-0.6,1.2) circle (6pt);
        \draw[thick,black, fill=black]  (0.6,1.2) circle (6pt);
                   \draw[decorate, decoration={snake,segment length=5pt, amplitude=1pt}, line width=1mm, black] (-0.6,1.2)--(0.6,1.2);
      \draw[thick, lime, fill=lime]  (-0.6,1.2) circle (5pt);
        \draw[thick, lime, fill=lime]  (0.6,1.2) circle (5pt);
      \draw[decorate, decoration={snake,segment length=5pt, amplitude=1pt}, line width=0.5mm, lime] (-0.6,1.2)--(0.6,1.2);
                \node[] at (-0.6,1.2) {$ \scriptstyle 1$};
           \draw[thick, fill=yellow]  (-0.6,0.5) circle (7pt);
    \node at (-0.6,0.5) {$ \scriptstyle f$};
              \draw[thick, fill=yellow]  (0.6,0.5) circle (7pt);
    \node at (0.6,0.5) {$ \scriptstyle g$};
      \end{tikzpicture}}
      =
      (-1)^{\bar f \bar g}
      \hackcenter{
\begin{tikzpicture}[scale=0.8]
  \draw[ultra thick,purple] (-0.6,0)--(-0.6,1.2);
   \draw[ultra thick,cyan] (0.6,0)--(0.6,1.2);
  \draw[ultra thick,purple] (0.6,1.2)--(0.6,2.4);
    \draw[ultra thick,cyan] (-0.6,1.2)--(-0.6,2.4);
     \draw[ultra thick,blue] (-0.6,1.9)--(-0.6,2.4);
       \draw[ultra thick,red] (0.6,1.9)--(0.6,2.4);
      \node[above] at (-0.6,2.4) {$ \scriptstyle \nakayamaj ^{ \scriptstyle (x)}$};
      \node[above] at (0.6,2.4) {$ \scriptstyle i^{ \scriptstyle (x)}$};
       \node[below] at (-0.6,0) {$ \scriptstyle k^{ \scriptstyle (x)}$};
      \node[below] at (0.6,0) {$ \scriptstyle \ell^{ \scriptstyle (x)}$};
        \node[left] at (-0.7,1.5) {$ \scriptstyle \nakayamal^{ \scriptstyle (x)}$};  
         \node[right] at (0.7,1.6) {$ \scriptstyle k^{ \scriptstyle (x)}$};
        \draw[thick, black, fill=black]  (-0.6,1.2) circle (6pt);
        \draw[thick,black, fill=black]  (0.6,1.2) circle (6pt);
                   \draw[decorate, decoration={snake,segment length=5pt, amplitude=1pt}, line width=1mm, black] (-0.6,1.2)--(0.6,1.2);
      \draw[thick, lime, fill=lime]  (-0.6,1.2) circle (5pt);
        \draw[thick, lime, fill=lime]  (0.6,1.2) circle (5pt);
      \draw[decorate, decoration={snake,segment length=5pt, amplitude=1pt}, line width=0.5mm, lime] (-0.6,1.2)--(0.6,1.2);
                \node[] at (-0.6,1.2) {$ \scriptstyle 1$};
           \draw[thick, fill=yellow]  (-0.6,1.9) circle (7pt);
    \node at (-0.6,1.9) {$ \scriptstyle \nakayamag$};  
              \draw[thick, fill=yellow]  (0.6,1.9) circle (7pt);
    \node at (0.6,1.9) {$ \scriptstyle f$};
      \end{tikzpicture}}
\end{equation*}
\end{lemma}
\begin{proof}
When \(x=1\), the claims follow from \cref{translem}. For \(x>1\), we note that if we multiply both sides of the equalities above by \(x!\), then this equality follows from the \(x=1\) case together with \cref{blowupteleporter,TAaSMRel}. But morphism spaces in \(\WebAaI\) are \(\k\)-free by \cite[Corollary 6.6.3]{DKMZ}, which implies the result.
\end{proof}

\section{The affine Frobenius web category \(\AffWebAaI\)}

\subsection{The category \texorpdfstring{$\AffWebAaI$}{AffWebAaI}}\label{defweb}  Let \(\calA = ((A, a)_I, \tr , \psi)\) be a Frobenius great pair. The affine web category \(\AffWebAaI\) is defined by adding additional ``polynomial'' generating morphisms to the finite web category \(\WebAaI\) that are drawn as a black dot.

\begin{definition}\label{defwebaa}
 Let \(\AffWebAaI\) be the strict monoidal \(\k\)-linear supercategory defined as follows.  
The objects and the monoidal structure on objects are the same as in \(\WebAaI\), given by \(\Ob(\AffWebAaI)=\widehat{ \Omega}_I\). 
The generating morphisms of \(\AffWebAaI\) are given by the diagrams:
\begin{align}\label{WebAaGens}
\hackcenter{
{}
}
\hackcenter{
\begin{tikzpicture}[scale=.8]
  \draw[ultra thick,blue] (0,0)--(0,0.2) .. controls ++(0,0.35) and ++(0,-0.35) .. (-0.4,0.9)--(-0.4,1);
  \draw[ultra thick,blue] (0,0)--(0,0.2) .. controls ++(0,0.35) and ++(0,-0.35) .. (0.4,0.9)--(0.4,1);
      \node[above] at (-0.4,1) {$ \scriptstyle i^{\scriptstyle (x)}$};
      \node[above] at (0.4,1) {$ \scriptstyle i^{\scriptstyle (y)}$};
      \node[below] at (0,0) {$ \scriptstyle i^{\scriptstyle (x+y)} $};
\end{tikzpicture}}
\qquad
\qquad
\hackcenter{
\begin{tikzpicture}[scale=.8]
  \draw[ultra thick,blue ] (-0.4,0)--(-0.4,0.1) .. controls ++(0,0.35) and ++(0,-0.35) .. (0,0.8)--(0,1);
\draw[ultra thick, blue] (0.4,0)--(0.4,0.1) .. controls ++(0,0.35) and ++(0,-0.35) .. (0,0.8)--(0,1);
      \node[below] at (-0.4,0) {$ \scriptstyle i^{ \scriptstyle (x)}$};
      \node[below] at (0.4,0) {$ \scriptstyle i^{ \scriptstyle (y)}$};
      \node[above] at (0,1) {$ \scriptstyle i^{ \scriptstyle (x+y)}$};
\end{tikzpicture}}
\qquad
\qquad
\hackcenter{
\begin{tikzpicture}[scale=.8]
  \draw[ultra thick,red] (0.4,0)--(0.4,0.1) .. controls ++(0,0.35) and ++(0,-0.35) .. (-0.4,0.9)--(-0.4,1);
  \draw[ultra thick,blue] (-0.4,0)--(-0.4,0.1) .. controls ++(0,0.35) and ++(0,-0.35) .. (0.4,0.9)--(0.4,1);
      \node[above] at (-0.4,1) {$ \scriptstyle j^{ \scriptstyle (y)}$};
      \node[above] at (0.4,1) {$ \scriptstyle i^{ \scriptstyle (x)}$};
       \node[below] at (-0.4,0) {$ \scriptstyle i^{ \scriptstyle (x)}$};
      \node[below] at (0.4,0) {$ \scriptstyle j^{ \scriptstyle (y)}$};
\end{tikzpicture}}
\qquad
\qquad
\hackcenter{
\begin{tikzpicture}[scale=.8]
  \draw[ultra thick, blue] (0,0)--(0,0.5);
   \draw[ultra thick, red] (0,0.5)--(0,1);
   \draw[thick, fill=yellow]  (0,0.5) circle (7pt);
    \node at (0,0.5) {$ \scriptstyle f$};
     \node[below] at (0,0) {$ \scriptstyle i^{ \scriptstyle (z)}$};
      \node[above] at (0,1) {$ \scriptstyle j^{ \scriptstyle (z)}$};
\end{tikzpicture}}
\qquad
\qquad
\hackcenter{
\begin{tikzpicture}[scale=.8]
  \draw[ultra thick, blue] (0,0)--(0,0.5);
   \draw[ultra thick, blue] (0,0.5)--(0,1);
   \draw[thick, fill=black]  (0,0.5) circle (5pt);
     \node[below] at (0,0) {$ \scriptstyle i^{ \scriptstyle (z)}$};
      \node[above] at (0,1) {$ \scriptstyle \nakayamai^{ \scriptstyle (z)}$};
\end{tikzpicture}}
\end{align}
for \(i,j \in I\), \(x,y \in \Z_{\geq 0}\), \(z \in \Z_{>0}\), \(f \in jA^{(z)}i\). 
The $\Z_{2}$-grading is given by declaring that splits, merges, and crossings  have parity \(\bar 0\), and the parity of the \(f\) coupon is \(\bar{f}\). The new generators are called \emph{affine dots} and are declared to be even. The relations between morphisms in \(\AffWebAaI\) are given as follows. The ``non-dotted'' relations (\ref{AssocRel})--(\ref{AaIntertwine}) hold as in \(\WebAaI\), together with the `dotted relations' below.
\\

{\em Affine dot relations.} For all \(i,j,k \in I, x,y \in \Z_{\geq 0}, f \in jA^{(x)}i\):
\begin{align}\label{AffDotRel1}
\hackcenter{
\begin{tikzpicture}[scale=.8]
  \draw[ultra thick, blue] (0,0)--(0,0.5);
    \draw[ultra thick, red] (0,0.5)--(0,1.3);
      \draw[ultra thick, red] (0,1.3)--(0,1.8);
     \draw[thick, fill=yellow]  (0,0.5) circle (7pt);
    \node at (0,0.5) {$ \scriptstyle f$};
   \draw[thick, fill=black]  (0,1.3) circle (5pt);
    \node at (0,1.3) {$\scriptstyle h$};
     \node[below] at (0,0) {$ \scriptstyle i^{ \scriptstyle (x)}$};
      \node[above] at (0,1.8) {$ \scriptstyle \nakayamaj^{\scriptstyle (x)}$};
         \node[left] at (-0.1,0.9) {$ \scriptstyle j^{ \scriptstyle (x)}$};
\end{tikzpicture}}
\;=\;
\hackcenter{
\begin{tikzpicture}[scale=.8]
  \draw[ultra thick, blue] (0,0)--(0,0.5);
    \draw[ultra thick, blue] (0,0.5)--(0,1.3);
      \draw[ultra thick, red] (0,1.3)--(0,1.8);
     \draw[thick, fill=black]  (0,0.5) circle (5pt);
    \node at (0,0.5) {$ \scriptstyle f$};
   \draw[thick, fill=yellow]  (0,1.3) circle (7pt);
    \node at (0,1.3) {$\scriptstyle \nakayamaf$}; 
     \node[below] at (0,0) {$ \scriptstyle i^{ \scriptstyle (x)}$};
      \node[above] at (0,1.8) {$ \scriptstyle \nakayamaj^{\scriptstyle (x)}$};
         \node[left] at (-0.1,0.9) {$ \scriptstyle \nakayamai^{ \scriptstyle (x)}$};
\end{tikzpicture}}
\qquad
\qquad
\hackcenter{}
\hackcenter{
\begin{tikzpicture}[scale=.8]
  \draw[ultra thick, blue] (0,0.5)--(0,1) .. controls ++(0,0.35) and ++(0,-0.35) .. (-0.4,1.7)--(-0.4,1.8);
  \draw[ultra thick, blue] (0,0.5)--(0,1) .. controls ++(0,0.35) and ++(0,-0.35) .. (0.4,1.7)--(0.4,1.8);
  \draw[ultra thick, blue] (0,0)--(0,0.5);
     \draw[thick, fill=black]  (0,0.5) circle (5pt);
         \node[below] at (0,0) {$ \scriptstyle i^{ \scriptstyle (x+y)}$};
      \node[above] at (-0.4,1.8) {$ \scriptstyle \nakayamai^{ \scriptstyle (x)}$};
            \node[above] at (0.4,1.8) {$ \scriptstyle \nakayamai^{ \scriptstyle (y)}$};
\end{tikzpicture}}
\;
=
\;
\hackcenter{
\begin{tikzpicture}[scale=.8]
  \draw[ultra thick, blue] (0,0)--(0,0.2) .. controls ++(0,0.35) and ++(0,-0.35) .. (-0.4,0.9)--(-0.4,1)--(-0.4,1.3);
  \draw[ultra thick, blue] (0,0)--(0,0.2) .. controls ++(0,0.35) and ++(0,-0.35) .. (0.4,0.9)--(0.4,1)--(0.4,1.3);
  \draw[ultra thick, blue] (-0.4,1.3)--(-0.4,1.8);
    \draw[ultra thick, blue] (0.4,1.3)--(0.4,1.8);
     \draw[thick, fill=black]  (-0.4,1.3) circle (5pt);
     \draw[thick, fill=black]  (0.4,1.3) circle (5pt);
         \node[below] at (0,0) {$ \scriptstyle i^{ \scriptstyle (x+y)}$};
      \node[above] at (-0.4,1.8) {$ \scriptstyle \nakayamai^{ \scriptstyle (x)}$};
            \node[above] at (0.4,1.8) {$ \scriptstyle \nakayamai^{ \scriptstyle (y)}$};
\end{tikzpicture}}
\qquad
\qquad
\hackcenter{
\begin{tikzpicture}[scale=.8]
  \draw[ultra thick, blue] (0,-0.5)--(0,-1) .. controls ++(0,-0.35) and ++(0,0.35) .. (-0.4,-1.7)--(-0.4,-1.8);
  \draw[ultra thick, blue] (0,-0.5)--(0,-1) .. controls ++(0,-0.35) and ++(0,0.35) .. (0.4,-1.7)--(0.4,-1.8);
  \draw[ultra thick, blue] (0,0)--(0,-0.5);
     \draw[thick, fill=black]  (0,-0.5) circle (5pt);
         \node[above] at (0,0) {$ \scriptstyle  \nakayamai^{ \scriptstyle (x+y)}$};
      \node[below] at (-0.4,-1.8) {$ \scriptstyle i^{ \scriptstyle (x)}$};
            \node[below] at (0.4,-1.8) {$ \scriptstyle i^{ \scriptstyle (y)}$};
\end{tikzpicture}}
\;
=
\;
\hackcenter{
\begin{tikzpicture}[scale=.8]
  \draw[ultra thick, blue] (0,0)--(0,-0.2) .. controls ++(0,-0.35) and ++(0,0.35) .. (-0.4,-0.9)--(-0.4,-1)--(-0.4,-1.3);
  \draw[ultra thick, blue] (0,0)--(0,-0.2) .. controls ++(0,-0.35) and ++(0,0.35) .. (0.4,-0.9)--(0.4,-1)--(0.4,-1.3);
  \draw[ultra thick, blue] (-0.4,-1.3)--(-0.4,-1.8);
    \draw[ultra thick, blue] (0.4,-1.3)--(0.4,-1.8);
     \draw[thick, fill=black]  (-0.4,-1.3) circle (5pt);
     \draw[thick, fill=black]  (0.4,-1.3) circle (5pt);
    \node at (-0.4,-1.3) {$ \scriptstyle f$};
      \node at (0.4,-1.3) {$ \scriptstyle f$};
         \node[above] at (0,0) {$ \scriptstyle \nakayamai^{\scriptstyle (x+y)}$};
      \node[below] at (-0.4,-1.8) {$ \scriptstyle i^{ \scriptstyle (x)}$};
            \node[below] at (0.4,-1.8) {$ \scriptstyle i^{ \scriptstyle (y)}$};
\end{tikzpicture}}
\end{align}

\begin{align}\label{DotCrossRel}
\hackcenter{
\begin{tikzpicture}[scale=0.8]
  \draw[ultra thick, blue] (0,0)--(0,0.3) .. controls ++(0,0.35) and ++(0,-0.35)  .. (1.8,2.3)--(1.8,2.6); 
    \draw[ultra thick, blue] (0,0)--(0,0.3) .. controls ++(0,0.35) and ++(0,-0.35)  .. (1.8,2.3)--(1.8,2.6);
    \draw[ultra thick, red] (1.8,0)--(1.8,0.3) .. controls ++(0,0.35) and ++(0,-0.35)  .. (0,2.3)-- (0,2.6); 
     \node[below] at (0,0) {$\scriptstyle i^{\scriptstyle (x)}$};
     \node[below] at (1.8,0) {$\scriptstyle j^{\scriptstyle (y)}$};
      \node[above] at (0,2.6) {$\scriptstyle j^{\scriptstyle (y)}$};
      \node[above] at (1.8,2.6) {$\scriptstyle \nakayamai^{\scriptstyle (x)}$};
         \draw[thick, fill=black]  (0.4,0.8) circle (5pt);
\end{tikzpicture}}
\;
=
\;
\sum_{t \in \Z_{\geq 0}}
(-1)^t\;\;
\hackcenter{
\begin{tikzpicture}[scale=0.8]
\draw[ultra thick, blue] (0,0)--(0,2.6);
\draw[ultra thick, blue] (1.8,0)--(1.8,2.6);
\draw[ultra thick, red] (1.8,0)--(1.8,1);
\draw[ultra thick, red] (0,1)--(0,2.6);
  \draw[ultra thick, blue] (0,0)--(0,0.3) .. controls ++(0,0.35) and ++(0,-0.35)  .. (1.8,2.3); 
    \draw[ultra thick, red] (1.8,0)--(1.8,0.3) .. controls ++(0,0.35) and ++(0,-0.35)  .. (0,2.3); 
     \node[below] at (0,0) {$\scriptstyle i^{\scriptstyle (x)}$};
     \node[below] at (1.8,0) {$\scriptstyle j^{\scriptstyle (y)}$};
      \node[above] at (0,2.6) {$\scriptstyle j^{\scriptstyle (y)}$};
      \node[above] at (1.8,2.6) {$\scriptstyle \nakayamai^{\scriptstyle (x)}$};
      \node[below] at (0.7,0.8) {$\scriptstyle i^{\scriptstyle (x-t)}$};
            \node[left] at (0,1.6) {$\scriptstyle j^{\scriptstyle (t)}$};
      \node[right] at (1.8,1.6) {$\scriptstyle \nakayamai^{\scriptstyle (t)}$};
          \draw[thick, black, fill=black]  (0,1) circle (6pt);
        \draw[thick,black, fill=black]  (1.8,1) circle (6pt);
                   \draw[decorate, decoration={snake,segment length=5pt, amplitude=1pt}, line width=1mm, black] (0,1)--(1.8,1);
      \draw[thick, lime, fill=lime]  (0,1) circle (5pt);
        \draw[thick, lime, fill=lime]  (1.8,1) circle (5pt);
      \draw[decorate, decoration={snake,segment length=5pt, amplitude=1pt}, line width=0.5mm, lime] (0,1)--(1.8,1);
         \draw[thick, fill=black]  (1.35,1.7) circle (5pt);
\end{tikzpicture}}
\end{align}
\end{definition}

\subsection{Affine degree}
We will say an affine dot on a strand of thickness \(z\) has \emph{affine degree} \(z\). More generally, we define the affine degree of a diagram representing a morphism in \(\AffWebAaI\) to be the sum of the affine degrees of the affine dot morphisms in the diagram.  Note that relation \cref{DotCrossRel} means a given morphism in $\AffWebAaI$ could be drawn using linear combinations of diagrams with differing affine degrees. For this reasons, when we use the notion of affine degree, it will be in reference to specific diagrams and linear combinations of diagrams.

\subsection{Some implied relations in \texorpdfstring{$\AffWebAaI$}{AffWebAaI}} 

There is an obvious $\k$-linear monoidal functor from \(\WebAaI\) to $\AffWebAaI$ which takes objects and generating morphisms to their mates in $\AffWebAaI$. This implies the relations established for \(\WebAaI\) in \cref{someimplied1,SSS:Additional-shorthand-coupons} may be freely transferred to \(\AffWebAaI\), and we will use these implicitly. 
We now highlight some additional implied relations in \(\AffWebAaI\).

\begin{lemma}\label{righttoleftthin}
For \(i,j \in I\), we have:
\begin{align*}
\hackcenter{
\begin{tikzpicture}[scale=1]
  \draw[ultra thick,red] (0.4,0)--(0.4,0.1) .. controls ++(0,0.35) and ++(0,-0.35) .. (-0.4,0.9)--(-0.4,1);
  \draw[ultra thick,blue] (-0.4,0)--(-0.4,0.1) .. controls ++(0,0.35) and ++(0,-0.35) .. (0.4,0.9)--(0.4,1);
      \node[above] at (-0.4,1) {$ \scriptstyle \nakayamaj^{ \scriptstyle (1)}$};
      \node[above] at (0.4,1) {$ \scriptstyle i^{ \scriptstyle (1)}$};
       \node[below] at (-0.4,0) {$ \scriptstyle i^{ \scriptstyle (1)}$};
      \node[below] at (0.4,0) {$ \scriptstyle j^{ \scriptstyle (1)}$};
                  \draw[thick, fill=black]  (0.3,0.3) circle (4pt);
\end{tikzpicture}}
-
\hackcenter{
\begin{tikzpicture}[scale=1]
  \draw[ultra thick,red] (0.4,0)--(0.4,0.1) .. controls ++(0,0.35) and ++(0,-0.35) .. (-0.4,0.9)--(-0.4,1);
  \draw[ultra thick,blue] (-0.4,0)--(-0.4,0.1) .. controls ++(0,0.35) and ++(0,-0.35) .. (0.4,0.9)--(0.4,1);
      \node[above] at (-0.4,1) {$ \scriptstyle \nakayamaj^{ \scriptstyle (1)}$};
      \node[above] at (0.4,1) {$ \scriptstyle i^{ \scriptstyle (1)}$};
       \node[below] at (-0.4,0) {$ \scriptstyle i^{ \scriptstyle (1)}$};
      \node[below] at (0.4,0) {$ \scriptstyle j^{ \scriptstyle (1)}$};
            \draw[thick, fill=black]  (-0.3,0.7) circle (4pt);
\end{tikzpicture}}
=
\hackcenter{
\begin{tikzpicture}[scale=1]
  \draw[ultra thick,blue] (-0.4,0)--(-0.4,0.5);
   \draw[ultra thick,red] (0.4,0)--(0.4,0.5);
  \draw[ultra thick,blue] (0.4,0.5)--(0.4,1);
    \draw[ultra thick,red] (-0.4,0.5)--(-0.4,1);
      \node[above] at (-0.4,1) {$ \scriptstyle \nakayamaj^{ \scriptstyle (1)}$};
      \node[above] at (0.4,1) {$ \scriptstyle i^{ \scriptstyle (1)}$};
       \node[below] at (-0.4,0) {$ \scriptstyle i^{ \scriptstyle (1)}$};
      \node[below] at (0.4,0) {$ \scriptstyle j^{ \scriptstyle (1)}$};
        \draw[thick, black, fill=black]  (-0.4,0.5) circle (5pt);
        \draw[thick,black, fill=black]  (0.4,0.5) circle (5pt);
                   \draw[decorate, decoration={snake,segment length=5pt, amplitude=1pt}, line width=1mm, black] (-0.4,0.5)--(0.4,0.5);
      \draw[thick, lime, fill=lime]  (-0.4,0.5) circle (4pt);
        \draw[thick, lime, fill=lime]  (0.4,0.5) circle (4pt);
      \draw[decorate, decoration={snake,segment length=5pt, amplitude=1pt}, line width=0.5mm, lime] (-0.4,0.5)--(0.4,0.5);
      \node[] at (-0.4,0.5) {$ \scriptstyle 1$};
\end{tikzpicture}}
\qquad
\textup{and}
\qquad
\hackcenter{
\begin{tikzpicture}[scale=1]
  \draw[ultra thick,red] (0.4,0)--(0.4,0.1) .. controls ++(0,0.35) and ++(0,-0.35) .. (-0.4,0.9)--(-0.4,1);
  \draw[ultra thick,blue] (-0.4,0)--(-0.4,0.1) .. controls ++(0,0.35) and ++(0,-0.35) .. (0.4,0.9)--(0.4,1);
      \node[above] at (-0.4,1) {$ \scriptstyle \nakayamaj ^{ \scriptstyle (1)}$};
      \node[above] at (0.4,1) {$ \scriptstyle \nakayamai ^{ \scriptstyle (1)}$};
       \node[below] at (-0.4,0) {$ \scriptstyle i^{ \scriptstyle (1)}$};
      \node[below] at (0.4,0) {$ \scriptstyle j^{ \scriptstyle (1)}$};
            \draw[thick, fill=black]  (-0.3,0.3) circle (4pt);
                  \draw[thick, fill=black]  (0.3,0.3) circle (4pt);
\end{tikzpicture}}
=
\hackcenter{
\begin{tikzpicture}[scale=1]
  \draw[ultra thick,red] (0.4,0)--(0.4,0.1) .. controls ++(0,0.35) and ++(0,-0.35) .. (-0.4,0.9)--(-0.4,1);
  \draw[ultra thick,blue] (-0.4,0)--(-0.4,0.1) .. controls ++(0,0.35) and ++(0,-0.35) .. (0.4,0.9)--(0.4,1);
      \node[above] at (-0.4,1) {$ \scriptstyle \nakayamaj^{ \scriptstyle (1)}$};
      \node[above] at (0.4,1) {$ \scriptstyle \nakayamai^{ \scriptstyle (1)}$};
       \node[below] at (-0.4,0) {$ \scriptstyle i^{ \scriptstyle (1)}$};
      \node[below] at (0.4,0) {$ \scriptstyle j^{ \scriptstyle (1)}$};
            \draw[thick, fill=black]  (-0.3,0.7) circle (4pt);
                  \draw[thick, fill=black]  (0.3,0.7) circle (4pt);
\end{tikzpicture}}.
\end{align*}
\end{lemma}
\begin{proof}
By \cref{DotCrossRel} we have:
\begin{equation}\label{thindotcross1}
\hackcenter{
\begin{tikzpicture}[scale=1]
  \draw[ultra thick,blue] (0.4,0)--(0.4,0.1) .. controls ++(0,0.35) and ++(0,-0.35) .. (-0.4,0.9)--(-0.4,1);
  \draw[ultra thick,red] (-0.4,0)--(-0.4,0.1) .. controls ++(0,0.35) and ++(0,-0.35) .. (0.4,0.9)--(0.4,1);
      \node[above] at (-0.4,1) {$ \scriptstyle i^{ \scriptstyle (1)}$};
      \node[above] at (0.4,1) {$ \scriptstyle \nakayamaj^{ \scriptstyle (1)}$};
       \node[below] at (-0.4,0) {$ \scriptstyle j^{ \scriptstyle (1)}$};
      \node[below] at (0.4,0) {$ \scriptstyle i^{ \scriptstyle (1)}$};
                  \draw[thick, fill=black]  (-0.3,0.3) circle (4pt);
\end{tikzpicture}}
-
\hackcenter{
\begin{tikzpicture}[scale=1]
  \draw[ultra thick,blue] (0.4,0)--(0.4,0.1) .. controls ++(0,0.35) and ++(0,-0.35) .. (-0.4,0.9)--(-0.4,1);
  \draw[ultra thick,red] (-0.4,0)--(-0.4,0.1) .. controls ++(0,0.35) and ++(0,-0.35) .. (0.4,0.9)--(0.4,1);
   \node[above] at (-0.4,1) {$ \scriptstyle i^{ \scriptstyle (1)}$};
      \node[above] at (0.4,1) {$ \scriptstyle \nakayamaj^{ \scriptstyle (1)}$};
       \node[below] at (-0.4,0) {$ \scriptstyle j^{ \scriptstyle (1)}$};
      \node[below] at (0.4,0) {$ \scriptstyle i^{ \scriptstyle (1)}$};
            \draw[thick, fill=black]  (0.3,0.7) circle (4pt);
\end{tikzpicture}}
=-
\hackcenter{
\begin{tikzpicture}[scale=1]
  \draw[ultra thick,red] (-0.4,0)--(-0.4,0.5);
   \draw[ultra thick,blue] (0.4,0)--(0.4,0.5);
  \draw[ultra thick,red] (0.4,0.5)--(0.4,1);
    \draw[ultra thick,blue] (-0.4,0.5)--(-0.4,1);
   \node[above] at (-0.4,1) {$ \scriptstyle i^{ \scriptstyle (1)}$};
      \node[above] at (0.4,1) {$ \scriptstyle \nakayamaj^{ \scriptstyle (1)}$};
       \node[below] at (-0.4,0) {$ \scriptstyle j^{ \scriptstyle (1)}$};
      \node[below] at (0.4,0) {$ \scriptstyle i^{ \scriptstyle (1)}$};
        \draw[thick, black, fill=black]  (-0.4,0.5) circle (5pt);
        \draw[thick,black, fill=black]  (0.4,0.5) circle (5pt);
                   \draw[decorate, decoration={snake,segment length=5pt, amplitude=1pt}, line width=1mm, black] (-0.4,0.5)--(0.4,0.5);
      \draw[thick, lime, fill=lime]  (-0.4,0.5) circle (4pt);
        \draw[thick, lime, fill=lime]  (0.4,0.5) circle (4pt);
      \draw[decorate, decoration={snake,segment length=5pt, amplitude=1pt}, line width=0.5mm, lime] (-0.4,0.5)--(0.4,0.5);
\end{tikzpicture}}
=
-
\sum_{b \in {}_iB_j}
\hackcenter{
\begin{tikzpicture}[scale=1]
  \draw[ultra thick,red] (-0.4,0)--(-0.4,0.5);
   \draw[ultra thick,blue] (0.4,0)--(0.4,0.5);
  \draw[ultra thick,red] (0.4,0.5)--(0.4,1);
    \draw[ultra thick,blue] (-0.4,0.5)--(-0.4,1);
   \node[above] at (-0.4,1) {$ \scriptstyle i^{ \scriptstyle (1)}$};
      \node[above] at (0.4,1) {$ \scriptstyle \nakayamaj^{ \scriptstyle (1)}$};
       \node[below] at (-0.4,0) {$ \scriptstyle j^{ \scriptstyle (1)}$};
      \node[below] at (0.4,0) {$ \scriptstyle i^{ \scriptstyle (1)}$};
        \draw[thick, black, fill=yellow]  (-0.4,0.5) circle (6pt);
        \draw[thick,black, fill=yellow]  (0.4,0.5) circle (6pt);
           \node[] at (-0.4,0.5) {$ \scriptstyle b$};
              \node[] at (0.4,0.5) {$ \scriptstyle b^\vee$};
\end{tikzpicture}}.
\end{equation}
Adding crossing morphisms to the top and bottom of each diagram, and simplifying yields:
\begin{align}\label{thindotcross2}
{}
\hackcenter{
\begin{tikzpicture}[scale=1]
  \draw[ultra thick,red] (0.4,0)--(0.4,0.1) .. controls ++(0,0.35) and ++(0,-0.35) .. (-0.4,0.9)--(-0.4,1);
  \draw[ultra thick,blue] (-0.4,0)--(-0.4,0.1) .. controls ++(0,0.35) and ++(0,-0.35) .. (0.4,0.9)--(0.4,1);
      \node[above] at (-0.4,1) {$ \scriptstyle j^{ \scriptstyle (1)}$};
      \node[above] at (0.4,1) {$ \scriptstyle i^{ \scriptstyle (1)}$};
       \node[below] at (-0.4,0) {$ \scriptstyle i^{ \scriptstyle (1)}$};
      \node[below] at (0.4,0) {$ \scriptstyle j^{ \scriptstyle (1)}$};
                  \draw[thick, fill=black]  (-0.3,0.7) circle (4pt);
\end{tikzpicture}}
\hspace{-2mm}
-
\hspace{-2mm}
\hackcenter{
\begin{tikzpicture}[scale=1]
  \draw[ultra thick,red] (0.4,0)--(0.4,0.1) .. controls ++(0,0.35) and ++(0,-0.35) .. (-0.4,0.9)--(-0.4,1);
  \draw[ultra thick,blue] (-0.4,0)--(-0.4,0.1) .. controls ++(0,0.35) and ++(0,-0.35) .. (0.4,0.9)--(0.4,1);
   \node[above] at (-0.4,1) {$ \scriptstyle \nakayamaj ^{ \scriptstyle (1)}$};
      \node[above] at (0.4,1) {$ \scriptstyle i^{ \scriptstyle (1)}$};
       \node[below] at (-0.4,0) {$ \scriptstyle i^{ \scriptstyle (1)}$};
      \node[below] at (0.4,0) {$ \scriptstyle j^{ \scriptstyle (1)}$};
            \draw[thick, fill=black]  (0.3,0.3) circle (4pt);
\end{tikzpicture}}
=
-
\sum_{b \in {}_iB_j}
(-1)^{\bar b}
\hackcenter{
\begin{tikzpicture}[scale=1]
  \draw[ultra thick,blue] (-0.4,0)--(-0.4,0.5);
   \draw[ultra thick,red] (0.4,0)--(0.4,0.5);
  \draw[ultra thick,blue] (0.4,0.5)--(0.4,1);
    \draw[ultra thick,red] (-0.4,0.5)--(-0.4,1);
   \node[above] at (-0.4,1) {$ \scriptstyle \nakayamaj^{ \scriptstyle (1)}$};
      \node[above] at (0.4,1) {$ \scriptstyle i^{ \scriptstyle (1)}$};
       \node[below] at (-0.4,0) {$ \scriptstyle i^{ \scriptstyle (1)}$};
      \node[below] at (0.4,0) {$ \scriptstyle j^{ \scriptstyle (1)}$};
        \draw[thick, black, fill=yellow]  (-0.4,0.5) circle (6pt);
        \draw[thick,black, fill=yellow]  (0.4,0.5) circle (6pt);
           \node[] at (-0.4,0.5) {$ \scriptstyle b^\vee$};
              \node[] at (0.4,0.5) {$ \scriptstyle b$};
\end{tikzpicture}}
\substack{ (\textup{L}.\ref{taulem}) \\ \textstyle = \\ }
-
\hspace{-2mm}
\sum_{b \in {}_jB_{i}}
\hspace{-2mm}
\hackcenter{
\begin{tikzpicture}[scale=1]
  \draw[ultra thick,blue] (-0.4,0)--(-0.4,0.5);
   \draw[ultra thick,red] (0.4,0)--(0.4,0.5);
  \draw[ultra thick,blue] (0.4,0.5)--(0.4,1);
    \draw[ultra thick,red] (-0.4,0.5)--(-0.4,1);
   \node[above] at (-0.4,1) {$ \scriptstyle \nakayamaj ^{ \scriptstyle (1)}$};
      \node[above] at (0.4,1) {$ \scriptstyle i^{ \scriptstyle (1)}$};
       \node[below] at (-0.4,0) {$ \scriptstyle i^{ \scriptstyle (1)}$};
      \node[below] at (0.4,0) {$ \scriptstyle j^{ \scriptstyle (1)}$};
        \draw[thick, black, fill=yellow]  (-0.4,0.5) circle (6pt);
        \draw[thick,black, fill=yellow]  (0.4,0.5) circle (6pt);
           \node[] at (-0.4,0.5) {$ \scriptstyle \nakayamab $};
              \node[] at (0.4,0.5) {$ \scriptstyle b^\vee$};
\end{tikzpicture}}
=
-
\hackcenter{
\begin{tikzpicture}[scale=1]
  \draw[ultra thick,blue] (-0.4,0)--(-0.4,0.5);
   \draw[ultra thick,red] (0.4,0)--(0.4,0.5);
  \draw[ultra thick,blue] (0.4,0.5)--(0.4,1);
    \draw[ultra thick,red] (-0.4,0.5)--(-0.4,1);
      \node[above] at (-0.4,1) {$ \scriptstyle \nakayamaj ^{ \scriptstyle (1)}$};
      \node[above] at (0.4,1) {$ \scriptstyle i^{ \scriptstyle (1)}$};
       \node[below] at (-0.4,0) {$ \scriptstyle i^{ \scriptstyle (1)}$};
      \node[below] at (0.4,0) {$ \scriptstyle j^{ \scriptstyle (1)}$};
        \draw[thick, black, fill=black]  (-0.4,0.5) circle (5pt);
        \draw[thick,black, fill=black]  (0.4,0.5) circle (5pt);
                   \draw[decorate, decoration={snake,segment length=5pt, amplitude=1pt}, line width=1mm, black] (-0.4,0.5)--(0.4,0.5);
      \draw[thick, lime, fill=lime]  (-0.4,0.5) circle (4pt);
        \draw[thick, lime, fill=lime]  (0.4,0.5) circle (4pt);
      \draw[decorate, decoration={snake,segment length=5pt, amplitude=1pt}, line width=0.5mm, lime] (-0.4,0.5)--(0.4,0.5);
      \node[] at (-0.4,0.5) {$ \scriptstyle 1$};
\end{tikzpicture}},
\end{align}
proving the first claim. 

For the second claim, applying the dot-crossing relations (\ref{thindotcross1}, \ref{thindotcross2}) in order yields
\begin{align}\label{dotcross3}
{}
\hackcenter{
\begin{tikzpicture}[scale=1]
  \draw[ultra thick,red] (0.4,0)--(0.4,0.1) .. controls ++(0,0.35) and ++(0,-0.35) .. (-0.4,0.9)--(-0.4,1);
  \draw[ultra thick,blue] (-0.4,0)--(-0.4,0.1) .. controls ++(0,0.35) and ++(0,-0.35) .. (0.4,0.9)--(0.4,1);
      \node[above] at (-0.4,1) {$ \scriptstyle \nakayamaj^{ \scriptstyle (1)}$};
      \node[above] at (0.4,1) {$ \scriptstyle \nakayamai^{ \scriptstyle (1)}$};
       \node[below] at (-0.4,0) {$ \scriptstyle i^{ \scriptstyle (1)}$};
      \node[below] at (0.4,0) {$ \scriptstyle j^{ \scriptstyle (1)}$};
            \draw[thick, fill=black]  (-0.3,0.3) circle (4pt);
                  \draw[thick, fill=black]  (0.3,0.3) circle (4pt);
\end{tikzpicture}}
-
\hackcenter{
\begin{tikzpicture}[scale=1]
  \draw[ultra thick,red] (0.4,0)--(0.4,0.1) .. controls ++(0,0.35) and ++(0,-0.35) .. (-0.4,0.9)--(-0.4,1);
  \draw[ultra thick,blue] (-0.4,0)--(-0.4,0.1) .. controls ++(0,0.35) and ++(0,-0.35) .. (0.4,0.9)--(0.4,1);
      \node[above] at (-0.4,1) {$ \scriptstyle \nakayamaj ^{ \scriptstyle (1)}$};
      \node[above] at (0.4,1) {$ \scriptstyle \nakayamai ^{ \scriptstyle (1)}$};
       \node[below] at (-0.4,0) {$ \scriptstyle i^{ \scriptstyle (1)}$};
      \node[below] at (0.4,0) {$ \scriptstyle j^{ \scriptstyle (1)}$};
            \draw[thick, fill=black]  (-0.3,0.7) circle (4pt);
                  \draw[thick, fill=black]  (0.3,0.7) circle (4pt);
\end{tikzpicture}}
&=
-
\sum_{b \in {}_{\nakayamaj}B_i}
\hackcenter{
\begin{tikzpicture}[scale=1]
  \draw[ultra thick,blue] (-0.4,0)--(-0.4,0.65);
   \draw[ultra thick,red] (0.4,0)--(0.4,0.65);
  \draw[ultra thick,blue] (0.4,0.65)--(0.4,1);
    \draw[ultra thick,red] (-0.4,0.65)--(-0.4,1);
   \node[above] at (-0.4,1) {$ \scriptstyle \nakayamaj^{ \scriptstyle (1)}$};
      \node[above] at (0.4,1) {$ \scriptstyle \nakayamai ^{ \scriptstyle (1)}$};
       \node[below] at (-0.4,0) {$ \scriptstyle i^{ \scriptstyle (1)}$};
      \node[below] at (0.4,0) {$ \scriptstyle j^{ \scriptstyle (1)}$};
        \draw[thick, black, fill=yellow]  (-0.4,0.65) circle (6pt);
        \draw[thick,black, fill=yellow]  (0.4,0.65) circle (6pt);
           \node[] at (-0.4,0.65) {$ \scriptstyle b$};
              \node[] at (0.4,0.65) {$ \scriptstyle b^\vee$};
      \draw[thick, fill=black]  (0.4,0.25) circle (4pt);
\end{tikzpicture}}
+
\sum_{c \in {}_{j}B_{\psi(i)}}
\hackcenter{
\begin{tikzpicture}[scale=1]
  \draw[ultra thick,blue] (-0.4,0)--(-0.4,0.35);
   \draw[ultra thick,red] (0.4,0)--(0.4,0.35);
  \draw[ultra thick,blue] (0.4,0.35)--(0.4,1);
    \draw[ultra thick,red] (-0.4,0.35)--(-0.4,1);
   \node[above] at (-0.4,1) {$ \scriptstyle \nakayamaj ^{ \scriptstyle (1)}$};
      \node[above] at (0.4,1) {$ \scriptstyle \nakayamai^{ \scriptstyle (1)}$};
       \node[below] at (-0.4,0) {$ \scriptstyle i^{ \scriptstyle (1)}$};
      \node[below] at (0.4,0) {$ \scriptstyle j^{ \scriptstyle (1)}$};
        \draw[thick, black, fill=yellow]  (-0.4,0.35) circle (6pt);
        \draw[thick,black, fill=yellow]  (0.4,0.35) circle (6pt);
           \node[] at (-0.4,0.35) {$ \scriptstyle \nakayamac$}; 
              \node[] at (0.4,0.35) {$ \scriptstyle c^\vee$};
      \draw[thick, fill=black]  (0.4,0.75) circle (4pt);
\end{tikzpicture}}
\end{align}
As noted in \cref{SS:Casimir-elements-in-A}, the element \(\sum_{b \in {}_{\nakayamaj}B_{i}} b \otimes b^\vee\) is independent of choice of basis for \({}_{\nakayamaj} A_{i}\). As \(\{ \psi^{-1}(c) \mid c \in {}_j B_{i} \}\) is an alternate choice of basis for \(\nakayamaj Ai\), we have
\begin{align*}
\sum_{b \in {}_{\nakayamaj} B_i} b \otimes b^\vee = \sum_{c \in {}_j B_{i}} \psi^{-1}(c) \otimes (\psi^{-1}(c))^\vee = \sum_{c \in {}_j B_{i}} \psi^{-1}(c) \otimes \psi^{-1}(c^\vee).
\end{align*}
Therefore we may rewrite \cref{dotcross3} as
\begin{equation*}
-
\sum_{c \in {}_{j}B_{i}}
\hackcenter{
\begin{tikzpicture}[scale=1]
  \draw[ultra thick,blue] (-0.4,0)--(-0.4,0.65);
   \draw[ultra thick,red] (0.4,0)--(0.4,0.65);
  \draw[ultra thick,blue] (0.4,0.65)--(0.4,1);
    \draw[ultra thick,red] (-0.4,0.65)--(-0.4,1);
   \node[above] at (-0.4,1) {$ \scriptstyle \nakayamaj^{ \scriptstyle (1)}$};
      \node[above] at (0.4,1) {$ \scriptstyle \nakayamai^{ \scriptstyle (1)}$};
       \node[below] at (-0.4,0) {$ \scriptstyle i^{ \scriptstyle (1)}$};
      \node[below] at (0.4,0) {$ \scriptstyle j^{ \scriptstyle (1)}$};
        \draw[thick, black, fill=yellow]  (-0.4,0.65) circle (6pt);
        \draw[thick,black, fill=yellow]  (0.4,0.65) circle (6pt);
           \node[] at (-0.4,0.65) {$ \scriptstyle \nakayamac$};
              \node[] at (0.4,0.65) {$ \scriptstyle \nakayamac^{\hspace{-0.5mm}\vee}$};
      \draw[thick, fill=black]  (0.4,0.25) circle (4pt);
\end{tikzpicture}}
+
\sum_{c \in {}_{j}B_{i}}
\hackcenter{
\begin{tikzpicture}[scale=1]
  \draw[ultra thick,blue] (-0.4,0)--(-0.4,0.35);
   \draw[ultra thick,red] (0.4,0)--(0.4,0.35);
  \draw[ultra thick,blue] (0.4,0.35)--(0.4,1);
    \draw[ultra thick,red] (-0.4,0.35)--(-0.4,1);
   \node[above] at (-0.4,1) {$ \scriptstyle \nakayamaj^{ \scriptstyle (1)}$};
      \node[above] at (0.4,1) {$ \scriptstyle \nakayamai^{ \scriptstyle (1)}$};
       \node[below] at (-0.4,0) {$ \scriptstyle i^{ \scriptstyle (1)}$};
      \node[below] at (0.4,0) {$ \scriptstyle j^{ \scriptstyle (1)}$};
        \draw[thick, black, fill=yellow]  (-0.4,0.35) circle (6pt);
        \draw[thick,black, fill=yellow]  (0.4,0.35) circle (6pt);
           \node[] at (-0.4,0.35) {$ \scriptstyle \nakayamac$};
              \node[] at (0.4,0.35) {$ \scriptstyle c^\vee$};
      \draw[thick, fill=black]  (0.4,0.75) circle (4pt);
\end{tikzpicture}}
\substack{ (\ref{AffDotRel1}) \\ \textstyle = \\ \;}
-
\sum_{c \in {}_{j}B_{i}}
\hackcenter{
\begin{tikzpicture}[scale=1]
  \draw[ultra thick,blue] (-0.4,0)--(-0.4,0.35);
   \draw[ultra thick,red] (0.4,0)--(0.4,0.35);
  \draw[ultra thick,blue] (0.4,0.35)--(0.4,1);
    \draw[ultra thick,red] (-0.4,0.35)--(-0.4,1);
   \node[above] at (-0.4,1) {$ \scriptstyle \nakayamaj^{ \scriptstyle (1)}$};
      \node[above] at (0.4,1) {$ \scriptstyle \nakayamai^{ \scriptstyle (1)}$};
       \node[below] at (-0.4,0) {$ \scriptstyle i^{ \scriptstyle (1)}$};
      \node[below] at (0.4,0) {$ \scriptstyle j^{ \scriptstyle (1)}$};
        \draw[thick, black, fill=yellow]  (-0.4,0.35) circle (6pt);
        \draw[thick,black, fill=yellow]  (0.4,0.35) circle (6pt);
           \node[] at (-0.4,0.35) {$ \scriptstyle \nakayamac$};
              \node[] at (0.4,0.35) {$ \scriptstyle c^\vee$};
      \draw[thick, fill=black]  (0.4,0.75) circle (4pt);
\end{tikzpicture}}
+
\sum_{c \in {}_{j}B_{i}}
\hackcenter{
\begin{tikzpicture}[scale=1]
  \draw[ultra thick,blue] (-0.4,0)--(-0.4,0.35);
   \draw[ultra thick,red] (0.4,0)--(0.4,0.35);
  \draw[ultra thick,blue] (0.4,0.35)--(0.4,1);
    \draw[ultra thick,red] (-0.4,0.35)--(-0.4,1);
   \node[above] at (-0.4,1) {$ \scriptstyle \nakayamaj^{ \scriptstyle (1)}$};
      \node[above] at (0.4,1) {$ \scriptstyle \nakayamai^{ \scriptstyle (1)}$};
       \node[below] at (-0.4,0) {$ \scriptstyle i^{ \scriptstyle (1)}$};
      \node[below] at (0.4,0) {$ \scriptstyle j^{ \scriptstyle (1)}$};
        \draw[thick, black, fill=yellow]  (-0.4,0.35) circle (6pt);
        \draw[thick,black, fill=yellow]  (0.4,0.35) circle (6pt);
           \node[] at (-0.4,0.35) {$ \scriptstyle \nakayamac$};
              \node[] at (0.4,0.35) {$ \scriptstyle c^\vee$};
      \draw[thick, fill=black]  (0.4,0.75) circle (4pt);
\end{tikzpicture}}
=0,
\end{equation*}
concluding the proof of the second statement.
\end{proof}

\section{Spanning results}\label{spansec}

\subsection{Additional shorthand coupons}

Let \(x, t \in \Z_{\geq 0}\), \(i,j \in I\). We indicate repeated affine dots with a label:
\begin{align*}
{}
\hackcenter{
\begin{tikzpicture}[scale=.8]
  \draw[ultra thick, blue] (0,-0.1)--(0,1.1);
   \draw[thick, fill=black]  (0,0.5) circle (5pt);
     \node[right] at  (0.15,0.5)  {$\scriptstyle t$};
     \node[below] at (0,-0.1) {$ \scriptstyle i^{ \scriptstyle (x)}$};
      \node[above] at (0,1.1) {$ \scriptstyle i^{ \scriptstyle (x)}$};
\end{tikzpicture}}
:=
\hackcenter{
\begin{tikzpicture}[scale=.8]
  \draw[ultra thick, blue] (0,-0.1)--(0,1.1);
   \draw[thick, fill=black]  (0,0.2) circle (5pt);
      \draw[thick, fill=black]  (0,0.8) circle (5pt);
     \node[below] at (0,-0.1) {$ \scriptstyle i^{ \scriptstyle (x)}$};
      \node[above] at (0,1.1) {$ \scriptstyle i^{ \scriptstyle (x)}$};
      \node[right] at (0.1,0.5) {$ \bigg  \} $};
          \node[right] at  (0.5,0.5)  {$\scriptstyle t \textup{\;times}$};
\end{tikzpicture}}
\end{align*}
For \(f \in jA i\), we recall the notation of \cref{greendotdef} and define morphisms \([t,f], [t,f]^\diamond \in \AffWebAaI(i^{(x)}, j^{(x)})\) by setting:

\begin{align*}
{}
\hackcenter{
\begin{tikzpicture}[scale=0.8]
  \draw[ultra thick, blue] (0,0)--(0,0.9);
  \draw[ultra thick, red] (0,0.9)--(0,1.8);
  \draw[draw=black, rounded corners, thick, fill=lime] (-0.5,0.6) rectangle ++(1,0.6);
    \node at (0,0.9) {$\scriptstyle [t,f]$};
     \node[below] at (0,0) {$\scriptstyle i^{\scriptstyle(x)}$};
      \node[above] at (0,1.8) {$\scriptstyle j^{\scriptstyle(x)}$};
\end{tikzpicture}}
:=
\hackcenter{
\begin{tikzpicture}[scale=0.8]
  \draw[ultra thick, blue] (0,0)--(0,0.9);
  \draw[ultra thick, red] (0,0.9)--(0,1.8);
  \draw[draw=black, rounded corners, thick, fill=lime] (-0.4,0.9) rectangle ++(0.8,0.6);
    \node at (0,1.2) {$\scriptstyle [f]$};
     \node[below] at (0,0) {$\scriptstyle i^{\scriptstyle(x)}$};
      \node[above] at (0,1.8) {$\scriptstyle j^{\scriptstyle(x)}$};
            \draw[thick, fill=black]  (0,0.5) circle (5pt);
             \node[right] at (0.1,0.5) {$\scriptstyle t$};
\end{tikzpicture}}
\qquad
\textup{and}
\qquad
\hackcenter{
\begin{tikzpicture}[scale=0.8]
  \draw[ultra thick, blue] (0,0)--(0,0.9);
  \draw[ultra thick, red] (0,0.9)--(0,1.8);
  \draw[draw=black, rounded corners, thick, fill=lime] (-0.5,0.6) rectangle ++(1,0.6);
    \node at (0,0.9) {$\scriptstyle [t,f]^{\scriptstyle \diamond}$};
     \node[below] at (0,0) {$\scriptstyle i^{\scriptstyle(x)}$};
      \node[above] at (0,1.8) {$\scriptstyle j^{\scriptstyle(x)}$};
\end{tikzpicture}}
:=
\hackcenter{
\begin{tikzpicture}[scale=0.8]
  \draw[ultra thick, blue] (0,0)--(0,0.9);
  \draw[ultra thick, red] (0,0.9)--(0,1.8);
  \draw[draw=black, rounded corners, thick, fill=lime] (-0.4,0.9) rectangle ++(0.8,0.6);
    \node at (0,1.2) {$\scriptstyle [f]^{\scriptstyle \diamond}$};
     \node[below] at (0,0) {$\scriptstyle i^{\scriptstyle(x)}$};
      \node[above] at (0,1.8) {$\scriptstyle j^{\scriptstyle(x)}$};
            \draw[thick, fill=black]  (0,0.5) circle (5pt);
             \node[right] at (0.1,0.5) {$\scriptstyle t$};
\end{tikzpicture}}.
\end{align*}

\subsection{Restricted \({}_j \BasisP_i\)-compositions}\label{resPcomp}
For \(i,j \in I\), define the set:
\begin{align*}
{}_j \BasisP_i &:= \{(t,b) \mid t \in \Z_{\geq 0}, b \in {}_j \BasisB_{i}\},
\end{align*}
and set \(\BasisP := \bigsqcup_{i,j \in I} {}_j \BasisP_i\). Fix a total order \(>\) on $\BasisB$.  We extend the total order on \(\BasisB\) to a total order on \(\BasisP\) by setting \([t,b]>[t',b']\) whenever \(t>t'\) or \(t=t'\) and \(b>b'\).
We say a finitely supported function \(\mu: {}_j \BasisP _i \to \Z_{\geq 0}\) is a {\em restricted \({}_j \BasisP_i\)-composition} provided that \(\mu(t,b) \leq 1\) whenever \(\bar b = \bar 1\). We write \(|\mu| = \sum_{(t,b) \in {}_j \BasisP_i} \mu(t,b) \in \Z_{\geq 0}\). For such \(\mu\), we enumerate
\begin{align*}
\textup{supp}(\mu) = \{(t_\mu^{(1)},b_\mu^{(1)}) > (t_\mu^{(2)}, b_\mu^{(2)}) > \cdots > (t_\mu^{(d_\mu)}, b_\mu^{(d_\mu)})\},
\end{align*}
where \(d_\mu = |\textup{supp}(\mu)|\). We write \({}_j\mathcal{P}_i\) for the set of restricted \({}_j \BasisP_i\)-compositions. We will freely associate \({}_j \mathcal{B}_i\) with the subset
\begin{align*}
\{\mu \in {}_j \mathcal{P}_i \mid \mu(t,b) = 0 \textup{ for all } t>0\} \subseteq {}_j\mathcal{P}_i. 
\end{align*}
We now generalize the notation of \cref{resmu} to any \(\mu \in {}_j\mathcal{P}_i\), by defining the associated morphisms in \(\AffWebAaI(i^{(|\mu|)}, j^{(|\mu|)})\):
\begin{align}\label{greenmudef}
{}
\hackcenter{
\begin{tikzpicture}[scale=0.8]
  \draw[ultra thick, blue] (0,0)--(0,0.9);
  \draw[ultra thick, red] (0,0.9)--(0,1.8);
  \draw[draw=black, rounded corners, thick, fill=lime] (-0.3,0.6) rectangle ++(0.6,0.6);
    \node at (0,0.9) {$\scriptstyle \mu^{\scriptstyle \diamond}$};
     \node[below] at (0,0) {$\scriptstyle i^{\scriptstyle(|\mu|)}$};
      \node[above] at (0,1.8) {$\scriptstyle j^{\scriptstyle(|\mu|)}$};
\end{tikzpicture}}
:=
\hackcenter{
\begin{tikzpicture}[scale=0.8]
  \draw[ultra thick, blue] (0,0)--(0,0.05) .. controls ++(0,0.35) and ++(0,-0.35) .. (-1.2,0.45)--(-1.2,1.2);
    \draw[ultra thick, red]  (-1.2,1.2)--(-1.2,1.35) 
  .. controls ++(0,0.35) and ++(0,-0.35) .. (0,1.75)--(0,1.8);
  \draw[ultra thick, blue] (0,0)--(0,0.05) .. controls ++(0,0.35) and ++(0,-0.35) .. (1.2,0.45)--(1.2,1.2);
    \draw[ultra thick, red]  (1.2,1.2)--(1.2,1.35) 
  .. controls ++(0,0.35) and ++(0,-0.35) .. (0,1.75)--(0,1.8);
  .. controls ++(0,0.35) and ++(0,-0.35) .. (0,1.75)--(0,1.8);
    \draw[draw=black, rounded corners, thick, fill=lime] (-2.4,0.6) rectangle ++(2.2,0.6);
        \draw[draw=black, rounded corners, thick, fill=lime] (0.2,0.6) rectangle ++(2.2,0.6);
  \node at (-1.3,0.9) {$\scriptstyle [t_{\mu}^{(1)},b_{\mu}^{(1)}]^{\scriptstyle \diamond}$};
    \node at (1.3,0.9) {$\scriptstyle [t_{\mu}^{(d_\mu)},b_{\mu}^{(d_\mu)}]^{\scriptstyle \diamond}$};
       \node[above] at (0,1.8) {$\scriptstyle j^{\scriptstyle(|\mu|)}$};
         \node[below] at (0,0) {$\scriptstyle i^{\scriptstyle(|\mu|)}$};
             \node[left] at (-0.6,0.05) {$\scriptstyle i^{\scriptstyle(\mu(t_{\mu}^{(1)},b_{\mu}^{(1)}))}$};
               \node[right] at (0.6,0.05) {$\scriptstyle i^{\scriptstyle(\mu(t_{\mu}^{(d_\mu)},b_{\mu}^{(d_\mu)}))}$};
                \node at (0.04,0.9) {$\scriptstyle \cdots $};
\end{tikzpicture}}
\qquad
\textup{and}
\qquad
\hackcenter{
\begin{tikzpicture}[scale=0.8]
  \draw[ultra thick, blue] (0,0)--(0,0.9);
  \draw[ultra thick, red] (0,0.9)--(0,1.8);
  \draw[draw=black, rounded corners, thick, fill=lime] (-0.3,0.6) rectangle ++(0.6,0.6);
    \node at (0,0.9) {$\scriptstyle \mu$};
     \node[below] at (0,0) {$\scriptstyle i^{\scriptstyle(|\mu|)}$};
      \node[above] at (0,1.8) {$\scriptstyle j^{\scriptstyle(|\mu|)}$};
\end{tikzpicture}}
:=
\hackcenter{
\begin{tikzpicture}[scale=0.8]
  \draw[ultra thick, blue] (0,0)--(0,0.05) .. controls ++(0,0.35) and ++(0,-0.35) .. (-1.2,0.45)--(-1.2,1.2);
    \draw[ultra thick, red]  (-1.2,1.2)--(-1.2,1.35) 
  .. controls ++(0,0.35) and ++(0,-0.35) .. (0,1.75)--(0,1.8);
  \draw[ultra thick, blue] (0,0)--(0,0.05) .. controls ++(0,0.35) and ++(0,-0.35) .. (1.2,0.45)--(1.2,1.2);
    \draw[ultra thick, red]  (1.2,1.2)--(1.2,1.35) 
  .. controls ++(0,0.35) and ++(0,-0.35) .. (0,1.75)--(0,1.8);
  .. controls ++(0,0.35) and ++(0,-0.35) .. (0,1.75)--(0,1.8);
    \draw[draw=black, rounded corners, thick, fill=lime] (-2.4,0.6) rectangle ++(2.2,0.6);
        \draw[draw=black, rounded corners, thick, fill=lime] (0.2,0.6) rectangle ++(2.2,0.6);
  \node at (-1.3,0.9) {$\scriptstyle [t_{\mu}^{(1)},b_{\mu}^{(1)}]$};
    \node at (1.3,0.9) {$\scriptstyle [t_{\mu}^{(d_\mu)},b_{\mu}^{(d_\mu)}]$};
       \node[above] at (0,1.8) {$\scriptstyle j^{\scriptstyle(|\mu|)}$};
         \node[below] at (0,0) {$\scriptstyle i^{\scriptstyle(|\mu|)}$};
             \node[left] at (-0.6,0.05) {$\scriptstyle i^{\scriptstyle(\mu(t_{\mu}^{(1)},b_{\mu}^{(1)}))}$};
               \node[right] at (0.6,0.05) {$\scriptstyle i^{\scriptstyle(\mu(t_{\mu}^{(d_\mu)},b_{\mu}^{(d_\mu)}))}$};
                \node at (0.04,0.9) {$\scriptstyle \cdots $};
\end{tikzpicture}}
\end{align}
It will also be useful to define the `cabled' morphism
\begin{align}\label{greenhatmudef}
\hackcenter{}
\hackcenter{
\begin{tikzpicture}[scale=0.8]
  \draw[ultra thick, blue] (0,0)--(0,0.9);
  \draw[ultra thick, red] (0,0.9)--(0,1.8);
    \draw[ultra thick, blue] (0.7,0)--(0.7,0.9);
  \draw[ultra thick, red] (0.7,0.9)--(0.7,1.8);
  \draw[draw=black, rounded corners, thick, fill=lime] (-0.3,0.6) rectangle ++(1.3,0.6);
    \node at (0.35,0.9) {$\scriptstyle \hat{\mu}$};
     \node[below] at (0,0) {$\scriptstyle i^{(1)}$};
      \node[above] at (0,1.8) {$\scriptstyle j^{(1)}$};
           \node[below] at (0.7,0) {$\scriptstyle i^{(1)}$};
      \node[above] at (0.7,1.8) {$\scriptstyle j^{(1)}$};
       \node[] at (0.4,1.5) {$\scriptstyle \cdots$};
        \node[] at (0.4,0.3) {$\scriptstyle \cdots$};
           \draw [decoration={brace,mirror},decorate] (-0.2,-0.7)--(0.9,-0.7) node[midway,below=2pt]{\tiny $|\mu|$ times} ;
            \node[white] at (0,3) {$\scriptstyle .$};
\end{tikzpicture}}
:=
\hackcenter{
\begin{tikzpicture}[scale=0.8]
  \draw[ultra thick, blue] (-1.2,0)--(-1.2,1.2);
    \draw[ultra thick, red]  (-1.2,1.2)--(-1.2,1.8);
  \draw[ultra thick, blue] (1.2,0)--(1.2,1.2);
    \draw[ultra thick, red]  (1.2,1.2)--(1.2,1.8);
    \draw[draw=black, rounded corners, thick, fill=lime] (-2.4,0.6) rectangle ++(2.2,0.6);
        \draw[draw=black, rounded corners, thick, fill=lime] (0.2,0.6) rectangle ++(2.2,0.6);
  \node at (-1.3,0.9) {$\scriptstyle [t_{\mu}^{(1)},b_{\mu}^{(1)}]$};
    \node at (1.3,0.9) {$\scriptstyle [t_{\mu}^{(1)},b_{\mu}^{(1)}]$};
     \node[below] at (-1.2,0) {$\scriptstyle i^{(1)}$};
      \node[above] at (-1.2,1.8) {$\scriptstyle j^{(1)}$};
           \node[below] at (1.2,0) {$\scriptstyle i^{(1)}$};
      \node[above] at (1.2,1.8) {$\scriptstyle j^{(1)}$};
       \node[] at (0,1.5) {$\scriptstyle \cdots$};
        \node[] at (0,0.3) {$\scriptstyle \cdots$};
                \node at (0.04,0.9) {$\scriptstyle \cdots $};
                %
                  \node at (2.8,0.9) {$\scriptstyle \cdots $};
                   \node at (2.8,1.5) {$\scriptstyle \cdots $};
                    \node at (2.8,0.3) {$\scriptstyle \cdots $};
                     \node at (2.8,-0.7) {$\scriptstyle \cdots $};
                    \draw [decoration={brace,mirror},decorate] (-1.8,-0.7)--(1.8,-0.7) node[midway,below=2pt]{\tiny $\mu(t_\mu^{(1)}, b_{\mu}^{(1)})$ times} ;
            \node[white] at (0,3.1) {$\scriptstyle .$};
\end{tikzpicture}}
\hackcenter{
\begin{tikzpicture}[scale=0.8]
  \draw[ultra thick, blue] (-1.2,0)--(-1.2,1.2);
    \draw[ultra thick, red]  (-1.2,1.2)--(-1.2,1.8);
  \draw[ultra thick, blue] (1.2,0)--(1.2,1.2);
    \draw[ultra thick, red]  (1.2,1.2)--(1.2,1.8);
    \draw[draw=black, rounded corners, thick, fill=lime] (-2.4,0.6) rectangle ++(2.2,0.6);
        \draw[draw=black, rounded corners, thick, fill=lime] (0.2,0.6) rectangle ++(2.2,0.6);
  \node at (-1.3,0.9) {$\scriptstyle [t_{\mu}^{(d_\mu)},b_{\mu}^{(d_\mu)}]$};
    \node at (1.3,0.9) {$\scriptstyle [t_{\mu}^{(d_\mu)},b_{\mu}^{(d_\mu)}]$};
     \node[below] at (-1.2,0) {$\scriptstyle i^{(1)}$};
      \node[above] at (-1.2,1.8) {$\scriptstyle j^{(1)}$};
           \node[below] at (1.2,0) {$\scriptstyle i^{(1)}$};
      \node[above] at (1.2,1.8) {$\scriptstyle j^{(1)}$};
       \node[] at (0,1.5) {$\scriptstyle \cdots$};
        \node[] at (0,0.3) {$\scriptstyle \cdots$};
                \node at (0.04,0.9) {$\scriptstyle \cdots $};
                %
                 \draw [decoration={brace,mirror},decorate] (-1.8,-0.7)--(1.8,-0.7) node[midway,below=2pt]{\tiny $\mu(t_\mu^{(d_\mu)}, b_{\mu}^{(d_\mu)})$ times} ;
            \node[white] at (0,3.1) {$\scriptstyle .$};
\end{tikzpicture}}
\in \AffWebAaI(i^{|\mu|}, j^{|\mu|}).
\end{align}

\subsection{\((\bi^{(\bx)}, \bj^{(\by)})\)-matrices}\label{ijmatsec}
We say \(\bmu = (\mu_{r,s})_{r \in [1,|\bx|], s \in [1,|\by|]}\) is an {\em \((\bi^{(\bx)}, \bj^{(\by)})\)-matrix} provided
\begin{enumerate}
\item  \(\mu_{r,s} \in {}_{j_s}\mathcal{P}_{i_r}\) for \(r \in [1,|\bx|]\), \(s \in [1,|\by|]\);
\item \(\sum_{t =1}^{|x|} |\mu_{t,s}| = y_s\) for \(s \in [1,|\by|]\);
\item \(\sum_{u = 1}^{|y|} |\mu_{r,u}| = x_r\) for \(r \in [1,|\bx|]\).
\end{enumerate}
We write  \(\mathcal{M}(\bi^{(\bx)}, \bj^{(\by)})\) for the set of all \((\bi^{(\bx)}, \bj^{(\by)})\)-matrices. For \(\bmu \in \mathcal{M}(\bi^{(\bx)}, \bj^{(\by)})\) we define the associated morphism which we call a `decorated double coset diagram':
\begin{align*}
\\
\eta_{\bmu} = 
\hackcenter{}
\hackcenter{
\begin{overpic}[height=35mm]{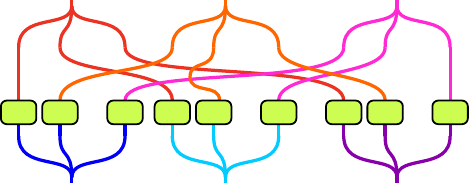}
  \put(15,-1){\makebox(0,0)[t]{$\scriptstyle i_1^{(x_1)}$}}    
    \put(48,-1){\makebox(0,0)[t]{$\scriptstyle i_2^{(x_2)}$}}    
     \put(67,-3){\makebox(0,0)[]{$\scriptstyle \cdots$}}
      \put(85,-1){\makebox(0,0)[t]{$\scriptstyle i_r^{(x_r)}$}}    
   \put(15,40){\makebox(0,0)[b]{$\scriptstyle j_1^{(x_1)}$}}    
    \put(48,40){\makebox(0,0)[b]{$\scriptstyle j_2^{(x_2)}$}}    
     \put(67,41){\makebox(0,0)[]{$\scriptstyle \cdots$}}
      \put(85,40){\makebox(0,0)[b]{$\scriptstyle j_s^{(x_s)}$}}   
      %
       \put(4,15){\makebox(0,0)[]{$\scriptstyle \mu_{1,1}$}}
       \put(13,15){\makebox(0,0)[]{$\scriptstyle \mu_{1,2}$}} 
       \put(27,15){\makebox(0,0)[]{$\scriptstyle \mu_{1,s}$}}
        \put(37,15){\makebox(0,0)[]{$\scriptstyle \mu_{2,1}$}}
       \put(45.6,15){\makebox(0,0)[]{$\scriptstyle \mu_{2,2}$}} 
       \put(59.5,15){\makebox(0,0)[]{$\scriptstyle \mu_{2,s}$}}          
         \put(73.2,15){\makebox(0,0)[]{$\scriptstyle \mu_{r,1}$}}
       \put(82,15){\makebox(0,0)[]{$\scriptstyle \mu_{r,2}$}} 
       \put(96,15){\makebox(0,0)[]{$\scriptstyle \mu_{r,s}$}}   
       %
         \put(7,3.5){\makebox(0,0)[t]{$\scriptstyle i_1^{(|\mu_{1,1}|)}$}}
          \put(13.7,9.5){\makebox(0,0)[l]{$\scriptstyle i_1^{(|\mu_{1\hspace{-.3mm},\hspace{-.2mm}2}|)}$}}
          \put(20,7){\makebox(0,0)[]{$\scriptstyle \cdots$}}
          \put(24,3.5){\makebox(0,0)[t]{$\scriptstyle i_1^{(|\mu_{1,s}|)}$}}
          \put(40,3.5){\makebox(0,0)[t]{$\scriptstyle i_2^{(|\mu_{2,1}|)}$}}
          \put(46.7,9.5){\makebox(0,0)[l]{$\scriptstyle  i_2^{(|\mu_{2\hspace{-.3mm},\hspace{-.2mm}2}|)}$}}
          \put(53,7){\makebox(0,0)[]{$\scriptstyle \cdots$}}
          \put(57,3.5){\makebox(0,0)[t]{$\scriptstyle i_2^{(|\mu_{2,s}|)}$}}
          \put(77,3.5){\makebox(0,0)[t]{$\scriptstyle i_r^{(|\mu_{r,1}|)}$}}
          \put(83,9.5){\makebox(0,0)[l]{$\scriptstyle  i_r^{(|\mu_{r\hspace{-.3mm},\hspace{-.2mm}2}|)}$}}
          \put(90,7){\makebox(0,0)[]{$\scriptstyle \cdots$}}
          \put(94,3.5){\makebox(0,0)[t]{$\scriptstyle i_r^{(|\mu_{r,s}|)}$}}
          \put(20,31){\makebox(0,0)[]{$\scriptstyle \cdots$}}
          \put(53,31){\makebox(0,0)[]{$\scriptstyle \cdots$}}
          \put(90,31){\makebox(0,0)[]{$\scriptstyle \cdots$}}
\end{overpic}
}
\in \AffWebAaI(\bi^{(\bx)}, \bj^{(\by)}).
\\
\end{align*}
Intuitively, this diagram can be thought of as follows. Let \(k \in [1,r]\), \(\ell \in [1,s]\), \(t \in \Z_{\geq 0}\), \(b \in {}_{j_\ell} \BasisB_{\psi^{-t}(i_k)}\), and assume  \(\mu_{k,\ell}(t,b) = m\). If \(b \in \a \), then \(\eta_{\bmu}\) has a strand of thickness \(m\) which emanates from the \(k\)th component at the bottom of the diagram, and carries \(t\) affine dots and the coupon \(b\) into the \(\ell\)th component at the top of the diagram. On the other hand, if \(b \notin \a \), then \(\eta_{\bmu}\) has \(m\) many thin strands which emanate from the \(k\)th component at the bottom of the diagram, each carrying \(t\) affine dots and the coupon \(b\) to the \(\ell\)th component at the top of the diagram. See \cref{cablex} for  an exemplar of the morphism \(\eta_{\bmu}\).

Given \(d \in \Z_{\geq 0}\), we set 
\begin{align*}
\AffWebAaI(\bi^{(\bx)}, \bj^{(\by)})_{\leq d} := \k\{ \eta_{\bmu} \mid \bmu \in \mathcal{M}(\bi^{(\bx)}, \bj^{(\by)}), \eta_{\bmu} \textup{ has affine degree}\leq d\}
\end{align*}

Recall that 
we write \(i^x:= (i^{(1)})^x\), and 
if \(\bi^{(\bx)} = i_1^{(x_1)} \cdots i_{r}^{(x_r)}\), then we write \(\bi^{\bx} = i_1^{x_1} \cdots i_r^{x_r}\).
Let \(\bmu \in \mathcal{M}(\bi^{(\bx)}, \bj^{(\by)})\). 
Recalling \cref{greenhatmudef}, we define \(\hat{\bmu} \in \mathcal{M}(\bi^{\bx}, \bj^{\bx})\) to be the `cabling' of \(\bmu\), so that

\begin{align*}
\\
\\
\eta_{\hat{\bmu}} = 
\hackcenter{}
\hackcenter{
\begin{overpic}[height=35mm]{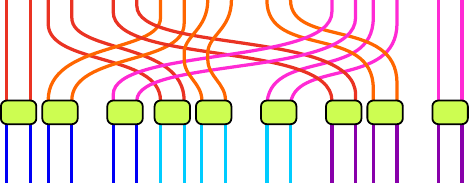}
            \put(15,-5){\makebox(0,0)[t]{$\scriptstyle  \underbrace{\scriptstyle{\color{white}xxxxxxxxxxxxxxxx\color{black}}}_{x_1 \textup{times}}$}}
               \put(48,-5){\makebox(0,0)[t]{$\scriptstyle  \underbrace{\scriptstyle{\color{white}xxxxxxxxxxxxxxxx\color{black}}}_{x_2 \textup{times}}$}}
                  \put(85,-5){\makebox(0,0)[t]{$\scriptstyle  \underbrace{\scriptstyle{\color{white}xxxxxxxxxxxxxxxx\color{black}}}_{x_m \textup{times}}$}}
                  \put(15,43){\makebox(0,0)[b]{$\scriptstyle  \overbrace{\scriptstyle{\color{white}xxxxxxxxxxxxxxxx\color{black}}}^{y_1 \textup{times}}$}}
               \put(48,43){\makebox(0,0)[b]{$\scriptstyle  \overbrace{\scriptstyle{\color{white}xxxxxxxxxxxxxxxx\color{black}}}^{y_2 \textup{times}}$}}
                  \put(85,43){\makebox(0,0)[b]{$\scriptstyle  \overbrace{\scriptstyle{\color{white}xxxxxxxxxxxxxxxx\color{black}}}^{y_n \textup{times}}$}}
  \put(2,-1){\makebox(0,0)[t]{$\scriptstyle i_1^{(1)}$}} 
   \put(6,-1){\makebox(0,0)[t]{$\scriptstyle i_1^{(1)}$}}    
    \put(10,-1){\makebox(0,0)[t]{$\scriptstyle i_1^{(1)}$}}    
     \put(15,-1){\makebox(0,0)[t]{$\scriptstyle i_1^{(1)}$}}       
       \put(24,-1){\makebox(0,0)[t]{$\scriptstyle i_1^{(1)}$}}     
         \put(29,-1){\makebox(0,0)[t]{$\scriptstyle i_1^{(1)}$}}     
        \put(34,-1){\makebox(0,0)[t]{$\scriptstyle i_2^{(1)}$}} 
   \put(39,-1){\makebox(0,0)[t]{$\scriptstyle i_2^{(1)}$}}    
    \put(43,-1){\makebox(0,0)[t]{$\scriptstyle i_2^{(1)}$}}    
     \put(48,-1){\makebox(0,0)[t]{$\scriptstyle i_2^{(1)}$}}       
       \put(57,-1){\makebox(0,0)[t]{$\scriptstyle i_2^{(1)}$}}     
         \put(62,-1){\makebox(0,0)[t]{$\scriptstyle i_2^{(1)}$}}     
                 \put(71,-1){\makebox(0,0)[t]{$\scriptstyle i_m^{(1)}$}} 
   \put(76,-1){\makebox(0,0)[t]{$\scriptstyle i_m^{(1)}$}}    
    \put(80,-1){\makebox(0,0)[t]{$\scriptstyle i_m^{(1)}$}}    
     \put(85,-1){\makebox(0,0)[t]{$\scriptstyle i_m^{(1)}$}}       
       \put(93,-1){\makebox(0,0)[t]{$\scriptstyle i_m^{(1)}$}}     
         \put(98,-1){\makebox(0,0)[t]{$\scriptstyle i_m^{(1)}$}}     
           \put(2,40){\makebox(0,0)[b]{$\scriptstyle j_1^{(1)}$}} 
   \put(6,40){\makebox(0,0)[b]{$\scriptstyle j_1^{(1)}$}}    
    \put(10,40){\makebox(0,0)[b]{$\scriptstyle j_1^{(1)}$}}    
     \put(15,40){\makebox(0,0)[b]{$\scriptstyle j_1^{(1)}$}}       
       \put(24,40){\makebox(0,0)[b]{$\scriptstyle j_1^{(1)}$}}     
         \put(29,40){\makebox(0,0)[b]{$\scriptstyle j_1^{(1)}$}}     
        \put(34,40){\makebox(0,0)[b]{$\scriptstyle j_2^{(1)}$}} 
   \put(39,40){\makebox(0,0)[b]{$\scriptstyle j_2^{(1)}$}}    
    \put(43,40){\makebox(0,0)[b]{$\scriptstyle j_2^{(1)}$}}    
     \put(48,40){\makebox(0,0)[b]{$\scriptstyle j_2^{(1)}$}}       
       \put(57,40){\makebox(0,0)[b]{$\scriptstyle j_2^{(1)}$}}     
         \put(62,40){\makebox(0,0)[b]{$\scriptstyle j_2^{(1)}$}}     
                 \put(71,40){\makebox(0,0)[b]{$\scriptstyle j_n^{(1)}$}} 
   \put(76,40){\makebox(0,0)[b]{$\scriptstyle j_n^{(1)}$}}    
    \put(80,40){\makebox(0,0)[b]{$\scriptstyle j_n^{(1)}$}}    
     \put(85,40){\makebox(0,0)[b]{$\scriptstyle j_n^{(1)}$}}       
       \put(93,40){\makebox(0,0)[b]{$\scriptstyle j_n^{(1)}$}}     
         \put(98,40){\makebox(0,0)[b]{$\scriptstyle j_n^{(1)}$}}     
         %
         %
      %
             \put(4,15){\makebox(0,0)[]{$\scriptstyle \hat{\mu}_{1,1}$}}
       \put(13,15){\makebox(0,0)[]{$\scriptstyle \hat{\mu}_{1,2}$}} 
       \put(27,15){\makebox(0,0)[]{$\scriptstyle \hat{\mu}_{1,s}$}}
        \put(37,15){\makebox(0,0)[]{$\scriptstyle \hat{\mu}_{2,1}$}}
       \put(45.6,15){\makebox(0,0)[]{$\scriptstyle \hat{\mu}_{2,2}$}} 
       \put(59.5,15){\makebox(0,0)[]{$\scriptstyle \hat{\mu}_{2,s}$}}          
         \put(73.2,15){\makebox(0,0)[]{$\scriptstyle \hat{\mu}_{r,1}$}}
       \put(82,15){\makebox(0,0)[]{$\scriptstyle \hat{\mu}_{r,2}$}} 
       \put(96,15){\makebox(0,0)[]{$\scriptstyle \hat{\mu}_{r,s}$}}   
       %
          \put(20,35){\makebox(0,0)[]{$\scriptstyle \cdots$}}
          \put(53,35){\makebox(0,0)[]{$\scriptstyle \cdots$}}
            \put(67,35){\makebox(0,0)[]{$\scriptstyle \cdots$}}
          \put(89,35){\makebox(0,0)[]{$\scriptstyle \cdots$}}
        \put(20,5){\makebox(0,0)[]{$\scriptstyle \cdots$}}
          \put(53,5){\makebox(0,0)[]{$\scriptstyle \cdots$}}
                   \put(67,5){\makebox(0,0)[]{$\scriptstyle \cdots$}}
          \put(89,5){\makebox(0,0)[]{$\scriptstyle \cdots$}} 
        \put(20,15){\makebox(0,0)[]{$\scriptstyle \cdots$}}
          \put(53,15){\makebox(0,0)[]{$\scriptstyle \cdots$}}
                   \put(67,15){\makebox(0,0)[]{$\scriptstyle \cdots$}}
          \put(89,15){\makebox(0,0)[]{$\scriptstyle \cdots$}}
\end{overpic}
}
\in \AffWebAaI(\bi^{\bx}, \bj^{\by}).
\\
\\
\end{align*}
It is straightforward to check that the function \(\mathcal{M}(\bi^{(\bx)}, \bj^{(\by)}) \to \mathcal{M}(\bi^{\bx}, \bj^{\by})\), \(\bmu \to \hat \bmu\) is injective.

\begin{example}\label{cablex}
Assume we have \(i,j \in I\), \(x \in {}_j\BasisB_i\), \(y \in {}_i \BasisB_j\), \(z \in {}_j \BasisB_j\), with \(x \in A_{\bar 0} \backslash a\), \(y \in a\), \(z \in A_{\bar 1}\), and \(i>j>x>y>z\) in the ordering on \(\BasisB\). Assume that \(\bmu\) is the \((i^{(4)}j^{(8)}, j^{(3)}i^{(5)}j^{(4)})\)-matrix with
\begin{align*}
\mu_{11}([1,x])=2; \qquad
\mu_{12}([2,i])=2; \qquad
\mu_{21}([0,j])=1; \qquad
\mu_{22}([0,y])=3;
\end{align*}
\begin{align*}
\mu_{23}([2,z])=1; \qquad
\mu_{23}([1,j])=2; \qquad
\mu_{23}([1,z])=1,
\end{align*}
and \(\mu_{r,s}([t,b])=0\) otherwise. Then we have:

\begin{align*}
\\
\hackcenter{}
\eta_{\bmu} = 
\hackcenter{
\begin{overpic}[height=35mm]{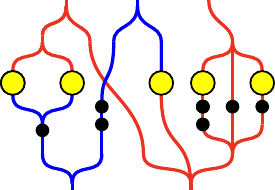}
  \put(26,-1){\makebox(0,0)[t]{$\scriptstyle i^{(4)}$}}    
    \put(69,-1){\makebox(0,0)[t]{$\scriptstyle j^{(8)}$}}    
        \put(24,69.5){\makebox(0,0)[b]{$\scriptstyle j^{(3)}$}}    
    \put(50,69.5){\makebox(0,0)[b]{$\scriptstyle i^{(5)}$}}    
        \put(76,69.5){\makebox(0,0)[b]{$\scriptstyle j^{(4)}$}}    
         \put(4.5,38.5){\makebox(0,0)[]{$\scriptstyle x$}}
          \put(26.5,38.5){\makebox(0,0)[]{$\scriptstyle x$}}       
           \put(58.5,38.5){\makebox(0,0)[]{$\scriptstyle y$}}       
            \put(73.5,38.5){\makebox(0,0)[]{$\scriptstyle z$}}       
             \put(95.5,38.5){\makebox(0,0)[]{$\scriptstyle z$}}     
 \put(10.5,10){\makebox(0,0)[b]{$\scriptstyle i^{(2)}$}}   
  \put(32.5,10){\makebox(0,0)[b]{$\scriptstyle i^{(2)}$}}   
    \put(47.5,10){\makebox(0,0)[b]{$\scriptstyle j^{(1)}$}}
    \put(60.8,10){\makebox(0,0)[b]{$\scriptstyle j^{(3)}$}}      
     \put(88,5){\makebox(0,0)[b]{$\scriptstyle j^{(4)}$}} 
     \put(8,28){\makebox(0,0)[t]{$\scriptstyle i^{(1)}$}}         
        \put(23,28){\makebox(0,0)[t]{$\scriptstyle i^{(1)}$}}
         \put(77,17){\makebox(0,0)[t]{$\scriptstyle j^{(1)}$}}      
         \put(90,26){\makebox(0,0)[t]{$\scriptstyle j^{(\hspace{-0.2mm}2\hspace{-0.2mm})}$}}   
          \put(101.5,30){\makebox(0,0)[t]{$\scriptstyle j^{(1)}$}} 
     \put(8,54){\makebox(0,0)[t]{$\scriptstyle j^{(1)}$}}         
        \put(23,56){\makebox(0,0)[t]{$\scriptstyle j^{(1)}$}}
          \put(17,65){\makebox(0,0)[t]{$\scriptstyle j^{(2)}$}}
         \put(65,50){\makebox(0,0)[]{$\scriptstyle i^{(3)}$}}   
\end{overpic}
}
\qquad
\textup{and}
\qquad
\hackcenter{}
\eta_{\hat{\bmu}}=
\hackcenter{
\begin{overpic}[height=35mm]{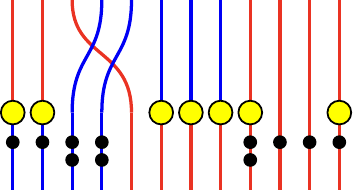}
  \put(3.5,-1){\makebox(0,0)[t]{$\scriptstyle i^{(1)}$}}    
    \put(12,-1){\makebox(0,0)[t]{$\scriptstyle i^{(1)}$}}    
       \put(20,-1){\makebox(0,0)[t]{$\scriptstyle i^{(1)}$}}   
          \put(28.5,-1){\makebox(0,0)[t]{$\scriptstyle i^{(1)}$}}   
        \put(37,-1){\makebox(0,0)[t]{$\scriptstyle j^{(1)}$}}    
    \put(45.5,-1){\makebox(0,0)[t]{$\scriptstyle j^{(1)}$}}    
       \put(54,-1){\makebox(0,0)[t]{$\scriptstyle j^{(1)}$}}   
          \put(62.5,-1){\makebox(0,0)[t]{$\scriptstyle j^{(1)}$}}   
                 \put(71,-1){\makebox(0,0)[t]{$\scriptstyle j^{(1)}$}}   
                        \put(79,-1){\makebox(0,0)[t]{$\scriptstyle j^{(1)}$}}   
                               \put(88,-1){\makebox(0,0)[t]{$\scriptstyle j^{(1)}$}}   
                                \put(96,-1){\makebox(0,0)[t]{$\scriptstyle j^{(1)}$}}   
        \put(3.5,55){\makebox(0,0)[b]{$\scriptstyle j^{(1)}$}}    
    \put(12,55){\makebox(0,0)[b]{$\scriptstyle j^{(1)}$}}    
       \put(20,55){\makebox(0,0)[b]{$\scriptstyle j^{(1)}$}}   
          \put(28.5,55){\makebox(0,0)[b]{$\scriptstyle i^{(1)}$}}   
        \put(37,55){\makebox(0,0)[b]{$\scriptstyle i^{(1)}$}}    
    \put(45.5,55){\makebox(0,0)[b]{$\scriptstyle i^{(1)}$}}    
       \put(54,55){\makebox(0,0)[b]{$\scriptstyle i^{(1)}$}}   
          \put(62.5,55){\makebox(0,0)[b]{$\scriptstyle j^{(1)}$}}   
                 \put(71,55){\makebox(0,0)[b]{$\scriptstyle j^{(1)}$}}   
                        \put(79,55){\makebox(0,0)[b]{$\scriptstyle j^{(1)}$}}   
                               \put(88,55){\makebox(0,0)[b]{$\scriptstyle j^{(1)}$}}   
                               \put(96,55){\makebox(0,0)[b]{$\scriptstyle j^{(1)}$}}   
        \put(3.5,22){\makebox(0,0)[]{$\scriptstyle x$}}    
            \put(12,22){\makebox(0,0)[]{$\scriptstyle x$}}    
             \put(45.5,22){\makebox(0,0)[]{$\scriptstyle y$}}   
              \put(54,22){\makebox(0,0)[]{$\scriptstyle y$}}   
               \put(62.5,22){\makebox(0,0)[]{$\scriptstyle y$}}   
                \put(71,22){\makebox(0,0)[]{$\scriptstyle z$}}   
                \put(96.3,22){\makebox(0,0)[]{$\scriptstyle z$}}   
\end{overpic}
}
\\
\end{align*}
\end{example}

\subsection{A spanning set for \(\AffWebAaI\)}
\begin{lemma}\label{LemAffLower}
Let \(D\in \AffWebAaI(\bi^{(\bx)}, \bj^{(\by)})\) be a diagram of affine degree \(d\). Then \(D \in \AffWebAaI(\bi^{(\bx)}, \bj^{(\by)})_{\leq d}\).
\end{lemma}
\begin{proof}
We proceed by induction on \(d\). When \(d=0\), this result follows as in the proof of \cite[Proposition 6.5.1]{DKMZ}. For \(d>0\), note that thanks to \cref{DotCrossRel}, crossings intertwine affine dots modulo \(\k\)-linear combinations of diagrams of lower affine degree. Hence by the induction assumption we may similarly follow the proof of \cite[Proposition 6.5.1]{DKMZ}, treating affine dots as coupons in that setting, arriving at the result. \end{proof}

\begin{corollary}\label{spancor}
The diagrams \(\{ \eta_{\bmu} \mid \bmu \in \mathcal{M}(\bi^{(\bx)}, \bj^{(\by)})\}\) form a \(\k\)-spanning set for \(\AffWebAaI(\bi^{(\bx)}, \bj^{(\by)})\).
\end{corollary}

\section{The thin dot web category \(\ThinAffWebAaI\)}

\subsection{The category \texorpdfstring{$\ThinAffWebAaI$}{ThinWebAaIthin}}

We next introduce the following simplified version of \(\AffWebAaI\) where the only allowed affine dot generators are those on strands of thickness one.   The defining relations are exactly those of \(\AffWebAaI\) applied to this restricted set of generators.    It will be helpful in certain upcoming arguments and can be viewed as an alternate presentation of $\AffWebAaI$ when $\k$ is a field (see \cref{T:Thin-dot-is-equivalent}).

\begin{definition}
 Let \(\ThinAffWebAaI\) be the strict monoidal \(\k\)-linear supercategory defined as follows.  
The objects and the monoidal structure on objects are the same as in \(\WebAaI\), given by \(\Ob(\ThinAffWebAaI)=\widehat{ \Omega}_I\). 
The generating morphisms of \(\ThinAffWebAaI\) are given by the diagrams:
\begin{align}\label{ThinWebAaGens}
\hackcenter{
{}
}
\hackcenter{
\begin{tikzpicture}[scale=.8]
  \draw[ultra thick,blue] (0,0)--(0,0.2) .. controls ++(0,0.35) and ++(0,-0.35) .. (-0.4,0.9)--(-0.4,1);
  \draw[ultra thick,blue] (0,0)--(0,0.2) .. controls ++(0,0.35) and ++(0,-0.35) .. (0.4,0.9)--(0.4,1);
      \node[above] at (-0.4,1) {$ \scriptstyle i^{\scriptstyle (x)}$};
      \node[above] at (0.4,1) {$ \scriptstyle i^{\scriptstyle (y)}$};
      \node[below] at (0,0) {$ \scriptstyle i^{\scriptstyle (x+y)} $};
\end{tikzpicture}}
\qquad
\qquad
\hackcenter{
\begin{tikzpicture}[scale=.8]
  \draw[ultra thick,blue ] (-0.4,0)--(-0.4,0.1) .. controls ++(0,0.35) and ++(0,-0.35) .. (0,0.8)--(0,1);
\draw[ultra thick, blue] (0.4,0)--(0.4,0.1) .. controls ++(0,0.35) and ++(0,-0.35) .. (0,0.8)--(0,1);
      \node[below] at (-0.4,0) {$ \scriptstyle i^{ \scriptstyle (x)}$};
      \node[below] at (0.4,0) {$ \scriptstyle i^{ \scriptstyle (y)}$};
      \node[above] at (0,1) {$ \scriptstyle i^{ \scriptstyle (x+y)}$};
\end{tikzpicture}}
\qquad
\qquad
\hackcenter{
\begin{tikzpicture}[scale=.8]
  \draw[ultra thick,red] (0.4,0)--(0.4,0.1) .. controls ++(0,0.35) and ++(0,-0.35) .. (-0.4,0.9)--(-0.4,1);
  \draw[ultra thick,blue] (-0.4,0)--(-0.4,0.1) .. controls ++(0,0.35) and ++(0,-0.35) .. (0.4,0.9)--(0.4,1);
      \node[above] at (-0.4,1) {$ \scriptstyle j^{ \scriptstyle (y)}$};
      \node[above] at (0.4,1) {$ \scriptstyle i^{ \scriptstyle (x)}$};
       \node[below] at (-0.4,0) {$ \scriptstyle i^{ \scriptstyle (x)}$};
      \node[below] at (0.4,0) {$ \scriptstyle j^{ \scriptstyle (y)}$};
\end{tikzpicture}}
\qquad
\qquad
\hackcenter{
\begin{tikzpicture}[scale=.8]
  \draw[ultra thick, blue] (0,0)--(0,0.5);
   \draw[ultra thick, red] (0,0.5)--(0,1);
   \draw[thick, fill=yellow]  (0,0.5) circle (7pt);
    \node at (0,0.5) {$ \scriptstyle f$};
     \node[below] at (0,0) {$ \scriptstyle i^{ \scriptstyle (z)}$};
      \node[above] at (0,1) {$ \scriptstyle j^{ \scriptstyle (z)}$};
\end{tikzpicture}}
\qquad
\qquad
\hackcenter{
\begin{tikzpicture}[scale=.8]
  \draw[ultra thick, blue] (0,0)--(0,0.5);
   \draw[ultra thick, blue] (0,0.5)--(0,1);
   \draw[thick, fill=black]  (0,0.5) circle (5pt);
     \node[below] at (0,0) {$ \scriptstyle i^{ \scriptstyle (1)}$};
      \node[above] at (0,1) {$ \scriptstyle \nakayamai ^{ \scriptstyle (1)}$};
\end{tikzpicture}},
\end{align}
where \(i,j \in I\), \(x,y \in \Z_{\geq 0}\), \(z \in \Z_{>0}\), \(f \in jA^{(z)}i\). 
The relations between morphisms in \(\ThinAffWebAaI\) are given as follows. The merge, split, crossings, and coupons satisfy \cref{AssocRel} through \cref{AaIntertwine} as in \(\WebAaI\). The relations involving the affine dot generators are listed below.

{\em Thin affine dot relations.} For all \(i,j \in I, f \in jAi\):
\begin{align}\label{ThinAffDotRel1}
\hackcenter{
\begin{tikzpicture}[scale=.8]
  \draw[ultra thick, blue] (0,0)--(0,0.5);
    \draw[ultra thick, red] (0,0.5)--(0,1.3);
      \draw[ultra thick, red] (0,1.3)--(0,1.8);
     \draw[thick, fill=yellow]  (0,0.5) circle (7pt);
    \node at (0,0.5) {$ \scriptstyle f$};
   \draw[thick, fill=black]  (0,1.3) circle (5pt);
    \node at (0,1.3) {$\scriptstyle h$};
     \node[below] at (0,0) {$ \scriptstyle i^{ \scriptstyle (1)}$};
      \node[above] at (0,1.8) {$ \scriptstyle \nakayamaj^{\scriptstyle (1)}$};
         \node[left] at (-0.1,0.9) {$ \scriptstyle j^{ \scriptstyle (1)}$};
\end{tikzpicture}}
\;=\;
\hackcenter{
\begin{tikzpicture}[scale=.8]
  \draw[ultra thick, blue] (0,0)--(0,0.5);
    \draw[ultra thick, blue] (0,0.5)--(0,1.3);
      \draw[ultra thick, red] (0,1.3)--(0,1.8);
     \draw[thick, fill=black]  (0,0.5) circle (5pt);
    \node at (0,0.5) {$ \scriptstyle f$};
   \draw[thick, fill=yellow]  (0,1.3) circle (7pt);
    \node at (0,1.3) {$\scriptstyle \nakayamaf$};
     \node[below] at (0,0) {$ \scriptstyle i^{ \scriptstyle (1)}$};
      \node[above] at (0,1.8) {$ \scriptstyle \nakayamaj ^{\scriptstyle (1)}$};
         \node[left] at (-0.1,0.9) {$ \scriptstyle \nakayamai ^{ \scriptstyle (1)}$};
\end{tikzpicture}},
\qquad
\qquad
{}
\hackcenter{
\begin{tikzpicture}[scale=0.8]
  \draw[ultra thick,red] (0.4,-0.4)--(0.4,0.1) .. controls ++(0,0.35) and ++(0,-0.35) .. (-0.4,0.9)--(-0.4,1.4);
  \draw[ultra thick,blue] (-0.4,-0.4)--(-0.4,0.1) .. controls ++(0,0.35) and ++(0,-0.35) .. (0.4,0.9)--(0.4,1.4);
      \node[above] at (-0.4,1.4) {$ \scriptstyle j^{ \scriptstyle (1)}$};
      \node[above] at (0.4,1.4) {$ \scriptstyle \nakayamai ^{ \scriptstyle (1)}$};
       \node[below] at (-0.4,-0.4) {$ \scriptstyle i^{ \scriptstyle (1)}$};
      \node[below] at (0.4,-0.4) {$ \scriptstyle j^{ \scriptstyle (1)}$};
                  \draw[thick, fill=black]  (-0.4,0) circle (5pt);
\end{tikzpicture}}
=
\hackcenter{
\begin{tikzpicture}[scale=0.8]
  \draw[ultra thick,red] (0.4,-0.4)--(0.4,0.1) .. controls ++(0,0.35) and ++(0,-0.35) .. (-0.4,0.9)--(-0.4,1.4);
  \draw[ultra thick,blue] (-0.4,-0.4)--(-0.4,0.1) .. controls ++(0,0.35) and ++(0,-0.35) .. (0.4,0.9)--(0.4,1.4);
   \node[above] at (-0.4,1.4) {$ \scriptstyle j^{ \scriptstyle (1)}$};
      \node[above] at (0.4,1.4) {$ \scriptstyle \nakayamai ^{ \scriptstyle (1)}$};
       \node[below] at (-0.4,-0.4) {$ \scriptstyle i^{ \scriptstyle (1)}$};
      \node[below] at (0.4,-0.4) {$ \scriptstyle j^{ \scriptstyle (1)}$};
            \draw[thick, fill=black]  (0.4,1) circle (5pt);
\end{tikzpicture}}
-
\hackcenter{
\begin{tikzpicture}[scale=0.8]
  \draw[ultra thick,blue] (-0.4,-0.4)--(-0.4,0.5);
   \draw[ultra thick,red] (0.4,-0.4)--(0.4,0.5);
  \draw[ultra thick,blue] (0.4,0.5)--(0.4,1.4);
    \draw[ultra thick,red] (-0.4,0.5)--(-0.4,1.4);
   \node[above] at (-0.4,1.4) {$ \scriptstyle j^{ \scriptstyle (1)}$};
      \node[above] at (0.4,1.4) {$ \scriptstyle \nakayamai ^{ \scriptstyle (1)}$};
       \node[below] at (-0.4,-0.4) {$ \scriptstyle i^{ \scriptstyle (1)}$};
      \node[below] at (0.4,-0.4) {$ \scriptstyle j^{ \scriptstyle (1)}$};
        \draw[thick, black, fill=black]  (-0.4,0.5) circle (6pt);
        \draw[thick,black, fill=black]  (0.4,0.5) circle (6pt);
                   \draw[decorate, decoration={snake,segment length=5pt, amplitude=1pt}, line width=1mm, black] (-0.4,0.5)--(0.4,0.5);
      \draw[thick, lime, fill=lime]  (-0.4,0.5) circle (5pt);
        \draw[thick, lime, fill=lime]  (0.4,0.5) circle (5pt);
      \draw[decorate, decoration={snake,segment length=5pt, amplitude=1pt}, line width=0.5mm, lime] (-0.4,0.5)--(0.4,0.5);
\end{tikzpicture}}.
\end{align}
\end{definition}

It is immediate from the definition that there is a monoidal superfunctor 
\begin{equation}\label{E:ThinAff-into-Aff-functor}
\iota=\iota_{\calA}: \ThinAffWebAaI \to \AffWebAaI
\end{equation}
which is the identity on objects and sends every web diagram in $\ThinAffWebAaI$ to its mate in $\AffWebAaI$.

\subsection{Implied relations in \texorpdfstring{$\ThinAffWebAaI$}{ThinAffWebAaI} Over a Field} Let $\KK$ be a characteristic zero field that contains $\k$.  Throughout this section we work over $\KK$.  Let $\calA  = \left((A,\a )_{I}, \tr , \psi \right)$ be a Frobenius great pair defined over $\KK$.   For example, $\calA$ could have been obtained from a Frobenius great pair defined over $\k$ by extending scalars to $\KK$.  While in general there are no thick affine dots in $\ThinAffWebAaI$, replacement morphisms can be defined over $\KK$ for $i \in I$ and $x \geq 1$ by
\begin{align}\label{whitedotdef}
\hackcenter{
}.
\end{align}  We refer to these as white affine dots.  We will establish that the white dots satisfies the same relations as the thick affine dots in \(\AffWebAaI\).

\begin{lemma} For all \(i,j,k \in I, x,y \in \Z_{\geq 0}, f \in jA^{(x)}i\), the following equalities hold in \(\KThinAffWebAaI\):
\begin{align*}
\hackcenter{
%
}
\end{align*}
\end{lemma}
\begin{proof}
The first equality is a straightforward consequence of \cref{whitedotdef,AffDotRel1}. We prove the second equality.  The third equality is similar and is left to the reader. We have
\begin{align*}
\hackcenter{}
\hackcenter{
%
}.
\end{align*}
\end{proof}


Since the definitions of the transporters given in \ref{thicktransporterdef} do not involve the thick affine dots, they make sense as morphisms in $\ThinAffWebAaI$.

\begin{lemma}\label{thintothickdotcrossing}
For all \(i,j,k \in I, x,y \in \Z_{\geq 0}\) the following equality holds in \(\ThinAffWebAaI\):
\begin{equation}\label{ThintoThickDotCrossRel}
\hackcenter{
%
}
\end{equation}
\end{lemma}
\begin{proof}
If \(x=0\) or \(y=0\), the claim is trivial. We first prove the claim holds for \(y=1\) and all \(x \in \Z_{>0}\), by induction on \(x\). The base case \(x=1\) follows from the defining relations. Now fix \(x>1\) and assume that \cref{ThintoThickDotCrossRel} holds for \(y=1\) and \(x-1\). Then we have:
\begin{align*}
&\hackcenter{
%
}.
\end{align*}
This proves the claim when $y=1$. 

Now assume that \(y>1\), and that \cref{ThintoThickDotCrossRel} holds for arbitrary \(x\) and \(y'<y\). Then we have:
\begin{align*}
&
\hackcenter{
%
}.
\end{align*}
This completes the proof.
\end{proof}

\subsection{Equivalence of \texorpdfstring{$\AffWebAaI$}{AffWebAaI} and \texorpdfstring{$\ThinAffWebAaI$}{ThinAffWebAaI} Over a Field}

 If $\KK$ is a field which contains $\k$ and $\calA$ is a Frobenius great pair defined over $\k$, then let $\KK \otimes \calA$ denote the Frobenius great pair obtained from $\calA$ by extending scalars to $\KK$.

\begin{theorem}\label{T:Affine-to-thin-dot-functor}  Let $\calA$ be a Frobenius great pair defined over $\k$.  Then there exists a unique functor of $\k$-linear monoidal categories, 
\[
\kappa = \kappa_{\calA}: \mathbf{Web}^{\textup{aff}}_{\calA} \to  \mathbf{Web}^{\textup{aff},1}_{\KK \otimes \calA},
\] which is the identity on objects and on merges, splits, coupons, and crossings; on affine dots the functor is given by
\[
\kappa\left(
\hackcenter{\begin{tikzpicture}[scale=0.8]
  \draw[ultra thick, blue] (0,0)--(0,0.9);
  \draw[ultra thick, blue] (0,0.9)--(0,1.8);
 \draw[thick, fill=black] (0,0.9) circle (5pt);
     \node[below] at (0,0) {$\scriptstyle i^{\scriptstyle(x)}$};
      \node[above] at (0,1.8) {$\scriptstyle \nakayamai^{\scriptstyle(x)}$};
\end{tikzpicture}} \right)
\; = \; 
\hackcenter{
\begin{tikzpicture}[scale=0.8]
  \draw[ultra thick, blue] (0,0)--(0,0.9);
  \draw[ultra thick, blue] (0,0.9)--(0,1.8);
 \draw[thick, fill=white] (0,0.9) circle (5pt);
     \node[below] at (0,0) {$\scriptstyle i^{\scriptstyle(x)}$};
      \node[above] at (0,1.8) {$\scriptstyle \nakayamai ^{\scriptstyle(x)}$};
\end{tikzpicture}}
\; =\;
\frac{1}{x!}
\hackcenter{
\begin{tikzpicture}[scale=0.8]
  \draw[ultra thick, blue] (0,0)--(0,0.1) .. controls ++(0,0.35) and ++(0,-0.35) .. (-0.6,0.6)--(-0.6,0.9);
    \draw[ultra thick, blue]  (-0.6,0.9)--(-0.6,1.2) 
  .. controls ++(0,0.35) and ++(0,-0.35) .. (0,1.7)--(0,1.8);
  \draw[ultra thick, blue] (0,0)--(0,0.1) .. controls ++(0,0.35) and ++(0,-0.35) .. (0.6,0.6)--(0.6,0.9);
    \draw[ultra thick, blue]  (0.6,0.9)--(0.6,1.2) 
  .. controls ++(0,0.35) and ++(0,-0.35) .. (0,1.7)--(0,1.8);
  .. controls ++(0,0.35) and ++(0,-0.35) .. (0,1.7)--(0,1.8);
   \draw[thick, fill=black] (-0.6,0.9) circle (5pt);
     \draw[thick, fill=black] (0.6,0.9) circle (5pt);
       \node[above] at (0,1.8) {$\scriptstyle \nakayamai ^{\scriptstyle(x)}$};
         \node[below] at (0,0) {$\scriptstyle i^{\scriptstyle(x)}$};
           \node at (-0.6,0.9) {$\scriptstyle f$};
             \node at (0.6,0.9) {$\scriptstyle f$};
             \node[right] at (0.5,0.3) {$\scriptstyle i^{\scriptstyle(1)}$};
               \node[left] at (-0.5,0.3) {$\scriptstyle i^{\scriptstyle(1)}$};
                \node at (0,0.9) {$\scriptstyle \cdots $};
\end{tikzpicture}}.
\]
\end{theorem}
\begin{proof}   Since $\mathbf{Web}^{\textup{aff},1}_{\KK \otimes \calA}$ is defined as a $\KK$-linear monoidal category, the white dots are sensible morphisms in this category and, hence, the given assignment of generating morphisms is allowed.  The claim is proven by verifying that the defining relations of $\AffWebAaI$ are satisfied.  This is immediate from the relations proven in this section.
\end{proof}

\begin{corollary}\label{T:Thin-dot-is-equivalent}  When $\k$ is a field of characteristic zero and $\calA$ is a Frobenius great pair defined over $\k$, then the functor from  \cref{E:ThinAff-into-Aff-functor},
\[
\iota_{\calA}:\ThinAffWebAaI \to \AffWebAaI,
\]  is an isomorphism of $\k$-linear monoidal categories
\end{corollary}
\begin{proof} An easy application of \cref{AffDotRel1,KnotholeRel} shows that $\iota_{\calA}$ and the functor $\kappa_{\calA}$ from \cref{T:Affine-to-thin-dot-functor} are mutual inverses.
\end{proof}

\begin{remark}  The functor $\iota_{\calA}$ will always be essentially surjective and faithful.  Essential surjectivity is obvious and faithfulness is an immediate consequence of Corollary~\ref{spancor}.
\end{remark}

\section{The defining representation of \(\WebAaI\)}

In this section we recall some of the main results from \cite{DKMZ} about the defining representation of \(\WebAaI\) on the supercategory of \(\gl_n(A)\)-supermodules. We upgrade this to a defining representation of \(\AffWebAaI\) in \cref{CasSec}.

\subsection{ The supercategory of \texorpdfstring{$\gl_n(A)$}{gln(A)}-supermodules}

For $n \geq 1$, let \(\gl_{n}(A)\) be the free $\k$-module of $n \times n$ matrices with entries in $A$.  For $r,s \in [1,n]$ and $f \in A$, write  \(E_{r,s}^{f}\) for the $n \times n$ matrix with $(r,s)$-entry equal to $f$ and all other entries equal to zero.  The $\Z _{2}$-grading on $\gl_{n}(A)$ is given by declaring $\overline{E_{r,s}^{f}} = \overline{f}$ for all homogeneous $f \in A$.  The set of matrices $\{E_{r,s}^{b} \mid r,s \in [1,n] \text{ and } b \in \BasisB\}$ form a homogeneous $\k$-basis for $\gl_{n}(A)$.

We view $\gl_{n}(A)$ as a Lie superalgebra over $\k$ via the graded commutator bracket,
\begin{align*}
[E_{r,s}^f, E_{r',s'}^{f'}] = \delta_{s, r'} E_{r,s'}^{ff'} - (-1)^{\bar f \bar f'} \delta_{s', r} E_{r',s}^{f'f},
\end{align*}
for all \(r,s, r', s' \in [1,n]\) and homogeneous \(f, f' \in A\).  Write \(U(\gl_{n}(A))\) for the universal enveloping superalgebra of \(\gl_n(A)\).

We will be interested in the category of right $\gl_{n}(A)$-supermodules, $\modgl$.  Our convention is to write $\gl_{n}(A)$-supermodule homomorphisms on the left and to allow for all (not just parity preserving) homomorphisms.  That is, a $\k$-linear map $f : M \to N$ between right $\gl_{n}(A)$-supermodules is a homomorphism if $f(mu) = f(m)u$ for all $m \in M$ and $u \in \gl_n(A)$.  The space of supermodule homomorphisms $\Hom_{\gl_n(A)}(M,N)$ is a $\k$-supermodule where the grading is given by declaring the parity of $f$ to be $s \in \Z_{2}$ if $f(M_{r}) \subset N_{r+s}$ for $r \in \Z_2$.

The coproduct on \(U(\gl_n(A))\) is given by \(\Delta(u) = 1 \otimes u + u \otimes 1\) for $u \in \gl_{n}(A)$.  Explicitly, if $M$ and $N$ are $\gl_{n}(A)$-supermodules, then  $M\otimes N$ is a right $\gl_n(A)$-supermodule via:
\begin{align*}
(m \otimes n) \cdot u = (-1)^{\overline u \,\overline n}mu \otimes n + m \otimes nu
\end{align*} for homogenous $m \in M$, $n \in N$, and $u \in \gl_{n}(A)$. The trivial $\gl_{n}(A)$-supermodule is $\k$ (concentrated in degree $\bar 0$) with $\gl_{n}(A)$ acting as zero.   This data makes $\modgl$ into a monoidal supercategory. The twist map given below makes it a symmetric monoidal supercategory.

\subsection{Symmetric powers}
For $n \geq 1$, let $V=V_{n} := A^{\oplus n}$ denote the set of length $n$ row vectors with entries in $A$.   For $t \in [1,n]$ and $a \in A$, we write $v_t^{a} \in V$ for the row vector which has $a$ in position $t$ and is zero elsewhere.  The set $\{v_t^{b} \mid t \in [1,n], b \in \BasisB\}$ is a homogeneous $\k$-basis for $V$.  Observe that $V_{n}$ is naturally a right $\gl_n(A)$-supermodule with action given by right multiplication:
$$
v_t^{a} \cdot E_{r,s}^f = \delta_{r,t} v_s^{a f},
$$  
for $r,s,t \in [1,n]$ and $a,f \in A$.
For each \(i \in I\) there is a $\gl_{n}(A)$-submodule ${}_i V={}_i V_{n} = (iA)^{\oplus n} \subseteq V_n$ with homogeneous basis $\{v_r^{b} \mid r \in [1,n], b \in  {}_i \BasisB \}$.

Via the coproduct, the tensor powers $({}_i V)^{\otimes x}$ for $x \geq 0$ are $\gl_n(A)$-supermodules, as is the tensor algebra $T^\bullet({}_i V) = \bigoplus_{x \geq  0} ({}_i V)^{\otimes x}$.  By convention, $({}_i V)^{\otimes 0}=\k$ for all $i \in I$.
 
Let $J$ denote the two-sided ideal in $T^\bullet({}_i V)$ generated by
\[
\left\{ v \otimes w - (-1)^{\bar v \bar w} w \otimes v \mid \text{ homogeneous } v,w \in {}_i V   \right\}  \cup \left\{v \otimes v \mid v \in \left({}_i V \right)_{\1}  \right\}.
\]
Let $S^{\bullet}({}_i V)$ denote the symmetric superalgebra $T^\bullet({}_i V) / J$ and let $p:T^\bullet({}_i V) \to S^\bullet({}_i V)$ be the canonical map.  Since $J$ is a homogeneous ideal in the $\Z$-grading on $T^{\bullet}({}_{i}V)$, we have $S^{\bullet}({}_i V) = \oplus_{x \geq 0} S^{x}_i V$, where $S^{x}_i V = S^{x}({}_i V)$ is the image of $({}_i V)^{\otimes x}$ under $p$.   If we write  $v_{i_1}^{\alpha_1}\cdots v_{i_x}^{\alpha_x}$ for $p(v_{i_1}^{\alpha_1} \otimes \cdots \otimes v_{i_x}^{\alpha_x})$, then for each \(x \in \Z_{\geq 0}\) and \(i \in I\), \(S^x_iV\) has a \(\k\)-basis given by the set 
\begin{equation*}
\left\{ (v_{1}^{b_{1}})^{r_{1}}\dotsb (v_{n}^{b_{n}})^{r_{n}} \left| \; b_{1}, \dotsc , b_{n} \in \BasisB, r_{k} \in \Z_{\geq 0}, \sum_{k=1}^{n} r_{k}=x, \text{ and } r_{k} \in \{0,1 \} \text{ whenever } \overline{b}_{k} = \1 \right. \right\}.
\end{equation*}  Note that for all $i \in I$, $S^0_i V \cong  \k$ and $S^1_i V \cong {}_{i}V$.  
 
 Since $J$ is a $\gl_n(A)$-submodule of $T^\bullet({}_i V)$, the symmetric superalgebra $S^{\bullet}({}_i V)$ and the symmetric powers $S^x_i V$ are also $\gl_n(A)$-supermodules with the action given on homogeneous elements by
\begin{align*}
v_{r_1}^{f_1} \cdots v_{r_x}^{f_x} \cdot E_{s,t}^{g}
=
\sum_{u =1}^x 
(-1)^{\bar g ( \bar f_{u+1} + \cdots + \bar f_x)} \delta_{r_u s}v_{r_1}^{f_1} \cdots v_{r_{u-1}}^{f_{u-1}} v_{t}^{f_u g} v_{u+1}^{f_{u+1}} \cdots v_x^{f_x}.
\end{align*}

\subsection{Distinguished morphisms}\label{S:DistinguishedMorphisms}

Having defined the symmetric powers $S_i^x V$ for $ i \in I$ and $x \geq 0$, we next define certain homomorphisms between tensor products of these  $\gl_{n}(A)$-modules.  These are discussed in greater detail in \cite[Section 5]{DKMZ}.

\subsubsection{The twist map} For any two $\gl_n(A)$-supermodules $M$ and $N$, there is an even ``twist'' supermodule homomorphism:
$$
\tau_{M,N} : M \otimes N \to N \otimes M, \qquad m \otimes n \mapsto (-1)^{\bar m \bar n} n \otimes m.
$$
In particular, taking $i,j \in I$ and $x,y \in \Z_{\geq 0}$, we have the special case of $M = S_i^x V$ and $N = S_j^y V$ and the associated $\gl_n(A)$-supermodule homomorphism
$$
\tau_{i^x,j^y} := \tau_{S_i^x V, S_j^y V} : S_i^x V \otimes S_j^y V \to  S_j^y V \otimes S_i^x V.
$$

\subsubsection{The split map}
For any $i \in I$, we have the \emph{split map} \(\textup{spl}_{i^x,i^y}: S^{x+y}_i V \to S^x_i V \otimes S^y_i V\) given by  
\begin{align}\label{ExpSpl}
\textup{spl}_{i^x,i^y}(w_1 w_2 \cdots w_{x+y})= \sum_{\substack{T = \{t_1< \cdots< t_x\} \\ U = \{u_1< \cdots<u_y\} \\ T \cup U = \{1,\ldots, x+y\} }} (-1)^{\varepsilon(T,U)}w_{t_1} \cdots w_{t_x} \otimes w_{u_1} \cdots w_{u_y},
\end{align}
for all homogeneous \(w_1, \ldots, w_{x+y} \in {}_i V\), where \(\varepsilon(T,U) \in \Z_2\) is defined by
\begin{align*}
\varepsilon(T,U)=\#\{(t,u) \in T \times U \mid t>u, \,\overline{w}_t = \overline{w}_u = \bar 1\}.
\end{align*}  The split map is an even supermodule homomorphism.

\subsubsection{The merge map}
Similarly, for $i \in I$ and $x,y \in \Z_{\geq 0}$, define the \emph{merge map}
\(
\textup{mer}_{i^x,i^y} : S^x_iV \otimes S^x_iV \to S^{x+y}_i V
\)
by 
\begin{align}\label{ExpMer}
\textup{mer}_{i^x,i^y}(w_1 \cdots w_x \otimes w_1' \cdots w_y') = w_1 \cdots w_x w_1' \cdots w_y',
\end{align}
for all \(w_1, \ldots, w_x, w_1', \ldots, w_y' \in {}_i V\).  Again, this is an even supermodule homomomorphism.

\subsubsection{The left multiplication map}
Fix \(i,j \in I\) and \(x \in \Z_{\geq 0}\), and suppose that \(f \in jA^{(x)}i\).  Left multiplication by \(f\) defines a $\k$-linear map ${}_i A \to {}_j A$.  With this in mind,  there is a $\gl_n(A)$-supermodule homomorphism with parity $\bar f$,
\(
L_x^f: S^x_iV \to S^x_jV,
\)
given by
\begin{align}\label{ExpL}
L_x^f(v_{r_{1}}^{g_1} \cdots v_{r_{x}}^{g_x}) = v_{r_{1}}^{fg_1} \cdots v_{r_{x}}^{fg_x}
\end{align}
for all $r_{k} \in [1,n]$ and homogenous \(g_{k} \in {}_iV\). Note that when \(x \geq 2\), \(\bar{f} = \bar 0\), so there is no need to worry about signs.

\subsection{The defining representation of \texorpdfstring{$\WebAaI$}{WebAaI}}\label{S:DefiningRepofFiniteWebs}

The following theorem is \cite[Theorem 5.5.1]{DKMZ}.  We call this functor the defining representation of $\WebAaI$.

\begin{theorem} \label{Gthm} 
Fix $n \geq 1$.  Then there is a monoidal superfunctor
\begin{align*}
G_n: \WebAaI \to \modgl
\end{align*}
given on generating objects by \(G_n(i^{(x)}) = S^x_iV \) and on generating morphisms by
\begin{align*}
\hackcenter{
{}
}
\hackcenter{
\begin{tikzpicture}[scale=.8]
  \draw[ultra thick,blue] (0,0)--(0,0.2) .. controls ++(0,0.35) and ++(0,-0.35) .. (-0.4,0.9)--(-0.4,1);
  \draw[ultra thick,blue] (0,0)--(0,0.2) .. controls ++(0,0.35) and ++(0,-0.35) .. (0.4,0.9)--(0.4,1);
      \node[above] at (-0.4,1) {$ \scriptstyle i^{\scriptstyle (x)}$};
      \node[above] at (0.4,1) {$ \scriptstyle i^{\scriptstyle (y)}$};
      \node[below] at (0,0) {$ \scriptstyle i^{\scriptstyle (x+y)} $};
\end{tikzpicture}}
\mapsto
 \textup{spl}_{i^x,i^y}
\qquad
\hackcenter{
\begin{tikzpicture}[scale=.8]
  \draw[ultra thick,blue ] (-0.4,0)--(-0.4,0.1) .. controls ++(0,0.35) and ++(0,-0.35) .. (0,0.8)--(0,1);
\draw[ultra thick, blue] (0.4,0)--(0.4,0.1) .. controls ++(0,0.35) and ++(0,-0.35) .. (0,0.8)--(0,1);
      \node[below] at (-0.4,0) {$ \scriptstyle i^{ \scriptstyle (x)}$};
      \node[below] at (0.4,0) {$ \scriptstyle i^{ \scriptstyle (y)}$};
      \node[above] at (0,1) {$ \scriptstyle i^{ \scriptstyle (x+y)}$};
\end{tikzpicture}}
\mapsto
 \textup{mer}_{i^x,i^y}
\qquad
\hackcenter{
\begin{tikzpicture}[scale=.8]
  \draw[ultra thick,red] (0.4,0)--(0.4,0.1) .. controls ++(0,0.35) and ++(0,-0.35) .. (-0.4,0.9)--(-0.4,1);
  \draw[ultra thick,blue] (-0.4,0)--(-0.4,0.1) .. controls ++(0,0.35) and ++(0,-0.35) .. (0.4,0.9)--(0.4,1);
      \node[above] at (-0.4,1) {$ \scriptstyle j^{ \scriptstyle (y)}$};
      \node[above] at (0.4,1) {$ \scriptstyle i^{ \scriptstyle (x)}$};
       \node[below] at (-0.4,0) {$ \scriptstyle i^{ \scriptstyle (x)}$};
      \node[below] at (0.4,0) {$ \scriptstyle j^{ \scriptstyle (y)}$};
\end{tikzpicture}}
\mapsto
\tau_{i^x,j^y}
\qquad
\hackcenter{
\begin{tikzpicture}[scale=.8]
  \draw[ultra thick, blue] (0,0)--(0,0.5);
   \draw[ultra thick, red] (0,0.5)--(0,1);
   \draw[thick, fill=yellow]  (0,0.5) circle (7pt);
    \node at (0,0.5) {$ \scriptstyle f$};
     \node[below] at (0,0) {$ \scriptstyle i^{ \scriptstyle (x)}$};
      \node[above] at (0,1) {$ \scriptstyle j^{ \scriptstyle (x)}$};
\end{tikzpicture}}
\mapsto L_x^f \; ,
\end{align*}
for all \(x,y \in \Z_{\geq 0}\), \(i,j \in I\), and \(f \in jA^{(x)}i\).
\end{theorem}

\subsection{The category \texorpdfstring{$\catEnd(\modgl)$}{End(modgln(A))} }

Let $\catEnd(\modgl)$ denote the monoidal supercategory of endofunctors of $\modgl$. That is, the objects are functors $F : \modgl \to \modgl$ and the morphisms are natural transformations between functors.  Our convention is that the monoidal product $\odot$ in $\catEnd(\modgl)$ is the left-to-right composition of functors.  That is, if $F, G : \modgl \to \modgl$ are objects in $\catEnd(\modgl)$, then their monoidal product $F \odot G$ is the endofunctor given by first applying $F$ and then applying $G$.   We refer to \cite[Example 1.5]{BE} for details on these notions in the graded setting, noting their convention is for the monoidal product to be given by the right-to-left composition of functors.

For each $n\geq 1$ there is a monoidal functor $\Phi : \modgl \to \catEnd(\modgl)$ defined by ``tensoring on the right'' as follows.  For a right $\gl_{n}(A)$-supermoduile $N$, $\Phi(N) = \Phi_N$  is the functor $- \otimes N : \modgl \to \modgl$ On morphisms, a $\gl_n(A)$-module homomorphism $f : N \to N'$ goes to the natural transformation given by $\Phi(f)_M := \id_M \otimes f : M \otimes N \to M \otimes N'$ for every right $\gl_n(A)$-module $M$.   Wanting $\Phi$ to be monoidal explains our convention for composing endofuctors left-to-right; namely, given our choices, we have
$$
(- \otimes M) \odot (- \otimes N) = - \otimes \left(  M \otimes N\right).
$$
Composing the monoidal functor $\Phi$ with the functor $G_n : \WebAaI \to \modgl$ from \cref{Gthm}  yields the following result.
 
  \begin{lemma} \label{L:finitewebfunctor} For every $n \geq 1$ there is a monoidal superfunctor
\begin{align*}
\widehat{G}_n := \Phi \circ G_{n}: \WebAaI \to \catEnd(\modgl)
\end{align*}
given on generating objects by \(\widehat{G}_n\left(i ^{(x)} \right) = - \otimes S^x_iV_n\), and on generating morphisms by
\begin{multicols}{2}
\begin{align*}
\hackcenter{
\begin{tikzpicture}[scale=.8]
  \draw[ultra thick,blue] (0,0)--(0,0.2) .. controls ++(0,0.35) and ++(0,-0.35) .. (-0.4,0.9)--(-0.4,1);
  \draw[ultra thick,blue] (0,0)--(0,0.2) .. controls ++(0,0.35) and ++(0,-0.35) .. (0.4,0.9)--(0.4,1);
      \node[above] at (-0.4,1) {$ \scriptstyle i^{\scriptstyle (x)}$};
      \node[above] at (0.4,1) {$ \scriptstyle i^{\scriptstyle (y)}$};
      \node[below] at (0,0) {$ \scriptstyle i^{\scriptstyle (x+y)} $};
\end{tikzpicture}}
&\mapsto - \otimes
\textup{spl}_{i^x, i^y} \\
\hackcenter{
\begin{tikzpicture}[scale=.8]
  \draw[ultra thick,red] (0.4,0)--(0.4,0.1) .. controls ++(0,0.35) and ++(0,-0.35) .. (-0.4,0.9)--(-0.4,1);
  \draw[ultra thick,blue] (-0.4,0)--(-0.4,0.1) .. controls ++(0,0.35) and ++(0,-0.35) .. (0.4,0.9)--(0.4,1);
      \node[above] at (-0.4,1) {$ \scriptstyle j^{ \scriptstyle (y)}$};
      \node[above] at (0.4,1) {$ \scriptstyle i^{ \scriptstyle (x)}$};
       \node[below] at (-0.4,0) {$ \scriptstyle i^{ \scriptstyle (x)}$};
      \node[below] at (0.4,0) {$ \scriptstyle j^{ \scriptstyle (y)}$};
\end{tikzpicture}}
&\mapsto
- \otimes \tau_{i^x, j^y},
\end{align*}
\columnbreak
\begin{align*}
\hackcenter{
\begin{tikzpicture}[scale=.8]
  \draw[ultra thick,blue ] (-0.4,0)--(-0.4,0.1) .. controls ++(0,0.35) and ++(0,-0.35) .. (0,0.8)--(0,1);
\draw[ultra thick, blue] (0.4,0)--(0.4,0.1) .. controls ++(0,0.35) and ++(0,-0.35) .. (0,0.8)--(0,1);
      \node[below] at (-0.4,0) {$ \scriptstyle i^{ \scriptstyle (x)}$};
      \node[below] at (0.4,0) {$ \scriptstyle i^{ \scriptstyle (y)}$};
      \node[above] at (0,1) {$ \scriptstyle i^{ \scriptstyle (x+y)}$};
\end{tikzpicture}}
&\mapsto
- \otimes \textup{mer}_{i^x,i^y},\\
\hackcenter{
\begin{tikzpicture}[scale=.8]
  \draw[ultra thick, blue] (0,0)--(0,0.5);
   \draw[ultra thick, red] (0,0.5)--(0,1);
   \draw[thick, fill=yellow]  (0,0.5) circle (7pt);
    \node at (0,0.5) {$ \scriptstyle f$};
     \node[below] at (0,0) {$ \scriptstyle i^{ \scriptstyle (x)}$};
      \node[above] at (0,1) {$ \scriptstyle j^{ \scriptstyle (x)}$};
\end{tikzpicture}}
&\mapsto - \otimes L_x^f,
\end{align*}
\end{multicols}
\noindent for all \(x,y \in \Z_{\geq 0}\), \(i,j \in I\), and \(f \in jA^{(x)}i\).
\end{lemma}

\section{Casimir natural transformations and the defining representation of \(\AffWebAaI\)}\label{CasSec}

The goal of this section is to construct a representation of the affine web category $\AffWebAaI$ defined over the integral domain $\k$ to the supercategory of endofunctors for the general linear Lie superalgebra. 

\subsection{The thin Casimir natural transformation}  \label{SS:ThinCasimir}
Fix $n \geq 1$.  For $i \in I$,  $\Phi({}_iV) = - \otimes {}_iV$ defines an endofunctor of $\modgl$.  The goal of this section is to define an (even) natural transformation $C^{i} :  - \otimes {}_i V  \to  - \otimes {}_{\nakayamai}V_{n}$ for every $i \in I$ using the Casimir elements from \cref{SS:Casimir-elements-in-A}.  

Let \(M\) be a right \(\gl_n(A)\)-module.  Define a $\k$-linear map $C_{M}: M \otimes V \to M \otimes V$ by
\begin{equation}\label{E:thincasimirformula}
C_M(m \otimes v^\alpha_t) = \sum_{\substack{b \in \BasisB \\ r \in [1,n]}} (-1)^{\bar{\alpha} \bar{b}}mE_{r,t}^{b \alpha} \otimes v_r^{b^{\vee}}
\end{equation}
for homogeneous \(m \in M\) and $\alpha \in A$, and $t \in [1,n]$. While \(\BasisB\) and the sum are potentially infinite, local finiteness ensures all but finitely many terms in the sum are zero.

\begin{remark} In contemplating the above formula, readers may find it helpful to note that, after applying a teleporter-like result as discussed in \cref{transrem} below, the linear map $C_M$ factors as the composition $$C_M = (1 \otimes \psi^{-1}) \circ \Omega : M \otimes V \to M \otimes V,$$
where:
\begin{itemize}
\item the $\k$-linear map $\psi^{-1} : V \to V$ is given by applying the inverse Nakayama automorphism to the entries of the row vectors. That is,  $\psi^{-1}(v_r^\alpha) = v_r^{\psi^{-1}(\alpha)}$, and 
\item the $\k$-linear map $\Omega: M \otimes V \to M \otimes V$ is defined as right multiplication by the element
$$\sum_{b \in \BasisB , r,s \in [1,n]}  E_{r,s}^{b} \otimes E_{s,r}^{\psi(b)^{\vee}},$$
where $\psi(b)^{\vee} \in A$ is the dual to $\psi(b)$ with respect to the basis $\psi(\BasisB) = \{\psi(c) \mid c \in \BasisB\}$.
\end{itemize}
In particular, when $\psi$ is the identity map, the linear map $C_M$ is equal to right multiplication by $\sum_{b \in \BasisB, r,s \in [1,n]} E_{r,s}^{b} \otimes E_{s,r}^{{b}^{\vee}}$.
\end{remark}
\begin{remark}\label{transrem}
In what follows, we will freely (and without further remark) apply teleporter results like Lemma~\ref{translem} more widely and involving other algebraic objects, whenever their use is justified via a map from \(A \otimes A\) into the object. For instance, if \(M, N\) are \(\gl_n(A)\)-modules, with \(f \in {}_j A_ i \), then we have
\begin{align*}
\sum_{b \in {}_i\BasisB_k} (-1)^{\bar b \bar n + \bar f \bar n}mE_{r,s}^{fb} \otimes nE_{t,u}^{b^\vee} = \sum_{b \in {}_i\BasisB_k} (-1)^{\bar b \bar n}mE_{r,s}^{b} \otimes nE_{t,u}^{b^\vee f},
\end{align*}
and
\begin{align*}
\sum_{b \in {}_k\BasisB_{j}} (-1)^{\bar b \bar n + \bar f \bar n}mE_{r,s}^{bf} \otimes nE_{t,u}^{b^\vee} = \sum_{b \in {}_k\BasisB_{j}} (-1)^{\bar b \bar n}mE_{r,s}^{b} \otimes nE_{t,u}^{\psi^{-1}(f) b^\vee }.
\end{align*}
Indeed, note that there is a linear map
\begin{align*}
h: A \otimes A \to M  \otimes N, \qquad g \otimes h \mapsto (-1)^{\bar g \bar n}mE_{r,s}^g \otimes nE_{t,u}^h.
\end{align*}
By \cref{translem} it follows that
\begin{align*}
h\left(  \sum_{b \in {}_i\BasisB_k} fb \otimes b^\vee\right) = h\left( \sum_{b \in {}_i\BasisB_k} b \otimes b^\vee f \right), 
\end{align*}
for all \(k \in I\), which yields the claimed result.
\end{remark}

\begin{lemma}
For all $\gl_{n}(A)$-modules $M$ the map \(C_M\) is a homomorphism of right \(\gl_n(A)\)-modules. 
\end{lemma}
\begin{proof}
Given homogeneous $m \in M$, $\alpha, f \in A$, and $u,w \in [1,n]$, one has
\begin{align*}
C_M((m \otimes v_t^\alpha) \cdot E^f_{u,w}) &= C_M((-1)^{\bar \alpha \bar f}mE_{u,w}^f  \otimes v_t^\alpha + \delta_{t,u}m \otimes v^{\alpha f}_w)\\
&=
\sum_{ \substack{ b \in \BasisB \\ r \in [1,n]}} 
(-1)^{\bar \alpha \bar f + \bar \alpha \bar b}
mE^f_{u,w}E^{b\alpha}_{r,t} \otimes v_r^{b^\vee}
+
\sum_{ \substack{ b \in \BasisB \\ r \in [1,n]}} \delta_{t,u}
(-1)^{\bar \alpha \bar b + \bar f \bar b}
 m E^{b\alpha f}_{r,w} \otimes v_r^{b^\vee}.
\end{align*}
On the other hand,
\begin{align*}
C_M(m \otimes v_t^\alpha) \cdot E_{u,w}^f
&=
\left( \sum_{\substack{b \in \BasisB \\ r \in [1,n]}} (-1)^{\bar \alpha \bar b}mE_{r,t}^{b \alpha} \otimes v_r^{b^\vee}\right) \cdot E_{u,w}^f\\
&=\sum_{\substack{b \in \BasisB \\ r \in [1,n]}} (-1)^{\bar \alpha \bar b + \bar b \bar f}mE_{r,t}^{b \alpha}E_{u,w}^f \otimes v_r^{b^\vee}
+
\sum_{\substack{b \in \BasisB \\ r \in [1,n]}} \delta_{r,u} (-1)^{\bar \alpha \bar b}mE_{u,t}^{b \alpha} \otimes v_w^{b^\vee f}\\
&=
\sum_{ \substack{ b \in \BasisB \\ r \in [1,n]}} 
(-1)^{ \bar \alpha \bar f + \bar \alpha \bar b}
mE^f_{u,w}E^{b\alpha}_{r,t} \otimes v_r^{b^\vee}
+
\sum_{ \substack{ b \in \BasisB \\ r \in [1,n]}} \delta_{t,u}
(-1)^{\bar \alpha \bar b + \bar f \bar b}
 m E^{b\alpha f}_{r,w} \otimes v_r^{b^\vee}\\
&\hspace{5mm}+\sum_{\substack{b \in \BasisB \\ r \in [1,n]}} \delta_{r,w} (-1)^{\bar \alpha \bar b   +\bar \alpha \bar f +  \color{black} 1} mE_{u,t}^{fb \alpha} \otimes v_r^{b^\vee} +
\sum_{\substack{b \in \BasisB \\ r \in [1,n]}} \delta_{r,u} (-1)^{\bar \alpha \bar b}mE_{u,t}^{b \alpha} \otimes v_w^{b^\vee f}
\end{align*}
where  the last equality follows from the commutator formula for $\gl_{n}(A)$.  From the above displayed equation, we see that $C_M(m \otimes v_t^\alpha) \cdot E_{u,w}^f$ is the sum of four expressions.  The first two expressions precisely match the calculation of $C_M(m \otimes v_t^\alpha \cdot E_{u,w}^f)$ and, hence, the difference \(C_M(m \otimes v_t^\alpha) \cdot E_{u,w}^f - C_M((m \otimes v_t^\alpha) \cdot E^f_{u,w}) \) equals the sum of the last two expressions.  By removing terms where the Kronecker delta function evaluates to zero this difference reduces to
\begin{equation}\label{leftover1}
\sum_{b \in \BasisB} (-1)^{\bar \alpha \bar b   +\bar \alpha \bar f + 1} mE_{u,t}^{fb \alpha} \otimes v_w^{b^\vee} +
\sum_{b \in \BasisB} (-1)^{\bar \alpha \bar b} mE_{u,t}^{b \alpha} \otimes v_w^{b^\vee f}.
\end{equation}
By a teleporter argument along the lines of \cref{transrem}, we have
\begin{equation*}
\sum_{b \in \BasisB} (-1)^{\bar \alpha \bar b   +\bar \alpha \bar f + 1} mE_{u,t}^{fb \alpha} \otimes v_w^{b^\vee}  = 
\sum_{b \in \BasisB} (-1)^{\bar \alpha \bar b   + 1} mE_{u,t}^{b \alpha} \otimes v_w^{b^\vee f}.
\end{equation*}
Therefore, \cref{leftover1} equals zero and this implies the result.
\end{proof}

\begin{lemma} \label{L:coloredthincasimir} Suppose $M$ is a $\gl_n(A)$-supermodule.  For any homogeneous $f \in A$, let $L_M^{f} : M \otimes V \to M \otimes V$ be the $\gl_n(A)$-supermodule homomorphism given by:
$$
L_M^f(m \otimes v_t^{\alpha} ) = (-1)^{\bar m \bar f} m \otimes v_t^{f \alpha},
$$
for $m \in M$, $1 \leq t \leq n$, and $\alpha \in A$.  Then,
\[
 C_M \circ L_M^f  =  L_M^{{}_1 f} \circ C_m.
\]

In particular, for every $i \in I$ the restriction of $C_M$ defines a $\gl_n(A)$-homomorphism
\[
C_M^i : M \otimes {}_i V \to M \otimes {}_{\nakayamai} V.
\]
\end{lemma}

\begin{proof}   For homogeneous $m \in M$, $v_t^\alpha \in V$, and $f \in A$, we have:
\begin{align*}
(L^{{}_1 f}_M \circ C_M )(m \otimes v_t^{\alpha}) &=L^{{}_1 f}_M\left( \sum_{\substack{b \in \BasisB \\ r \in [1,n]}} (-1)^{\bar \alpha \bar b}mE_{r,t}^{b \alpha} \otimes v_r^{b^\vee}\right) \\
& = \sum_{\substack{b \in \BasisB \\ r \in [1,n]}} (-1)^{\bar \alpha \bar b + \bar f \bar m + \bar f \, \bar b + \bar \alpha \bar f}mE_{r,t}^{b \alpha} \otimes v_r^{ {{}_1 f} b^\vee} \\
&= \sum_{\substack{b \in \BasisB \\ r \in [1,n]}} (-1)^{\bar \alpha \bar b + \bar f \bar m + \bar f \,  \bar b + \bar \alpha \bar f + \bar \alpha \bar f}mE_{r,t}^{b f \alpha} \otimes v_r^{b^\vee} \\
&= \sum_{\substack{b \in \BasisB \\ r \in [1,n]}} (-1)^{(\bar \alpha + \bar f )  \bar b + \bar f \bar m }mE_{r,t}^{b f \alpha} \otimes v_r^{b^\vee} \\
&= C_M((-1)^{\overline{f}  \bar m} m \otimes v_t^{f \alpha}) \\
&= (C_M \circ L^{f}_M)(m \otimes v_t^\alpha).
\end{align*}
The second equality follows from an application of the transporter trick as described in \cref{transrem}.   This proves the main claim of the lemma.

For the next claim, fix $ i \in I$.  Since the Nakayama automorphism fixes the elements of $I$, the previous calculation shows that
$$
C_{M} \circ L^{i}_{M}  = L^{\nakayamai}_{M} \circ C_{M}.
$$
However, since $M \otimes {}_i V$ is the image of $ L^i_M$ this equality implies the restriction of the supermodule homomorphism $C_M$ defines an endomorphism of $M \otimes {}_i V$.  
\end{proof}

\begin{proposition} Fix $ i \in I$.  As $M$ varies over all $\gl_n(A)$-supermodules, the collection of maps $C^i_M : M \otimes {}_i V \to M \otimes {}_{\nakayamai}V $ defines the data of a natural transformation 
\[
C^i : - \otimes {}_i V \to - \otimes {}_{\nakayamai }V.
\]
\end{proposition}

\begin{proof}

Assume that $M,N$ are right $\mathfrak{gl}_n(A)$-supermodules and that $\phi : M \to N$ is a $\gl_n(A)$-homomorphism.  We must verify that the following diagram commutes:
\[
\begin{tikzcd}
M \otimes {}_iV \arrow{r}{C^i_M}  \arrow[swap]{d}{\phi \otimes \Id_{{}_{i}V}} & M \otimes {}_{\nakayamai} V \arrow{d}{\phi \otimes \Id_{{}_{i}V}}  \\ N \otimes {}_i V \arrow{r}{C^i_N} & N \otimes {}_{\nakayamai} V
\end{tikzcd} \; .
\]
Given homogeneous $m \in M$ and  $v^{\alpha}_{t} \in  {}_iV$, 
\begin{align*}
\left( \phi \otimes \Id_{{}_{i}V} \circ C^i_M \right)(m \otimes v^{i \alpha}_t) &= \sum_{\substack{b \in \BasisB \\ r \in [1,n]}} (-1)^{\bar \alpha \bar b} \phi(mE_{r,t}^{b  \alpha})  \otimes v_r^{b^\vee}  \\
&= \sum_{\substack{b \in \BasisB \\ r \in [1,n]}} (-1)^{\bar \alpha \bar b}\phi(m)E_{r,t}^{b  \alpha} \otimes v_r^{b^\vee} \\
&= \left( C^i_N \circ \phi \otimes \Id_{{}_{i}V}\right)(m \otimes v_t^{\alpha}).
\end{align*}
\end{proof}

\subsection{A representation of the thin dot category \texorpdfstring{$\ThinAffWebAaI$}{ThinAffWebAaI}}

The following lemma establishes a representation for the thin dot category $\ThinAffWebAaI$ into the endofunctor category $\catEnd(\modgl)$.  The remainder of this section will be devoted to extending this representation to one of the full affine web category $\AffWebAaI$.

\begin{proposition} \label{L:hatH} For every $n \geq 1$, there is a unique monoidal superfunctor
\[
\widehat{H}_{n} : \ThinAffWebAaI \to \catEnd(\modgl)
\] given on generating objects by \(\widehat{H}_{n}\left( i^{(x)}\right) = - \otimes S^x_iV_n\) and on generating morphisms by
\begin{equation*}
\begin{gathered}
\hackcenter{
\begin{tikzpicture}[scale=.8]
  \draw[ultra thick,blue] (0,0)--(0,0.2) .. controls ++(0,0.35) and ++(0,-0.35) .. (-0.4,0.9)--(-0.4,1);
  \draw[ultra thick,blue] (0,0)--(0,0.2) .. controls ++(0,0.35) and ++(0,-0.35) .. (0.4,0.9)--(0.4,1);
      \node[above] at (-0.4,1) {$ \scriptstyle i^{\scriptstyle (x)}$};
      \node[above] at (0.4,1) {$ \scriptstyle i^{\scriptstyle (y)}$};
      \node[below] at (0,0) {$ \scriptstyle i^{\scriptstyle (x+y)} $};
\end{tikzpicture}}
\mapsto - \otimes
 \textup{spl}_{i^x,i^y},
\qquad
\hackcenter{
\begin{tikzpicture}[scale=.8]
  \draw[ultra thick,blue ] (-0.4,0)--(-0.4,0.1) .. controls ++(0,0.35) and ++(0,-0.35) .. (0,0.8)--(0,1);
\draw[ultra thick, blue] (0.4,0)--(0.4,0.1) .. controls ++(0,0.35) and ++(0,-0.35) .. (0,0.8)--(0,1);
      \node[below] at (-0.4,0) {$ \scriptstyle i^{ \scriptstyle (x)}$};
      \node[below] at (0.4,0) {$ \scriptstyle i^{ \scriptstyle (y)}$};
      \node[above] at (0,1) {$ \scriptstyle i^{ \scriptstyle (x+y)}$};
\end{tikzpicture}}
\mapsto
- \otimes \textup{mer}_{i^x,i^y}
\qquad
\hackcenter{
\begin{tikzpicture}[scale=.8]
  \draw[ultra thick,red] (0.4,0)--(0.4,0.1) .. controls ++(0,0.35) and ++(0,-0.35) .. (-0.4,0.9)--(-0.4,1);
  \draw[ultra thick,blue] (-0.4,0)--(-0.4,0.1) .. controls ++(0,0.35) and ++(0,-0.35) .. (0.4,0.9)--(0.4,1);
      \node[above] at (-0.4,1) {$ \scriptstyle j^{ \scriptstyle (y)}$};
      \node[above] at (0.4,1) {$ \scriptstyle i^{ \scriptstyle (x)}$};
       \node[below] at (-0.4,0) {$ \scriptstyle i^{ \scriptstyle (x)}$};
      \node[below] at (0.4,0) {$ \scriptstyle j^{ \scriptstyle (y)}$};
\end{tikzpicture}}
\mapsto
- \otimes \tau_{i^x, i^y},
\\
\hackcenter{
\begin{tikzpicture}[scale=.8]
  \draw[ultra thick, blue] (0,0)--(0,0.5);
   \draw[ultra thick, red] (0,0.5)--(0,1);
   \draw[thick, fill=yellow]  (0,0.5) circle (7pt);
    \node at (0,0.5) {$ \scriptstyle f$};
     \node[below] at (0,0) {$ \scriptstyle i^{ \scriptstyle (x)}$};
      \node[above] at (0,1) {$ \scriptstyle j^{ \scriptstyle (x)}$};
\end{tikzpicture}}
\mapsto - \otimes L_x^f,
\qquad
\hackcenter{
\begin{tikzpicture}[scale=.8]
  \draw[ultra thick, blue] (0,0)--(0,0.5);
   \draw[ultra thick, blue] (0,0.5)--(0,1);
   \draw[thick, fill=black]  (0,0.5) circle (5pt);
     \node[below] at (0,0) {$ \scriptstyle i^{ \scriptstyle (1)}$};
      \node[above] at (0,1) {$ \scriptstyle \nakayamai^{ \scriptstyle (1)}$};
\end{tikzpicture}}
\mapsto  C^{i},
\end{gathered}
\end{equation*}
for all \(x,y \in \Z_{\geq 0}\), \(i,j \in I\), and \(f \in jA^{(x)}i\).
\end{proposition}

\begin{proof}
It suffices to check that the images of the generating morphisms of $\ThinAffWebAaI$ satisfy the defining relations, namely the  webs relations given in  \cref{AssocRel} through \cref{AaIntertwine}, and the thin affine dot relations given in \cref{ThinAffDotRel1}.   The fact that the non-affine web relations hold follows immediately from \cref{L:finitewebfunctor}.   That the relation on the left of \cref{ThinAffDotRel1} holds follows from \cref{L:coloredthincasimir}.

It remains to check the relation on the right of \cref{ThinAffDotRel1}.  For a right $\mathfrak{gl}_n(A)$-supermodule, $M$, the supernatural transformation
\[
\widehat{H}_n \left( \hackcenter{
\begin{tikzpicture}[scale=0.8]
  \draw[ultra thick,red] (0.4,-0.4)--(0.4,0.1) .. controls ++(0,0.35) and ++(0,-0.35) .. (-0.4,0.9)--(-0.4,1.4);
  \draw[ultra thick,blue] (-0.4,-0.4)--(-0.4,0.1) .. controls ++(0,0.35) and ++(0,-0.35) .. (0.4,0.9)--(0.4,1.4);
      \node[above] at (-0.4,1.4) {$ \scriptstyle j^{ \scriptstyle (1)}$};
      \node[above] at (0.4,1.4) {$ \scriptstyle \nakayamai^{ \scriptstyle (1)}$};
       \node[below] at (-0.4,-0.4) {$ \scriptstyle i^{ \scriptstyle (1)}$};
      \node[below] at (0.4,-0.4) {$ \scriptstyle j^{ \scriptstyle (1)}$};
                  \draw[thick, fill=black]  (-0.4,0) circle (5pt);
\end{tikzpicture}}
\right)
\]
assigns the supermodule homomorphism
\[
\id_M \otimes \tau_{{}_{\nakayamai}V, {}_{j}V} \circ C^i \otimes \id_{{}_j V} : M \otimes {}_i V \otimes {}_j V \to M \otimes {}_{j} V \otimes {}_{\nakayamai} V
\]
For $s,t \in [1,n]$ and homogeneous $m \in M$, $\alpha \in iA$, and $\beta \in jA$ applying this homomorphism yields
\begin{align*}
\left( \id_M  \otimes \tau_{{}_{\nakayamai}V, {}_{j}V} \circ C^i \otimes \id_{{}_j V}\right)\left(  m \otimes v_s^\alpha \otimes v_t^\beta\right)
&= \id_M  \otimes \tau_{{}_{\nakayamai}V, {}_{j}V}  \left( \sum_{\substack{b \in \BasisB \\ r \in [1,n]}} (-1)^{\bar \alpha \bar b}mE_{r,s}^{b \alpha} \otimes v_r^{b^\vee} \otimes v_t^{\beta} \right)\\
&=  \sum_{\substack{b \in \BasisB \\ r \in [1,n]}} (-1)^{\bar \alpha \bar b + \bar b \bar \beta}mE_{r,s}^{b \alpha} \otimes  v_t^{\beta} \otimes v_r^{b^\vee}.
\end{align*}

On the other hand, the supernatural transformation
\[
\widehat{H}_n \left(
\hackcenter{
\begin{tikzpicture}[scale=0.8]
  \draw[ultra thick,red] (0.4,-0.4)--(0.4,0.1) .. controls ++(0,0.35) and ++(0,-0.35) .. (-0.4,0.9)--(-0.4,1.4);
  \draw[ultra thick,blue] (-0.4,-0.4)--(-0.4,0.1) .. controls ++(0,0.35) and ++(0,-0.35) .. (0.4,0.9)--(0.4,1.4);
   \node[above] at (-0.4,1.4) {$ \scriptstyle j^{ \scriptstyle (1)}$};
      \node[above] at (0.4,1.4) {$ \scriptstyle \nakayamai^{ \scriptstyle (1)}$};
       \node[below] at (-0.4,-0.4) {$ \scriptstyle i^{ \scriptstyle (1)}$};
      \node[below] at (0.4,-0.4) {$ \scriptstyle j^{ \scriptstyle (1)}$};
            \draw[thick, fill=black]  (0.4,1) circle (5pt);
\end{tikzpicture}} \right)
\]
assigns to $M$ the supermodule homomorphism
\(
C^{i}_{M \otimes \, {}_j V} \circ \Id_{M} \otimes  \tau_{{}_{\nakayamai}V, {}_{j}V} : M \otimes {}_i V \otimes {}_j V \to M \otimes {}_{j} V \otimes {}_{\nakayamai} V.
\)
Computing this on the same input yields:
\begin{align*}
C^{i}_{M \otimes {}_j V} \circ \tau_{i^1,j^1} \left( m \otimes v_s^\alpha \otimes v_t^\beta\right) &= (-1)^{\bar \alpha \bar \beta} C^i_{M \otimes {}_j V} \left( m \otimes v_t^\beta \otimes v_s^{\alpha}\right) \\
&=  \sum_{\substack{b \in B \\ r \in [1,n]}} (-1)^{\bar \alpha \bar b + \bar \alpha \bar \beta}(m \otimes v_t^{\beta}) E_{rs}^{b \alpha} \otimes v_r^{b ^\vee} \\
&= \sum_{\substack{b \in B \\ r \in [1,n]}} (-1)^{\bar \alpha \bar b + \bar b \bar \beta} mE_{rs}^{b \alpha}  \otimes v_t^{\beta} \otimes v_r^{b ^\vee}
 +  \sum_{\substack{b \in B \\ r \in [1,n]}} (-1)^{\bar \alpha \bar b  + \bar \alpha \bar \beta}m \otimes v_t^{\beta} E_{rs}^{b \alpha} \otimes v_r^{b ^\vee}
\end{align*}
Taking the difference yields:
\begin{align*}
- \sum_{\substack{b \in B \\ r \in [1,n]}} (-1)^{\bar \alpha \bar b  + \bar \alpha \bar \beta}m \otimes v_t^{\beta} E_{rs}^{b \alpha} \otimes v_r^{b ^\vee} 
&= - \sum_{b \in B } (-1)^{\bar \alpha \bar b  + \bar \alpha \bar \beta}m \otimes v_s^{\beta b \alpha} \otimes v_t^{b ^\vee} \\
&= - \sum_{b \in B } (-1)^{\bar \alpha \bar b }m \otimes v_s^{ b \alpha} \otimes v_t^{b^\vee \beta}.
\end{align*}
Another direct calculation shows this equals the result of evaluating the homomorphism assigned to $M$ by the supernatural transformation
\[
\widehat{H}_n \left(
-
\hackcenter{
\begin{tikzpicture}[scale=0.8]
  \draw[ultra thick,blue] (-0.4,-0.4)--(-0.4,0.5);
   \draw[ultra thick,red] (0.4,-0.4)--(0.4,0.5);
  \draw[ultra thick,blue] (0.4,0.5)--(0.4,1.4);
    \draw[ultra thick,red] (-0.4,0.5)--(-0.4,1.4);
   \node[above] at (-0.4,1.4) {$ \scriptstyle j^{ \scriptstyle (1)}$};
      \node[above] at (0.4,1.4) {$ \scriptstyle \nakayamai^{ \scriptstyle (1)}$};
       \node[below] at (-0.4,-0.4) {$ \scriptstyle i^{ \scriptstyle (1)}$};
      \node[below] at (0.4,-0.4) {$ \scriptstyle j^{ \scriptstyle (1)}$};
        \draw[thick, black, fill=black]  (-0.4,0.5) circle (6pt);
        \draw[thick,black, fill=black]  (0.4,0.5) circle (6pt);
                   \draw[decorate, decoration={snake,segment length=5pt, amplitude=1pt}, line width=1mm, black] (-0.4,0.5)--(0.4,0.5);
      \draw[thick, lime, fill=lime]  (-0.4,0.5) circle (5pt);
        \draw[thick, lime, fill=lime]  (0.4,0.5) circle (5pt);
      \draw[decorate, decoration={snake,segment length=5pt, amplitude=1pt}, line width=0.5mm, lime] (-0.4,0.5)--(0.4,0.5);
\end{tikzpicture}} \right)
\]
on $m \otimes v_s^\alpha \otimes v_t^{\beta}$.   This verifies the relation, and the lemma follows.
 \end{proof}

\subsection{Extension of scalars and torsion}\label{SS:Extension-of-scalars}
Let $\KK$ be the field of fractions for $\k$.  For a $\k$-linear supercategory $\catC$ write ${}_{\mathbb{K}} \catC$ for the $\mathbb{K}$-linear category which has the same objects as $\catC$ and whose morphisms  are given by
\[
\Hom_{{}_\mathbb{K} \catC}(X,Y) := \KK \otimes_{\k} \Hom_{\catC}(X,Y).
\]
Composition of morphisms in $\catC$ is extended linearly to define composition in ${}_\mathbb{K} \catC$.  It is straightforward to verify that ${}_\mathbb{K} \catC$ is a $\mathbb{K}$-linear supercategory. 
There is a canonical functor 
\begin{equation}\label{E:Extension-of-scalars-functor}
\xi = \xi_{\catC}: \catC \to {}_{\mathbb{K}} \catC
\end{equation}
which is the identity on objects and  $ \alpha \mapsto \Id_{\KK} \otimes \alpha$ on morphisms. If $\catD$ is a $\KK$-linear category, then every functor of $\k$-linear categories $\catC \to \catD$ extends uniquely to a functor of $\KK$-linear categories ${}_{\mathbb{K}} \catC \to \catD$.  Similarly, a functor of $\k$-linear categories $F: \catC \to \catD$ may be extended linearly to a $\KK$-linear functor ${}_{\KK}F:{}_{\KK}\catC \to {}_{\KK}\catD$.

\begin{remark}  Given a Frobenius great pair $\calA$ defined over $\k$, write $\KK \otimes \mathcal{A}$ for the Frobenius great pair defined over $\KK$ by extending scalars. There are obvious functors between ${}_{\KK}\ThinAffWebAaI$ and  $\mathbf{Web}^{\textup{aff},1}_{\KK \otimes \calA}$, and between ${}_{\KK}\AffWebAaI$ and $\mathbf{Web}^{\textup{aff}}_{\KK \otimes \calA}$. Each pair of functors define mutually inverse isomorphisms of monoidal $\KK$-linear categories.  We identify these categories via these functors.
\end{remark}

Recall that if $M$ is a $\k$-module, then $m \in M$ is a \emph{torsion element} if there is some non-zero $c \in \k$ with $c m = 0$.  The set of torsion elements of $M$ is the \emph{torsion submodule} of $M$, denoted by $T(M)$.  One says that $M$ is \emph{torsion-free} if $T(M) = 0$.  For example, every free $\k$-module is torsion-free.

\begin{lemma} \label{L:homintofree} Suppose $M$ and $F$ are $\k$-supermodules and assume $F$ is torsion-free.  Then, the $\k$-supermodule of $\k$-linear maps
$
\Hom_{\k}(M,F)
$
is torsion-free.
\end{lemma}

The following lemma is immediate from \cite[Exercise 3.12.iv]{AtiyahMacdonald}.

\begin{lemma} \label{L:torsion} For objects $X$ and $Y$ in a supercategory $\catC$, let 
\[
\xi_{X,Y} : \Hom_{\catC}(X,Y) \to \mathbb{K} \otimes \Hom_{\mathcal C}(X,Y).
\] denote the $\k$-linear map given by the canonical functor $\xi : \mathcal C \to {}_{\mathbb{K}} \mathcal{C}$.
Then the kernel of $\xi_{X,Y}$ is the torsion submodule $T( \Hom_{\catC}(X,Y))$.
\end{lemma}

\begin{lemma} \label{L:torsionfreenattransforms} Suppose $\Psi$ and $\Phi$ are functors in $\catEnd(\modgl)$ where $\Psi$ is exact and $\Phi$ preserves the subcategory of $\gl_n(A)$-supermodules which are free over $\k$.  Then, the space of natural transformations $\Hom_{\catEnd(\modgl)}(\Psi,\Phi)$ is torsion-free as a $\k$-module.
\end{lemma}
\begin{proof} Suppose that $\eta : \Psi \to \Phi$ is a non-zero natural transformation, and that $c \in \k$ is non-zero.  We need to verify that the natural transformation $c \eta$ is non-zero.  

Because $\eta$ is non-zero by assumption, there exists some $\gl_n(A)$-supermodule $M$ for which the associated homomorphism $\eta_M:\Psi(M) \to \Phi (M)$ is not the zero map.  Fix a $\gl_n(A)$-supermodule $F$ which is free over $\k$ and admits an exact sequence:
\[
\begin{tikzcd}
F \arrow{r} & M \arrow{r} & 0.
\end{tikzcd}
\]  That such an $F$ exists follows by taking $F$ to be a free cover of $M$ as a right $U(\gl_n(A))$-supermodule. The PBW theorem guarantees that such an $F$ is free over $\k$.  

Given the assumptions, there is the following commutative diagram with an exact top row:
\[
\begin{tikzcd} 
\Psi(F) \arrow{r} \arrow{d}{\eta_F} & \Psi(M) \arrow{d}{\eta_M} \arrow{r} & 0 \\
\Phi(F) \arrow{r} & \Phi(M).
\end{tikzcd}
\]
This diagram and the fact that $\eta_M$ is non-zero immediately implies $\eta_F$ is also non-zero.

The $\gl_n(A)$-homomorphism $\eta_F$ is $\k$-linear and may be regarded as an element of $\Hom_{\k}(\Psi(F), \Phi(F))$.  However, $F$ is a free $\k$-module which implies $\Phi(F)$ is also $\k$-free and, hence, torsion-free. This, in turn, implies $\Hom_{\k}(\Psi(F), \Phi(F))$ is torsion-free by \cref{L:homintofree}.  Since $\eta_{F}$ is non-zero, this implies $c \eta_F$ is also nonzero and, hence, the natural transformation $c\eta$ is nonzero.\end{proof}
\color{black}

\subsection{The thick Casimir}  The goal of this section is to define a Casimir element for the thick affine dots.

Given a finite subset of positive integers, \(S \subset \Z_{\geq 1}\), and \(t  \leq |S|\), we call an ordered sequence \(\bu = u_1 \cdots u_t\) a \emph{\(t\)-permutation of \(S\)} (or simply \emph{permutation of $S$} if the $t$ is left implicit) if \(u_1, \ldots, u_t\) are distinct elements of \(S\). We write \(\ell(\bu):= t\) for the length of \(\bu\). We call a permutation \(\bu\) \emph{least-last} provided that \(u_j > u_{\ell(\bu)}\) for \(j \in [1, \ell(\bu)]\).
An \emph{ordered \(t\)-partition of \(S\)} is a sequence \(\bu^\bullet = \bu^{(1)} | \bu^{(2)} | \cdots | \bu^{(d)}\) of permutations \(\bu^{(1)}, \ldots, \bu^{(d)}\) of \(S\) such that \(\bu=\bu^{(1)}  \bu^{(2)}  \cdots  \bu^{(d)}\) is a \(t\)-permutation of \(S\). In other words, an ordered \(t\)-partition of \(S\) is a \(t\)-permutation of \(S\) equipped with a choice of grouping. We write \(\ell(\bu^{\bullet}) :=d\) for the number of permutations in \(\bu^\bullet\). Write \(\ord_t(S)\) for the set of ordered \(t\)-partitions of \(S\).  For short we write $\ord_{k}$ for $\ord_{k}([1,k])$ for all $k \geq 1$.

For \(\bu^\bullet \in \ord_t(S)\), we say that \(\bu\) is {\em good} provided that:
\begin{itemize}
\item \(\bu^{(1)}, \ldots, \bu^{(\ell(\bu^{\bullet}))}\) are least-last permutations, and;
\item \(u^{(1)}_{\ell(\bu^{(1)})} < \cdots < u^{(\ell(\bu^\bullet))}_{\ell(\bu^{(\ell(\bu^\bullet))})}\).
\end{itemize}
We write \(\ord_t(S)^\textup{g}\) for the set of good ordered \(t\)-partitions of \(S\) and write \(\ord_{k}^\textup{g}\) for \(\ord_k([1,k])^\textup{g}\). For a given \(t\)-permutation \(\bu\) of \(S\), there is a unique way to group the entries of \(\bu\) to form a good ordered \(t\)-partition of \(S\). Thus there is a bijection between \(\ord_t(S)^\textup{g}\) and \(t\)-permutations of \(S\), and \(|\ord_t(S)^\textup{g}| = |S|!/(|S|-t)!\). 

\begin{example}  Let $S= \left\{1,2,3,5,6,7,11,13,29,74 \right\}$.  Then:
\begin{itemize}
\item $\bu_{1} = 6,29,74$, $\bu_{2} = 5,3,11,7,2,13$, and $\bu_{3} = 29,74,5,6,7,11,2,3,5$ are permutations of $S$ with $\ell(\bu_{1})=3$, $\ell(\bu_{2})=6$, and $\ell(\bu_{3})=9$.  None are least-last.  However, both $29,74,6$ and $74,29,6$ are least-last $3$-partitions of $S$.
\item $\bu_{1}^{\bullet}=2,3|5,7,11|13$, $\bu_{2}^{\bullet}=2|7,5|11,13,3$, and $\bu_{3}^{\bullet}=2|11,13,3|7,5$ are ordered $6$-partitions of $S$ with $\ell(\bu_{k}^{\bullet})=3$ for $k=1,2,3$. $\bu_{1}^{\bullet}$ fails both conditions for being good, $\bu_{2}^{\bullet}$ satisfies the least-last condition but not the second condition, and $\bu_{3}^{\bullet}$ satisfies both.  
\end{itemize}
\end{example}

For fixed $k,n \geq 1$, let \(\bbf = (f_1, \ldots, f_k) \in A^k\), \(\br = (r_1,\ldots, r_k),  \bt = (t_1, \ldots, t_k) \in [1,n]^k\). Let \(\bw=w_{1}\dotsb w_{t}\) be a $t$-partition of $[1,k]$ for some $t \leq k$. Using this data, set
\begin{align*}
E^{\bbf}_{\br, \bt}(\bw) =  \delta_{t_{w_1}, r_{w_2}} \cdots \delta_{t_{w_{\ell(\bw)-1}}, r_{w_{\ell(\bw)}}} E_{r_{w_1}, t_{\ell(\bw)}}^{f_{w_1} \cdots  f_{w_{\ell(\bw)}}}.
\end{align*}
For \(\bu^\bullet = \bu^{(1)} | \cdots | \bu^{(\ell(\bu^\bullet))} \in \ord_{k}\) set \(\bu = u_1 \ldots u_k\) to be the $k$-partition of $[1,k]$ obtained by removing the $|$ symbols. With this data along with $\bbf$, $\br$, and $\bt$ from above, define:
\begin{align*}
E^{\bbf}_{\br, \bt}(\bu^\bullet) := 
(-1)^{\langle \bu ; \bbf \rangle}
\prod_{x = 1}^{\ell(\bu^\bullet)} E^{\bbf}_{\br, \bt}(\bu^{(x)}), 
\end{align*}
where 
\begin{align*}
\langle \bu; \bbf \rangle := \sum_{\substack{1 \leq i< j \leq k \\ u_i > u_j}} \bar f_{u_i} \bar f_{u_j}.
\end{align*}

For \(\bbf = (f_1, \ldots, f_k), \bbg = (g_1, \ldots, g_k) \in A^k\) write \(\bbf \bbg = (f_1g_1, \ldots, f_k g_k) \in A^k\), and set 
\begin{align*}
\langle \bbg, \bbf \rangle = \sum_{1 \leq i \leq j \leq k} \bar f_i \bar g_j.
\end{align*}

\begin{definition}\label{thickcasmap} 
Let \(M\) be a right \(\gl_n(A)\)-supermodule.  For $i \in I$ and $k \geq 1$, define a $\k$-linear map
\begin{align*}
C_M^{i^{(k)}} = M \otimes S^k_i V \to M \otimes S^k_{\nakayamai } V
\end{align*}
by setting
\begin{align*}
m \otimes v_{t_{1}}^{f_{1}}\dotsb v_{t_{k}}^{f_{k}}  \mapsto \sum_{\substack{ \bb \in \BasisB^k \\ \br \in [1,n]^k \\ \bu^{\bullet} \in \ord_k^\textup{g}} }
(-1)^{\langle \bbf, \bb \rangle}m E^{\bb \bbf}_{\br, \bt}(\bu^\bullet) \otimes v_{r_k}^{b_k^\vee} \cdots v_{r_1}^{b_1^\vee}.
\end{align*}
for all homogeneous \(m \in M\) and  \(\bbf = (f_{1}, \dotsc , f_{k}) \in (iA)^k\), and all \(\bt = (t_{1}, \dotsc , t_{k}) \in [1,n]^k\).
\end{definition}

\begin{remark} Observe that $C_M^{i^{(1)}}$ equals $C_M^i$.
\end{remark}

Clearly the definition of the map $C_M^{i^{(k)}}$ is rather complicated and would be difficult to use as given.    The following result will allow us to avoid working with the definition of $C_{M}^{i^{(k)}}$ directly. It can be interpreted as saying that $C_M^{i^{(k)}}$ satisfies a special case of the rightmost relation in \cref{AffDotRel1}.  See \cref{P:Thick-White-Dot-Action}.

\begin{proposition} \label{T:thickcasimirrecursion}
Let \(M\) be a right \(\gl_n(A)\)-supermodule. Then, 
\[
(\id_M \otimes \textup{mer}_{\nakayamai^{k-1},\nakayamai^{1}})\circ C_{M \otimes S_{\nakayamai}^{k-1}V}^{i^{(1)}}\circ (C_M^{i^{(k-1)}} \otimes \id_{S^{1}_{i}V})
= C_M^{i^{(k)}} \circ  (\id_M \otimes \textup{mer}_{i^{k-1}, i^{1}})
\] as maps \(M \otimes S_i^{k-1}V \otimes S_i^1V \to M \otimes S_{\nakayamai}^kV\).
\end{proposition}

\begin{proof}
First we define some combinatorial maps. For \(\bu^\bullet \in \ord_{k-1}^\textup{g}\), and \(w \in [1,k-1]\), let \(\textup{bef}_w(\bu^\bullet) \in \ord_k^\textup{g}\) be defined by inserting \(k\) directly before \(w\) in \(\bu^\bullet\) (in the same subpermutation as \(w\)). Let \(\textup{suf}(\bu^\bullet) = \bu^\bullet | k \in \ord_k^\textup{g}\). It is straightforward to check that
\begin{align}\label{combbij}
\{ \textup{bef}_w(\bu^\bullet) \mid w \in [1,k-1], \bu^\bullet \in \ord_{k-1}^\textup{g}\} \sqcup \{ \textup{suf}(\bu^\bullet) \mid  \bu^\bullet \in \ord_{k-1}^\textup{g}\} = \ord_k^\textup{g}.
\end{align}
Now, let \(m \in M\), \(\bt \in [1,n]^k\), \(\bbf \in (iA)^k\). In what follows we will use primes to indicate the truncation of a \(k\)-tuple to the first \(k-1\) terms, writing \(\bbf' = (f_1, \ldots, f_{k-1})\), and so forth. We have
\begin{align}
&\textup{mer}_M^{\nakayamai^{(k-1)},\psi^{-1}( i)^{(1)}}\circ C_{M \otimes S_{\nakayamai}^{k-1}V}^{i^{(1)}}\circ (C_M^{i^{(k-1)}} \otimes \id) (m \otimes v^{\bbf'}_{\bt'} \otimes v_{t_k}^{f_k})
\notag\\
&\hspace{10mm}=\textup{mer}_M^{\nakayamai^{(k-1)},\psi^{-1}( i)^{(1)}}\circ C_{M \otimes S_{\nakayamai}^{k-1}V}^{i^{(1)}}
\left( 
 \sum_{\substack{ \bb \in \BasisB^{k-1} \\ \br \in [1,n]^{k-1} \\ \bu^\bullet \in \ord_{k-1}^\textup{g}} }
(-1)^{\langle \bbf', \bb \rangle}m E^{\bb \bbf'}_{\br, \bt'}(\bu^\bullet) \otimes v_{r_{k-1}}^{b_{k-1}^\vee} \cdots v_{r_1}^{b_1^\vee} \otimes v_{t_k}^{f_k}
\right)
\notag\\
&\hspace{10mm}=
\textup{mer}_M^{\nakayamai^{(k-1)},\psi^{-1}( i)^{(1)}}
\left( 
 \sum_{\substack{ \bb \in \BasisB^{k} \\ \br \in [1,n]^{k} \\ \bu^\bullet \in \ord_{k-1}^\textup{g}} }
(-1)^{\langle \bbf', \bb' \rangle + \bar b_k \bar f_k} (m E^{\bb' \bbf'}_{\br', \bt'}(\bu^\bullet) \otimes v_{r_{k-1}}^{b_{k-1}^\vee} \cdots v_{r_1}^{b_1^\vee})E_{r_k, t_k}^{b_k f_k}  \otimes v_{r_k}^{b_k^\vee}
\right)
\notag\\
&\hspace{10mm}=
 \sum_{\substack{ \bb \in \BasisB^{k} \\ \br \in [1,n]^{k} \\ \bu^\bullet \in \ord_{k-1}^\textup{g}} }
(-1)^{\langle \bbf, \bb \rangle + \bar b_k (\bar b_1 + \cdots + \bar b_{k-1}) }m E^{\bb' \bbf'}_{\br', \bt'}(\bu^\bullet)E_{r_k, t_k}^{b_k f_k} \otimes v_{r_{k-1}}^{b_{k-1}^\vee} \cdots v_{r_1}^{b_1^\vee}  v_{r_k}^{b_k^\vee}
\notag \\
&\hspace{15mm}+
 \sum_{\substack{ \bb \in \BasisB^{k} \\ \br \in [1,n]^{k} \\ \bu^\bullet \in \ord_{k-1}^\textup{g} \\ w \in [1,k-1]}  }
(-1)^{\langle \bbf', \bb' \rangle + (\bar b_k + \bar f_k)( \bar b_1 + \cdots + \bar b_{w-1}) + \bar b_k \bar f_k}
\delta_{r_w, r_k} m E^{\bb' \bbf'}_{\br', \bt'}(\bu^\bullet) \otimes v_{r_{k-1}}^{b_{k-1}^\vee} 
\cdots    
v_{t_k}^{b_w^\vee b_k f_k}
\cdots
v_{r_1}^{b_1^\vee}  v_{r_k}^{b_k^\vee}
\notag \\
&\hspace{10mm}=
 \sum_{\substack{ \bb \in \BasisB^{k} \\ \br \in [1,n]^{k} \\ \bu^\bullet \in \ord_{k-1}^\textup{g}} }
(-1)^{\langle \bbf, \bb \rangle}m E^{\bb \bbf}_{\br, \bt}(\textup{suf}(\bu^\bullet)) \otimes v_{r_{k}}^{b_{k}^\vee} \cdots v_{r_1}^{b_1^\vee} 
\notag \\
&\hspace{15mm}+
 \sum_{\substack{ \bb \in \BasisB^{k} \\ \br \in [1,n]^{k} \\ \bu^\bullet \in \ord_{k-1}^\textup{g}\\ w \in [1,k-1]}  }
(-1)^{\langle \bbf, \bb \rangle}
m E^{\bb \bbf}_{\br, \bt}(\textup{bef}_w(\bu^\bullet)) \otimes v_{r_{k}}^{b_{k}^\vee} 
\cdots    
v_{r_1}^{b_1^\vee}
\label{transport5}\\
&\hspace{10mm}=
 \sum_{\substack{ \bb \in \BasisB^{k} \\ \br \in [1,n]^{k} \\ \bv^\bullet \in \ord_{k}^\textup{g}} }(-1)^{\langle \bbf, \bb \rangle}
m E^{\bb \bbf}_{\br, \bt}(\bv^\bullet) \otimes v_{r_{k}}^{b_{k}^\vee} 
\cdots    
v_{r_1}^{b_1^\vee}
\label{usecombbij}\\
&\hspace{10mm}=
 C_M^{i^{(k)}} \circ \;\textup{mer}_M^{i^{(k-1)}, i^{(1)}}(m \otimes v^{\bbf'}_{\bt'} \otimes v_{t_k}^{f_k}),
 \notag
\end{align}
where \cref{transport5} follows by an application of the transporter rule, and \cref{usecombbij} follows by \cref{combbij}, completing the proof.
\end{proof}

\begin{theorem} \label{T:thickcasisnaturaltrasform} The maps $C_{M}^{i^{(k)}} : M \otimes S^k_i V \to M \otimes S^k_{\nakayamai} V$ define an even supernatural transformation, 
\[
C^{i^{(k)}} : - \otimes S^k_{i}V \to  - \otimes S^k_{\nakayamai}V,
\] in  $\catEnd(\modgl)$.
\end{theorem}

\begin{proof}  Since $C_{M}^{i^{(k)}}$ is only known to be $\k$-linear, we must demonstrate that it is a $\gl_n(A)$-supermodule homomorphism and that this family of maps satisfies the appropriate naturality condition.   In fact, it turns out to be easier to first verify naturality.

Let $M, N$ be two right $\gl_n(A)$-supermodules and let $\varphi : M \to N$ be a $\gl_n(A)$-supermodule homomorphism.  Verifying the naturality condition amounts to demonstrating the commutativity of the following the square of $\k$-linear maps:
\begin{equation} \label{naturalitysquare}
\begin{tikzcd}
M \otimes S^k_i V \arrow{r}{\varphi \otimes \id }  \arrow{d}[swap]{C_M^{i^{(k)}}} & N \otimes S^k_i V \arrow{d}{C_{N}^{i^{(k)}}}\\
M \otimes S^k_{\nakayamai}(V) \arrow{r}{\varphi \otimes \id} & N \otimes S^k_{\nakayamai} V
\end{tikzcd}
\end{equation} For homogeneous $m\in M$ and $\bbf = (f_{1}, \dotsc , f_{k}) \in ({}_{i}A)^{k}$, and $\bt = (t_{1}, \dotsc , t_{k}) \in [1,n]^{k}$, write $m \otimes v_{\bt}^{\bbf} = m \otimes  v_{t_{1}}^{f_{1}} \dotsb v_{t_{k}}^{f_{k}}  \in M \otimes S^k_i V$.
For such a vector, traveling along the top and then rightmost edge of the square results in:
$$
 \sum_{\substack{ \bb \in \BasisB^k \\ \br \in [1,n]^k \\ \bu \in \ord_k^\textup{g}} }
(-1)^{\langle \bbf, \bb \rangle} \varphi(m) E^{\bb \bbf}_{\br, \bt}(\bu^\bullet) \otimes v_{r_k}^{b_k^\vee} \cdots v_{r_1}^{b_1^\vee} \in N \otimes S^k_{\nakayamai} V.
$$
Because $\varphi$ is a $\gl_n(A)$-supermodule homomorphism and $E^{\bb \bbf}_{\br, \bt}(\bu^\bullet)$ is the action of an element in $U(\gl_n(A))$, we have: 
$$
\sum_{\substack{ \bb \in \BasisB^k \\ \br \in [1,n]^k \\ \bu \in \ord_k^\textup{g}} }
(-1)^{\langle \bbf, \bb \rangle} \varphi(m) E^{\bb \bbf}_{\br, \bt}(\bu^\bullet) \otimes v_{r_k}^{b_k^\vee} \cdots v_{r_1}^{b_1^\vee} = 
\sum_{\substack{ \bb \in \BasisB^k \\ \br \in [1,n]^k \\ \bu \in \ord_k^\textup{g}} }
(-1)^{\langle \bbf, \bb \rangle} \varphi \left( m E^{\bb \bbf}_{\br, \bt}(\bu^\bullet) \right) \otimes v_{r_k}^{b_k^\vee} \cdots v_{r_1}^{b_1^\vee}.
$$
However, the expression on the right side of the equality is also the result of instead traveling along the leftmost and bottom edges of the square.  The commutativity of the square in \cref{naturalitysquare} follows.

Next, we establish that the $\k$-linear map $C_{M}^{i^{(k)}}$ is a $\gl_n(A)$-supermodule homomorphism for all $\gl_n(A)$-supermodules $M$, all $i \in I$, and all $k \geq 1$.  We prove this by induction on $k$. When $k = 1$, the map $C_{M}^{i^{(1)}}$ is the same as the thin Casimir $C^i_M$ and the base case follows from \cref{L:coloredthincasimir}.  For the inductive step, given $K > 1$ suppose  that $C_{M}^{i^{(k)}}$ is a supermodule homomorphism for all $k < K$. We will show that $C_{M}^{i^{(K)}}$ is also a $\gl_n(A)$-homomorphism. First we establish that the scalar multiple $K C_{M}^{i^{(K)}}$ is a $\gl_n(A)$-supermodule homomorphism.  We have
\begin{align*} K C_{M}^{i^{(K)}} &=  C_{M}^{i^{(K)}} \circ (K \id_{M} \otimes \id_{S_{i}^K V}) \\ 
&= C_{M}^{i^{(K)}} \circ  \widehat{H}_n\left(K \hackcenter{
\begin{tikzpicture}[scale=0.4]
  \draw[ultra thick, blue] (0,-0.1)--(0,0.9);
  \draw[ultra thick, blue] (0,0.9)--(0,1.9);
     \node[below] at (0,-0.1) {$\scriptstyle i^{\scriptstyle(K)}$};
\end{tikzpicture}}  \right)_M \\
&= C_{M}^{i^{(K)}} \circ  \widehat{H}_n\left(\hackcenter{
\begin{tikzpicture}[scale=0.5]
  \draw[ultra thick, blue] (0,-0.1)--(0,0.1) .. controls ++(0,0.35) and ++(0,-0.35) .. (-0.4,0.6)--(-0.4,0.9);
    \draw[ultra thick, blue]  (-0.4,0.9)--(-0.4,1.2) 
  .. controls ++(0,0.35) and ++(0,-0.35) .. (0,1.7)--(0,1.9);
  \draw[ultra thick, blue] (0,-0.1)--(0,0.1) .. controls ++(0,0.35) and ++(0,-0.35) .. (0.4,0.6)--(0.4,0.9);
    \draw[ultra thick, blue]  (0.4,0.9)--(0.4,1.2) 
  .. controls ++(0,0.35) and ++(0,-0.35) .. (0,1.7)--(0,1.9);
  .. controls ++(0,0.35) and ++(0,-0.35) .. (0,1.7)--(0,1.9);
       \node[above] at (0,1.8) {$\scriptstyle i^{\scriptstyle(K)}$};
         \node[below] at (0,-0.1) {$\scriptstyle i^{\scriptstyle(K)}$};
             \node[right] at (0.4,0.9) {$\scriptstyle i^{\scriptstyle(1)}$};
               \node[left] at (-0.4,0.9) {$\scriptstyle i^{\scriptstyle(K-1)}$};
\end{tikzpicture}}
  \right)_M \\
&=  C_M^{i^{(K)}} \circ \;(\id_M \otimes \textup{mer}_{i^{(K-1)}, i^{(1)}}) \circ \; (\id_M \otimes \textup{spl}_{i^{(K-1)}, i^{(1)}}) \\ 
&= (\id_M \otimes \textup{mer}_{\nakayamai^{(K-1)},\nakayamai^{(1)}})\circ C_{M \otimes S_{\nakayamai}^{K-1}V}^{i^{(1)}}\circ (C_M^{i^{(K-1)}} \otimes \id_{ {}_i \! V}) \circ \; (\id_M \otimes \textup{spl}_{i^{(K-1)}, i^{(1)}}),
\end{align*}
where the final equality follows from \cref{T:thickcasimirrecursion}.    Using the induction hypothesis, the expression in the last line is a composite of  $\gl_n(A)$-supermodule homomorphisms.  This proves $K C_M^{i^{(K)}}$ is a $\gl_n(A)$-homomorphism for every right $\gl_n(A)$-module $M$.

We use this to show that $C_M^{i^{(K)}}$ itself is a supermodule homomorphism.  For any fixed $X \in \gl_n(A)$ define a $\k$-linear map
\[
\eta_{M,X} : M \otimes S_i^K V \to M \otimes S_{\nakayamai}^K V
\] by
\[
\eta_{M,X}(w) = C_{M}^{i^{(K)}}(wX) - C_{M}^{i^{(K)}}(w)X, 
\] where $w \in  M \otimes S_i^K V$. Then $C_M^{i^{(K)}}$ is a $\gl_n(A)$-homomorphism if and only if $\eta_{M,X}$ is identically zero for every $X \in \gl_n(A)$.  Since $K C^{i^{(K)}}_M$ is a $\gl_n(A)$-homomorphism for every $M$, it follows that $K\eta_{M,X}$ is identically zero for all $X \in \gl_n(A)$ and, hence, $\eta_{M,X}$ is a torsion element in $\End_{\k}\left( M \otimes S_i^K V \right)$.

Find a $\gl_n(A)$-supermodule $F$ which is free over $\k$ and admits an exact sequence
\[
F \xrightarrow{\phi} M \to 0.
\]  That such an $F$ exists follows by taking a free cover of $M$ as a right $U(\gl_n(A))$-supermodule. The PBW theorem guarantees $F$ is free over $\k$. 

Because $S_i^K V$ is free over $\k$, the functor $- \otimes_{\k} S_{i}^{K}(V)$  is exact and the following diagram has exact rows:
\begin{equation}\label{D:commutingdiagramhomomorphism}
\begin{tikzcd}
F \otimes S^K_i V \arrow{r}{\phi \otimes \id_{S_{i}^{K}V}} \arrow{d}[swap]{\eta_{F,X}} & M \otimes S^K_i V \arrow{r} \arrow{d}{\eta_{M,X}} & 0 \\
F \otimes S^K_i V \arrow{r}{\phi \otimes \id_{S_{i}^{K}V}}  & M \otimes S^K_i V \arrow{r}  & 0 
\end{tikzcd}
\end{equation}
The fact that this diagram commutes follows from the definition of the map $\eta$, the naturality of $C^{i^{(K)}}$ established above, and the fact that $\phi \otimes \id_{S_{i}^{K}V}$ is a $\gl_n(A)$-supermodule homomorphism. 

From the previous paragraph, $\eta_{F,X} \in \End_{\k}\left(  F \otimes S^K_i V \right)$ is a torsion element. By \cref{L:homintofree} this implies $\eta_{F,X} = 0$.  A diagram chase using \cref{D:commutingdiagramhomomorphism}  implies $\eta_{M,X} = 0$, as well.  Since this is true for all $X \in \gl_{n}(A)$, we conclude that $C_{M}^{i^{(K)}}$ is a $\gl_n(A)$-supermodule homomorphism.
\end{proof}

\begin{proposition}\label{P:Thick-White-Dot-Action}  Let $\k$ be a field, let $n \geq 1$, and let $\widehat{H}_{n}: \ThinAffWebAaI \to \catEnd (\modgl)$ be the functor given in \cref{L:hatH}. Then, for all $k \geq 1$,
\[
\widehat{H}_{n}\left( \hackcenter{
\begin{tikzpicture}[scale=0.6]
  \draw[ultra thick, blue] (0,-0.1)--(0,0.9);
  \draw[ultra thick, blue] (0,0.9)--(0,1.9);
  \draw[thick, fill=white] (0,0.9) circle (5pt);
     \node[below] at (0,-0.1) {$\scriptstyle i^{\scriptstyle(k)}$};
      \node[above] at (0,1.9) {$\scriptstyle i^{\scriptstyle(k)}$};
\end{tikzpicture}} \right) = C^{i^{(k)}}.
\]
\end{proposition}
\begin{proof} The claim follows because the white dots satisfy the defining relations of $\AffWebAaI$ along with \cref{T:thickcasimirrecursion}.  Namely, for any $\gl_{n}(A)$-module $M$ we have:
\begin{align*}
\widehat{H}_{n} \left( 
\hackcenter{
\begin{tikzpicture}[scale=0.6]
  \draw[ultra thick, blue] (0,-0.3)--(0,0.9);
  \draw[ultra thick, blue] (0,0.9)--(0,2.1);
  \draw[thick, fill=white] (0,0.9) circle (5pt);
     \node[below] at (0,-0.3) {$\scriptstyle i^{\scriptstyle(k)}$};
      \node[above] at (0,2.1) {$\scriptstyle i^{\scriptstyle(k)}$};
\end{tikzpicture}}\right)_{M}
\;
& =
\;
\widehat{H}_{n}\left( \frac{1}{k}
\hackcenter{
\begin{tikzpicture}[scale=0.6]
  \draw[ultra thick, blue] (0,-0.3)--(0,0.1) .. controls ++(0,0.35) and ++(0,-0.35) .. (-0.4,0.6)--(-0.4,0.9);
    \draw[ultra thick, blue]  (-0.4,0.9)--(-0.4,1.2) 
  .. controls ++(0,0.35) and ++(0,-0.35) .. (0,1.7)--(0,1.9);
  \draw[ultra thick, blue] (0,-0.1)--(0,0.1) .. controls ++(0,0.35) and ++(0,-0.35) .. (0.4,0.6)--(0.4,0.9);
    \draw[ultra thick, blue]  (0.4,0.9)--(0.4,1.2) 
  .. controls ++(0,0.35) and ++(0,-0.35) .. (0,1.7)--(0,2.1);
  .. controls ++(0,0.35) and ++(0,-0.35) .. (0,1.7)--(0,2.1);
       \draw[thick, fill=white] (0,1.8) circle (5pt);
       \node[above] at (0,2) {$\scriptstyle i^{\scriptstyle(k)}$};
         \node[below] at (0,-0.3) {$\scriptstyle i^{\scriptstyle(k)}$};
             \node[right] at (0.4,0.9) {$\scriptstyle i^{\scriptstyle(k)}$};
               \node[left] at (-0.4,0.9) {$\scriptstyle i^{\scriptstyle(k-1)}$};
\end{tikzpicture}}\right)_{M} \\
& =
\;
\widehat{H}_{n}\left( \frac{1}{k}
\hackcenter{
\begin{tikzpicture}[scale=0.6]
  \draw[ultra thick, blue] (0,-0.3)--(0,0.1) .. controls ++(0,0.35) and ++(0,-0.35) .. (-0.4,0.6)--(-0.4,0.9);
    \draw[ultra thick, blue]  (-0.4,0.9)--(-0.4,1.2) 
  .. controls ++(0,0.35) and ++(0,-0.35) .. (0,1.7)--(0,1.9);
  \draw[ultra thick, blue] (0,-0.1)--(0,0.1) .. controls ++(0,0.35) and ++(0,-0.35) .. (0.4,0.6)--(0.4,0.9);
    \draw[ultra thick, blue]  (0.4,0.9)--(0.4,1.2) 
  .. controls ++(0,0.35) and ++(0,-0.35) .. (0,1.7)--(0,2.1);
  .. controls ++(0,0.35) and ++(0,-0.35) .. (0,1.7)--(0,2.1);
       \draw[thick, fill=white] (0.4,0.9) circle (5pt);
       \draw[thick, fill=white] (-0.4,0.9) circle (5pt);
       \node[above] at (0,2) {$\scriptstyle i^{\scriptstyle(k)}$};
         \node[below] at (0,-0.3) {$\scriptstyle i^{\scriptstyle(k)}$};
             \node[right] at (0.4,0.9) {$\scriptstyle i^{\scriptstyle(1)}$};
               \node[left] at (-0.4,0.9) {$\scriptstyle i^{\scriptstyle(k-1)}$};
\end{tikzpicture}}\right)_{M} \\
&= \frac{1}{k}(\id_M \otimes \textup{mer}_{\nakayamai^{k-1},\nakayamai^{1}})\circ C_{M \otimes S_{\nakayamai}^{k-1}V}^{i^{(1)}}\circ (C_M^{i^{(k-1)}} \otimes \id_{S^{1}_{i}V}) \circ (\id_M \otimes \textup{spl}_{\nakayamai^{k-1},\nakayamai^{1}}) \\
& = \frac{1}{k} C_M^{i^{(k)}} \circ  (\id_M \otimes \textup{mer}_{i^{k-1}, i^{1}})\circ (\id_M \otimes \textup{spl}_{\nakayamai^{k-1},\nakayamai^{1}}) \\
&= \frac{1}{k} C_M^{i^{(k)}} \circ \widehat{H}_{n}\left(
\hackcenter{
\begin{tikzpicture}[scale=0.6]
  \draw[ultra thick, blue] (0,-0.3)--(0,0.1) .. controls ++(0,0.35) and ++(0,-0.35) .. (-0.4,0.6)--(-0.4,0.9);
    \draw[ultra thick, blue]  (-0.4,0.9)--(-0.4,1.2) 
  .. controls ++(0,0.35) and ++(0,-0.35) .. (0,1.7)--(0,1.9);
  \draw[ultra thick, blue] (0,-0.1)--(0,0.1) .. controls ++(0,0.35) and ++(0,-0.35) .. (0.4,0.6)--(0.4,0.9);
    \draw[ultra thick, blue]  (0.4,0.9)--(0.4,1.2) 
  .. controls ++(0,0.35) and ++(0,-0.35) .. (0,1.7)--(0,2.1);
  .. controls ++(0,0.35) and ++(0,-0.35) .. (0,1.7)--(0,2.1);
       \node[above] at (0,2) {$\scriptstyle i^{\scriptstyle(k)}$};
         \node[below] at (0,-0.3) {$\scriptstyle i^{\scriptstyle(k)}$};
             \node[right] at (0.4,0.9) {$\scriptstyle i^{\scriptstyle(k)}$};
               \node[left] at (-0.4,0.9) {$\scriptstyle i^{\scriptstyle(k-1)}$};
\end{tikzpicture}}\right)_{M} \\
&=
C_M^{i^{(k)}} \circ
\widehat{H}_{n}\left( \hackcenter{
\begin{tikzpicture}[scale=0.6]
  \draw[ultra thick, blue] (0,-0.3)--(0,0.9);
  \draw[ultra thick, blue] (0,0.9)--(0,2.1);
     \node[below] at (0,-0.3) {$\scriptstyle i^{\scriptstyle(k)}$};
      \node[above] at (0,2.1) {$\scriptstyle i^{\scriptstyle(k)}$};
\end{tikzpicture}}\right)_{M}\\
&=
C_M^{i^{(k)}} \circ \Id_{M \otimes S_{i}^{k}V} = C_M^{i^{(k)}}.
\end{align*}
\end{proof}

\subsection{The defining representation of \texorpdfstring{$\AffWebAaI$}{AffWebAaI}}  

The goal of this section is to establish the existence of what we call the \emph{defining representation} of $\AffWebAaI$ over $\k$.

\begin{theorem}\label{Hthm} For every $n \geq 1$, there is a monoidal superfunctor
\begin{align*}
H_n: \AffWebAaI \to \catEnd(\modgl)
\end{align*}
given on generating objects by \(H_n (i ^{(x)}) = - \otimes S^x_iV_n\), and on generating morphisms by
\begin{equation*}
\begin{gathered}
\hackcenter{
{}
}
\hackcenter{
\begin{tikzpicture}[scale=.8]
  \draw[ultra thick,blue] (0,0)--(0,0.2) .. controls ++(0,0.35) and ++(0,-0.35) .. (-0.4,0.9)--(-0.4,1);
  \draw[ultra thick,blue] (0,0)--(0,0.2) .. controls ++(0,0.35) and ++(0,-0.35) .. (0.4,0.9)--(0.4,1);
      \node[above] at (-0.4,1) {$ \scriptstyle i^{\scriptstyle (x)}$};
      \node[above] at (0.4,1) {$ \scriptstyle i^{\scriptstyle (y)}$};
      \node[below] at (0,0) {$ \scriptstyle i^{\scriptstyle (x+y)} $};
\end{tikzpicture}}
\mapsto - \otimes
 \textup{spl}_{i^x,i^y},
\qquad
\hackcenter{
\begin{tikzpicture}[scale=.8]
  \draw[ultra thick,blue ] (-0.4,0)--(-0.4,0.1) .. controls ++(0,0.35) and ++(0,-0.35) .. (0,0.8)--(0,1);
\draw[ultra thick, blue] (0.4,0)--(0.4,0.1) .. controls ++(0,0.35) and ++(0,-0.35) .. (0,0.8)--(0,1);
      \node[below] at (-0.4,0) {$ \scriptstyle i^{ \scriptstyle (x)}$};
      \node[below] at (0.4,0) {$ \scriptstyle i^{ \scriptstyle (y)}$};
      \node[above] at (0,1) {$ \scriptstyle i^{ \scriptstyle (x+y)}$};
\end{tikzpicture}}
\mapsto
- \otimes \textup{mer}_{i^x,i^y}
\qquad
\hackcenter{
\begin{tikzpicture}[scale=.8]
  \draw[ultra thick,red] (0.4,0)--(0.4,0.1) .. controls ++(0,0.35) and ++(0,-0.35) .. (-0.4,0.9)--(-0.4,1);
  \draw[ultra thick,blue] (-0.4,0)--(-0.4,0.1) .. controls ++(0,0.35) and ++(0,-0.35) .. (0.4,0.9)--(0.4,1);
      \node[above] at (-0.4,1) {$ \scriptstyle j^{ \scriptstyle (y)}$};
      \node[above] at (0.4,1) {$ \scriptstyle i^{ \scriptstyle (x)}$};
       \node[below] at (-0.4,0) {$ \scriptstyle i^{ \scriptstyle (x)}$};
      \node[below] at (0.4,0) {$ \scriptstyle j^{ \scriptstyle (y)}$};
\end{tikzpicture}}
\mapsto
- \otimes \tau_{i^x, i^y},
\\
\hackcenter{
\begin{tikzpicture}[scale=.8]
  \draw[ultra thick, blue] (0,0)--(0,0.5);
   \draw[ultra thick, red] (0,0.5)--(0,1);
   \draw[thick, fill=yellow]  (0,0.5) circle (7pt);
    \node at (0,0.5) {$ \scriptstyle f$};
     \node[below] at (0,0) {$ \scriptstyle i^{ \scriptstyle (x)}$};
      \node[above] at (0,1) {$ \scriptstyle j^{ \scriptstyle (x)}$};
\end{tikzpicture}}
\mapsto - \otimes L_x^f,
\qquad
\hackcenter{
\begin{tikzpicture}[scale=.8]
  \draw[ultra thick, blue] (0,0)--(0,0.5);
   \draw[ultra thick, blue] (0,0.5)--(0,1);
   \draw[thick, fill=black]  (0,0.5) circle (5pt);
     \node[below] at (0,0) {$ \scriptstyle i^{ \scriptstyle (x)}$};
      \node[above] at (0,1) {$ \scriptstyle \nakayamai^{ \scriptstyle (x)}$};
\end{tikzpicture}}
\mapsto  C^{i^{(x)}},
\end{gathered}
\end{equation*}
for all \(x,y \in \Z_{\geq 0}\), \(i,j \in I\), and \(f \in jA^{(x)}i\).
\end{theorem}
\begin{proof} Let $\KK$ be the field of fractions for $\k$.  Recall the functor $\kappa: \AffWebAaI \to \KKThinAffWebAaI \cong \KThinAffWebAaI $ given by \cref{T:Affine-to-thin-dot-functor} and the canonical functor $\xi: \catEnd(\modgl) \to {}_{\mathbb{K}}\catEnd(\modgl)$ discussed in \cref{SS:Extension-of-scalars}. Starting with the representation $\widehat{H}_{n}$ given in \cref{L:hatH} and extending scalars linearly yields the representation ${}_{\KK}\widehat{H}_{n}: \KThinAffWebAaI \to {}_{\mathbb{K}}\catEnd(\modgl)$.  These functors give three sides of the following diagram:  
\[
\begin{tikzcd}
\KThinAffWebAaI \arrow{r}{{}_{\KK}\widehat{H}_n}  &  {}_{\mathbb{K}}\catEnd(\modgl)    \\
 \AffWebAaI \arrow[r,dashed]{}{H_{n}}\arrow[u]{}{\kappa}   &  \catEnd(\modgl)  \arrow[u,swap]{}{\xi}
\end{tikzcd}.
\]  The claim is that we can define the functor $H_{n}$ as indicated and complete to a commuting square.
The definition of $H_{n}$ on the generating objects and generating morphisms is clear given what each of the other three functors do to generating objects and morphisms, either by definition or by \cref{P:Thick-White-Dot-Action}. 

 To check that $H_n$ is well-defined, it remains to verify that the defining relations for $\AffWebAaI$ are satisfied.  Given one of the defining relations, move terms to one side of the equality to obtain a $\k$-linear combination of diagrams, $R$, so that the condition $R=0$ is equivalent to the relation being satisfied.   For example, to the relation \cref{KnotholeRel}, we associate the $\k$-linear combination of diagrams:
$$
R = \hackcenter{
\begin{tikzpicture}[scale=0.8]
  \draw[ultra thick, blue] (0,-0.1)--(0,0.1) .. controls ++(0,0.35) and ++(0,-0.35) .. (-0.4,0.6)--(-0.4,0.9);
    \draw[ultra thick, blue]  (-0.4,0.9)--(-0.4,1.2) 
  .. controls ++(0,0.35) and ++(0,-0.35) .. (0,1.7)--(0,1.9);
  \draw[ultra thick, blue] (0,-0.1)--(0,0.1) .. controls ++(0,0.35) and ++(0,-0.35) .. (0.4,0.6)--(0.4,0.9);
    \draw[ultra thick, blue]  (0.4,0.9)--(0.4,1.2) 
  .. controls ++(0,0.35) and ++(0,-0.35) .. (0,1.7)--(0,1.9);
  .. controls ++(0,0.35) and ++(0,-0.35) .. (0,1.7)--(0,1.9);
       \node[above] at (0,1.8) {$\scriptstyle i^{\scriptstyle(x+y)}$};
         \node[below] at (0,-0.1) {$\scriptstyle i^{\scriptstyle(x+y)}$};
             \node[right] at (0.4,0.9) {$\scriptstyle i^{\scriptstyle(y)}$};
               \node[left] at (-0.4,0.9) {$\scriptstyle i^{\scriptstyle(x)}$};
\end{tikzpicture}}
\;
-
\;
{x+y \choose x}
\hspace{-2mm}
\hackcenter{
\begin{tikzpicture}[scale=0.8]
  \draw[ultra thick, blue] (0,-0.1)--(0,0.9);
  \draw[ultra thick, blue] (0,0.9)--(0,1.9);
     \node[below] at (0,-0.1) {$\scriptstyle i^{\scriptstyle(x+y)}$};
      \node[above] at (0,1.9) {$\scriptstyle i^{\scriptstyle(x+y)}$};
\end{tikzpicture}}.
$$
For a general defining relation, we can express $R$ as a $\k$-linear combination of compositions of generating diagrams.  We write $\tilde{R}$ for the natural transformation in $\catEnd(\modgl)$ given by applying $H_n$ to these generating diagrams, and then forming the corresponding $\k$-linear combination of compositions of natural transformation in the obvious way (i.e., if $H_{n}$ was a well-defined functor we would have $H_{n}(R) = \tilde{R}$).  Checking that $H_n$ is well-defined then amounts to checking that $\tilde{R} = 0$ as a natural transformation in $\catEnd (\modgl)$.  

Say that $R \in \Hom_{\AffWebAaI}\left( {\bf i}^{\bf x}, {\bf j}^{\bf y} \right)$, where ${\bf i}^{\bf x}$ and ${\bf j}^{\bf y}$ are objects in $\AffWebAaI$.     Since the above square commutes on generators, we have $\xi(\tilde{R})= {}_{\KK}\widehat{H}_{n}(\kappa(R)) = 0$ as a morphism in  ${}_{\KK}\catEnd(\modgl)$.  Therefore, the morphism $\tilde{R}$ is in the kernel of the $\k$-linear map induced by $\xi$:
\begin{gather*}
\Hom(-  \otimes S_{\bf i}^{\bf x}, - \otimes S^{\bf y}_{\bf j}) \to
\mathbb K \otimes \Hom(-  \otimes S^{x_{1}}_{i_{1}}V \otimes \dotsb \otimes S^{x_{t}}_{i_{t}}V, - \otimes S^{y_{1}}_{j_{1}}V \otimes \dotsb \otimes S^{y_{t'}}_{j_{t'}}V ), \\
 \eta \mapsto \xi(\eta) = 1_{\mathbb{K}} \otimes \eta.
\end{gather*}
Because the functors in the domain and codomain of $H_n(R)$ are exact and and take $\gl_n(A)$-modules which are free over $\k$ to others of the same kind, \cref{L:torsionfreenattransforms} implies this $\k$-linear map is injective and $\tilde{R} = 0$, as desired.
\color{black}
\end{proof}

\section{The defining representation and the thin subcategory \texorpdfstring{$\WebAaIthin$}{WebAaIthin}}
Throughout this section we continue to let $\k$ be an integral domain.  Our goal is to compute information about the action of $H_{n}(D)$ for certain web diagrams $D$.  This will be used to establish that the family of functors $\left\{H_{n}\right\}_{n\geq 1}$ is asymptotically faithful on those diagrams.

\subsection{PBW bases and filtrations}   Let $\calA = (A, \tr , \psi)$ be a locally unital Frobenius superalgebra with set of distinguished idempotents $I$.  Let $\BasisB$ be a homogeneous basis for $A$ with a fixed total order.

 Let $\fh \subseteq \gl_{n}(A)$ be the Lie subalgebra of diagonal matrices, $\fb \subseteq \gl_{n}(A)$ the subalgebra of upper triangular matricies, $\fu^{+} \subseteq \gl_{n}(A)$ the subalgebra of strictly upper triangular matrices, and $\fu^{-} \subseteq \gl_{n}(A)$ the subalgebra of strictly lower triangular matrices.   There is a triangular decomposition $\gl_{n}(A)  = \fu^{+} \oplus \fh \oplus \fu^{-}$.  Furthermore, $\fu^{+}$ is an ideal of $\fb$ and we have a canonical surjection 
\[
\fb \to \fb /\fu^{+} \cong \fh.
\]  In particular, given a right $\fh$-supermodule we can and will view it as a right $\fb$-supermodule via inflation through this homomorphism.

For short, write $h_{k}^{a} = E_{k,k}^{a}$ for $k=1, \dotsc , n$ and $a \in A$.  Then,
\begin{equation}\label{E:gbasis}
\left\{E_{i,j}^{b} \mid 1 \leq i, j \leq n, b \in \BasisB  \right\}
\end{equation}
and
\begin{equation}\label{E:hbasis}
\left\{h_{k}^{b} \mid 1 \leq k \leq n, b \in \BasisB  \right\}
\end{equation}
are homogeneous bases for $\gl_{n}(A)$ and $\fh$, respectively. The total order on $\BasisB$ and the natural ordering on $[1,n]$ can be used to define a total order on the bases for $\gl_{n}(A)$ by declaring  $E_{p_{1}, q_{1}}^{b_{1}} \leq E_{p_{2},q_{2}}^{b_{2}}$ if and only if $(p_{1}, q_{1}, b_{1})$ is less than or equal to $(p_{2},q_{2},b_{2})$ in the lexicographic ordering.  With this ordering there is a corresponding PBW basis for $U(\gl_{n}(A) )$ and $U(\fh )$ given by the set of all ordered monomials in the above bases for $\gl_{n}(A)$ and $\fh$, respectively.  
Specifically, a PBW basis element of $U(\fh )$ is an ordered product
\begin{equation}\label{E:U(h)-basis}
\prod_{\substack{1 \leq i \leq n \\ b \in \BasisB}} \left(h^{b}_{i} \right)^{t_{i,b}},
\end{equation} where all but finitely many $t_{i,b}$ are equal to zero, and 
where $t_{i,b} \in \Z_{\geq 0}$ if $\bar{b}=\bar{0}$, and $t_{i,b} \in \{0,1 \}$ if $\bar{b}=\bar{1}$. We define the \emph{degree} of this PBW basis element to be $\sum_{1 \leq i \leq n, b \in \BasisB} t_{i,b}$.    For a nonnegative integer $t$, set $ U(\fh )^{\leq t}$ to be the $\k $-span of the PBW basis elements of degree less than or equal to $t$. The commutator formula,
\begin{equation}\label{E:hcommutator}
h_{k}^{b}h_{r}^{a} = (-1)^{\bar{b}\bar{a}} h_{r}^{b}h_{k}^{a} + \delta_{k,r}h_{k}^{[b,a]},
\end{equation}
shows that there is a multiplicative filtration
\[
0 \subseteq U(\fh )^{\leq 0} \subseteq U(\fh )^{\leq 1} \subseteq U(\fh )^{\leq 2} \subseteq (\fh )^{\leq 3} \subseteq \dotsb .
\]  For $t \geq  0$ we write $U(\fh )^{t}$ for the $\k$-span of the PBW basis elements of degree exactly $t$.  Observe that the associated graded superalgebra $\gr  U(\fh )$ is the free supercommutative $\k$-superalgebra on generators $\left\{h_{k}^{b} \mid 1 \leq k \leq  n, b \in \BasisB  \right\}$.  In particular, $U(\fh )^{t}$ is isomorphic as a $\k$-module in an obvious way  to the degree $t$ component of $\gr U(\fh )$ and we will identify them.

We define the \emph{generic Verma module}, $\calM$, to be the  $(U(\fh ), U(\gl_{n}(A) ))$-bimodule 
\[
\calM :=  U(\fh ) \otimes_{U(\fb)} U(\gl_{n}(A) ) .
\] The PBW theorem applied to $\calM $ shows that
\[
\calM  \cong U(\fh ) \otimes_{\k} U(\fu^{-} )
\] as left $U(\fh )$-modules.  In particular, $\calM $ is a free $U(\fh )$-module and is a filtered by setting  $\calM ^{\leq t} = U(\fh )^{\leq t} \otimes U(\fu^{-})$ for $t \geq  0$.  For short, given $t \geq 0$, write $\calM^{t}$ for $U(\fh )^{t} \otimes U(\fu^{-})$.

\subsection{The Casimir map and filtrations}

Recall that $V=V_{n} = A^{\oplus n}$ denotes the natural right $\gl_{n}(A)$-module. Recall from \cref{SS:ThinCasimir} that for each $\gl_{n}(A)$-module $M$ there is a homomorphism
\begin{align*}
C=C_M: M \otimes V \to M \otimes V
\end{align*}
given on homogeneous $m \in M$ and $v_{t}^{\alpha} \in V$ by  
\begin{align}\label{E:CasimirAction}
C_M(m \otimes v^\alpha_t) = \sum_{\substack{b \in \BasisB \\ r \in [1,n]}} (-1)^{\bar \alpha \bar b}mE_{r,t}^{b \alpha} \otimes v_r^{b^\vee} =\sum_{\substack{b \in \BasisB \\ r \in [1,n]}} (-1)^{\bar \alpha \bar b}mE_{r,t}^{b } \otimes v_r^{[{}_{1}\alpha]b^\vee}.
\end{align}
The second sum is obtained from the first by using the teleporter formula given in \cref{transrem}.

\begin{lemma}\label{L:E-on-filtration} Let $\alpha \in A$ and let $1 \leq r,s \leq n$. Then the right action of $E_{r,s}^{\alpha}$ on $\calM$ satisfies $\calM^{t}E_{r,s}^{\alpha} \subseteq  \calM^{t}$ if $r > s$, and $\calM^{t}E_{r,s}^{\alpha} \subseteq  \calM^{\leq (t+1)}$ if $r \leq s$.
\end{lemma}

\begin{proof}   Let $m \in \calM^{t}$.  By linearity, we may assume without loss that
\[
 m = \prod_{\substack{1 \leq k \leq n \\ b \in \BasisB}} \left(h^{b}_{k} \right)^{t_{k,b}} \otimes \prod_{\substack{1 \leq j < i \leq n \\ b \in \BasisB}} \left(E^{b}_{i,j} \right)^{u_{i,j,b}}
\] where the two products are PBW basis elements and $t = \sum_{1 \leq k \leq n, b \in \BasisB} t_{k,b}$.  If $r > s$, then since $U(\fu^{-})$ is a subalgebra one has $mE_{r,s}^{\alpha} \in \prod_{\substack{1 \leq k \leq n \\ b \in \BasisB}} \left(h^{b}_{k} \right)^{t_{k,b}} \otimes U(\fu^{-}) = \calM^{t}$, as claimed.

For $r \leq s$ we argue by induction on $u:= \sum_{1 \leq j < i \leq n, b \in \BasisB} u_{i,j,b}$.
Consider the base case of $u = 0$.
If $r=s$, then 
\[
mE_{r,s}^{\alpha} = mh_{r}^{\alpha} = \left( \prod_{\substack{1 \leq k \leq n \\ b \in \BasisB}} \left(  h^{b}_{k} \right)^{t_{k,b}}   \otimes 1 \right) h_{r}^{\alpha} = \left( \prod_{\substack{1 \leq k \leq n \\ b \in \BasisB}} \left(  h^{b}_{k} \right)^{t_{k,b}}  \right) h_{r}^{\alpha} \otimes 1
\]  That this lies in $\calM^{\leq (t+1)}$ follows from the multiplicativity of the filtration on $U(\fh )$.  If $r < s$, then 
\[
 mE_{r,s}^{\alpha} = \left(  \prod_{\substack{1 \leq k \leq n \\ b \in \BasisB}} \left( h^{b}_{k} \right)^{t_{k,b}}   \otimes 1\right) E_{r,s}^{\alpha} =  \left(  \prod_{\substack{1 \leq k \leq n \\ b \in \BasisB}} \left( h^{b}_{k} \right)^{t_{k,b}}  \right) E_{r,s}^{\alpha} \otimes 1 = 0.
\]

Now consider the case of $u > 0$. For short, let $X=\prod_{\substack{1 \leq k \leq n \\ b \in \BasisB}} \left(h^{b}_{k} \right)^{t_{k,b}}$.  Let $E_{p,q}^{c}$ be the rightmost term in the monomial $\prod_{\substack{1 \leq j < i \leq n \\ b \in \BasisB}} \left(E^{b}_{i,j}\right)^{u_{i,j,b}}$.  Then 
\[
\prod_{\substack{1 \leq j < i \leq n \\ b \in \BasisB}} \left(E^{b}_{i,j}\right)^{u_{i,j,b}} = YE_{p,q}^{c},
\]
 where $Y=\prod_{\substack{1 \leq j < i \leq n \\ b \in \BasisB}} \left(E^{b}_{i,j}\right)^{u'_{i,j,b}}$ with $u'_{p,q,c}=u_{p,q,c}-1$ and $u'_{i,j,b}=u_{i,j,b}$ for all other triples $(i,j,b)$.  In particular, $Y$ is a PBW basis element and $u':= \sum_{1 \leq j < i \leq n, b \in \BasisB} u'_{i,j,b} = u - 1$.  In this case, applying the commutator formula yields:
\begin{align*}
mE_{r,s}^{\alpha} & = X \otimes YE_{p,q}^{c}E_{r,s}^{\alpha} \\
            & = X \otimes Y\left((-1)^{\bar{c}\bar{\alpha}}E_{r,s}^{\alpha}E_{p,q}^{c} +\delta_{r,q}E_{p,s}^{c\alpha}  - (-1)^{\bar{\alpha}\bar{c}}  \delta_{s,p}E_{r,q}^{\alpha c} \right) \\
            & = (-1)^{\bar{c}\bar{\alpha}} X \otimes  YE_{r,s}^{\alpha}E_{p,q}^{c} 
  +  \delta_{r,q} X \otimes Y  E_{p,s}^{c\alpha} 
  - (-1)^{\bar{\alpha}\bar{c}}\delta_{s,p} X \otimes YE_{r,q}^{\alpha c} . 
\end{align*}  Since $u' < u$, it follows by induction that $X \otimes YE_{r,s}^{\alpha} \in \calM^{\leq (t+1)}$ and, hence, $X \otimes  YE_{r,s}^{\alpha}E_{p,q}^{c} \in  \calM^{\leq (t+1)}$ from observing that $ p > q$ and applying earlier arguments.  For the second term of the sum,  $X \otimes Y  E_{p,s}^{c\alpha} \in \calM^{\leq (t+1)}$ because either $p > s$ and it follows from earlier arguments, or $p \leq  s$ and the inductive assumption applies.  Identical reasoning applies to the third term.  Since all three terms lie in $\calM^{\leq (t+1)}$, the claim follows.
\end{proof}

If $T$ is a right $U(\gl_{n}(A) )$-module, then $\calM  \otimes T$ is naturally a $((U(\fh ), U(\gl_{n}(A) ))$-bimodule where $U(\fh )$ acts on the left on $\calM $ and $U(\gl_{n}(A) )$ acts on the right via its coproduct. Furthermore, $\calM \otimes  T$ is a filtered left $U(\fh )$-module by declaring $\left(\calM \otimes T \right)^{\leq t} = \calM^{\leq t} \otimes T$. The next result describe how the map $C$ interacts with certain filtrations of this kind.

\begin{lemma} If $U$ is a right $U(\gl_{n}(A) )$-module, then  the action of 
\[
C = C_{\calM  \otimes U }: \left( \calM  \otimes U \right) \otimes V \to \left( \calM  \otimes U \right) \otimes V
\] satisfies 
\[
C\left( \left( \calM^{\leq t} \otimes U \right) \otimes V\right) \subseteq  \left( \calM^{\leq (t+1)} \otimes U \right) \otimes V
\]
for all $t \geq 0$.
\end{lemma}

\begin{proof}  Let $1 \leq  k \leq  n$ and $\alpha \in A$. It suffices to consider $ (m \otimes u) \otimes v^{\alpha}_{k} \in \left( \calM^{ t} \otimes U \right) \otimes V  $.  Applying the definition of $C$ along with the action of $\gl_{n}(A)$ on $\calM  \otimes U $ yields:
\begin{align*}
C\left( (m \otimes u) \otimes v^{\alpha}_{k}\right) & = \sum_{\substack{b \in \BasisB \\ r \in [1,n]}} (-1)^{\bar{\alpha} \bar{b}}(m \otimes u)E_{r,k}^{b \alpha} \otimes v_r^{b^\vee}\\
  &=  \sum_{\substack{b \in \BasisB \\ r \in [1,n]}} (-1)^{\overline{\alpha} \overline{b} + (\overline{b}+\overline{\alpha})\overline{u}}(mE_{r,k}^{b \alpha} \otimes u) \otimes v_r^{b^\vee}  + \sum_{\substack{b \in \BasisB \\ r \in [1,n]}} (-1)^{\bar{\alpha} \bar{b}}m \otimes \left( uE_{r,k}^{b \alpha} \right)\otimes v_r^{b^\vee}.
\end{align*}   By \cref{L:E-on-filtration} the terms in the first sum lie in $\left( \calM^{\leq (t+1)} \otimes U \right) \otimes V$ and, self-evidently, the terms in the second sum lie in $\left( \calM^{\leq t} \otimes  U \right) \otimes V$.  The result follows.
\end{proof}

\subsection{On the image of the Casimir operator under the defining representation}\label{imcasop}

Let us establish some useful notation.  Given a filtered $\k$-module $T$, $t \geq 0$, and $x, y \in T$, we write 
\begin{equation}\label{E:equiv-def}
x \equiv_{t+1} y
\end{equation}
if $x,y \in  T^{\leq (t+1)}$ and if $x-y \in  T^{\leq t}$.

As an example of this notation that will be used later, say $c_{1}, \dotsc , c_{t} \in \BasisB$ with $c_{1} \leq c_{2} \leq \dotsb \leq c_{t}$ in our fixed ordering on $\BasisB$, and set $1 \leq  i_{1} \leq \dotsb \leq  i_{t} \leq n$.  The product of $h_{i_{1}}^{c_{1}}, \dotsc , h_{i_{t}}^{c_{t}}$ in the given order is the unique PBW basis element of $U(\fh )$ that can be formed as a product of these elements.  For any other ordering, it follows from \eqref{E:hcommutator} that 
\begin{equation}\label{E:permute-the-h}
h_{i_{\tau(1)}}^{c_{\tau(1)}}\dotsb h_{i_{\tau(t)}}^{c_{\tau(t)}} \equiv_{t} \pm h_{i_{1}}^{c_{1}}\dotsb h_{i_{t}}^{c_{t}},
\end{equation} where $\tau \in S_{t}$ and where the sign is determined by $\tau$ and the parities of $c_{1}, \dotsc , c_{t}$.

\begin{lemma} Let $u \in U(\fh )^{\leq t}$, $1 \leq k \leq n$, and $\alpha \in A$.  The map 
\[
C=C_{\calM} : \calM \otimes V \to \calM \otimes V
\] satisfies
\begin{align*}
C\left( (u \otimes 1) \otimes v_{k}^{\alpha} \right) & \equiv_{t+1} \sum_{\substack{b \in \BasisB}} (-1)^{\bar \alpha \bar b}(uh_{k}^{b \alpha} \otimes 1) \otimes v_{k}^{b^\vee} \\
  & \equiv_{t+1}\sum_{\substack{b \in \BasisB}} (-1)^{\bar \alpha \bar b}(uh_{k}^{b } \otimes 1) \otimes v_{k}^{({}_{1}\alpha) b^\vee}. 
\end{align*}
\end{lemma}

\begin{proof} By definition, 
\[
C((u \otimes 1) \otimes v_{k}^{\alpha}) = \sum_{\substack{b \in \BasisB \\ r \in [1,n]}} (-1)^{\bar \alpha \bar b}(u \otimes 1)E_{r,k}^{b \alpha} \otimes v_r^{b^\vee}.
\]  Since $(u \otimes 1)E_{r,k}^{b \alpha} =0$ whenever $r<k$, we have
\[
C((u \otimes 1) \otimes v_{t}^{\alpha}) = \sum_{\substack{b \in \BasisB \\ k = r}} (-1)^{\bar \alpha \bar b}(uh_{k}^{b \alpha} \otimes 1) \otimes v_{k}^{b^\vee}+\sum_{\substack{b \in \BasisB \\ 1 \leq k < r \leq n }} (-1)^{\bar \alpha \bar b}(u \otimes E_{r,k}^{b \alpha}) \otimes v_r^{b^\vee}.
\]  Since the first sum lies in $\calM^{\leq (t+1)}$ and the second sum lies in $\calM^{\leq t}$, the first claimed $\equiv_{t+1}$ follows.  The second formula is argued in the same fashion using the other formula for $C$ given in \cref{E:CasimirAction}.
\end{proof}

For $1 \leq k \leq d$, let $C_{k}$ denote the operator given by taking the sum of all natural transformations given diagrams on thin strands which are an affine dot on the $k$th strand:
\[
C_{k}= \sum_{(i_{1}, \dotsc , i_{d}) \in I^{d}} H_{n}\left(
\hackcenter{
\begin{tikzpicture}[scale=.8]
  \draw[ultra thick, red] (-1.5,-0.1)--(-1.5,1.1);
     \node at  (-0.85,0.5)  {$\cdots$};
     \node[below] at (-1.5,-0.1) {$ \scriptstyle i_{1}^{ \scriptstyle (1)}$};
      \node[above] at (-1.5,1.1) {$ \scriptstyle i_{1}^{ \scriptstyle (1)}$};
  \draw[ultra thick, blue] (0,-0.1)--(0,1.1);
   \draw[thick, fill=black]  (0,0.5) circle (5pt);
     \node[below] at (0,-0.1) {$ \scriptstyle i_{k}^{ \scriptstyle (1)}$};
      \node[above] at (0,1.1) {$ \scriptstyle i_{k}^{ \scriptstyle (1)}$};
        \draw[ultra thick, violet] (1.5,-0.1)--(1.5,1.1);
     \node at  (0.85,0.5)  {$\cdots$};
     \node[below] at (1.5,-0.1) {$ \scriptstyle i_{d}^{ \scriptstyle (1)}$};
      \node[above] at (1.5,1.1) {$ \scriptstyle i_{d}^{ \scriptstyle (1)}$};
\end{tikzpicture}} \right).
\]
Similarly, given $1 \leq p \neq q \leq d$, let $\sigma_{p,q}$ denote operator given by taking the sum of all natural transformations given by diagrams on thin strands which are the crossing of the $p$th and $q$th strands:
\[
\sigma_{p,q}= \sum_{(i_{1}, \dotsc , i_{d}) \in I^{d}} H_{n}\left(
\hackcenter{
\begin{tikzpicture}[scale=.8]
  \draw[ultra thick, violet] (-1.5,-0.1)--(-1.5,1.1);
     \node at  (-0.95,0.5)  {$\cdots$};
     \node[below] at (-1.5,-0.1) {$ \scriptstyle i_{1}^{ \scriptstyle (1)}$};
      \node[above] at (-1.5,1.1) {$ \scriptstyle i_{1}^{ \scriptstyle (1)}$};
  \draw[ultra thick,red] (0.4,-0.1)--(0.4,0.1) .. controls ++(0,0.35) and ++(0,-0.35) .. (-0.4,0.9)--(-0.4,1.1);
  \draw[ultra thick,blue] (-0.4,-0.1)--(-0.4,0.1) .. controls ++(0,0.35) and ++(0,-0.35) .. (0.4,0.9)--(0.4,1.1);
      \node[above] at (-0.4,1.1) {$ \scriptstyle i_{q}^{ \scriptstyle (1)}$};
      \node[above] at (0.4,1.1) {$ \scriptstyle i_{p}^{ \scriptstyle (1)}$};
       \node[below] at (-0.4,-0.1) {$ \scriptstyle i_{p}^{ \scriptstyle (1)}$};
      \node[below] at (0.4,-0.1) {$ \scriptstyle i_{q}^{ \scriptstyle (1)}$};
     \node at  (0.05,0.9)  {$\cdots$};
     \node at  (0.05,0.1)  {$\cdots$};
        \draw[ultra thick, orange] (1.5,-0.1)--(1.5,1.1);
     \node at  (0.95,0.5)  {$\cdots$};
     \node[below] at (1.5,-0.1) {$ \scriptstyle i_{d}^{ \scriptstyle (1)}$};
      \node[above] at (1.5,1.1) {$ \scriptstyle i_{d}^{ \scriptstyle (1)}$};
\end{tikzpicture}} \right).
\]
By local finiteness, $C_{k}$ and $\sigma_{p,q}$ give well-defined endomorphisms $C_{k,M}$ and $\sigma_{p,q,M}$, respectively, of $M \otimes  V^{\otimes d}$ for any $\gl_{n}(A)$-module $M$.  Moreover, on $\calM \otimes V^{\otimes d}$, the map given by $\sigma_{p,q}$ only acts on the vectors of $V^{\otimes d}$ and, hence, preserves filtration degrees.

\begin{lemma}\label{L:Affine-dot-r-times} Let $u \in  U(\fh )^{t}$, $1 \leq k \leq d$, and $\alpha_{1}, \dotsc , \alpha_{d} \in A$.  Then the map
\[
C_{k} = C_{k,\calM}: \calM \otimes V^{\otimes d} \to \calM \otimes V^{\otimes d}
\]
 satisfies 
\begin{align*}
& C_{k}\left( (u \otimes 1) \otimes v_{t_{1}}^{\alpha_{1}}  \otimes \dotsb \otimes v_{t_{d}}^{\alpha_{d}}\right)\\
& \hspace{1in}  \equiv_{t+1} \sum_{\substack{b \in \BasisB}} (-1)^{\overline{\alpha}_{k} \overline{b} +  \overline{b}(\overline{\alpha}_{1}+\dotsb + \overline{\alpha}_{k-1}) }(uh_{t_{k}}^{b } \otimes 1) \otimes v_{t_{1}}^{\alpha_{1}} \otimes \dotsb \otimes  v_{t_{k}}^{({}_{1}\alpha_{k}) b^\vee}\otimes \dotsb  \otimes v_{t_{d}}^{\alpha_{d}} 
\end{align*}

\end{lemma}

\begin{proof} Consider $C_{k} - \sigma_{1,k}C_{1}\sigma_{1,k}$.  It follows from repeated applications of \cref{righttoleftthin} that this can be written as a linear combination of webs that have no affine dots.  These webs act only on the vectors in $V^{\otimes d}$ and, hence, preserve filtration degree when they act on $\calM \otimes V^{\otimes d}$.   Therefore, to prove the asserted formulas it suffices to prove them when $C_{k}$ is replaced by $\sigma_{1,k}C_{1}\sigma_{1,k}$.  But this case follows from the fact that $C_{1}= C_{\calM} \otimes \Id_{V}^{\otimes (d-1)}$, the previous lemma, and that $\sigma_{1,k}$ acts on $V^{\otimes d}$ by signed place permutation.
\end{proof}

For $1 \leq  k \leq  d$ and $r \geq 1$, let $C_{k}^{r}$ denote the operator given by taking the sum of all natural transformations given by diagrams on thin strands which have $r$ affine dots on the $k$th strand:
\[
C^{r}_{k}= \sum_{(i_{1}, \dotsc , i_{d}) \in I^{d}} H_{n}\left(
\hackcenter{
\begin{tikzpicture}[scale=.8]
  \draw[ultra thick, red] (-1.5,-0.1)--(-1.5,1.1);
     \node at  (-0.85,0.5)  {$\cdots$};
     \node[below] at (-1.5,-0.1) {$ \scriptstyle i_{1}^{ \scriptstyle (1)}$};
      \node[above] at (-1.5,1.1) {$ \scriptstyle i_{1}^{ \scriptstyle (1)}$};
  \draw[ultra thick, blue] (0,-0.1)--(0,1.1);
   \draw[thick, fill=black]  (0,0.5) circle (5pt);
   \node at (0.25,0.75) {$\scriptstyle r$};
     \node[below] at (0,-0.1) {$ \scriptstyle i_{k}^{ \scriptstyle (1)}$};
      \node[above] at (0,1.1) {$ \scriptstyle i_{k}^{ \scriptstyle (1)}$};
        \draw[ultra thick, violet] (1.5,-0.1)--(1.5,1.1);
     \node at  (0.85,0.5)  {$\cdots$};
     \node[below] at (1.5,-0.1) {$ \scriptstyle i_{d}^{ \scriptstyle (1)}$};
      \node[above] at (1.5,1.1) {$ \scriptstyle i_{d}^{ \scriptstyle (1)}$};
\end{tikzpicture}} \right).
\]

Iterating the previous lemma yields the following result.  Note that we implicitly use that the Nakayama automorphism preserves the $\Z_{2}$-grading in order to simplify certain sign formulas, and use the fact that $({}_{t}a)^{\vee}={}_{t}(a^{\vee})$ for $a \in A$ and $t \geq  0$ by \cref{FrobFacts} to omit certain parentheses. 

\begin{lemma} Let $u \in  U(\fh )^{t}$, $1 \leq k \leq d$, and $\alpha_{1}, \dotsc , \alpha_{d} \in A$.  The map 
\[
C^{r}_{k} : \calM \otimes V^{\otimes d} \to \calM \otimes V^{\otimes d}
\]
satisfies 
\begin{align*}
& C^{r}_{k}\left( u \otimes 1 \otimes v_{t_{1}}^{\alpha_{1}} \otimes \dotsb \otimes v_{t_{d}}^{\alpha_{d}}\right) \\
 & \hspace{0.5in} \equiv_{t+1} \sum_{\substack{(b_{1}, \dotsc , b_{r}) \in \BasisB^{r}}} (-1)^{\heartsuit }\left( uh_{t_{k}}^{b_{1} }\dotsb h_{t_{k}}^{b_{r}}\right) \otimes 1 \otimes v_{t_{1}}^{\alpha_{1}} \otimes \dotsb \otimes v_{t_{k-1}}^{\alpha_{k-1}} \otimes  v_{t_{k}}^{z_{k}}\otimes v_{t_{k+1}}^{\alpha_{k+1}} \otimes  \dotsb  \otimes v_{t_{d}}^{\alpha_{d}},
\end{align*} where
\[
\heartsuit = \heartsuit(\alpha_{1},\dotsc , \alpha_{d}; b_{1}, \dotsc , b_{r}) = \bar{\alpha}_{k} \bar{b}_{1} + (\bar{\alpha}_{k}+\bar{b}_{1})\bar{b}_{2}+\dotsb + (\bar{\alpha}_{k}+\dotsb + \bar{b}_{r-1})\bar{b}_{r} + (\bar{b}_{1} + \dotsb + \bar{b}_{r})(\bar{\alpha}_{1}+\dotsb + \bar{\alpha}_{k-1}), 
\]
and where 
\[
z_{k} = z_{k}(\alpha_{k};b_{1}, \dotsc , b_{r}) :=({}_{r}\alpha_{k}) ({}_{r-1}b_{1}^\vee) ({}_{r-2}b_{2}^\vee) \dotsb ({}_{1}b^{\vee}_{r-1}) (b_{r}^{\vee}).
\]  

\end{lemma}

\subsection{On the image of thin diagrams under the defining representation}\label{SS:thin-category}

Let $\WebAaIthin$ be the full subcategory of $\Webaff$ consisting of objects which are words from the set $\left\{i^{(1)}  \right\}_{i \in I}$.  For short, for $i \in I$ we will write $i$ for $i^{(1)}$.  By Corollary~\ref{spancor}, $\Hom_{\WebAaIthin}(i_{1}\dotsb i_{d_{1}}, j_{1}\dotsb j_{d_{2}})$ is spanned by diagrams consisting only of thin strands involving affine dots, crossings, and coupons.  In particular, if this morphism space is nonzero, then $d_{1}=d_{2}$.

For short, we will call a diagram $D$ in $\Hom_{\WebAaIthin}(i_{1}\dotsb i_{d}, j_{1}\dotsb j_{d})$ a \emph{normally ordered} diagram if it is an element of our putative basis.  In particular, $D = X \circ C \circ T$, where $X$ is a diagram that consists solely of crossings, $C$ is a diagram that consists solely of coupons (with at most one on each strand), and $T$ is a diagram that consists solely of affine dots.  Given a normally ordered $D$ in $\Hom_{\WebAaIthin}(i_{1}\dotsb i_{d}, j_{1}\dotsb j_{d})$, let $\operatorname{undot}(D)$ denote the diagram obtained from $D$ by deleting the affine dots. For $k=1, \dotsc , d$, let $\beta_{k}(D)$ equal the number of affine dots on the strand labeled by $i_{k}$ on the bottom of $D$.  In particular, if $x_{k} \in \Hom_{\WebAaIthin}(i_{1}\dotsb i_{d}, i_{1}\dotsb i_{d})$ is short for the diagram which has an affine dot on the $k$th strand, then for a normally ordered diagram $D$ we have 
\begin{equation}\label{E:Decomposition-of-D}
D = \operatorname{undot}(D) \circ x_{d}^{\beta_{d}(D)} \circ \dotsb \circ x_{1}^{\beta_{1}(D)}.
\end{equation}

  Given a normally ordered diagram $D$ in $\Hom_{\WebAaIthin}(i_{1}\dotsb i_{d}, j_{1}\dotsb j_{d})$, let 
\[
\BasisB^{\beta(D)} = \BasisB^{\beta_{1}(D)} \times \dotsb \times \BasisB^{\beta_{d}(D)},
\] and write $\mathbf{b} = (\mathbf{b}_{1}, \dotsc , \mathbf{b}_{d})$ for an element of $\BasisB^{\beta(D)}$ where $\mathbf{b}_{k} = (b_{k,1}, \dotsc , b_{k, \beta_{k}(D)}) \in \BasisB^{\beta_{k}(D)}$.

\begin{proposition}\label{P:Key-basis-result}  Let $D \in \Hom_{\WebAaIthin}(i_{1}\dotsb i_{d}, j_{1}\dotsb j_{d})$ be a normally ordered diagram and let $r=\sum_{k=1}^{d} \beta_{k}(D)$ be the number of affine dots in $D$. Fix $n \geq d$.  Consider the module homomorphism 
\[
H_{n}(D)_{\calM }: \calM \otimes {}_{i_{1}}V \otimes  \dotsb \otimes {}_{i_{d}}V \to \calM \otimes {}_{j_{1}}V \otimes  \dotsb \otimes {}_{j_{d}}V.
\]
  Then,
\begin{align*}
& H_{n}(D)_{\calM } \left( 1 \otimes 1 \otimes v_{1}^{i_{1}} \otimes \dotsb \otimes v_{d}^{i_{d}} \right) \\
& \hspace{0in} \equiv_{r} \sum_{\mathbf{b} \in \BasisB^{\beta(D)}} (-1)^{\diamondsuit}h_{D}(\mathbf{b}) \otimes 1 \otimes G_{n}(\operatorname{undot}(D)) \left( v_{1}^{z_{1}} \otimes v_{2}^{z_{2}} \otimes \dotsb \otimes v_{d}^{z_{d}}\right),
\end{align*} where $h_{D}(\mathbf{b}) \in U(\fh )$ is given by
\begin{align*}
h_{D}(\mathbf{b}) & = \left( h_{1}^{b_{1,1}}h_{1}^{b_{1,2}}\dotsb h_{1}^{b_{1,\beta_{1}(D)}}\right)\left( h_{2}^{b_{2,1}}h_{2}^{b_{2,2}}\dotsb h_{2}^{b_{2, \beta_{2}(D)}}\right)\dotsb \left( h_{d}^{b_{d, 1}}h_{d}^{b_{d,2}}\dotsb h_{d}^{b_{d,\beta_{d}(D)}}\right),
\end{align*} and where $\diamondsuit \in  \Z_{2}$ is given by 
\[
\diamondsuit  = \diamondsuit(i_{1}, \dotsc , i_{d}; \mathbf{b}) := \overline{D}\cdot\overline{h_{D}(\mathbf{b})}+ \sum_{k=1}^{d}\heartsuit \left( z_{1},\dotsc, z_{k-1}, i_{k},\dotsc  , i_{d};  b_{k,1}, \dotsc , b_{k, \beta_{k}(D)}\right),
\]
and where $z_{k} \in A$ is given by
\[
z_{k}=z_{k}(i_{k};\mathbf{b}_{k})=z_{k}(i_{k};b_{k, 1}, \dotsc , b_{k,\beta_{k}(D)})  := ({}_{\beta_{k}(D) }i_{k}) ({}_{\beta_{k}(D) -1}b_{k,1}^{\vee}) ({}_{\beta_{k}(D)-2}b_{k, 2}^{\vee}) \dotsb ({}_{1}b_{k, \beta_{k}(D)-1}^{\vee}) (b_{k,\beta_{k}(D) }^{\vee}).
\]
\end{proposition}

\begin{proof} The result follows by repeated applications of \cref{L:Affine-dot-r-times} to compute the effect of the affine dots in the description of $D$ given in \eqref{E:Decomposition-of-D} followed by the fact that the remainder of $D$, which equals $\operatorname{undot}(D)$, is a web without affine dots and so and acts on $V^{\otimes d}$, as indicated.
\end{proof}

\subsection{Decomposition of the associated graded module for \texorpdfstring{$\calM \otimes  {}_{j_{1}}V \otimes \dotsb \otimes {}_{j_{d}}V$}{MVotimesd}}

Observe that the PBW basis, filtrations, and results of the previous sections depend on a fixed choice of an ordered homogenous basis for $A$.  However, these results can be applied to any choice of homogenous basis.  Fix a totally ordered, homogeneous basis $\BasisB$ for $A$ that contains $I$ (e.g., a good basis for $A$). Let $\BasisB^{\vee}$ be the dual basis for $A$ with ordering given by declaring $b_{1}^{\vee} \leq b_{2}^{\vee}$ if and only if $b_{1}\leq b_{2}$.   In what follows the PBW basis, filtrations, and the results of the previous sections will be taken with respect to the basis $\BasisB^{\vee}$.   In particular, replace references to elements of $\BasisB$ by references to elements of $\BasisB^{\vee}$.

For example, using the ordered basis $\BasisB^{\vee}$ there is a PBW basis for $U(\fh )$:
\begin{equation}\label{E:PBW-Basis-for-Uh-Bvee}
\left\{ \left.  \prod_{\substack{1 \leq i \leq n \\ b \in \BasisB}} \left(h^{b^{\vee}}_{i} \right)^{t_{i,b^{\vee}}} \right| 1 \leq i\leq n, b \in B,  \text{ where } t_{i,b^{\vee}} \in  \Z \text{ if } \bar{b}=\bar{0}, \text{ and } t_{i,b^{\vee}} \in \{0,1 \} \text{ if } \bar{b}=\bar{1} \right\}.
\end{equation}

Likewise, the ordered basis $\BasisB^{\vee}$ and arguments in the previous sections define a filtration on $ \calM \otimes {}_{j_{1}}V \otimes \dotsb \otimes {}_{j_{d}}V$. In particular, for any fixed $r \geq 0$ the degree $r$ component of the associated graded $\gr U(\fh )$-module $\gr \left(  \calM \otimes {}_{j_{1}}V \otimes \dotsb \otimes {}_{j_{d}}V \right)$ can be described as a $\k$-module by 
\begin{align}\label{E:Decomp-Equation}
 \left(\calM \otimes  {}_{j_{1}}V \otimes \dotsb \otimes {}_{j_{d}}V \right)^{\leq r} & /\left(\calM \otimes  {}_{j_{1}}V \otimes \dotsb \otimes {}_{j_{d}}V \right)^{\leq (r-1)} \nonumber \\
    & \cong \left(  \calM^{\leq r}/\calM^{\leq r-1}\right) \otimes  {}_{j_{1}}V \otimes \dotsb \otimes {}_{j_{d}}V \nonumber \\
    & \cong \left( U(\fh )^{\leq r}/U(\fh )^{\leq (r-1)} \right)  \otimes U(\fu^{-})  \otimes  {}_{j_{1}}V \otimes \dotsb \otimes {}_{j_{d}}V \nonumber \\
    & \cong \bigoplus_{X} X \otimes U(\fu^{-})  \otimes  {}_{j_{1}}V \otimes \dotsb \otimes {}_{j_{d}}V,
\end{align}
where the direct sum is over the PBW basis elements of $U(\fh )$ of degree $r$ given in \eqref{E:PBW-Basis-for-Uh-Bvee}.  

\subsection{Computation of the action of normally ordered diagrams in \texorpdfstring{$\WebAaIthin$}{WebAaIthin}}

Using the previous results, we are now able to obtain enough information about the action of the homomorphism $H_{n}(D)_{\mathcal{M}}$ to prove the normally ordered diagrams form a basis for the morphism spaces of $\WebAaIthin$.

Fix a normally ordered diagram $D \in \Hom_{\WebAaIthin}(i_{1}\dotsb i_{d}, j_{1}\dotsb j_{d})$.   Reading from bottom-to-top, $D$ consists of some number of affine dots on each strand, then a single coupon from $pAq$ for some $p,q \in I$ on each strand, then a diagram consisting of only crossings.  Recall that $\beta_{k}=\beta_{k}(D)$ denotes the number of affine dots on the $k$th strand of $D$.  For $k=1, \dotsc , d$, let $\alpha_{k} \in j_{\tau^{-1}(k)}Ai_{k}$ be the coupon on the $k$th strand of $D$, where $\tau$ denotes both the subdiagram of the crossings at the top of $D$ and the element of $S_{d}$ that this diagram defines.

With this notation the description of $D$ as $X \circ C \circ T$ given in Section~\ref{SS:thin-category} can be expressed more explicitly as 
\begin{equation}\label{E:D-Decomposition}
D = \tau \circ \alpha_{1} \otimes \dotsb \otimes \alpha_{d} \circ  x_{d}^{\beta_{d}} \circ \dotsb \circ x_{1}^{\beta_{1}}.   
\end{equation}

Recall that we assume $I \subseteq \BasisB$. For $k=1, \dotsc , d$, let $\mathbf{b}_{k}  \in I^{\beta_{k}(D)} \subseteq \BasisB^{\beta_{k}(D)}$ be given by
\[
\mathbf{b}_{k} = (b_{k,1}, \dotsc , b_{k, \beta_{k}(D)})= (i_{k}, \dotsc , i_{k}) \in I^{\beta_{k}(D)} \subseteq \BasisB^{\beta_{k}(D)},
\] and
\[
\mathbf{b}^{\vee}_{k} =  (b^{\vee}_{k,1}, \dotsc , b^{\vee}_{k, \beta_{k}(D)})=(i^{\vee}_{k},\dotsc , i_{k}^{\vee}) \in (I^{\vee})^{\beta_{k}(D)} \subseteq (\BasisB^{\vee})^{\beta_{k}(D)}.
\]
Let $\mathbf{b} \in I^{\beta(D)} \subseteq \BasisB^{\beta(D)}$  be given by 
\begin{equation}\label{E:boldb-def}
\mathbf{b}  = (\mathbf{b}_{1}, \dotsc , \mathbf{b}_{d}),
\end{equation}  and let $\mathbf{b}^{\vee} \in (I^{\vee})^{\beta(D)} \subseteq (\BasisB^{\vee})^{\beta(D)}$ be given by
\begin{equation}\label{E:boldbvee-def}
\mathbf{b}^{\vee}  = (\mathbf{b}^{\vee}_{1}, \dotsc , \mathbf{b}^{\vee}_{d}).
\end{equation}

For any fixed $\mathbf{c} \in (\BasisB^{\vee})^{\beta(D)}$ the discussion before \eqref{E:permute-the-h} applied to $h_{D}(\mathbf{c})$ defined in \ref{P:Key-basis-result} implies that there is a unique PBW basis element $\widetilde{h_{D}(\mathbf{c})}$ of degree $r:=\sum_{k=1}^{d}\beta_{k}(D)$ such that
\begin{equation}\label{E:permute-the-h}
h_{D}(\mathbf{c}) \equiv_{r} \pm\widetilde{h_{D}(\mathbf{c})}.
\end{equation}

Observe that the set $\left\{v_{k}^{\alpha} \mid 1 \leq  k \leq n, \alpha \in {}_{j}\BasisB  \right\}$ is a $\k$-basis for ${}_{j}V$.  More generally, if $\mathbb{Y}$ is the set of PBW basis elements for $U(\fu^{-})$ and $X$ is a fixed PBW basis element of $U(\fh )$, then the set
\begin{equation}\label{E:basis-element}
\left\{ X \otimes Y \otimes v_{t_{1}}^{\alpha_{1}} \otimes \dotsb v_{t_{d}}^{\alpha_{d}} \mid Y \in \mathbb{Y},  t_{k} \in [1,n] \text{ and } \alpha_{k} \in  {}_{j_{k}}\BasisB \text{ for $k=1, \dotsc , d$} \right\}
\end{equation}
is  a $\k$-basis for the summand 
\[
X \otimes U(\fu^{-})  \otimes  {}_{j_{1}}V \otimes \dotsb \otimes {}_{j_{d}}V
\]
in \eqref{E:Decomp-Equation}.

\begin{theorem}\label{P:KeyLinearIndepedenceResult}  Let $D \in \Hom_{\WebAaIthin}(i_{1}\dotsb i_{d}, j_{1}\dotsb j_{d})$ be a normally ordered diagram, let $r=\sum_{k=1}^{d}\beta_{k}(D)$ be the total number of affine dots in $D$, and let $\mathbf{b}^{\vee}$ be given by \cref{E:boldbvee-def}. Fix $n \geq d$. Consider the morphism 
\[
H_{n}(D)_{\calM}: \calM \otimes {}_{i_{1}}V \otimes  \dotsb \otimes {}_{i_{d}}V \to \calM \otimes {}_{j_{1}}V \otimes  \dotsb \otimes {}_{j_{d}}V.
\]

Then, the image of $H_{n}(D)_{\calM}\left( 1 \otimes 1 \otimes v_{1}^{i_{1}} \otimes \dotsb \otimes v_{d}^{i_{d}}\right)$ in the summand of \eqref{E:Decomp-Equation} indexed by $\widetilde{h_{D}(\mathbf{b}^{\vee})}$ equals
\[
K_{D}  \widetilde{h_{D}(\mathbf{b}^{\vee})} \otimes 1 \otimes  \left( v_{\tau^{-1}(1)}^{\alpha_{\tau^{-1}(1)}} \otimes v_{\tau^{-1}(2)}^{\alpha_{\tau^{-1}(2)}} \otimes \dotsb \otimes v_{\tau^{-1}(d)}^{\alpha_{\tau^{-1}(d)}}\right),
\] where $K_{D} \in \k$ is nonzero and $\tau$, $\alpha_{1}, \dotsc , \alpha_{d}$ are as in \cref{E:D-Decomposition}.
\end{theorem}

\begin{proof} 

First,  note that since $ 1 \otimes 1 \otimes v_{1}^{i_{1}} \otimes \dotsb \otimes v_{d}^{i_{d}}$ lies in $\left(\calM \otimes  {}_{i_{1}}V \otimes \dotsb \otimes {}_{i_{d}}V \right)^{0}$, \cref{P:Key-basis-result} implies $H_{n}(D)\left( 1 \otimes 1 \otimes v_{1}^{i_{1}} \otimes \dotsb \otimes v_{d}^{i_{d}}\right)$ lies in $\left(\calM \otimes  {}_{j_{1}}V \otimes \dotsb \otimes {}_{j_{d}}V \right)^{\leq r}$. Thus we can consider the image of this element in 
\[
\left(\calM \otimes  {}_{j_{1}}V \otimes \dotsb \otimes {}_{j_{d}}V \right)^{\leq r}/\left(\calM \otimes  {}_{j_{1}}V \otimes \dotsb \otimes {}_{j_{d}}V \right)^{\leq (r-1)} \cong  \bigoplus_{X}  X \otimes U(\fu^{-}) \otimes  {}_{j_{1}}V \otimes \dotsb \otimes {}_{j_{d}}V.
\]  This element can further be projected onto the direct summand of \eqref{E:Decomp-Equation} indexed by the PBW basis element $\widetilde{h_{D}(\mathbf{b}^{\vee})}$:
\begin{equation}\label{E:Projection-Summand}
\widetilde{h_{D}(\mathbf{b}^{\vee})} \otimes U(\fu^{-}) \otimes {}_{j_{1}}V \otimes  \dotsb \otimes {}_{j_{d}}V.
\end{equation}
 Let $\mathcal{I}$ be the image of  $H_{n}(D)\left( 1 \otimes 1 \otimes v_{1}^{i_{1}} \otimes \dotsb \otimes v_{d}^{i_{d}}\right)$ in this summand.

We now show this image is as claimed.  Let $\mathbb{X} \subseteq (\BasisB^{\vee})^{\beta(D)}$ be the set of all $\mathbf{c} \in (\BasisB^{\vee})^{\beta(D)}$ for which $c_{1, 1}, \dotsc, c_{d, \beta_{d}(D)}$ equals $b^{\vee}_{1, 1}, \dotsc, b^{\vee}_{d, \beta_{d}(D)}$ as an unordered list of elements from $\BasisB^{\vee}$.  It follows from \eqref{E:permute-the-h} that the only terms in the sum given in Proposition~\ref{P:Key-basis-result} that contribute to summand indexed by $\widetilde{h_{D}(\mathbf{b}^{\vee})}$ are those of the form $h_{D}(\mathbf{c})$ for some $\mathbf{c} \in \mathbb{X}$.  Thus, $\mathcal{I}$ equals the projection of 
\begin{equation}\label{E:another-sum}
\sum_{\mathbf{c} \in \mathbb{X}} (-1)^{\diamondsuit(i_{1}, \dotsc , i_{d};\mathbf{c})}h_{D}(\mathbf{c}) \otimes 1 \otimes G_{n}(\operatorname{undot}(D)) \left( v_{1}^{z_{1}} \otimes v_{2}^{z_{2}} \otimes \dotsb \otimes v_{d}^{z_{d}}\right)
\end{equation} onto \cref{E:Projection-Summand}.

We next study the $z_{k}=z_{k}(i_{k};\mathbf{c}_{k}) \in A$ which appear in \cref{E:another-sum}.  Let $\mathbf{c} \in \mathbb{X}$ and let $c$ be one of the entries of $\mathbf{c}$.  Then $c = b_{p,q}^{\vee} = i_{p}^{\vee}$ for some $p=1,\dotsc ,d$ and, since $(\ell^{\vee})^{\vee}=\ell$ for any $\ell \in I$, we have: 
\[
{}_{t}(c^{\vee})={}_{t}((i_{k}^{\vee})^{\vee}) = i_{k} \in I.
\]  
Therefore, each $z_{k}(i_{k};\mathbf{c}_{k})$ is the product of $i_{k} \in I$ along with potentially other elements of $I$.  But $I$ is a set of orthogonal idempotents.  That is, $z_{k}(i_{k};\mathbf{c}_{k})$ is  either $i_{k}$ or $0$.

Let $\mathbb{X}'$ be the set of all $\mathbf{c} \in \mathbb{X}$ for which $z_{k}(i_{k};\mathbf{c}_{k})=i_{k}$ for all $k=1, \dotsc , d$.  Since a summand in \eqref{E:another-sum} is zero whenever $z_{k}(i_{k};\mathbf{c}_{k})=0$ for some $k=1,\dotsc ,d$, we may further assume that sum runs only over the elements of $\mathbb{X}'$. Therefore, we have that $ \mathcal{I}$ is equal to the projection of:
\begin{align*}
 & = \sum_{\mathbf{c} \in \mathbb{X}'}  (-1)^{\diamondsuit(i_{1}, \dotsc , i_{d};\mathbf{c})}h_{D}(\mathbf{c}) \otimes 1 \otimes G_{n}(\operatorname{undot}(D)) \left( v_{1}^{i_{1}} \otimes v_{2}^{i_{2}} \otimes \dotsb \otimes v_{d}^{i_{d}}\right) \\
    & = \sum_{\mathbf{c} \in \mathbb{X}'}  (-1)^{\diamondsuit(i_{1}, \dotsc , i_{d};\mathbf{c})}h_{D}(\mathbf{c}) \otimes 1 \otimes G_{n}(\tau) \left( v_{1}^{\alpha_{1}} \otimes v_{2}^{\alpha_{2}} \otimes \dotsb \otimes v_{d}^{\alpha_{d}}\right) \\
    &= \sum_{\mathbf{c} \in \mathbb{X}'}  (-1)^{\diamondsuit(i_{1}, \dotsc , i_{d};\mathbf{c})}h_{D}(\mathbf{c}) \otimes 1 \otimes  \left( v_{\tau^{-1}(1)}^{\alpha_{\tau^{-1}(1)}} \otimes v_{\tau^{-1}(2)}^{\alpha_{\tau^{-1}(2)}} \otimes \dotsb \otimes v_{\tau^{-1}(d)}^{\alpha_{\tau^{-1}(d)}}\right)  
\end{align*}  Note that since the elements of $I$ are even, and both the Nakayama automorphism and the taking of the dual of an element preserve parity, it follows that the entries of $\mathbf{c}$ are even for any $\mathbf{c} \in \mathbb{X}'$.  This implies $\diamondsuit(i_{1}, \dotsc , i_{d};\mathbf{c}) = \bar{0}$ for all $\mathbf{c}\in \mathbb{X}'$.  It also implies 
\[
h_{D}(\mathbf{c}) \equiv_{r} \widetilde{h_{D}(\mathbf{b}^{\vee})}
\] for all $\mathbf{c} \in \mathbb{X}'$.  These observations allow further simplification.  Namely, $\mathcal{I}$ equals the projection of:
\begin{align*}
 & = \sum_{\mathbf{c} \in \mathbb{X}'}  h_{D}(\mathbf{c}) \otimes 1 \otimes  \left( v_{\tau^{-1}(1)}^{\alpha_{\tau^{-1}(1)}} \otimes v_{\tau^{-1}(2)}^{\alpha_{\tau^{-1}(2)}} \otimes \dotsb \otimes v_{\tau^{-1}(d)}^{\alpha_{\tau^{-1}(d)}}\right)   \\
          & = \sum_{\mathbf{c} \in \mathbb{X}'}  \widetilde{h_{D}(\mathbf{b}^{\vee})} \otimes 1 \otimes  \left( v_{\tau^{-1}(1)}^{\alpha_{\tau^{-1}(1)}} \otimes v_{\tau^{-1}(2)}^{\alpha_{\tau^{-1}(2)}} \otimes \dotsb \otimes v_{\tau^{-1}(d)}^{\alpha_{\tau^{-1}(d)}}\right) \\
          & = \left| \mathbb{X}' \right| \widetilde{h_{D}(\mathbf{b}^{\vee})} \otimes 1 \otimes  \left( v_{\tau^{-1}(1)}^{\alpha_{\tau^{-1}(1)}} \otimes v_{\tau^{-1}(2)}^{\alpha_{\tau^{-1}(2)}} \otimes \dotsb \otimes v_{\tau^{-1}(d)}^{\alpha_{\tau^{-1}(d)}}\right).
\end{align*} 

The  claim follows once we verify $\mathbf{b}^{\vee} \in \mathbb{X}'$.  Since $\mathbf{b}^{\vee}$ is obviously an element of $\mathbb{X}$, it remains to verify 
\begin{align*}
z_{k}(i_{k};\mathbf{b}^{\vee}_{k})
        &= ({}_{\beta_{k}(D) }i_{k}) ({}_{\beta_{k}(D) -1}(b_{k, 1}^{\vee \vee})) ({}_{\beta_{k}(D) -2}(b_{k, 2}^{\vee \vee}))\dotsb ({}_{1}(b_{k, \beta_{k}(D)-1}^{\vee \vee})) (b_{k, \beta_{k}(D) }^{\vee \vee} ) \\
        &= ({}_{\beta_{k}(D) }i_{k}) ({}_{\beta_{k}(D) -1}(i_{k})) ({}_{\beta_{k}(D) -2}(i_{k}))\dotsb ({}_{1}(i_{k})) ({}_{0}(i_{k})^{\vee \vee}) \\
        &= (i_{k}) (i_{k}) (i_{k}) \dotsb (i_{k}) (i_{k}) \\
        & = i_{k}.
\end{align*}
\end{proof}

\begin{theorem}\label{T:asymptotic-faithfulness}  Let 
\[
\sum_{t=1}^{p} \nu_{t} D_{t} \in \Hom_{\WebAaIthin}(i_{1}\dotsb i_{d}, j_{1}\dotsb j_{d})
\] be a $\k$-linear combination of normally ordered diagrams.  Let $n \geq d$. If $H_{n}\left(\sum_{t=1}^{p} \nu_{t} D_{t} \right)=0$, then $\nu_{t}=0$ for $t=1, \dotsc , p$.
\end{theorem}

\begin{proof} Assume $\sum_{t=1}^{p} \nu_{t} D_{t}$ is a $\k$-linear combination of distinct normally ordered diagrams with $\nu_{t} \neq 0$ for all $t=1,\dotsc ,p$. Further assume that $H_{n}\left(\sum_{t=1}^{p} \nu_{t} D_{t} \right)=0$ for some $n \geq d$. Then, 
\[
H_{n}\left(\sum_{t=1}^{p} \nu_{t} D_{t} \right)_{\calM}: \calM \otimes {}_{i_{1}}V \otimes  \dotsb \otimes {}_{i_{d}}V \to \calM \otimes {}_{j_{1}}V \otimes  \dotsb \otimes {}_{j_{d}}V.
\] is the zero map.

Without loss of generality, we may assume $D_{1}$ is the diagram in the sum with the greatest number of affine dots.  Let $r = \sum_{k=1}^{d}\beta_{k}(D_{1})$ be the total number of affine dots in $D_{1}$.  Let $\mathbf{b}^{\vee}$ be defined as in \cref{E:boldbvee-def} using $D=D_{1}$. Let us consider the image of $H_{n}\left(\sum_{t=1}^{p} \nu_{t} D_{t} \right)_{\calM}(1 \otimes 1 \otimes v_{1}^{i_{1}} \otimes \dotsb \otimes v_{d}^{i_{d}})$ in the summand of \cref{E:Decomp-Equation} indexed by $\widetilde{h_{D_{1}}(\mathbf{b}^{\vee})}$.

If $D_{t}$ has fewer than $r$ affine dots, then by \cref{P:Key-basis-result} $H_{n}\left( D_{t} \right)_{\calM}(1 \otimes 1 \otimes v_{1}^{i_{1}} \otimes \dotsb \otimes v_{d}^{i_{d}}) =0$ in $\left(\calM \otimes  {}_{j_{1}}V \otimes \dotsb \otimes {}_{j_{d}}V \right)^{\leq r}  /\left(\calM \otimes  {}_{j_{1}}V \otimes \dotsb \otimes {}_{j_{d}}V \right)^{\leq (r-1)}$ and does not contribute to the summand of interest.  Without loss, we may assume $D_{1}, \dotsc , D_{p}$ each have exactly $r$ affine dots.   

Furthermore, by considering the multiplicity of each $h_{1}, \dotsc , h_{n}$ in \cref{P:Key-basis-result} we see that the projection of $H_{n}( D_{t})_{\calM}\left(1 \otimes  1 \otimes v_{1}^{i_{1}} \otimes \dotsb \otimes v_{d}^{i_{d}} \right)$ into the summand indexed by $\widetilde{h_{D_{1}}(\mathbf{b}^{\vee})}$ is identically zero unless $\beta_{k}(D_{t}) = \beta_{k}(D_{1})$ for all $k=1, \dotsc , d$.   That is, we may further assume that $D_{1}, \dotsc , D_{p}$ have exactly the same number of affine dots on each strand.

For these normally ordered diagrams, the image of  $H_{n}( D_{t})_{\calM }\left(1 \otimes  1 \otimes v_{1}^{i_{1}} \otimes \dotsb \otimes v_{d}^{i_{d}} \right)$ is given by Proposition~\ref{P:KeyLinearIndepedenceResult}.  For $a=1, \dotsc , p$, let $\alpha_{t,1}, \dotsc \alpha_{t,d} \in  \BasisB$ and $\tau_{t} \in S_{d}$ be the notation for the normally ordered diagram $D_{t}$ established in \eqref{E:D-Decomposition}.  Using this notation and  Proposition~\ref{P:KeyLinearIndepedenceResult},  the image in the summand of \eqref{E:Decomp-Equation} indexed by $\widetilde{h_{D_{1}}(\mathbf{b}^{\vee})}$ is given by 
\[
\sum_{t=1}^{p} K_{D_{t}} \nu_{t}  \widetilde{h_{D_{1}}(\mathbf{b}^{\vee})} \otimes 1 \otimes  \left( v_{\tau_{t}^{-1}(1)}^{\alpha_{t,\tau_{t}^{-1}(1)}} \otimes v_{\tau_{t}^{-1}(2)}^{\alpha_{t,\tau_{t}^{-1}(2)}} \otimes \dotsb \otimes v_{\tau_{t}^{-1}(d)}^{\alpha_{t,\tau_{t}^{-1}(d)}}\right),
\] where $K_{D_{t}} \in \k \setminus \{0 \} $  for $t=1, \dotsc , p$.
However, this is a $\k$-linear combination of distinct basis elements of $X \otimes  U(\fu^{-}) \otimes {}_{i_{1}}V \otimes \dotsb \otimes {}_{i_{d}}V$ as described in \eqref{E:basis-element}.  From this we deduce $\nu_{t}=0$ for $t=1, \dotsc , p$ and the claim follows.
\end{proof}

Since the normally ordered diagrams span the morphism spaces of $\WebAaIthin$, the following results immediately follow from the previous theorem.

\begin{corollary}\label{normordind77} The family of functors $\{H_n\}_{n \geq 1}$ is asymptotically faithful on $\WebAaIthin$, meaning that on any fixed morphism in $\WebAaIthin$, the functor $H_n$ is faithful for $n$ sufficiently large.  \end{corollary}

\begin{corollary}\label{normordbasis77} The normally ordered diagrams give bases for the morphism spaces of $\WebAaIthin$.
\end{corollary}

\subsection{The (degenerate) affine wreath product algebra}\label{degenaffcon}
Savage has defined in \cite{SavageAff}  the {\em affine wreath product algebra} \(\mathcal{H}_d^{\textup{aff}}(A)\) (we use this notation instead of Savage's notation \(\mathcal{A}_d(A)\) to avoid confusion with our use of the symbol \(\mathcal{A}\)). Writing \(\mathcal{H}^{\textup{aff}}(A) := \bigoplus_{\bi, \bj \in I^d} \bj \mathcal{H}^{\textup{aff}}_d(A)\bi\), we may view \(\mathcal{H}^\textup{aff}(A)\) as a monoidal \(\k\)-supercategory with objects \(\{\bi \in I^d \mid d \in \Z_{\geq 0}\}\). In view of \cite[Definition 3.1, Theorem 4.6]{SavageAff}, and Corollary~\ref{normordbasis77} it is straightforward to check the following

\begin{proposition}\label{affwreathisom}
We have an isomorphism of monoidal categories \(\mathcal{H}^\textup{aff}(A) \xrightarrow{\sim} \WebAaIthin\) which sends (in Savage's notation in \cite[Definition 3.1]{SavageAff}) \(f_r, x_r, s_r\) to the coupon \(f\) on the \(r\)th strand, the affine dot on the \(r\)th strand, and the crossing of the \(r, r+1\) strands, respectively.
\end{proposition} 

\begin{remark}
In light of Proposition~\ref{affwreathisom}, one may view \(\AffWebAaI\) as a `thickened' analogue of the (degenerate) affine wreath product algebra.
\end{remark}

\section{Asympototic faithfulness of the defining representation and a basis theorem}

\subsection{A useful ordering on thin morphisms}\label{OrderSecTemp} Let \(\bi= i_1  \cdots i_d\), \(\bj  = j_1 \cdots j_d\). We associate \(\bnu \in \mathcal{M}(\bi, \bj)\) with data \(([\bt, \bb], \sigma)\), where
\begin{align*}
[\bt, \bb] = ([t_1, b_1], \ldots, [t_{d}, b_{d}]), \qquad \sigma \in \mathfrak{S}_{d}
\end{align*}
and \(\nu_{u,\sigma u}([t_u,b_u]) = 1\) for all \(u \in [1,d]\). For instance, if we take \(\bnu = \hat{\bmu}\) as in Example~\ref{cablex}, then we associate \(\bnu\) with 
\begin{align*}
([\bt, \bb], \sigma) = (([1,x],[1,x],[2,i],[2,i],[0,j],[0,y],[0,y],[0,y],[2,z],[1,j],[1,j],[1,z]),(3,4,5)).
\end{align*}
We put a total order \(\lessdot\) on \(\mathcal{M}(\bi, \bj)\) as follows. Choose any order on \(\mathfrak{S}_{d}\) which refines the Bruhat order to a total order. We say \(([\bt, \bb], \sigma) \lessdot ([\bt', \bb'], \sigma')\) provided
\begin{enumerate}
\item \(|\bt| < |\bt'|\) (comparison of affine degrees), or;
\item \(|\bt| = |\bt'|\) and \(\sigma > \sigma'\) (comparison in the Bruhat order), or;
\item \(|\bt| = |\bt'|\), \(\sigma = \sigma'\), and \([\bt, \bb] < [\bt', \bb']\) in the lexicographic total order induced by the total order on \(\BasisP\) as in \S\ref{resPcomp}.
\end{enumerate}
We write
\begin{align*}
\WebAaIthin(\bi, \bj)_{\lessdot \bmu} = \k\{ \eta_{\bnu} \mid \bnu \in \mathcal{M}(\bi, \bj), \bnu \lessdot \bmu\}.
\end{align*}

\subsection{Explosion and cabling}
Recall the explosion map from \S\ref{ExplSec} and the morphisms \(\eta_{\bmu}\) and its cabled mate \(\eta_{\hat{\bmu}}\) from \S\ref{ijmatsec}.

\begin{lemma}\label{explodebasis}
Let \(\bmu \in \mathcal{M}(\bi^{(\bx)}, \bj^{(\by)})\). 
We have \(|\bmu|^!_{a} \exp_{\bi^{(\bx)},\bj^{(\by)}} (\eta_{\bmu}) - |\bmu|^! \eta_{\hat{\bmu}} \in \WebAaIthin(\bi^{\bx}, \bj^{\by})_{\lessdot \hat{\bmu}} \).\end{lemma}

\begin{proof}
Assume \(\bi^{(x)} = i_1^{(x_1)} \cdots i_m^{(x_m)}\) and \(\bj^{(y)} = j_1^{(y_1)} \cdots j_n^{(y_n)}\).
We compute
\begin{align*}
\\
\hackcenter{}
|\bmu|^!_{a} \exp_{\bi^{(\bx)},\bj^{(\by)}} (\eta_{\bmu})
&= 
|\bmu|^!_{a} 
\;\;
\hackcenter{
\begin{overpic}[height=35mm]{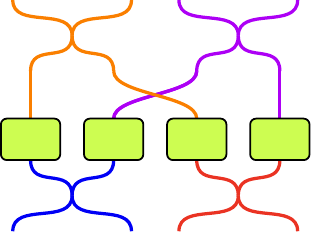}
  \put(4,-1){\makebox(0,0)[t]{$\scriptstyle i_1^{(1)}$}}  
    \put(23,-1){\makebox(0,0)[t]{$\scriptstyle \cdots$}}      
   \put(42,-1){\makebox(0,0)[t]{$\scriptstyle i_1^{(1)}$}}
      \put(50,12){\makebox(0,0)[t]{$\scriptstyle \cdots$}}
    \put(58,-1){\makebox(0,0)[t]{$\scriptstyle i_m^{(1)}$}}      
        \put(77,-1){\makebox(0,0)[t]{$\scriptstyle \cdots$}}      
     \put(96,-1){\makebox(0,0)[t]{$\scriptstyle i_m^{(1)}$}} 
     \put(4,75){\makebox(0,0)[b]{$\scriptstyle j_1^{(1)}$}}  
    \put(23,75){\makebox(0,0)[b]{$\scriptstyle \cdots$}}      
   \put(42,75){\makebox(0,0)[b]{$\scriptstyle j_1^{(1)}$}}
      \put(50,62){\makebox(0,0)[b]{$\scriptstyle \cdots$}}
    \put(58,75){\makebox(0,0)[b]{$\scriptstyle j_n^{(1)}$}}      
        \put(77,75){\makebox(0,0)[b]{$\scriptstyle \cdots$}}      
     \put(96,75){\makebox(0,0)[b]{$\scriptstyle j_n^{(1)}$}}      
       \put(10,29){\makebox(0,0)[]{$\scriptstyle \mu_{1,1}$}}  
          \put(36,29){\makebox(0,0)[]{$\scriptstyle \mu_{1,n}$}}  
             \put(63,29){\makebox(0,0)[]{$\scriptstyle \mu_{m,1}$}}  
                \put(90,29){\makebox(0,0)[]{$\scriptstyle \mu_{m,n}$}}
                \put(24,29){\makebox(0,0)[]{$\scriptstyle \cdots$}}   
                 \put(77,29){\makebox(0,0)[]{$\scriptstyle \cdots$}}   
                      \put(24,49){\makebox(0,0)[]{$\scriptstyle \cdots$}}   
                 \put(77,49){\makebox(0,0)[]{$\scriptstyle \cdots$}}   
         \put(25,12){\makebox(0,0)[l]{$\scriptstyle i_1^{(x_1)}$}}             
   \put(79,12){\makebox(0,0)[l]{$\scriptstyle i_m^{(x_m)}$}}       
      \put(25,63){\makebox(0,0)[l]{$\scriptstyle j_1^{(y_1)}$}}             
   \put(79,63){\makebox(0,0)[l]{$\scriptstyle j_n^{(y_n)}$}}    
\end{overpic}
}
\;
\substack{
(\ref{KnotholeRel})
\\
=
\\
{}
}
\;
\hackcenter{
\begin{overpic}[height=35mm]{WebAffImages/wp1.pdf}
  \put(4,-1){\makebox(0,0)[t]{$\scriptstyle i_1^{(1)}$}}  
    \put(23,-1){\makebox(0,0)[t]{$\scriptstyle \cdots$}}      
   \put(42,-1){\makebox(0,0)[t]{$\scriptstyle i_1^{(1)}$}}
      \put(50,12){\makebox(0,0)[t]{$\scriptstyle \cdots$}}
    \put(58,-1){\makebox(0,0)[t]{$\scriptstyle i_m^{(1)}$}}      
        \put(77,-1){\makebox(0,0)[t]{$\scriptstyle \cdots$}}      
     \put(96,-1){\makebox(0,0)[t]{$\scriptstyle i_m^{(1)}$}} 
     \put(4,75){\makebox(0,0)[b]{$\scriptstyle j_1^{(1)}$}}  
    \put(23,75){\makebox(0,0)[b]{$\scriptstyle \cdots$}}      
   \put(42,75){\makebox(0,0)[b]{$\scriptstyle j_1^{(1)}$}}
      \put(50,62){\makebox(0,0)[b]{$\scriptstyle \cdots$}}
    \put(58,75){\makebox(0,0)[b]{$\scriptstyle j_n^{(1)}$}}      
        \put(77,75){\makebox(0,0)[b]{$\scriptstyle \cdots$}}      
     \put(96,75){\makebox(0,0)[b]{$\scriptstyle j_n^{(1)}$}}      
       \put(10,29){\makebox(0,0)[]{$\scriptstyle \mu_{1,1}^\diamond$}}  
          \put(36,29){\makebox(0,0)[]{$\scriptstyle \mu_{1,n}^\diamond$}}  
             \put(63,29){\makebox(0,0)[]{$\scriptstyle \mu_{m,1}^\diamond$}}  
                \put(90,29){\makebox(0,0)[]{$\scriptstyle \mu_{m,n}^\diamond$}}
                \put(24,29){\makebox(0,0)[]{$\scriptstyle \cdots$}}   
                 \put(77,29){\makebox(0,0)[]{$\scriptstyle \cdots$}}   
                      \put(24,49){\makebox(0,0)[]{$\scriptstyle \cdots$}}   
                 \put(77,49){\makebox(0,0)[]{$\scriptstyle \cdots$}}   
         \put(25,12){\makebox(0,0)[l]{$\scriptstyle i_1^{(x_1)}$}}             
   \put(79,12){\makebox(0,0)[l]{$\scriptstyle i_m^{(x_m)}$}}       
      \put(25,63){\makebox(0,0)[l]{$\scriptstyle j_1^{(y_1)}$}}             
   \put(79,63){\makebox(0,0)[l]{$\scriptstyle j_n^{(y_n)}$}}    
\end{overpic}
}
\\
\\
\\
&
\substack{
(\ref{AssocRel})
\\
(\ref{MergeIntertwineRel})
\\
=
\\
{}
\\
{}
}
\hackcenter{
\begin{overpic}[height=35mm]{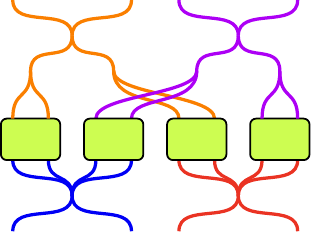}
  \put(4,-1){\makebox(0,0)[t]{$\scriptstyle i_1^{(1)}$}}  
    \put(23,-1){\makebox(0,0)[t]{$\scriptstyle \cdots$}}      
   \put(42,-1){\makebox(0,0)[t]{$\scriptstyle i_1^{(1)}$}}
      \put(50,12){\makebox(0,0)[t]{$\scriptstyle \cdots$}}
    \put(58,-1){\makebox(0,0)[t]{$\scriptstyle i_m^{(1)}$}}      
        \put(77,-1){\makebox(0,0)[t]{$\scriptstyle \cdots$}}      
     \put(96,-1){\makebox(0,0)[t]{$\scriptstyle i_m^{(1)}$}} 
     \put(4,75){\makebox(0,0)[b]{$\scriptstyle j_1^{(1)}$}}  
    \put(23,75){\makebox(0,0)[b]{$\scriptstyle \cdots$}}      
   \put(42,75){\makebox(0,0)[b]{$\scriptstyle j_1^{(1)}$}}
      \put(50,62){\makebox(0,0)[b]{$\scriptstyle \cdots$}}
    \put(58,75){\makebox(0,0)[b]{$\scriptstyle j_n^{(1)}$}}      
        \put(77,75){\makebox(0,0)[b]{$\scriptstyle \cdots$}}      
     \put(96,75){\makebox(0,0)[b]{$\scriptstyle j_n^{(1)}$}}      
         \put(10,39){\makebox(0,0)[]{$\scriptstyle \cdots$}}
       \put(10,29){\makebox(0,0)[]{$\scriptstyle \hat\mu_{1,1}$}}
        \put(11,20){\makebox(0,0)[]{$\scriptstyle \cdots$}}  
         \put(38,39){\makebox(0,0)[]{$\scriptstyle \cdots$}}
          \put(36,29){\makebox(0,0)[]{$\scriptstyle \hat\mu_{1,n}$}}  
           \put(35,20){\makebox(0,0)[]{$\scriptstyle \cdots$}}
           \put(61,39){\makebox(0,0)[]{$\scriptstyle \cdots$}}
             \put(63,29){\makebox(0,0)[]{$\scriptstyle \hat\mu_{m,1}$}}  
               \put(64,20){\makebox(0,0)[]{$\scriptstyle \cdots$}}  
                  \put(90,39){\makebox(0,0)[]{$\scriptstyle \cdots$}}
                \put(90,29){\makebox(0,0)[]{$\scriptstyle \hat\mu_{m,n}$}}
                  \put(89,20){\makebox(0,0)[]{$\scriptstyle \cdots$}}  
                \put(24,29){\makebox(0,0)[]{$\scriptstyle \cdots$}}   
                 \put(77,29){\makebox(0,0)[]{$\scriptstyle \cdots$}}   
                      \put(24,49){\makebox(0,0)[]{$\scriptstyle \cdots$}}   
                 \put(77,49){\makebox(0,0)[]{$\scriptstyle \cdots$}}   
         \put(25,12){\makebox(0,0)[l]{$\scriptstyle i_1^{(x_1)}$}}             
   \put(79,12){\makebox(0,0)[l]{$\scriptstyle i_m^{(x_m)}$}}       
      \put(25,63){\makebox(0,0)[l]{$\scriptstyle j_1^{(y_1)}$}}             
   \put(79,63){\makebox(0,0)[l]{$\scriptstyle j_n^{(y_n)}$}}    
\end{overpic}
}
\substack{
\textup{L.}\ref{bigwishtosym}
\\
=
\\
{}
}
\sum_{\substack{\bw \in \mathfrak{S}_{\bx} \\ \bu \in \mathfrak{S}_{\by}}}
\hackcenter{
\begin{overpic}[height=35mm]{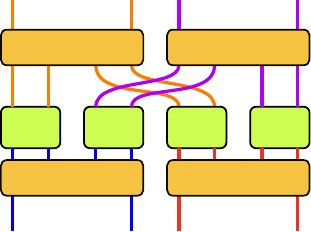}
  \put(4,-1){\makebox(0,0)[t]{$\scriptstyle i_1^{(1)}$}}  
    \put(23,-1){\makebox(0,0)[t]{$\scriptstyle \cdots$}}      
   \put(42,-1){\makebox(0,0)[t]{$\scriptstyle i_1^{(1)}$}}
    \put(58,-1){\makebox(0,0)[t]{$\scriptstyle i_m^{(1)}$}}      
        \put(77,-1){\makebox(0,0)[t]{$\scriptstyle \cdots$}}      
     \put(96,-1){\makebox(0,0)[t]{$\scriptstyle i_m^{(1)}$}} 
     \put(4,75){\makebox(0,0)[b]{$\scriptstyle j_1^{(1)}$}}  
    \put(23,75){\makebox(0,0)[b]{$\scriptstyle \cdots$}}      
   \put(42,75){\makebox(0,0)[b]{$\scriptstyle j_1^{(1)}$}}
    \put(58,75){\makebox(0,0)[b]{$\scriptstyle j_m^{(1)}$}}      
        \put(77,75){\makebox(0,0)[b]{$\scriptstyle \cdots$}}      
     \put(96,75){\makebox(0,0)[b]{$\scriptstyle j_m^{(1)}$}}      
         \put(10,43){\makebox(0,0)[]{$\scriptstyle \cdots$}}
       \put(10,33){\makebox(0,0)[]{$\scriptstyle \hat\mu_{1,1}$}}
        \put(10,24.5){\makebox(0,0)[]{$\scriptstyle \cdots$}}  
         \put(38,43){\makebox(0,0)[]{$\scriptstyle \cdots$}}
          \put(36,33){\makebox(0,0)[]{$\scriptstyle \hat\mu_{1,n}$}}  
           \put(36,24.5){\makebox(0,0)[]{$\scriptstyle \cdots$}}
           \put(61,43){\makebox(0,0)[]{$\scriptstyle \cdots$}}
             \put(63,33){\makebox(0,0)[]{$\scriptstyle \hat\mu_{m,1}$}}  
               \put(63,24.5){\makebox(0,0)[]{$\scriptstyle \cdots$}}  
                  \put(90,43){\makebox(0,0)[]{$\scriptstyle \cdots$}}
                \put(90,33){\makebox(0,0)[]{$\scriptstyle \hat\mu_{m,n}$}}
                  \put(90,24.5){\makebox(0,0)[]{$\scriptstyle \cdots$}}  
                 \put(50.5,33){\makebox(0,0)[]{$\scriptstyle \cdots$}}    
                \put(24,33){\makebox(0,0)[]{$\scriptstyle \cdots$}}   
                 \put(77,33){\makebox(0,0)[]{$\scriptstyle \cdots$}}   
                      \put(24,47){\makebox(0,0)[]{$\scriptstyle \cdots$}}   
                 \put(77,47){\makebox(0,0)[]{$\scriptstyle \cdots$}}   
         \put(23,17){\makebox(0,0)[]{$\scriptstyle w_1$}}        
         \put(50.5,17){\makebox(0,0)[]{$\scriptstyle \cdots$}}        
   \put(77,17){\makebox(0,0)[]{$\scriptstyle w_m$}}       
      \put(23,59){\makebox(0,0)[]{$\scriptstyle u_1$}}  
         \put(50.5,59){\makebox(0,0)[]{$\scriptstyle \cdots$}}             
   \put(77,59){\makebox(0,0)[]{$\scriptstyle u_n$}}    
\end{overpic}
}
\\
\end{align*}
where \(\bw = (w_1, \ldots, w_m) \in \mathfrak{S}_{\bx}\) and \(\bu = (u_1, \ldots, u_n) \in \mathfrak{S}_{\by}\). Now write 
\begin{align*}
\mathfrak{S}_{\bmu} := \mathfrak{S}_{|\mu_{1,1}|} \times \cdots \times \mathfrak{S}_{|\mu_{1,n}|} \times \cdots \times \mathfrak{S}_{|\mu_{m,1}|} \times \cdots \times \mathfrak{S}_{|\mu_{m,n}|},
\end{align*}
and let \(\mathscr{D}_{\bx}^{\bmu}\) designate a set of minimal coset representatives for \(\mathfrak{S}_{\bx} / \mathfrak{S}_{\bmu}\).
Continuing on then, we have that the sum above is equal to 
\begin{align*}
\\
\sum_{\substack{\bom \in \mathcal{D}^{ \bmu }_{\bx} \\ \bu \in \mathfrak{S}_{\by}   \\ \bsi \in \mathfrak{S}_{ \bmu } }}
\hackcenter{
\begin{overpic}[height=35mm]{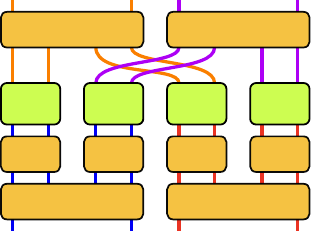}
  \put(4,-1){\makebox(0,0)[t]{$\scriptstyle i_1^{(1)}$}}  
    \put(23,-1){\makebox(0,0)[t]{$\scriptstyle \cdots$}}      
   \put(42,-1){\makebox(0,0)[t]{$\scriptstyle i_1^{(1)}$}}
    \put(58,-1){\makebox(0,0)[t]{$\scriptstyle i_m^{(1)}$}}      
        \put(77,-1){\makebox(0,0)[t]{$\scriptstyle \cdots$}}      
     \put(96,-1){\makebox(0,0)[t]{$\scriptstyle i_m^{(1)}$}} 
     \put(4,75){\makebox(0,0)[b]{$\scriptstyle j_1^{(1)}$}}  
    \put(23,75){\makebox(0,0)[b]{$\scriptstyle \cdots$}}      
   \put(42,75){\makebox(0,0)[b]{$\scriptstyle j_1^{(1)}$}}
    \put(58,75){\makebox(0,0)[b]{$\scriptstyle j_m^{(1)}$}}      
        \put(77,75){\makebox(0,0)[b]{$\scriptstyle \cdots$}}      
     \put(96,75){\makebox(0,0)[b]{$\scriptstyle j_m^{(1)}$}}      
        \put(10,50){\makebox(0,0)[]{$\scriptstyle \cdots$}}
      \put(10,41){\makebox(0,0)[]{$\scriptstyle \hat\mu_{1,1}$}}
         \put(10,32){\makebox(0,0)[]{$\scriptstyle \cdots$}}
       \put(10,24){\makebox(0,0)[]{$\scriptstyle \sigma_{1,1}$}}
        \put(10,17){\makebox(0,0)[]{$\scriptstyle \cdots$}} 
 \put(38,50){\makebox(0,0)[]{$\scriptstyle \cdots$}}
      \put(36,41){\makebox(0,0)[]{$\scriptstyle \hat\mu_{1,n}$}}
         \put(36,32){\makebox(0,0)[]{$\scriptstyle \cdots$}}
       \put(36,24){\makebox(0,0)[]{$\scriptstyle \sigma_{1,n}$}}
        \put(36,17){\makebox(0,0)[]{$\scriptstyle \cdots$}} 
 \put(61,50){\makebox(0,0)[]{$\scriptstyle \cdots$}}
      \put(63,41){\makebox(0,0)[]{$\scriptstyle \hat\mu_{m,1}$}}
         \put(63,32){\makebox(0,0)[]{$\scriptstyle \cdots$}}
       \put(63,24){\makebox(0,0)[]{$\scriptstyle \sigma_{m,1}$}}
        \put(63,17){\makebox(0,0)[]{$\scriptstyle \cdots$}} 
 \put(90,50){\makebox(0,0)[]{$\scriptstyle \cdots$}}
      \put(90,41){\makebox(0,0)[]{$\scriptstyle \hat\mu_{m,n}$}}
         \put(90,32){\makebox(0,0)[]{$\scriptstyle \cdots$}}
       \put(90,24){\makebox(0,0)[]{$\scriptstyle \sigma_{m,n}$}}
        \put(90,17){\makebox(0,0)[]{$\scriptstyle \cdots$}} 
                 \put(50.5,33){\makebox(0,0)[]{$\scriptstyle \cdots$}}    
                \put(24,24){\makebox(0,0)[]{$\scriptstyle \cdots$}}   
                 \put(77,24){\makebox(0,0)[]{$\scriptstyle \cdots$}}   
                      \put(24,41){\makebox(0,0)[]{$\scriptstyle \cdots$}}   
                 \put(77,41){\makebox(0,0)[]{$\scriptstyle \cdots$}}   
         \put(23,9){\makebox(0,0)[]{$\scriptstyle \omega_1$}}        
         \put(50.5,9){\makebox(0,0)[]{$\scriptstyle \cdots$}}        
   \put(77,9){\makebox(0,0)[]{$\scriptstyle \omega_m$}}       
      \put(23,64){\makebox(0,0)[]{$\scriptstyle u_1$}}  
         \put(50.5,64){\makebox(0,0)[]{$\scriptstyle \cdots$}}             
   \put(77,64){\makebox(0,0)[]{$\scriptstyle u_n$}}    
\end{overpic}
}
\hspace{5mm}
\substack{
(\ref{AaIntertwine})
\\
(\ref{DotCrossRel})
\\
=
\\
{}
\\
{}
}
\sum_{\substack{\bom \in \mathcal{D}^{ \bmu }_{\bx} \\ \bu \in \mathfrak{S}_{\by}   \\ \bsi \in \mathfrak{S}_{ \bmu } }}
\hackcenter{
\begin{overpic}[height=35mm]{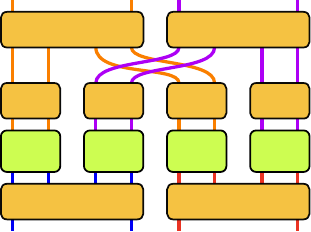}
  \put(4,-1){\makebox(0,0)[t]{$\scriptstyle i_1^{(1)}$}}  
    \put(23,-1){\makebox(0,0)[t]{$\scriptstyle \cdots$}}      
   \put(42,-1){\makebox(0,0)[t]{$\scriptstyle i_1^{(1)}$}}
    \put(58,-1){\makebox(0,0)[t]{$\scriptstyle i_m^{(1)}$}}      
        \put(77,-1){\makebox(0,0)[t]{$\scriptstyle \cdots$}}      
     \put(96,-1){\makebox(0,0)[t]{$\scriptstyle i_m^{(1)}$}} 
     \put(4,75){\makebox(0,0)[b]{$\scriptstyle j_1^{(1)}$}}  
    \put(23,75){\makebox(0,0)[b]{$\scriptstyle \cdots$}}      
   \put(42,75){\makebox(0,0)[b]{$\scriptstyle j_1^{(1)}$}}
    \put(58,75){\makebox(0,0)[b]{$\scriptstyle j_m^{(1)}$}}      
        \put(77,75){\makebox(0,0)[b]{$\scriptstyle \cdots$}}      
     \put(96,75){\makebox(0,0)[b]{$\scriptstyle j_m^{(1)}$}}      
        \put(10,50){\makebox(0,0)[]{$\scriptstyle \cdots$}}
      \put(10,41){\makebox(0,0)[]{$\scriptstyle \sigma_{1,1}$}}
         \put(10,34){\makebox(0,0)[]{$\scriptstyle \cdots$}}
       \put(10.5,25){\makebox(0,0)[]{$\scriptstyle \hat \mu_{1,1}^{\sigma_{1,1}}$}}
        \put(10,17){\makebox(0,0)[]{$\scriptstyle \cdots$}} 
 \put(38,50){\makebox(0,0)[]{$\scriptstyle \cdots$}}
      \put(36,41){\makebox(0,0)[]{$\scriptstyle \sigma_{1,n}$}}
         \put(36,34){\makebox(0,0)[]{$\scriptstyle \cdots$}}
       \put(36.5,25){\makebox(0,0)[]{$\scriptstyle \hat \mu_{1,n}^{\sigma_{1,n}}$}}
        \put(36,17){\makebox(0,0)[]{$\scriptstyle \cdots$}} 
 \put(61,50){\makebox(0,0)[]{$\scriptstyle \cdots$}}
      \put(63,41){\makebox(0,0)[]{$\scriptstyle \sigma_{m,1}$}}
         \put(63,34){\makebox(0,0)[]{$\scriptstyle \cdots$}}
       \put(63.5,25){\makebox(0,0)[]{$\scriptstyle \hat \mu_{m,1}^{\sigma_{m,1}}$}}
        \put(63,17){\makebox(0,0)[]{$\scriptstyle \cdots$}} 
 \put(90,50){\makebox(0,0)[]{$\scriptstyle \cdots$}}
      \put(90,41){\makebox(0,0)[]{$\scriptstyle \sigma_{m,n}$}}
         \put(90,34){\makebox(0,0)[]{$\scriptstyle \cdots$}}
       \put(90.5,25){\makebox(0,0)[]{$\scriptstyle \hat \mu_{m,n}^{\sigma_{m,n}}$}}
        \put(90,17){\makebox(0,0)[]{$\scriptstyle \cdots$}} 
                 \put(50.5,33){\makebox(0,0)[]{$\scriptstyle \cdots$}}    
                \put(24,24){\makebox(0,0)[]{$\scriptstyle \cdots$}}   
                 \put(77,24){\makebox(0,0)[]{$\scriptstyle \cdots$}}   
                      \put(24,41){\makebox(0,0)[]{$\scriptstyle \cdots$}}   
                 \put(77,41){\makebox(0,0)[]{$\scriptstyle \cdots$}}   
         \put(23,9){\makebox(0,0)[]{$\scriptstyle \omega_1$}}        
         \put(50.5,9){\makebox(0,0)[]{$\scriptstyle \cdots$}}        
   \put(77,9){\makebox(0,0)[]{$\scriptstyle \omega_m$}}       
      \put(23,64){\makebox(0,0)[]{$\scriptstyle u_1$}}  
         \put(50.5,64){\makebox(0,0)[]{$\scriptstyle \cdots$}}             
   \put(77,64){\makebox(0,0)[]{$\scriptstyle u_n$}}    
\end{overpic}
}
+X
\\
\end{align*}
where \(X\) is a \(\k\)-linear combination of diagrams of lesser affine degree than \(\eta_{\bmu}\), and hence \(X \in \WebAaIthin(\bi^{\bx}, \bj^{\by})_{\lessdot \hat{\bmu}} \) by Lemma~\ref{LemAffLower}. Continuing on then, we have
\begin{align}\label{wpeq4}
\nonumber\\
\sum_{\substack{\bom \in \mathcal{D}^{ \bmu }_{\bx} \\ \bu \in \mathfrak{S}_{\by}   \\ \bsi \in \mathfrak{S}_{ \bmu } }}
\hackcenter{
\begin{overpic}[height=30mm]{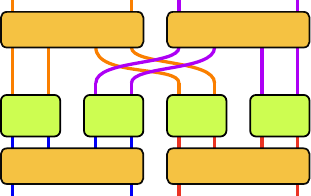}
  \put(4,-1){\makebox(0,0)[t]{$\scriptstyle i_1^{(1)}$}}  
    \put(23,-1){\makebox(0,0)[t]{$\scriptstyle \cdots$}}      
   \put(42,-1){\makebox(0,0)[t]{$\scriptstyle i_1^{(1)}$}}
    \put(58,-1){\makebox(0,0)[t]{$\scriptstyle i_m^{(1)}$}}      
        \put(77,-1){\makebox(0,0)[t]{$\scriptstyle \cdots$}}      
     \put(96,-1){\makebox(0,0)[t]{$\scriptstyle i_m^{(1)}$}} 
     \put(4,64){\makebox(0,0)[b]{$\scriptstyle j_1^{(1)}$}}  
    \put(23,64){\makebox(0,0)[b]{$\scriptstyle \cdots$}}      
   \put(42,64){\makebox(0,0)[b]{$\scriptstyle j_1^{(1)}$}}
    \put(58,64){\makebox(0,0)[b]{$\scriptstyle j_m^{(1)}$}}      
        \put(77,64){\makebox(0,0)[b]{$\scriptstyle \cdots$}}      
     \put(96,64){\makebox(0,0)[b]{$\scriptstyle j_m^{(1)}$}}      
         \put(10,35){\makebox(0,0)[]{$\scriptstyle \cdots$}}
       \put(10.5,25){\makebox(0,0)[]{$\scriptstyle \hat \mu_{1,1}^{\sigma_{1,1}}$}}
        \put(10,17){\makebox(0,0)[]{$\scriptstyle \cdots$}} 
         \put(36,35){\makebox(0,0)[]{$\scriptstyle \cdots$}}
       \put(36.5,25){\makebox(0,0)[]{$\scriptstyle \hat \mu_{1,n}^{\sigma_{1,n}}$}}
        \put(36,17){\makebox(0,0)[]{$\scriptstyle \cdots$}} 
         \put(63,35){\makebox(0,0)[]{$\scriptstyle \cdots$}}
       \put(63.5,25){\makebox(0,0)[]{$\scriptstyle \hat \mu_{m,1}^{\sigma_{m,1}}$}}
        \put(63,17){\makebox(0,0)[]{$\scriptstyle \cdots$}} 
         \put(90,35){\makebox(0,0)[]{$\scriptstyle \cdots$}}
       \put(90.5,25){\makebox(0,0)[]{$\scriptstyle \hat \mu_{m,n}^{\sigma_{m,n}}$}}
        \put(90,17){\makebox(0,0)[]{$\scriptstyle \cdots$}} 
                 \put(50.5,24){\makebox(0,0)[]{$\scriptstyle \cdots$}}    
                \put(24,24){\makebox(0,0)[]{$\scriptstyle \cdots$}}   
                 \put(77,24){\makebox(0,0)[]{$\scriptstyle \cdots$}}   
                      \put(24,41){\makebox(0,0)[]{$\scriptstyle \cdots$}}   
                 \put(77,41){\makebox(0,0)[]{$\scriptstyle \cdots$}}   
         \put(23,9){\makebox(0,0)[]{$\scriptstyle \omega_1$}}        
         \put(50.5,9){\makebox(0,0)[]{$\scriptstyle \cdots$}}        
   \put(77,9){\makebox(0,0)[]{$\scriptstyle \omega_m$}}       
      \put(23,53){\makebox(0,0)[]{$\scriptstyle u_1$}}  
         \put(50.5,53){\makebox(0,0)[]{$\scriptstyle \cdots$}}             
   \put(77,53){\makebox(0,0)[]{$\scriptstyle u_n$}}    
\end{overpic}
}
&
=
\sum_{\substack{\bom \in \mathcal{D}^{ \bmu }_{\bx} \\ \bu \in \mathfrak{S}_{\by}   \\ \bsi \in \mathfrak{S}_{ \bmu } }}
\hackcenter{
\begin{overpic}[height=30mm]{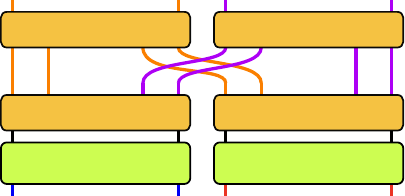}
  \put(4,-1){\makebox(0,0)[t]{$\scriptstyle i_1^{(1)}$}}  
    \put(23,-1){\makebox(0,0)[t]{$\scriptstyle \cdots$}}      
   \put(42,-1){\makebox(0,0)[t]{$\scriptstyle i_1^{(1)}$}}
    \put(58,-1){\makebox(0,0)[t]{$\scriptstyle i_m^{(1)}$}}      
        \put(77,-1){\makebox(0,0)[t]{$\scriptstyle \cdots$}}      
     \put(96,-1){\makebox(0,0)[t]{$\scriptstyle i_m^{(1)}$}} 
     \put(4,49){\makebox(0,0)[b]{$\scriptstyle j_1^{(1)}$}}  
    \put(23,49){\makebox(0,0)[b]{$\scriptstyle \cdots$}}      
   \put(42,49){\makebox(0,0)[b]{$\scriptstyle j_1^{(1)}$}}
    \put(58,49){\makebox(0,0)[b]{$\scriptstyle j_m^{(1)}$}}      
        \put(77,49){\makebox(0,0)[b]{$\scriptstyle \cdots$}}      
     \put(96,49){\makebox(0,0)[b]{$\scriptstyle j_m^{(1)}$}}      
                                    \put(24,31){\makebox(0,0)[]{$\scriptstyle \cdots$}}   
                 \put(77,31){\makebox(0,0)[]{$\scriptstyle \cdots$}}   
         \put(23,20){\makebox(0,0)[]{$\scriptstyle \omega_1$}}        
         \put(50.5,20){\makebox(0,0)[]{$\scriptstyle \cdots$}}        
   \put(77,20){\makebox(0,0)[]{$\scriptstyle \omega_m$}}       
      \put(23,41){\makebox(0,0)[]{$\scriptstyle u_1$}}  
         \put(50.5,41){\makebox(0,0)[]{$\scriptstyle \cdots$}}             
   \put(77,41){\makebox(0,0)[]{$\scriptstyle u_n$}}    
               \put(23.5,8){\makebox(0,0)[]{$\scriptstyle \left( \hat\mu_{1,1}^{\sigma_{1,1}} \otimes \cdots \otimes\hat \mu_{1,n}^{\sigma_{1,n}}\right)^{\omega_1}$}}        
         \put(50.5,9){\makebox(0,0)[]{$\scriptstyle \cdots$}}        
 \put(77,8){\makebox(0,0)[]{$\scriptstyle \left(\hat \mu_{m,1}^{\hspace{-0.5mm}\sigma_{m,1}} \otimes \cdots \otimes\hat \mu_{m,n}^{\hspace{-0.5mm}\sigma_{m,n}}\right)^{\omega_m}$}}  
\end{overpic}
}
+Y
\nonumber\\
\end{align}
where again \(Y\) is a \(\k\)-linear combination of diagrams of lesser affine degree than \(\eta_{\bmu}\), and hence \(Y \in \WebAaIthin(\bi^{\bx}, \bj^{\by})_{\lessdot \hat{\bmu}} \) by Lemma~\ref{LemAffLower}. 

Now note that the diagrammatic permutations at the top of the diagram above are already in reduced form. Thus the Bruhat-lowest element is the summand with \(\bom\), \(\bu\) trivial. Then we have that (\ref{wpeq4}) is equal to
\begin{align*}
\nonumber\\
\sum_{ \bsi \in \mathfrak{S}_{ \bmu  }}
\hackcenter{
\begin{overpic}[height=17mm]{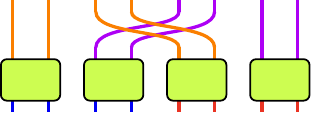}
  \put(4,-1){\makebox(0,0)[t]{$\scriptstyle i_1^{(1)}$}}  
    \put(14,-1){\makebox(0,0)[t]{$\scriptstyle i_1^{(1)}$}}  
    \put(23,-1){\makebox(0,0)[t]{$\scriptstyle \cdots$}}      
            \put(32,-1){\makebox(0,0)[t]{$\scriptstyle i_1^{(1)}$}}  
   \put(42,-1){\makebox(0,0)[t]{$\scriptstyle i_1^{(1)}$}}
       \put(58,-1){\makebox(0,0)[t]{$\scriptstyle i_m^{(1)}$}}    
    \put(68,-1){\makebox(0,0)[t]{$\scriptstyle i_m^{(1)}$}}      
        \put(77,-1){\makebox(0,0)[t]{$\scriptstyle \cdots$}}      
     \put(86,-1){\makebox(0,0)[t]{$\scriptstyle i_m^{(1)}$}} 
         \put(96,-1){\makebox(0,0)[t]{$\scriptstyle i_m^{(1)}$}}    
     \put(4,37){\makebox(0,0)[b]{$\scriptstyle j_1^{(1)}$}}  
          \put(14,37){\makebox(0,0)[b]{$\scriptstyle j_1^{(1)}$}}  
    \put(23,37){\makebox(0,0)[b]{$\scriptstyle \cdots$}}    
      \put(32,37){\makebox(0,0)[b]{$\scriptstyle j_1^{(1)}$}}  
   \put(42,37){\makebox(0,0)[b]{$\scriptstyle j_1^{(1)}$}}
    \put(58,37){\makebox(0,0)[b]{$\scriptstyle j_m^{(1)}$}}      
       \put(68,37){\makebox(0,0)[b]{$\scriptstyle j_m^{(1)}$}}    
        \put(77,37){\makebox(0,0)[b]{$\scriptstyle \cdots$}}     
         \put(86,37){\makebox(0,0)[b]{$\scriptstyle j_m^{(1)}$}}      
     \put(96,37){\makebox(0,0)[b]{$\scriptstyle j_m^{(1)}$}}      
                                    \put(24,31){\makebox(0,0)[]{$\scriptstyle \cdots$}}   
                 \put(77,31){\makebox(0,0)[]{$\scriptstyle \cdots$}}   
          \put(10,9.5){\makebox(0,0)[]{$\scriptstyle \hat\mu_{1,1}^{\sigma_{1,1}}$}}  
          \put(36.5,9.5){\makebox(0,0)[]{$\scriptstyle \hat\mu_{1,n}^{\sigma_{1,n}}$}}  
             \put(63.5,9.5){\makebox(0,0)[]{$\scriptstyle \hat\mu_{m,1}^{\sigma_{m,1}}$}}  
                \put(90.5,9.5){\makebox(0,0)[]{$\scriptstyle \hat\mu_{m,n}^{\sigma_{m,n}}$}}
                \put(24,9){\makebox(0,0)[]{$\scriptstyle \cdots$}}   
                 \put(77,9){\makebox(0,0)[]{$\scriptstyle \cdots$}}   
              \end{overpic}
}
+Z
=
|\bmu|^!
\;
\hackcenter{
\begin{overpic}[height=17mm]{WebAffImages/wp8.pdf}
  \put(4,-1){\makebox(0,0)[t]{$\scriptstyle i_1^{(1)}$}}  
    \put(14,-1){\makebox(0,0)[t]{$\scriptstyle i_1^{(1)}$}}  
    \put(23,-1){\makebox(0,0)[t]{$\scriptstyle \cdots$}}      
            \put(32,-1){\makebox(0,0)[t]{$\scriptstyle i_1^{(1)}$}}  
   \put(42,-1){\makebox(0,0)[t]{$\scriptstyle i_1^{(1)}$}}
       \put(58,-1){\makebox(0,0)[t]{$\scriptstyle i_m^{(1)}$}}    
    \put(68,-1){\makebox(0,0)[t]{$\scriptstyle i_m^{(1)}$}}      
        \put(77,-1){\makebox(0,0)[t]{$\scriptstyle \cdots$}}      
     \put(86,-1){\makebox(0,0)[t]{$\scriptstyle i_m^{(1)}$}} 
         \put(96,-1){\makebox(0,0)[t]{$\scriptstyle i_m^{(1)}$}}    
     \put(4,37){\makebox(0,0)[b]{$\scriptstyle j_1^{(1)}$}}  
          \put(14,37){\makebox(0,0)[b]{$\scriptstyle j_1^{(1)}$}}  
    \put(23,37){\makebox(0,0)[b]{$\scriptstyle \cdots$}}    
      \put(32,37){\makebox(0,0)[b]{$\scriptstyle j_1^{(1)}$}}  
   \put(42,37){\makebox(0,0)[b]{$\scriptstyle j_1^{(1)}$}}
    \put(58,37){\makebox(0,0)[b]{$\scriptstyle j_m^{(1)}$}}      
       \put(68,37){\makebox(0,0)[b]{$\scriptstyle j_m^{(1)}$}}    
        \put(77,37){\makebox(0,0)[b]{$\scriptstyle \cdots$}}     
         \put(86,37){\makebox(0,0)[b]{$\scriptstyle j_m^{(1)}$}}      
     \put(96,37){\makebox(0,0)[b]{$\scriptstyle j_m^{(1)}$}}      
                                    \put(24,31){\makebox(0,0)[]{$\scriptstyle \cdots$}}   
                 \put(77,31){\makebox(0,0)[]{$\scriptstyle \cdots$}}   
          \put(10,9.5){\makebox(0,0)[]{$\scriptstyle \hat\mu_{1,1}$}}  
          \put(36.5,9.5){\makebox(0,0)[]{$\scriptstyle\hat \mu_{1,n}$}}  
             \put(63.5,9.5){\makebox(0,0)[]{$\scriptstyle \hat\mu_{m,1}$}}  
                \put(90.5,9.5){\makebox(0,0)[]{$\scriptstyle \hat\mu_{m,n}$}}
                \put(24,9){\makebox(0,0)[]{$\scriptstyle \cdots$}}   
                 \put(77,9){\makebox(0,0)[]{$\scriptstyle \cdots$}}   
              \end{overpic}
}
+W.
\\
\nonumber
\end{align*}
where, on the left, \(Z\) is a \(\k\)-linear combination of diagrams associated with \(\bnu \in \mathcal{M}(\bi^{\bx}, \bj^{\by})\) which are lower in the total order on \(\mathcal{M}(\bi^{\bx}, \bj^{\by})\) than \(\hat{\bmu}\), thanks to \S\ref{OrderSecTemp}(2). Then, on the right \(W\), is a \(\k\)-linear combination of diagrams associated with \(\bnu \in \mathcal{M}(\bi^{\bx}, \bj^{\by})\) which are lower in the total order on \(\mathcal{M}(\bi^{\bx}, \bj^{\by})\) than \(\hat{\bmu}\), thanks to \S\ref{OrderSecTemp}(3).
But this last diagram is \(|\bmu|^!\eta_{\hat \bmu}\), and the coefficient occurs because there are exactly \(|\bmu|^!\) elements \(\bsi\) which stabilize \(\hat \bmu\).
\end{proof}

\subsection{Asymptotic faithfulness and a basis for \(\AffWebAaI\)}

\begin{theorem}\label{defrepind77}
The family of functors $\{H_n\}_{n \geq 1}$ is asympotically faithful on $\AffWebAaI$.  That is, for any given morphism space in $\AffWebAaI$ the functor $H_{n}$ is faithful for $n \gg 0$.  
\end{theorem}
\begin{proof}
We have that \(\{ \eta_{\bmu} \mid \bmu \in \mathcal{M}(\bi^{(\bx)}, \bj^{(\by)})\}\) is a \(\k\)-spanning set for \(\AffWebAaI(\bi^{(\bx)}, \bj^{(\by)})\) by Corollary~\ref{spancor}. Letting \(0 \neq \rho \in \AffWebAaI(\bi^{(\bx)}, \bj^{(\by)})\), we have that 
\begin{align*}
\rho =  \sum_{\bnu \in \mathcal{M}(\bi^{(\bx)}, \bj^{(\by)})} c_{\bnu} \eta_{\bnu} \in \AffWebAaI(\bi^{(\bx)}, \bj^{(\by)})
\end{align*}
for some coefficients \(c_{\bnu} \in \k\), not all zero.
Let \(\bmu \in \mathcal{M}(\bi^{(\bx)}, \bj^{(\by)})\) be chosen such that \(c_{\bmu} \neq 0\) and \( \hat{\bnu} \lessdot \hat{\bmu} \) for all \(\bnu \in \mathcal{M}(\bi^{(\bx)}, \bj^{(\by)})\) with \(\bnu \neq \bmu\) and \(c_{\bnu} \neq 0\). Take
\begin{align*}
C = \prod_{
\substack{
\bmu \neq \bnu \in \mathcal{M}(\bi^{(\bx)}, \bj^{(\by)})\\
c_{\bnu} \neq 0
}
}
|\bnu|^!_{\a} \in \k.
\end{align*}
We have then by Lemma~\ref{explodebasis} that
\begin{align*}
C|\bmu|^!_{\a} \exp_{\bi^{(\bx)}, \bj^{(\by)}}(\rho)  &= C|\bmu|^!_{\a } \sum_{\bnu \in \mathcal{M}(\bi^{(\bx)}, \bj^{(\by)})} c_{\bnu} \exp_{\bi^{(\bx)}, \bj^{(\by)}} (\eta_{\bnu}) = C |\bmu|^! c_{\bmu} \eta_{\hat{\bmu}} + X,
\end{align*}
where \(X \in \AffWebAaI(\bi^{\bx}, \bj^{\by})_{\lessdot \hat{\bmu}}\). Thus Corollary~\ref{normordind77} implies that there exists some \(n \gg 0\) such that 
\begin{align*}
0 \neq H_n(C|\bmu|^!_{\a } \exp_{\bi^{(\bx)}, \bj^{(\by)}}(\rho)) = C |\bmu|^!_{\a } H_n(Y_{\bj^{(\by)}}) \circ H_n(\rho) \circ H_n(Z_{\bi^{(\bx)}}),
\end{align*}
and thus \(H_n(\rho) \neq 0\), completing the proof.
\end{proof}

We have the following corollary to (the proof of) Theorem~\ref{defrepind77}:

\begin{corollary}\label{affwebbasis}
The diagrams \(\{ \eta_{\bmu} \mid \bmu \in \mathcal{M}(\bi^{(\bx)}, \bj^{(\by)})\}\) form a \(\k\)-basis for \(\AffWebAaI(\bi^{(\bx)}, \bj^{(\by)})\).
\end{corollary}
\section{The cyclotomic quotient category \(\WebLam\)}\label{cycquosecmain}
In this section we define cyclotomic quotients of the affine Frobenius web category \(\AffWebAaI\), and establish a basis theorem for these objects.
Let \(\bL = (L_1, \ldots, L_t) \in \Z_{>0}^t\) and set \(\ell = L_1 + \cdots + L_t\). Let \(\bc = (c_{i,k})_{i \in I, k \in [1,t]}\) be such that
\begin{align}\label{cycdatreq}
c_{i,k} \in (iAi)_{\bar 0}
\qquad
\textup{and}
\qquad
c_{j,k} x = \psi^{-L_k}(x) c_{i,k}
\end{align}
for all \(k \in [1,t]\), \(i,j \in I\) and \(x \in jAi\). Roughly speaking, the requirement on the right side of (\ref{cycdatreq}) ensures that \(x\) ``commutes past the coupons \(\{c_{j,k}\}\) like it commutes past \(k\) affine dots''. 
We call \(\Lambda = (\bL, \bc)\) a {\em cyclotomic datum} of {\em level \(\ell\)}.

\subsection{The thin cyclotomic quotient category \(\WebLamThin\)}\label{degencycconn}
For \(i \in I\), set
\begin{align*}
{}
f_{i,k} = 
\hackcenter{
\begin{tikzpicture}[scale=.8]
  \draw[ultra thick, blue] (0,-0.1)--(0,1.1);
   \draw[thick, fill=black]  (0,0.5) circle (5pt);
     \node[right] at  (0.15,0.5)  {$\scriptstyle L_k$};
     \node[below] at (0,-0.1) {$ \scriptstyle i^{ \scriptstyle (1)}$};
      \node[above] at (0,1.1) {$ i^{ \scriptstyle (1)}$};
\end{tikzpicture}}
-
\hackcenter{
\begin{tikzpicture}[scale=.8]
  \draw[ultra thick, blue] (0,-0.1)--(0,1.1);
  \draw[thick, fill=yellow]  (0,0.5) circle (9pt);
       \node[] at (0.02,0.47) {$  \scriptstyle c_{i\hspace{-0.3mm},\hspace{-0.3mm}k}$};
     \node[below] at (0,-0.1) {$ \scriptstyle i^{ \scriptstyle (1)}$};
      \node[above] at (0,1.1) {$ i^{ \scriptstyle (1)}$};
\end{tikzpicture}}
\qquad
\textup{and}
\qquad
f_i = f_{i,t} \circ \cdots \circ f_{i, 3} \circ  f_{i, 2}\circ f_{i,1}
\end{align*}
Let \(\WebLamThin\) be the quotient of \(\WebAffThin\) by the right tensor ideal
\begin{align*}
\mathcal{I}^{\Lambda, \textup{thin}} := \big \langle f_i \mid i \in I \big \rangle.
\end{align*}
We will often make use of bars to indicate morphisms considered in the quotient category.
Under the isomorphism \(\mathcal{H}^\textup{aff}(A) \xrightarrow{\sim} \WebAffThin\) in \cref{affwreathisom}, one may view \(\WebLamThin\) as a special case (for a much stricter type of cyclotomic datum) of Savage's \emph{cyclotomic wreath quotient algebra} (see \cite[Section 6.2]{SavageAff}). 
Set
\begin{align*}
\mathcal{M}^\Lambda(\bi^{(x)},\bj^{(y)}) = \{\bmu \in \mathcal{M}(\bi^{(x)},\bj^{(y)}) \mid \mu_{r,s}(t,b) = 0 \textup{ for all }t\geq \ell\},
\end{align*}
In other words, \(\bmu \in \mathcal{M}^\Lambda(\bi^{(x)},\bj^{(y)})\) only if \(\eta_{\bmu}\) carries at most \(\ell -1\) affine dots on any given strand (for example, for \(\bmu\) as in \cref{cablex}, we would have \(\bmu \in \mathcal{M}^\Lambda(\bi^{(\bx)}, \bj^{(\by)})\) if and only if \(\ell \geq 3\).
The next result follows from \cref{affwreathisom} and \cite[Theorem 6.11]{SavageAff}:

\begin{lemma}\label{cycthinisombasis}
Let \(\bi, \bj \in I^d\). Then \(\{\bar{\eta}_{\bmu} \mid \bmu \in \mathcal{M}^\Lambda(\bi,\bj)\}\) is a \(\k\)-basis for \(\WebLamThin(\bi,\bj)\).
\end{lemma}

\subsection{The thick cyclotomic quotient category \(\WebLam\)}
We now endeavor to extend the \emph{thin} cyclotomic quotient category to the \emph{thick} setting in a sensible fashion. For us, `sensible' means that the resulting category should be \(\k\)-free with a natural basis which is a thick analogue of the basis in \cref{cycthinisombasis}. In order to do so we will require a certain regularity condition on the cyclotomic datum, which we now describe.

\begin{definition}
By \cref{cycthinisombasis}, we have that 
\begin{align}\label{kappaexpand}
\underbrace{
\hackcenter{
\begin{tikzpicture}[scale=.8]
  \draw[ultra thick, blue] (0,-0.1)--(0,1.1);
   \draw[thick, fill=black]  (0,0.5) circle (5pt);
     \node[right] at  (0.15,0.5)  {$\scriptstyle \ell$};
     \node[below] at (0,-0.1) {$ \scriptstyle i^{ \scriptstyle (1)}$};
      \node[above] at (0,1.1) {$ i^{ \scriptstyle (1)}$};
      \node[] at  (0.85,0.5)  {$\scriptstyle \cdots$};
      \draw[ultra thick, blue] (1.5,-0.1)--(1.5,1.1);
   \draw[thick, fill=black]  (1.5,0.5) circle (5pt);
     \node[right] at  (1.65,0.5)  {$\scriptstyle \ell$};
     \node[below] at (1.5,-0.1) {$ \scriptstyle i^{ \scriptstyle (1)}$};
      \node[above] at (1.5,1.1) {$ i^{ \scriptstyle (1)}$};
      \node[above] at  (0.75,1.1)  {$\scriptstyle \cdots$};
       \node[below] at  (0.75,-0.1)  {$\scriptstyle \cdots$};
\end{tikzpicture}
}
}_{d \textup{ times}}
= \sum_{\bmu\in \mathcal{M}^\Lambda(i^d,i^d)}\kappa^{\Lambda,d,i}_{\bmu } \overline{\eta}_{\bmu} \in \WebLamThin(i^d, i^d),
\end{align}
for some coefficients \(\kappa^{\Lambda, d,i}_{\bmu } \in \k \). 
We say that a cyclotomic datum \(\Lambda = (\bL, \bc)\) is {\em regular} provided that for all \(d \in \Z_{\geq 0}\) and \(i \in I\), there exists a morphism \(X^{\Lambda, d,i} \in \AffWebAaI(i^{(d)},i^{(d)})\) such that
\begin{align}\label{regcrit}
\sum_{\bmu\in \mathcal{M}^\Lambda(i^d,i^d)} \kappa^{\Lambda,d,i}_{\bmu }
\hackcenter{
\begin{tikzpicture}[scale=.8]
  \draw[ultra thick, blue] (0.75, -0.3)--(0.75, -0.2) .. controls ++(0,0.35) and ++(0,-0.35) .. (0,0.3)--(0,0.5)--(0,0.7) .. controls ++(0,0.35) and ++(0,-0.35) .. (0.75,1.2)--(0.75,1.3);
    \draw[ultra thick, blue] (0.75, -0.3)--(0.75, -0.2) .. controls ++(0,0.35) and ++(0,-0.35) .. (1.5,0.3)--(1.5,0.5)--(1.5,0.7) .. controls ++(0,0.35) and ++(0,-0.35) .. (0.75,1.2)--(0.75,1.3);
  \draw[draw=black, rounded corners, thick, fill=lime] (-0.25,0.3) rectangle ++(2,0.45);
     \node[] at  (0.75,0.5)  {$\scriptstyle \eta_{\bmu}$};
     \node[below] at (0.75,-0.3) {$ \scriptstyle i^{ \scriptstyle (d)}$};
       \node[above] at (0.75,1.3) {$ \scriptstyle i^{ \scriptstyle (d)}$};
         \node[below] at (0,0.1) {$ \scriptstyle i^{ \scriptstyle (1)}$};
         \node[below] at (1.5,0.1) {$ \scriptstyle i^{ \scriptstyle (1)}$};
      \node[] at  (0.75,0.85)  {$\scriptstyle \cdots$};
      \node[] at  (0.75,0.15)  {$\scriptstyle \cdots$};
\end{tikzpicture}
}
= d! X^{\Lambda, d,i} \in \AffWebAaI(i^{(d)},i^{(d)}),
\end{align}
where the coefficients \(\kappa_{\bmu}^{\Lambda, d,i}\) are as in \cref{kappaexpand}.
\end{definition}

\begin{remark}
Empirically speaking, many choices of cyclotomic datum wind up being regular. But this is not always the case; a counterexample arises for instance in the case of the truncated polynomial algebra \(A = A_{\bar 0} = \Z[x]/(x^4)\), \(\a = \Z\{1, x^2 \}\), as discussed in \cref{SS:Examples}, and cyclotomic datum \(\Lambda = (\bL, \bc)\) with \(\bL = (1)\), \(\bc = (x)\). It is easy to see that regularity fails when \(d=2\).
\end{remark}

\begin{definition}
Let \(\Lambda = (\bL,\bc)\) be a regular cyclotomic datum. For $d \in \Z_{> 0}$, \(i \in I\), set
\begin{align}\label{cycrelthick}
{}
f_{i}^{(d)} = 
\hackcenter{
\begin{tikzpicture}[scale=.8]
  \draw[ultra thick, blue] (0,-0.1)--(0,1.1);
   \draw[thick, fill=black]  (0,0.5) circle (5pt);
     \node[right] at  (0.15,0.5)  {$\scriptstyle \ell$};
     \node[below] at (0,-0.1) {$ \scriptstyle i^{ \scriptstyle (d)}$};
      \node[above] at (0,1.1) {$ i^{ \scriptstyle (d)}$};
\end{tikzpicture}}
-
X^{\Lambda,d,i} \in \AffWebAaI(i^{(d)}, i^{(d)}),
\end{align}
and define \(\WebLam\) to be the quotient of \(\AffWebAaI\) by the right tensor ideal
\begin{align*}
\mathcal{I}^\Lambda := \big \langle f_i^{(d)} \mid d \in \Z_{\geq 0}, i \in I \big \rangle.
\end{align*}
\end{definition}

\begin{proposition}\label{FaithThinCyc}
We have a faithful functor \(\iota^{\Lambda, \textup{thin}}: \WebLamThin \to \WebLam\) which is the identity on objects and diagrams.  
\end{proposition}
\begin{proof}
First, note that \(f_i = f_i^{(1)}\) for all \(i \in I\), so \(\mathcal{I}^{\Lambda, \textup{thin}} \subseteq \mathcal{I}^\Lambda\), and so the functor \(\iota^{\Lambda, \textup{thin}}\) is well-defined. We focus now on faithfulness.

Note that for any \(i \in I\), \(d \in \Z_{>0}\) we have
\begin{align*}
d!f_i^{(d)} = 
d!
\hackcenter{
\begin{tikzpicture}[scale=.8]
  \draw[ultra thick, blue] (0,-0.3)--(0,1.3);
   \draw[thick, fill=black]  (0,0.5) circle (5pt);
     \node[right] at  (0.15,0.5)  {$\scriptstyle \ell$};
     \node[below] at (0,-0.3) {$ \scriptstyle i^{ \scriptstyle (d)}$};
      \node[above] at (0,1.3) {$ i^{ \scriptstyle (d)}$};
\end{tikzpicture}}
-
d!X^{\Lambda,d,i} 
\;\;
\substack{
(\ref{KnotholeRel})\\
(\ref{AffDotRel1})\\
(\ref{regcrit})\\
=
\\
{}
\\
{}
\\
{}
}
\hackcenter{
\begin{tikzpicture}[scale=.8]
  \draw[ultra thick, blue] (0.75, -0.3)--(0.75, -0.2) .. controls ++(0,0.35) and ++(0,-0.35) .. (0,0.3)--(0,0.5)--(0,0.7) .. controls ++(0,0.35) and ++(0,-0.35) .. (0.75,1.2)--(0.75,1.3);
    \draw[ultra thick, blue] (0.75, -0.3)--(0.75, -0.2) .. controls ++(0,0.35) and ++(0,-0.35) .. (1.5,0.3)--(1.5,0.5)--(1.5,0.7) .. controls ++(0,0.35) and ++(0,-0.35) .. (0.75,1.2)--(0.75,1.3);
   \draw[thick, fill=black]  (0,0.5) circle (5pt);
     \node[right] at  (0.15,0.5)  {$\scriptstyle \ell$};
     \node[below] at (0.75,-0.3) {$ \scriptstyle i^{ \scriptstyle (d)}$};
       \node[above] at (0.75,1.3) {$ \scriptstyle i^{ \scriptstyle (d)}$};
         \node[below] at (0,0.1) {$ \scriptstyle i^{ \scriptstyle (1)}$};
         \node[below] at (1.5,0.1) {$ \scriptstyle i^{ \scriptstyle (1)}$};
      \node[] at  (0.85,0.5)  {$\scriptstyle \cdots$};
   \draw[thick, fill=black]  (1.5,0.5) circle (5pt);
     \node[right] at  (1.65,0.5)  {$\scriptstyle \ell$};
\end{tikzpicture}
}
-\sum_{\bmu\in \mathcal{M}^\Lambda(i^d,i^d)} \kappa^{\Lambda,d,i}_{\bmu } 
\hackcenter{
\begin{tikzpicture}[scale=.8]
  \draw[ultra thick, blue] (0.75, -0.3)--(0.75, -0.2) .. controls ++(0,0.35) and ++(0,-0.35) .. (0,0.3)--(0,0.5)--(0,0.7) .. controls ++(0,0.35) and ++(0,-0.35) .. (0.75,1.2)--(0.75,1.3);
    \draw[ultra thick, blue] (0.75, -0.3)--(0.75, -0.2) .. controls ++(0,0.35) and ++(0,-0.35) .. (1.5,0.3)--(1.5,0.5)--(1.5,0.7) .. controls ++(0,0.35) and ++(0,-0.35) .. (0.75,1.2)--(0.75,1.3);
  \draw[draw=black, rounded corners, thick, fill=lime] (-0.25,0.3) rectangle ++(2,0.45);
     \node[] at  (0.75,0.5)  {$\scriptstyle \eta_{\bmu}$};
     \node[below] at (0.75,-0.3) {$ \scriptstyle i^{ \scriptstyle (d)}$};
       \node[above] at (0.75,1.3) {$ \scriptstyle i^{ \scriptstyle (d)}$};
         \node[below] at (0,0.1) {$ \scriptstyle i^{ \scriptstyle (1)}$};
         \node[below] at (1.5,0.1) {$ \scriptstyle i^{ \scriptstyle (1)}$};
      \node[] at  (0.75,0.85)  {$\scriptstyle \cdots$};
      \node[] at  (0.75,0.15)  {$\scriptstyle \cdots$};
\end{tikzpicture}
}
=
\con_{i^{(d)}, i^{(d)}}(F_i^{(d)}),
\end{align*}
where \(F_i^{(d)} \in \WebAffThin(i^d,i^d)\) and \(\overline{F}_i^{(d)} = 0 \) in \(\WebLamThin\) by \cref{kappaexpand}, so \(F_i^{(d)} \in \mathcal{I}^{\Lambda, \textup{thin}}\).

Now assume \(n \geq d\), \(\bj = j_1 \cdots j_n, \bk = k_1 \cdots k_n \in I^n\). We consider any morphism of the form
\begin{align}\label{Q3one}
Q_3 \circ (f_i^{(d)} \otimes Q_2) \circ Q_1 \in \WebAffThin(\bj, \bk),
\end{align}
where \(Q_1 \in \AffWebAaI(\bj, i^{(d)} \bp^{(\bx)})\), \(Q_2 \in \AffWebAaI(\bp^{(\bx)}, \bq^{(\by)})\), and \(Q_3 \in \AffWebAaI(i^{(d)}\bq^{(\by)}, \bk)\) for some \(\bp^{(\bx)}, \bq^{(\by)} \in \widehat{\Omega}\), noting that \(\WebAffThin(\bj, \bk) \cap \mathcal{I}^\Lambda\) is \(\k\)-spanned by such morphisms by definition. By the above paragraph we have
\begin{align*}
d!\;Q_3 \circ (f_i^{(d)} \otimes Q_2) \circ Q_1 &=Q_3 \circ (\con_{i^{(d)}, i^{(d)}}(F_i^{(d)}) \otimes Q_2) \circ Q_1 \\
&= Q_3' \circ ((\exp_{i^{(d)}, i^{(d)}} \circ \con_{i^{(d)}, i^{(d)}}(F_i^{(d)})) \otimes\exp_{\bp^{(x)}, \bq^{(y)}}(Q_2))) \circ Q_1' 
\end{align*}
for some \(Q_1' \in \WebAffThin(\bj, i^d \bp^{\bx})\) and \(Q_3' \in \WebAffThin(i^d \bq^{\by}, \bk)\), thanks to Corollary~\ref{spancor}. Now, by Lemma~\ref{bigwishtosym} and Corollary~\ref{spancor}, the above is equal to
\begin{align}\label{Q7one}
\sum_{\sigma, \omega \in \mathfrak{S}_d} Q_3' \circ (( \sigma \circ F_i^{(d)} \circ \omega) \otimes Q_2') \circ Q_1' 
\end{align}
for some \(Q_2' \in \WebAffThin(\bp^{\bx}, \bq^{\by})\). Now note that each constituent morphism in (\ref{Q7one}) is in \(\WebAffThin\), and therefore (\ref{Q7one}) is in \(\mathcal{I}^{\Lambda, \textup{thin}}\) since \(F_i^{(d)} \in \mathcal{I}^{\Lambda, \textup{thin}}\). Thus \(d!Q_3 \circ (f_i^{(d)} \otimes Q_2) \circ Q_1  \in \mathcal{I}^{\Lambda, \textup{thin}}\), so \(Q_3 \circ (f_i^{(d)} \otimes Q_2) \circ Q_1  \in \mathcal{I}^{\Lambda, \textup{thin}}\) since \(\WebLamThin\) is torsion-free by Lemma~\ref{cycthinisombasis}.

Now, let \(\rho \in \AffWebAaI(\bj, \bk)\), and assume \(\iota^{\Lambda, \textup{thin}}(\bar \rho) = 0\). We have then that \(\rho \in \mathcal{I}^\Lambda\). Since \(\AffWebAaI(\bj, \bk) \cap \mathcal{I}^\Lambda\) is spanned by morphisms of the form (\ref{Q3one}), we have then that \(\rho \in \mathcal{I}^{\Lambda, \textup{thin}}\) by the above paragraph. Thus \(\bar{\rho} = 0 \in \WebLamThin(\bj, \bk)\) initially, so \(\iota^{\Lambda, \textup{thin}}\) is faithful, as required.
\end{proof}

\begin{proposition}\label{cycwebbasisfull}
We have that 
\(
\{\overline{\eta}_{\bmu} \mid \bmu \in \mathcal{M}^\Lambda(\bi^{(\bx)}, \bj^{(\by)})\}
\)
is a \(\k\)-basis for \(\WebLam(\bi^{(\bx)},\bj^{(\by)})\).
\end{proposition}
\begin{proof}
First we prove spanning.
We have by \cref{spancor} that \(\mathcal{M}(\bi^{(\bx)}, \bj^{(\by)})\) is a spanning set for \(\WebLam(\bi^{(\bx)}, \bj^{(\by)})\). Now, assume a diagram \(D \in \AffWebAaI(\bi^{(\bx)}, \bj^{(\by)})\) has a strand which carries \(\geq \ell\) affine dots. We will show that we may write \(\overline{D} \in \WebLam(\bi^{(\bx)}, \bj^{(\by)})\) as a \(\k\)-linear combination of diagrams with lesser overall affine degree, and then the claim follows by induction.
Indeed, using \cref{DotCrossRel}, we may pull the strand with \(\geq\ell\) dots all the way to the left (modulo a \(\k\)-linear combination of diagrams with lesser affine degree), and then use the cyclotomic relation \cref{cycrelthick} to rewrite this diagram as a \(\k\)-linear combination of diagrams with lesser affine degree, completing the spanning proof.

Now we prove independence. 
Assume \( \rho=\sum_{\bmu \in \mathcal{M}^{\Lambda}(\bi^{(\bx)}, \bj^{(\by)})} \gamma_{\bmu} \eta_{\bmu} \in \AffWebAaI(\bi^{(\bx)}, \bj^{(\by)})\) is a \(\k\)-linear combination such that not all coefficients \(\gamma_{\bmu}\) are zero. Let \(\bnu\) be such that \(\hat{\bnu}\) is maximally \(\lessdot\)-dominant with \(\gamma_{\bnu} \neq 0\). Let \(K = \prod_{\bmu \in \mathcal{M}^\Lambda(\bi^{(\bx)}, \bj^{(\by)})} |\bmu|^!\).
Then by \cref{explodebasis} we have
\begin{align*}
C \overline{\exp_{\bi^{(\bx)}, \bj^{(\by)}} (\rho)}- C' \overline{\eta}_{\hat{\bnu}} \in \Web(\bi^{\bx}, \bj^{\by})_{\lessdot \hat{\bmu}}^{\Lambda, \textup{thin}}
\end{align*}
for some \(C' \in \Z_{\neq 0}\). By \cref{cycthinisombasis} we have then that \(\overline{\exp_{\bi^{(\bx)}, \bj^{(\by)}} (\rho)} \) is nonzero in \(\WebLamThin\), and thus \(\bar \rho\) is nonzero in \(\WebLam\) by \cref{FaithThinCyc}, completing the proof.
\end{proof}

\cref{cycwebbasisfull} implies the immediate
\begin{corollary}
Let \(\Lambda\) be a regular cyclotomic datum of level \(\ell = 1\). Then we have an isomorphism of categories \(\WebAaI \xrightarrow{\sim} \WebLam\) which is the identity on objects and diagrams.
\end{corollary}

\subsection{Regularity and the thick cyclotomic quotient category in special cases}
In general, \emph{regularity} of a cyclotomic datum \(\Lambda = (\bL, \bc)\) and the explicit definition of the associated quotient ideal \(\mathcal{I}^\Lambda\) of \(\WebLam\) requires direct computation. In this section we outline some cases where regularity can be established without computation.

\subsubsection{Regularity over a characteristic zero field}\label{reg1} If \(\k\) is a field of characteristic zero, then every morphism in \(\AffWebAaI(i^{(d)}, i^{(d)})\) is divisible by \(d!\), and hence any choice of cyclotomic datum \(\Lambda\) is regular and so \(\WebLam\) is well-defined. Indeed, in this setting we have \(\mathcal{I}^{\Lambda, \textup{thin}} = \mathcal{I}^\Lambda\) as may be seen in the proof of \cref{FaithThinCyc}.

\subsubsection{Regularity under some nice idempotent conditions}\label{reg2}
In the following proposition, we show that for various important classes of algebras \(A\) (see \cref{remzigreg}), a natural choice of cyclotomic datum is regular, and we directly describe the associated quotient ideal (note that the below proposition places no extra conditions on the underlying ring \(\k\) or the even subalgebra \(\a\)).

\begin{proposition}\label{pureidemprop}
Assume \(A\) is such that \((i^\vee)^2 = 0\) and \( i^\vee b= \delta_{b,i} i^\vee\) for all \(b \in \BasisB\), \(i \in I\), and \(\Lambda = (\bL, \bc)\) is such that \(L_k = 1\) and \(c_{i,k} = \gamma_{i,k} i^\vee\) for all \(i \in I, k \in [1,t]\), and some integers \(\gamma_{i,k} \in \Z \subseteq \k\). Then \(\Lambda\) is regular, and the defining quotient ideal of \(\WebLam\) is
\begin{align*}
\mathcal{I}^\Gamma = 
\Bigg \langle
\hackcenter{
\begin{tikzpicture}[scale=.8]
  \draw[ultra thick, blue] (0,-0.3)--(0,1.8);
   \draw[thick, fill=black]  (0,0.8) circle (5pt);
     \node[left] at  (-0.15,0.8)  {$\scriptstyle \ell$};
     \node[below] at (0,-0.3) {$ \scriptstyle i^{ \scriptstyle (d)}$};
      \node[above] at (0,1.8) {$ i^{ \scriptstyle (d)}$};
\end{tikzpicture}}
-
\;
{ \gamma_i + d -1 \choose d}
\hspace{-1mm}
\hackcenter{
\begin{tikzpicture}[scale=.8]
  \draw[ultra thick, blue] (0.75, -0.3)--(0.75, -0.2) .. controls ++(0,0.35) and ++(0,-0.35) .. (0,0.3)--(0,0.5)--(0,1.2) .. controls ++(0,0.35) and ++(0,-0.35) .. (0.75,1.7)--(0.75,1.8);
    \draw[ultra thick, blue] (0.75, -0.3)--(0.75, -0.2) .. controls ++(0,0.35) and ++(0,-0.35) .. (1.5,0.3)--(1.5,0.5)--(1.5,1.2) .. controls ++(0,0.35) and ++(0,-0.35) .. (0.75,1.7)--(0.75,1.8);
    \draw[thick, fill=yellow]  (0,1) circle (9pt);
       \node[] at (0,1) {$  \scriptstyle i^\vee$};
           \draw[thick, fill=yellow]  (1.5,1) circle (9pt);
       \node[] at (1.5,1) {$  \scriptstyle i^\vee$};
     \node[below] at (0.75,-0.3) {$ \scriptstyle i^{ \scriptstyle (d)}$};
       \node[above] at (0.75,1.8) {$ \scriptstyle i^{ \scriptstyle (d)}$};
         \node[below] at (0,0.1) {$ \scriptstyle i^{ \scriptstyle (1)}$};
         \node[below] at (1.5,0.1) {$ \scriptstyle i^{ \scriptstyle (1)}$};
      \node[] at  (0.75,1)  {$\scriptstyle \cdots$};
      \node[] at  (0.75,0.5)  {$\scriptstyle \cdots$};
         \draw[thick, fill=black]  (0,0.4) circle (5pt);
              \node[right] at (1.6,0.4) {$  \scriptstyle \ell-1$};
            \draw[thick, fill=black]  (1.5,0.4) circle (5pt);
              \node[left] at (-0.1,0.4) {$  \scriptstyle \ell-1$};
\end{tikzpicture}
}
\Bigg|
\;\;
i \in I, d \in \Z_{>0}
\Bigg \rangle
\end{align*}
where \(\gamma_i = \gamma_{i,1} + \cdots + \gamma_{i,t}\).
\end{proposition}
\begin{proof}
The condition \((i^\vee)^2 = 0\) ensures that
\begin{align}\label{fiexpzig}
{}
f_i = 
\hackcenter{
\begin{tikzpicture}[scale=.8]
  \draw[ultra thick, blue] (0,-0.1)--(0,1.3);
   \draw[thick, fill=black]  (0,0.65) circle (5pt);
     \node[right] at  (0.15,0.65)  {$\scriptstyle \ell$};
     \node[below] at (0,-0.1) {$ \scriptstyle i^{ \scriptstyle (1)}$};
      \node[above] at (0,1.3) {$ i^{ \scriptstyle (1)}$};
\end{tikzpicture}}
-
\gamma_i
\hackcenter{
\begin{tikzpicture}[scale=.8]
  \draw[ultra thick, blue] (0,-0.1)--(0,1.3);
  \draw[thick, fill=yellow]  (0,0.85) circle (9pt);
       \node[] at (0,0.85) {$  \scriptstyle i^\vee$};
        \draw[thick, fill=black]  (0,0.25) circle (5pt);
        \node[right] at  (0.15,0.25)  {$\scriptstyle \ell - 1$};
     \node[below] at (0,-0.1) {$ \scriptstyle i^{ \scriptstyle (1)}$};
      \node[above] at (0,1.3) {$ i^{ \scriptstyle (1)}$};
\end{tikzpicture}}
\end{align}
for all \(i \in I\). The condition \(i^\vee b= \delta_{b,i} i^\vee\) ensures that \(\sum_{b \in \BasisB} i^\vee b \otimes b^\vee = i^\vee \otimes i^\vee = \sum_{b \in \BasisB} b \otimes b^\vee  i^\vee\), and we will use this fact repeatedly in the computations that follow.

The {\em absolute length} \(u(\sigma)\) of a permutation \(\sigma \in \mathfrak{S}_d\) is the minimum number of (not necessarily simple) transpositions whose product is \(\sigma\). It is a well-known fact that for \(x \in \Z\), we have
\begin{align}\label{binomeq}
\sum_{\sigma \in \mathfrak{S}_d} x^{d - u(\sigma)} = x(x + 1) \cdots (x + d-1) = d! { x + d -1 \choose d}.
\end{align}

Now we prove by induction on \(d\) that the following claim holds in \(\WebLamThin(i^d, i^d)\):
\begin{align}\label{kappaexpand2}
\hackcenter{

}
\end{align}
The base case \(d=1\) is given by \cref{fiexpzig}. For the induction step we have 
%
which yields the claim (\ref{kappaexpand2}).
The proposition statement now follows from (\ref{kappaexpand2}, \ref{cycrelthick}), \cref{CrossAbsorb} and (\ref{binomeq}).
\end{proof}

\begin{remark}\label{remzigreg}
In many cases of interest (such as the zigzag algebras (see \cref{zigzagcat}), or trivial extension algebras of quiver algebras more generally (see \cref{CrossAbsorb})), the algebra \(A\) is non-negatively graded, with the degree-zero component spanned by the orthogonal idempotents \(I\), and with the trace map having degree \(-m\) for some \(m \in \Z_{>0}\). In this setting \cref{pureidemprop} gives a direct construction of most natural choices of cyclotomic quotient categories \(\WebLam\).

Indeed, in this setting the top degree \(m\) component of \(A\) is spanned by the dual elements \(I^\vee\), so \(A\) satisfies the first condition in \cref{pureidemprop}. Moreover, in this setting, \(\AffWebAaI\) inherits a natural grading as well, taking the \(x\)-thick affine dot to have degree \(xm\). Thus if one wants to construct a cyclotomic quotient \(\WebLam\) which preserves this grading, the only coherent choices of cyclotomic datum are as presented in \cref{pureidemprop}.
\end{remark}

\subsubsection{Regularity in the \(\a = A_{\bar 0}\) case}\label{reg3}
 In this section we prove that when \(\a = A_{\bar 0}\), {\em every} choice of cyclotomic datum \(\Lambda = (\bL, \bc)\) is regular, and thus \(\WebLam\) is well-defined. Understanding the explicit definition of the cyclotomic ideal \(\mathcal{I}^\Gamma\) still requires computation in general. 

Throughout this subsection we assume \(\a = A_{\bar 0}\). Note that the symmetric group \(\mathfrak{S}_d\) acts on the left of \(W_{d,i}:=\WebAffThin(i^d,i^d)\) by conjugating diagrams:
\begin{align*}
\\
\omega \cdot
\hackcenter{
\begin{overpic}[height=18mm]{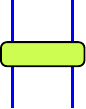}
  \put(12,-1){\makebox(0,0)[t]{$\scriptstyle i^{(1)}$}}
   \put(67,-1){\makebox(0,0)[t]{$\scriptstyle i^{(1)}$}}      
     \put(12,101){\makebox(0,0)[b]{$\scriptstyle i^{(1)}$}}
   \put(67,101){\makebox(0,0)[b]{$\scriptstyle i^{(1)}$}}       
                         \put(40,10){\makebox(0,0)[]{$\scriptstyle \cdots$}}   
                          \put(40,90){\makebox(0,0)[]{$\scriptstyle \cdots$}}   
        \put(42,50){\makebox(0,0)[]{$\scriptstyle D$}}   
\end{overpic}
}
\;\;
=
\;\;
\hackcenter{
\begin{overpic}[height=18mm]{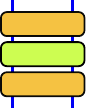}
  \put(12,-1){\makebox(0,0)[t]{$\scriptstyle i^{(1)}$}}
   \put(67,-1){\makebox(0,0)[t]{$\scriptstyle i^{(1)}$}}      
     \put(12,101){\makebox(0,0)[b]{$\scriptstyle i^{(1)}$}}
   \put(67,101){\makebox(0,0)[b]{$\scriptstyle i^{(1)}$}}       
                         \put(40,-1){\makebox(0,0)[]{$\scriptstyle \cdots$}}   
                          \put(40,101){\makebox(0,0)[]{$\scriptstyle \cdots$}}   
        \put(40,50){\makebox(0,0)[]{$\scriptstyle D$}}   
                \put(40,22){\makebox(0,0)[]{$\scriptstyle \omega$}}  
        \put(40,78){\makebox(0,0)[]{$\scriptstyle \omega^{-1}$}}  
\end{overpic}
}
\\
\end{align*}
We write \(W_{d,i}^{\mathfrak{S}_d}\) for the \(\mathfrak{S}_d\)-invariants in \(W_{d,i}\). By \cref{normordbasis77} we have that \(W_{d,i}\) has \(\k\)-basis given by normally-ordered diagrams of the form
\begin{align*}
\\
\hackcenter{}
\eta_{([\bt, \bb], \sigma)}:=
\hackcenter{
\begin{overpic}[height=15mm]{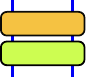}
  \put(15,-1){\makebox(0,0)[t]{$\scriptstyle i^{(1)}$}}
   \put(85,-1){\makebox(0,0)[t]{$\scriptstyle i^{(1)}$}}      
     \put(15,92){\makebox(0,0)[b]{$\scriptstyle i^{(1)}$}}
   \put(85,92){\makebox(0,0)[b]{$\scriptstyle i^{(1)}$}}       
                         \put(50,-1){\makebox(0,0)[]{$\scriptstyle \cdots$}}   
                          \put(50,101){\makebox(0,0)[]{$\scriptstyle \cdots$}}   
        \put(50,62){\makebox(0,0)[]{$\scriptstyle \sigma$}}   
                \put(50,27){\makebox(0,0)[]{$\scriptstyle [\bt, \bb]$}}  
\end{overpic}
}
=
\hackcenter{
\begin{overpic}[height=15mm]{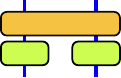}
  \put(15,-1){\makebox(0,0)[t]{$\scriptstyle i^{(1)}$}}
   \put(85,-1){\makebox(0,0)[t]{$\scriptstyle i^{(1)}$}}      
     \put(15,65){\makebox(0,0)[b]{$\scriptstyle i^{(1)}$}}
   \put(85,65){\makebox(0,0)[b]{$\scriptstyle i^{(1)}$}}       
                         \put(50,-1){\makebox(0,0)[]{$\scriptstyle \cdots$}}   
                          \put(50,65){\makebox(0,0)[]{$\scriptstyle \cdots$}}   
        \put(50,45){\makebox(0,0)[]{$\scriptstyle \sigma$}}   
                \put(20,20){\makebox(0,0)[]{$\scriptstyle [t_1,b_1]$}}  
                \put(80,20){\makebox(0,0)[]{$\scriptstyle [t_d,b_d]$}}  
\end{overpic}
}
\\
\end{align*}
where \(\sigma \in \mathfrak{S}_d\) and \([\bt, \bb] = ([t_1, b_1], \ldots, [t_d,b_d]) \in {}_i\BasisP_i^d\), 
 indexing the basis morphisms for \(W_{d,i}\) using the nomenclature in \cref{OrderSecTemp}.
We write \(Y_{d,i}:=\{([\bt, \bb], \sigma) \mid [\bt, \bb] \in {}_i\BasisP_i^d, \sigma \in \mathfrak{S}_d\}\) for the indexing set of these normally-ordered diagrams. (We of course have \(\eta_{([\bt, \bb], \sigma)} = \eta_{\bmu}\) for some \(\bmu \in \mathcal{M}(i^d,i^d)\), but in this section it will be more convenient to index via \(Y_{d,i}\)).

 There is a left \(\mathfrak{S}_d\)-action on \(\pm Y_{d,i}\) by 
\begin{align*}
([\bt, \bb], \sigma)^\omega:= (-1)^{\langle [\bt, \bb] ; \omega\rangle}([\bt, \bb]^\omega, \sigma^\omega) = (-1)^{\langle \bb ; \omega\rangle}(([t_{\omega 1}, b_{\omega 1}], \ldots, [t_{\omega d}, b_{\omega d}], \omega \sigma \omega^{-1}),
\end{align*}
where
\begin{align*}
\langle \bb; \omega\rangle = \sum_{\substack{r< s \\ \omega r > \omega s}} \bar b_r \bar b_s.
\end{align*}
For \(([\bt, \bb], \sigma) \in Y_{d,i}\), write:
\begin{align*}
\mathfrak{S}_{([\bt, \bb], \sigma)} &:= \{\omega \in \mathfrak{S}_d \mid ([\bt, \bb], \sigma)^\omega = ([\bt, \bb], \sigma)\}\\
\mathfrak{S}^\pm_{([\bt, \bb], \sigma)} &:= \{\omega \in \mathfrak{S}_d \mid ([\bt, \bb], \sigma)^\omega = \pm([\bt, \bb], \sigma)\}\\
\mathcal{O}^\pm_{([\bt, \bb], \sigma)} &:= \{([\bt, \bb]', \sigma') \mid ([\bt, \bb], \sigma)^\omega = \pm ([\bt, \bb]', \sigma')\}.
\end{align*}
We designate by \(\mathscr{D}_{([\bt, \bb], \sigma)}\) a set of minimal coset representatives for \(\mathfrak{S}_{([\bt, \bb], \sigma)} \backslash \mathfrak{S}_d\), and we designate by \(\mathscr{D}^\pm_{([\bt, \bb], \sigma)}\) a set of minimal coset representatives for \(\mathfrak{S}_{([\bt, \bb], \sigma)}^\pm \backslash \mathfrak{S}_d\). 

We choose an ordering \(\succ\) on \(Y_{d,i}\) such that \(([\bt, \bb], \sigma) \succ ([\bt, \bb]', \sigma')\) provided:
\begin{enumerate}
\item \(t_1 + \cdots + t_d > t_1' + \cdots +t_d'\) (i.e. \((([\bt, \bb], \sigma)\) has higher affine degree than \(([\bt, \bb]', \sigma')\)), or;
\item \(t_1 + \cdots + t_d = t_1' + \cdots +t_d'\), and \([\bt, \bb] > [\bt, \bb]'\) in the lexicographic order induced by the total order on \({}_i \BasisP_i\), or;
\item \([\bt, \bb] = [\bt, \bb]'\) and \(\omega > \omega'\).
\end{enumerate}
Note that \(\prec\) is slightly different than the order \(\lessdot\) on \(Y_{d,i}\) defined in \cref{OrderSecTemp}, and is more useful for the purposes of this section.
Under the ordering \(\prec\), every orbit \(\mathcal{O}^\pm_{([\bt, \bb], \sigma)}\) has a unique dominant member. We write 
\begin{align*}
\hat{Y}_{d, i}^\pm = \{([\bt, \bb], \sigma) \in Y_{d,i} \mid ([\bt, \bb], \sigma) \textup{ is \(\prec\)-dominant in }\mathcal{O}_{([\bt, \bb], \sigma)}^\pm\}.
\end{align*}
Note that if \(([\bt, \bb], \sigma) \in \hat{Y}_{d, i}^\pm\), then \([\bt, \bb] = ([t_{j_1}, b_{j_1}]^{m_1}, \ldots, [t_{j_q}, b_{j_q}]^{m_q})\) for some \(q, m_1, \ldots, m_q > 0\), where \([t_{j_r}, b_{j_r}] > [t_{j_s}, b_{j_s}]\) if \(r < s\). Writing \(\mathfrak{S}_{[\bt, \bb]} := \mathfrak{S}_{m_1} \times \cdots \times \mathfrak{S}_{m_q}\), note that we have \(\mathfrak{S}_{([\bt, \bb], \sigma)} \leq \mathfrak{S}_{[\bt, \bb]}\). 
Finally, define
\begin{align*}
\hat{Y}_{d,i} := \{([\bt, \bb], \sigma) \in \hat{Y}_{d,i}^\pm \mid \mathfrak{S}_{([\bt, \bb], \sigma)} = \mathfrak{S}_{([\bt, \bb], \sigma)}^\pm\}.
\end{align*}

\begin{lemma}\label{L0A}
Let \(([\bt, \bb], \sigma) \in Y_{d,i}\).  
\begin{enumerate}
\item We have that \(\omega^{-1} \eta_{([\bt, \bb], \sigma)} \omega = (-1)^{\langle [\bt, \bb]; \omega \rangle} \eta_{([\bt, \bb]^{\omega}, \sigma^\omega)}\) modulo terms of lower affine degree.
\item We have \(\omega^{-1} \eta_{([\bt, \bb], \sigma)} \omega = \eta_{([\bt, \bb], \sigma)}\) if \(([\bt, \bb], \sigma) \in \hat{Y}_{d,i}\) and \(\omega \in \mathfrak{S}_{([\bt, \bb], \sigma)}\).
\end{enumerate}
\end{lemma}
\begin{proof}
(1) is obvious from relations in \(W_{d,i}\). For (2), note that if \(([\bt, \bb], \sigma) \in \hat{Y}_{d,i}\), then \(\mathfrak{S}_{([\bt, \bb], \sigma)} < \mathfrak{S}_{[\bt, \bb]}\), so by \cref{righttoleftthin} we have 
\begin{align*}
\omega^{-1} \eta_{([\bt, \bb], \sigma)} \omega  &= \omega^{-1} \eta_{(i^d, \sigma)}\eta_{([\bt, \bb], \textup{id})} \omega =
 \omega^{-1} \eta_{(i^d, \sigma)} \omega \circ \omega^{-1} \eta_{([\bt, \bb], \textup{id})} \omega\\
 &=  \eta_{(i^d, \sigma)} \circ (-1)^{\langle [\bt, \bb]; \omega \rangle} \eta_{([\bt, \bb], \textup{id})} = \eta_{([\bt, \bb], \sigma)},
\end{align*}
where \((-1)^{\langle ([\bt, \bb], \sigma) \rangle} = 1\) since \(\mathfrak{S}_{([\bt, \bb], \sigma)} = \mathfrak{S}_{([\bt, \bb], \sigma)}^\pm\).
\end{proof}

Now, for \(([\bt, \bb], \sigma) \in \hat{Y}_{d,i}\), we define:
\begin{align*}
\xi_{([\bt, \bb], \sigma)} = \sum_{\tau \in \mathscr{D}_{([\bt, \bb], \sigma)}} \tau^{-1} \eta_{([\bt, \bb], \sigma)} \tau \in W_{d,i}.
\end{align*}

\begin{lemma}\label{L0B}
For all \(([\bt, \bb], \sigma) \in \hat{Y}_{d,i}\), we have that \(\xi_{([\bt, \bb], \sigma)} \in W_{d,i}^{\mathfrak{S}_d}\) and \(\xi_{([\bt, \bb], \sigma)} - \eta_{([\bt, \bb], \sigma)} \) is a \(\k\)-linear combination of terms of the form \(\eta_{([\bt, \bb]', \sigma')}\), where \(([\bt, \bb]', \sigma') \prec ([\bt, \bb], \sigma)\).
\end{lemma}
\begin{proof}
Let \(\omega \in \mathfrak{S}_d\). For each \(\tau \in \mathscr{D}_{([\bt, \bb], \sigma)}\), there exists \(\tau_{\omega} \in \mathscr{D}_{([\bt, \bb], \sigma)}\), \(\nu_{\omega} \in \mathfrak{S}_{([\bt, \bb], \sigma)}\) such that \(\tau \omega = \nu_{\omega} \tau_{\omega}\). Then we have
\begin{align*}
\omega^{-1} \xi_{([\bt, \bb], \sigma)} \omega
&= 
\sum_{\tau \in \mathscr{D}_{([\bt, \bb], \sigma)}}
\omega^{-1} \tau^{-1}
\eta_{([\bt,\bb], \sigma)}
\tau \omega
=
\sum_{\tau \in \mathscr{D}_{([\bt, \bb], \sigma)}}
\tau_{\omega}^{-1} \nu_{\omega}^{-1}
\eta_{([\bt,\bb], \sigma)} 
\nu_{\omega} \tau_{\omega}\\
&=
\sum_{\tau \in \mathscr{D}_{([\bt, \bb], \sigma)}}
\tau_{\omega}^{-1}
\eta_{([\bt,\bb], \sigma)} 
 \tau_{\omega}
 =
 \xi_{([\bt, \bb], \sigma)},
\end{align*}

where the last equality follows from \cref{L0A}(ii). This proves the first claim. For the second claim, note that by \cref{L0A}(i), we have the following equalities modulo the \(\k\)-span of terms of the form \(\eta_{([\bt, \bb]', \sigma')}\) with lower affine degree than \(([\bt, \bb], \sigma)\):
\begin{align*}
\xi_{([\bt, \bb], \sigma)}
&\equiv \sum_{\tau \in \mathscr{D}_{([\bt, \bb], \sigma)}} (-1)^{\langle [\bt, \bb]; \tau \rangle} \eta_{([\bt, \bb]^\tau, \sigma^\tau)} 
\equiv \sum_{\tau \in \mathscr{D}^\pm_{([\bt, \bb], \sigma)}} (-1)^{\langle [\bt, \bb]; \tau \rangle} \eta_{([\bt, \bb]^\tau, \sigma^\tau)} \\
&\equiv \eta_{([\bt, \bb], \sigma)} + \sum_{\substack{([\bt, \bb]', \sigma') \in \mathcal{O}_{([\bt, \bb], \sigma)}^\pm \\ ([\bt, \bb]', \sigma') \neq ([\bt, \bb], \sigma)}} c_{([\bt, \bb]', \sigma')} \eta_{([\bt, \bb]', \sigma')},
\end{align*}
for some \(c_{([\bt, \bb]', \sigma')} \in \k\). Since \(([\bt, \bb], \sigma) \in \hat{Y}_{d,i}\), it is the most dominant member of \(\mathcal{O}_{([\bt, \bb], \sigma)}^\pm\), proving the claim.
\end{proof} 

\begin{lemma}\label{L0C}
We have \(W_{d,i}^{\mathfrak{S}_d} = \k\{ \xi_{([\bt, \bb], \sigma)} \mid ([\bt, \bb], \sigma) \in \hat{Y}_{d,i}\}\).
\end{lemma}
\begin{proof}
Let \(z \in W_{d,i}^{\mathfrak{S_d}}\). By \cref{L0B} it is enough to show that \(z \in \k\{ \xi_{([\bt, \bb], \sigma)} \mid ([\bt, \bb], \sigma) \in \hat{Y}_{d,i}\}\). By \cref{normordbasis77} we may write \(z = \sum_{([\bt, \bb], \sigma) \in S} p_{([\bt, \bb], \sigma)} \eta_{([\bt, \bb], \sigma)} + q\), where \(S \subseteq Y_{d,i}\) is a subset of elements of equal maximal affine degree, and \(q\) is a sum of terms of lower affine degree.

{\em Claim I: If \(([\bt, \bb], \sigma)\) is such that \(p_{([\bt, \bb], \sigma)} \neq 0\), then \(\mathfrak{S}_{([\bt, \bb], \sigma)} = \mathfrak{S}_{([\bt, \bb], \sigma)}^\pm\). } To see this, assume not. Then there exists \(([\bt, \bb]', \sigma')\) with \(p_{([\bt, \bb]', \sigma')} \neq 0\) and \(\omega \in \mathfrak{S}^\pm_{([\bt, \bb]', \sigma')} \backslash \mathfrak{S}_{([\bt, \bb]', \sigma')}\). By \cref{L0A}(i), \(\omega^{-1} z \omega \equiv \sum p_{([\bt, \bb], \sigma)} (-1)^{\langle [\bt, \bb]; \omega\rangle} \eta_{([\bt, \bb]^\omega, \sigma^\omega)}\) modulo terms of lower affine degree. Now, \(([\bt, \bb]', \sigma') = (([\bt, \bb]')^\omega, (\sigma')^\omega)\), so \(\eta_{([\bt, \bb]', \sigma')}\) arises with coefficient \(p_{([\bt, \bb]', \sigma')}\) in \(z\), and with coefficient \((-1)^{\langle [\bt, \bb]', \omega' \rangle}p_{([\bt, \bb]', \sigma')} = -p_{([\bt, \bb]', \sigma')} \) in \(\omega^{-1} z \omega\). But \(z = \omega^{-1} z \omega\), giving the desired contradiction and proving Claim I.

Now, let \(([\bu, \bd], \epsilon) \in S\) be \(\prec\)-maximal such that \(p_{([\bu, \bd], \epsilon)} \neq 0\).\\
\indent{\em Claim II: We have \(([\bu, \bd], \epsilon) \in \hat{Y}_{d,i}\).} To see this, assume not. Then there exists a more \(\prec\)-dominant \(([\bu, \bd]^\omega, \epsilon^\omega) \in \mathcal{O}_{([\bu, \bd], \epsilon)}^\pm\). Then by \cref{L0A}(i), we have that \(\omega^{-1}z \omega \equiv \sum (-1)^{\langle [\bt, \bb]; \omega \rangle} p_{([\bt, \bb], \omega)} \eta_{([\bt, \bb]^\omega, \sigma^\omega)}\) modulo terms of lower affine degree. The coefficient of \(\eta_{([\bu, \bd]^\omega, \epsilon^\omega)}\) in \(\omega^{-1} z \omega\) is \(\pm p_{([\bu, \bd], \epsilon)} \neq 0\), which contradicts the \(\prec\)-maximality of \(([\bu, \bd], \epsilon)\) since \(\omega^{-1} z \omega = z\) and proves Claim II.\\
\indent Finally, we have by \cref{L0B} that \(z- p_{([\bu, \bd], \epsilon)} \xi_{([\bu, \bd], \epsilon)} \in W_{d,i}^{\mathfrak{S}_d}\), and \(([\bt, \bb], \sigma) < ([\bu,\bd], \epsilon)\) for any \(([\bt, \bb], \sigma)\) such that \(\eta_{([\bt, \bb], \sigma)}\) arises in the support of \(z- p_{([\bu, \bd], \epsilon)} \xi_{([\bu, \bd], \epsilon)} \). The lemma statement then follows by induction on the dominance order on \(Y_{d,i}\).
\end{proof}

\begin{proposition}
Assume \(\Lambda = (\bL, \bc)\) is a cyclotomic datum for \(\mathcal{A} = (A, A_{\bar 0})\). Then \(\Lambda\) is regular.
\end{proposition}
\begin{proof}
Considering (\ref{kappaexpand}), it follows from \cref{righttoleftthin} and \cref{L0C} that 
\begin{align}
\sum_{\bmu\in \mathcal{M}(i^d,i^d)_{< \ell}} \kappa^{\Lambda,d,i}_{\bmu } \overline{\eta}_{\bmu}=
\underbrace{
\hackcenter{
\begin{tikzpicture}[scale=.8]
  \draw[ultra thick, blue] (0,-0.1)--(0,1.1);
   \draw[thick, fill=black]  (0,0.5) circle (5pt);
     \node[right] at  (0.15,0.5)  {$\scriptstyle \ell$};
     \node[below] at (0,-0.1) {$ \scriptstyle i^{ \scriptstyle (1)}$};
      \node[above] at (0,1.1) {$ i^{ \scriptstyle (1)}$};
      \node[] at  (0.85,0.5)  {$\scriptstyle \cdots$};
      \draw[ultra thick, blue] (1.5,-0.1)--(1.5,1.1);
   \draw[thick, fill=black]  (1.5,0.5) circle (5pt);
     \node[right] at  (1.65,0.5)  {$\scriptstyle \ell$};
     \node[below] at (1.5,-0.1) {$ \scriptstyle i^{ \scriptstyle (1)}$};
      \node[above] at (1.5,1.1) {$ i^{ \scriptstyle (1)}$};
      \node[above] at  (0.75,1.1)  {$\scriptstyle \cdots$};
       \node[below] at  (0.75,-0.1)  {$\scriptstyle \cdots$};
\end{tikzpicture}
}
}_{d \textup{ times}}
\in W_{d,i}^{\mathfrak{S}_d} =  \k\{ \xi_{([\bt, \bb], \sigma)} \mid ([\bt, \bb], \sigma) \in \hat{Y}_{d,i}\},
\end{align}
so in view of (\ref{regcrit}), to show regularity of \(\Lambda\) it suffices to show that 
\begin{align*}
\hackcenter{}
\hackcenter{
\begin{tikzpicture}[scale=.8]
  \draw[ultra thick, blue] (0.75, -0.3)--(0.75, -0.2) .. controls ++(0,0.35) and ++(0,-0.35) .. (0,0.3)--(0,0.5)--(0,0.7) .. controls ++(0,0.35) and ++(0,-0.35) .. (0.75,1.2)--(0.75,1.3);
    \draw[ultra thick, blue] (0.75, -0.3)--(0.75, -0.2) .. controls ++(0,0.35) and ++(0,-0.35) .. (1.5,0.3)--(1.5,0.5)--(1.5,0.7) .. controls ++(0,0.35) and ++(0,-0.35) .. (0.75,1.2)--(0.75,1.3);
  \draw[draw=black, rounded corners, thick, fill=lime] (-0.25,0.3) rectangle ++(2,0.45);
     \node[] at  (0.75,0.5)  {$\scriptstyle \xi_{([\bt, \bb], \sigma)}$};
     \node[below] at (0.75,-0.3) {$ \scriptstyle i^{ \scriptstyle (d)}$};
       \node[above] at (0.75,1.3) {$ \scriptstyle i^{ \scriptstyle (d)}$};
         \node[below] at (0,0.1) {$ \scriptstyle i^{ \scriptstyle (1)}$};
         \node[below] at (1.5,0.1) {$ \scriptstyle i^{ \scriptstyle (1)}$};
      \node[] at  (0.75,0.85)  {$\scriptstyle \cdots$};
      \node[] at  (0.75,0.15)  {$\scriptstyle \cdots$};
\end{tikzpicture}
}
\substack{
\textup{L.\ref{CrossAbsorb}}\\
=\\
\\
}
\sum_{\tau \in \mathscr{D}_{([\bt, \bb],\sigma)}}
\hackcenter{
\begin{tikzpicture}[scale=.8]
  \draw[ultra thick, blue] (0.75, -0.3)--(0.75, -0.2) .. controls ++(0,0.35) and ++(0,-0.35) .. (0,0.3)--(0,0.5)--(0,0.7) .. controls ++(0,0.35) and ++(0,-0.35) .. (0.75,1.2)--(0.75,1.3);
    \draw[ultra thick, blue] (0.75, -0.3)--(0.75, -0.2) .. controls ++(0,0.35) and ++(0,-0.35) .. (1.5,0.3)--(1.5,0.5)--(1.5,0.7) .. controls ++(0,0.35) and ++(0,-0.35) .. (0.75,1.2)--(0.75,1.3);
  \draw[draw=black, rounded corners, thick, fill=lime] (-0.25,0.3) rectangle ++(2,0.45);
     \node[] at  (0.75,0.5)  {$\scriptstyle [\bt, \bb]$};
     \node[below] at (0.75,-0.3) {$ \scriptstyle i^{ \scriptstyle (d)}$};
       \node[above] at (0.75,1.3) {$ \scriptstyle i^{ \scriptstyle (d)}$};
         \node[below] at (0,0.1) {$ \scriptstyle i^{ \scriptstyle (1)}$};
         \node[below] at (1.5,0.1) {$ \scriptstyle i^{ \scriptstyle (1)}$};
      \node[] at  (0.75,0.85)  {$\scriptstyle \cdots$};
      \node[] at  (0.75,0.15)  {$\scriptstyle \cdots$};
\end{tikzpicture}
}
=
[\mathfrak{S}_d:\mathfrak{S}_{([\bt, \bb], \sigma)}]
\hackcenter{
\begin{tikzpicture}[scale=.8]
  \draw[ultra thick, blue] (0.75, -0.3)--(0.75, -0.2) .. controls ++(0,0.35) and ++(0,-0.35) .. (0,0.3)--(0,0.5)--(0,0.7) .. controls ++(0,0.35) and ++(0,-0.35) .. (0.75,1.2)--(0.75,1.3);
    \draw[ultra thick, blue] (0.75, -0.3)--(0.75, -0.2) .. controls ++(0,0.35) and ++(0,-0.35) .. (1.5,0.3)--(1.5,0.5)--(1.5,0.7) .. controls ++(0,0.35) and ++(0,-0.35) .. (0.75,1.2)--(0.75,1.3);
  \draw[draw=black, rounded corners, thick, fill=lime] (-0.25,0.3) rectangle ++(2,0.45);
     \node[] at  (0.75,0.5)  {$\scriptstyle [\bt, \bb]$};
     \node[below] at (0.75,-0.3) {$ \scriptstyle i^{ \scriptstyle (d)}$};
       \node[above] at (0.75,1.3) {$ \scriptstyle i^{ \scriptstyle (d)}$};
         \node[below] at (0,0.1) {$ \scriptstyle i^{ \scriptstyle (1)}$};
         \node[below] at (1.5,0.1) {$ \scriptstyle i^{ \scriptstyle (1)}$};
      \node[] at  (0.75,0.85)  {$\scriptstyle \cdots$};
      \node[] at  (0.75,0.15)  {$\scriptstyle \cdots$};
\end{tikzpicture}
}
=
[\mathfrak{S}_d:\mathfrak{S}_{([\bt, \bb], \sigma)}] \cdot |\mathfrak{S}_{[\bt, \bb]}|
\hackcenter{
\begin{tikzpicture}[scale=.8]
  \draw[ultra thick, blue] (0.75, -0.3)--(0.75,1.3);
  \draw[draw=black, rounded corners, thick, fill=lime] (0.25,0.3) rectangle ++(1,0.45);
     \node[] at  (0.75,0.5)  {$\scriptstyle \eta_{\bmu}$};
     \node[below] at (0.75,-0.3) {$ \scriptstyle i^{ \scriptstyle (d)}$};
       \node[above] at (0.75,1.3) {$ \scriptstyle i^{ \scriptstyle (d)}$};
\end{tikzpicture}
},
\end{align*}
where \(\bmu \in \mathcal{M}(i^{(d)}, i^{(d)})\) is the composition such that 
\begin{align*}
\bmu([m_r,t_r]) = \#\{u \mid [m_u, t_u] = [m_r,t_r]\} 
\end{align*}
for all \(r \in [1,t]\). In essence, we get the last equality above by closing up all the knotholes with identical coupons and dots, which we may do given the assumption that \(\a = A_{\bar 0}\). Now note that
\begin{align*}
[\mathfrak{S}_d:\mathfrak{S}_{([\bt, \bb], \sigma)}] \cdot |\mathfrak{S}_{[\bt, \bb]}| &= [\mathfrak{S}_d: \mathfrak{S}_{[\bt, \bb]}] \cdot [\mathfrak{S}_{[\bt, \bb]}: \mathfrak{S}_{([\bt, \bb], \sigma)}] \cdot |\mathfrak{S}_{[\bt, \bb]}|\\
&= |\mathfrak{S}_d| \cdot [\mathfrak{S}_{[\bt, \bb]}: \mathfrak{S}_{([\bt, \bb], \sigma)}]  = d! [\mathfrak{S}_{[\bt, \bb]}: \mathfrak{S}_{([\bt, \bb], \sigma)}].
\end{align*}
Thus it follows from (\ref{cycrelthick}) that \(\Lambda\) is regular.
\end{proof}

\bibliographystyle{eprintamsplain}
\bibliography{Biblio}

\end{document}